\documentclass{NP-IV}
\usepackage[latin1]{inputenc}
\usepackage{amsmath}
\usepackage{amscd}
\usepackage{mathrsfs}
\usepackage[all]{xy}
\usepackage{graphicx}
\usepackage{pstricks} 
\usepackage{amsfonts}   
\usepackage{pict2e}   
\usepackage{makeidx}
\usepackage[colorlinks=true,linkcolor=red,citecolor=red]{hyperref}

\usepackage{amssymb}
\usepackage{latexsym}
\usepackage{bold-extra}
\usepackage[shortlabels]{enumitem}

\newcommand{\C}{\mathscr{C}}
\newcommand{\chiabs}{\chi^{\mathrm{abs}}}

\newcommand{\chidr}{\chi^{\phantom{1}}_{\mathrm{dR}}}
\newcommand{\Ds}{\mathscr{D}}

\newcommand{\End}{\mathrm{End}}

\newcommand{\Fin}{\mathrm{(Fin)}}
\newcommand{\Fingen}{\mathrm{(\chi^{\mathrm{gen}})}}

\renewcommand{\H}{\mathscr{H}}

\newcommand{\Hdr}{\mathrm{H}_{\mathrm{dR}}}
\newcommand{\hdr}{\mathrm{h}_{\mathrm{dR}}}

\newcommand{\M}{\mathrm{M}}

\newcommand{\N}{\mathrm{N}}

\newcommand{\NE}{E}

\renewcommand{\O}{\mathscr{O}}

\newcommand{\sol}{\mathrm{sol}}

\newcommand{\R}{\mathcal{R}}
\newcommand{\QS}{\textrm{quasi-Stein} }

\newcommand{\SF}{{S_{\Fs}}}

\newcommand{\sing}[1]{\mathrm{Sing}(#1)}

\newcommand{\sm}[1]{\begin{smallmatrix}#1\end{smallmatrix}}

\newcommand{\simto}{\xrightarrow{\sim}}

\DeclareMathOperator\rk{rank}
\DeclareMathOperator\Irr{Irr}

\DeclareMathOperator\Coker{Coker}
\DeclareMathOperator\Ker{Ker}
\DeclareMathOperator\Hom{Hom}

\newcommand\an{\mathrm{an}}
\newcommand\alg{\mathrm{alg}}

\newcommand{\ho}{\ensuremath{\hat{\otimes}}}

\newcommand\wti{\widetilde}

\newcommand{\eps}{\varepsilon}

\newcommand\Es{\mathscr{E}}
\newcommand\Fs{\mathscr{F}}
\newcommand\Gs{\mathscr{G}}
\newcommand\Hs{\mathscr{H}}

\newcommand\Os{\mathscr{O}}

\newcommand\Ls{\mathscr{L}}

\newcommand\Ns{\mathscr{N}}

\newcommand\Gc{\mathcal{G}}
\newcommand\Fc{\mathcal{F}}

\newcommand\Rc{\mathcal{R}}

\newcommand\Fk{\mathfrak{F}}

\newcommand\Yk{\mathfrak{Y}}
\newcommand\Zk{\mathfrak{Z}}

\newcommand\NN{\mathbb{N}}
\newcommand\QQ{\mathbb{Q}}
\newcommand\ERRE{\mathbb{R}}
\newcommand\Z{\mathbb{Z}}
\renewcommand\Im{\mathrm{Im}}
\newcommand{\E}[2]{\ensuremath{\mathbb{A}^{#1,\mathrm{an}}_{#2}}}
\newcommand{\AfK}{\mathbb{A}^{1,\an}_K}

\newcommand\of[3]{\mathopen#1 #2 \mathopen#3}
\newcommand{\type}{\textrm{type}}

\newcommand\wc{{\mkern 2mu\cdot\mkern 2mu}}
\newcommand\va{|\wc|}

\newcommand\wKa{\widehat{K^\alg}}

\newcommand{\REF}{{\color{red}REF}}

\newcommand{\comment}[1]{
\noindent{\red \framebox{\framebox{\framebox{
\begin{minipage}{450pt}#1
\end{minipage}}}}}
}

\newcommand{\smallcomment}[1]{
{\red \framebox{\framebox{\framebox{
#1}}}}
}

\newcommand{\comm}[1]{
\noindent{\magenta \framebox{\framebox{\framebox{
\begin{minipage}{450pt}#1
\end{minipage}}}}}
}

\usepackage{amsthm}
\swapnumbers

\makeatletter
\def\swappedhead#1#2#3{%
  \thmname{#1}\;%
  \thmnumber{\@upn{\the\thm@headfont#2\@ifnotempty{#1}}}%
  \thmnote{\,{\the\thm@notefont(#3)}}{.~}}
\makeatother

\newtheoremstyle{dotless-thm}
  {10pt}
  {10pt}
  {\itshape}
  {}
  {\bfseries}
  {}
  {.0em}
  {}
\theoremstyle{dotless-thm}
\newtheorem{theorem}{\textbf{Theorem}}[subsection]
\newtheorem{def-intro}{\textbf{\textsc{Definition}}}
\newtheorem{thm-intro}{\textbf{\textsc{Theorem}}}
\newtheorem{rk-intro}[thm-intro]{\textbf{\textsc{Remark}}}
\newtheorem{cor-intro}[thm-intro]{\textbf{\textsc{Corollary}}}
\newtheorem{proposition}[theorem]{\textbf{Proposition}}
\newtheorem{lemma}[theorem]{\textbf{Lemma}}
\newtheorem{corollary}[theorem]{\textbf{Corollary}}
\newtheorem{definition}[theorem]{\textbf{Definition}}
\newtheorem{remark}[theorem]{\textbf{Remark}}
\newtheorem{example}[theorem]{\textbf{Example}}

\newtheorem{setting}[theorem]{\textbf{Setting}}

\newtheorem{notation}[theorem]{\textbf{Notation}}
\numberwithin{equation}{section}

\title[Convergence Newton 
polygon V: local index theorems]{
The convergence Newton polygon of a $p$-adic 
differential equation V : local index theorems}

\author{Jérôme Poineau}
\email{jerome.poineau@unicaen.fr}
\address{Laboratoire de Mathématiques Nicolas Oresme, Université de Caen, BP 5186, 14032 Caen, France}

\author{Andrea Pulita}
\email{andrea.pulita@univ-grenoble-alpes.fr}
\address{Univ. Grenoble Alpes, CNRS, IF, 38000 Grenoble, France.}

\date{\today}



\begin{abstract}
In this paper and its sequel 
we consider locally-free $\O_X$-modules 
together with a connection over a quasi-smooth Berkovich curve $X$.
We obtain necessary and sufficient 
conditions for the finite dimensionality of their de Rham 
cohomology over local domains such as disks and annuli. 
We deal with both analytic and 
meromorphic connections and we derive index formulas 
relating the index to the behavior of the radii of 
convergence of their solutions at the boundary of the 
curve $X$.
We introduce the notion of absolute local index. We prove 
that it is an intrinsic notion extending the previous notions 
of Robba's generalized index and that 
of $p$-adic exponents. This condition arises at the 
boundary of the curve $X$ 
and it is an exact condition for the finite 
dimensionality of the de Rham cohomology. 
We derive comparison results between formal, 
meromorphic and analytic de Rham cohomologies.
\end{abstract}

\begin{document}
\maketitle

\begin{center}
Version of \today
\end{center}

\makeatletter
\renewcommand\tableofcontents{%
    \subsection*{\contentsname}%
    \@starttoc{toc}%
    }
\makeatother

\begin{small}
\setcounter{tocdepth}{3} \tableofcontents
\end{small}

\setcounter{section}{0}

\section*{Introduction}
\addcontentsline{toc}{section}{Introduction}

Over the last 60 years, $p$-adic cohomologies for varieties in positive characteristic~$p$ have been 
experiencing an intensive development. Beside the original goal of Grothendieck's program to construct of a Weil cohomology with $p$-adic coefficients,
a new panoply of applications recently emerged 
(see for instance the survey 
\cite{Ked-p-adic-cohomology}). 
Compared to their $\ell$-adic counterpart, one advantage of the 
$p$-adic cohomologies lies in their more constructive nature, which 
allows a large range of explicit examples and 
effective computations (cf. \cite{Kedlaya-practice}). 

The theory that eventually emerged to describe the 
cohomology of a scheme in positive characteristic 
is known today as \emph{rigid cohomology}. 
It was introduced by P.Berthelot 
\cite{Berthelot-rigide, 
Le-Stum-Book, Kedlaya-Weil-II, Tsuzuki-rigide} and it 
unifies the previous cohomology theories by several authors.
\if{ such as J.P.Serre \cite{Serre-Witt-vectors-coh}, B.Dwork \cite{Dw-hyp,Robba-naive-Dwork}, P.Monsky and G.Washnitzer \cite{Monsky-Washnitzer-F1, Monsky-F2,Monsky-F3}, P.Berthelot and A.Ogus \cite{Berthelot-these, Berthelot-Ogus-Book},  L.Illusie \cite{Illusie-de-Rham-Witt-complex} and S.Lubkin \cite{Lubkin-bounded-Witt}. 
}\fi
To have an idea about the panorama we quote \cite{Kedlaya-finiteness, Ked-p-adic-cohomology} and therein introductions and bibliographies, see also \cite{Caro-Tsuzuki, Caro, 
Caro-Abe, Scholze-Bhatt} for more recent evolutions.

Following the original ideas of 
Monsky and Washnitzer (relying of prior ideas of 
B.Dwork), the construction involves associating to a
variety in characteristic~$p$ some appropriate liftings, which are $p$-adic analytic spaces, together with certain coefficients on them. In the key situations, 
these coefficients are nice vector bundles with 
integrable connections (here called simply differential 
equations) satisfying some specific conditions such as \emph{overconvergence}, and equipped with  a \emph{Frobenius structure}. Rigid cohomology is finally computed as the de Rham cohomology of those connections.

%

\if{Notably B.Dwork 
(cf. \cite{Dw-hyp,Robba-naive-Dwork}) who notably introduced the key notion of Frobenius 
structure and overconvergence; P.Monsky and 
G.Washnitzer 
\cite{Monsky-Washnitzer-F1, Monsky-F2,Monsky-F3}; 
and the cristalline cohomology developed by
Berthelot, Ogus et al. (cf. \cite{Berthelot-these, Berthelot-Ogus-Book}); 
Serre's Witt vectors cohomology (cf. \cite{Serre-Witt-vectors-coh}); and more generally the 
de Rham-Witt complex of P.Deligne and L.Illusie 
(cf. \cite{Illusie-de-Rham-Witt-complex}); 
Lubkin's bounded Witt vectors 
(cf. \cite{Lubkin-bounded-Witt}). 
}\fi 
While in algebraic geometry over a field of characteristic~$0$, the de Rham cohomology spaces are 
finite-dimensional under mild assumptions
(cf. \cite{Grothendieck-DR-coh, Monsky-finiteness-alg, Hartshorne-de-Rham} and 
\cite[Proposition 8.4.1]{ Mebkhout-book}), this property often fails in the $p$-adic analytic realm. 
Finding the exact conditions to ensure finite dimensionality is one of the 
central problems of the theory (needed to 
obtain a Weil cohomology theory). In the specific case of 
rigid cohomology, it is the result 
of the contributions by several authors such as 
P.Berthelot, G.Christol, R.Crew, 
E.Grosse-Kl\"onne, Z.Mebkhout, 
P.Monsky, N.Tsuzuki, and it was eventually achieved by K.S.Kedlaya (cf. \cite{Kedlaya-finiteness}).


%

The problem of the finite dimensionality of rigid 
cohomology spaces fits into the larger problem of the 
finite dimensionality for the more general theory 
of \emph{de Rham cohomology of differential 
equations over non-archimedean analytic spaces}, that 
we study here using the language of Berkovich spaces. 
In this extended framework, few results exist and they are 
mainly devoted to the one dimensional case. 
The case of a projective space is known since 
\cite{Adolphson, Balda-Comparaison, 
Balda-Comparison-II, 
Chiarellotto-Comparison-2, Chiarellotto-Comparison, 
Andre-Balda-Book, Andre-Comparison}
and relies on the comparison between the algebraic and 
analytic de Rham cohomologies (cf. 
\cite[Section 1.5]{NP-V} for a general claim in our 
context).

Beyond that, the case of the affine line has aroused much interest and conditional results are known, 
mainly due to P.Robba and E.Pons 
\cite{Clark, Ro-I, Ro-II, RoIII, RoIV, Pons}. In the $p$-adic 
setting, Robba introduced the 
key notion of \emph{generalized index} of a differential 
equation, which is an algebraic condition at the boundary 
of the curve for the finiteness of the global index of the 
differential equation. These boundary conditions follow the general 
philosophy introduced by M.F.Atiyah and I.M.Singer (cf. 
\cite{Atiyah-Singer-announcement,Atiyah-Singer}) 
according to which the global 
finiteness of the index is controlled by some 
quantities of topological nature at the boundary. Robba 
was indeed able to relate the generalized index of a 
differential equation to the 
\emph{behavior of the radii of convergence of its 
solutions at the boundary of the curve}
(cf. \cite{RoIV,Young}). 
However, surprisingly enough, even over a Berkovich 
curve as simple as a disk or an annulus, the precise 
conditions ensuring the finite dimensionality of the de 
Rham cohomology of a general differential equation remained 
essentially unknown. 

In the case of open annuli, the existing results are 
mainly due to G.Christol and Z.Mebkhout in 
\cite{Ch-Me-I, Ch-Me-II, Ch-Me-III, Ch-Me-IV}, 
relying on prior ideas and results of 
P.Robba, as mentioned before, and some 
other authors such as \cite{Crew-Finiteness} for 
instance. Christol and Mebkhout focus 
on rigid cohomology and they quickly restrict 
the range of their study to the Robba ring, 
and to differential modules satisfying a condition on 
the radii of convergence called \emph{solvability} and a 
technical condition about the nature of the exponents 
called \textbf{DNL} (diff\'erence non Liouville), which are 
all satisfied in presence of a Frobenius structure. Moreover, they consider a discretely 
valued base field of mixed characteristic, which excludes 
many possible base fields. 
Although their techniques are quite general, 
these conditions drastically restrict the class of 
differential equations for which their results hold.

In this paper and its sequel \cite{NP-V}, 
we deal with general vector bundles with connections
(with no extra conditions) over quasi-smooth Berkovich curves over 
non-archimedean fields of characteristic zero and 
provide necessary and sufficient conditions for the finite 
dimensionality of their \emph{analytic} and 
\emph{meromorphic} de Rham cohomologies. Recall that such curves are known to admit triangulations, that is to say locally finite subsets of points whose complements are disjoint unions of virtual disks and virtual annuli. This structure result determines our strategy. In this paper, we focus on  differential equations over the basic pieces: disks and annuli (and their generalizations). In this setting, we manage to characterize the finite dimensionality of the de Rham cohomology and prove index formulas (cf. Theorems 
\ref{Thm : index of finite opens} and \ref{Thm : Index meromorphic disk} and Corollary \ref{cor:cohomologypseudoannulus}).

 In~\cite{NP-V}, we will build on those results to investigate differential equations on arbitrary curves. We refer to the introduction of \cite{NP-V} for a 
panorama of the strategy. Roughly speaking, 
we show that the finite dimensionality of the de Rham 
cohomology is controlled by a finiteness property of the 
radii of convergence of the solutions 
at the boundary of the Berkovich curve. 
\if{Indirectly, this relies on a result about the 
local finiteness of the behavior of these radii 
(cf. \cite{NP-I,NP-II,Kedlaya-draft}, see also 
\cite{RoIV, Pons, Ch-Dw,  DV-Balda, 
Balda-Inventiones} for prior results).

\comment{Je ne comprends pas la derni\`ere phrase. On 
n'a rien prouv\'e sur le comportement des rayons au 
bord, si ? Je pense qu'on peut enlever la fin, \`a partir de 
``a result which''.}
\comm{J'ai reformulÃ©. On a montrÃ© la finitude locale. Cela entraine implicitement que les conneries peuvent arriver seulement quand l'on approche la frontiÃ¨re 
ouverte. Sans finitude locale on aurait pu avoir le bordel 
absolu pour les rayons quand on approche la frontiÃ¨re. Plus spÃ©cifiquÃ©ment, dans nos Ã©noncÃ©s et preuve sur la finitude de la cohomologie pour les courbes on se place dans deux contextes: soit on suppose directement que les rayons sont affines Ã  la frontiÃ¨re et on fait cela car on sait implictement que notre resultat de finitude locale entraine que les rayons sont constants en dÃ©hors du squelette d'une couronne au bord de la courbe; soit si les rayons ne sont pas affines on retreci un peu la courbe de sorte que les rayons sont linÃ©aires au bord de la courbe retrecie, ce qui est vrai encore une fois grace aux propriÃ©tÃ©s de finitude locales qu'on a dÃ©montrÃ©... }
\comment{OK, je comprends ce que tu veux dire maintenant. Je trouve cependant que c'est mal plac\'e ici. En plus, ce contr\^ole au bord est vraiment fait dans ce papier. Je propose d'enlever la phrase. Je vais redire \c ca un peu plus loin.}
}\fi

\if{
\comm{
Tu as en mÃÂªme temps raison et tort. Tu as raison parce que je parle d'un voisinage d'un point de Berkovich et je vais reprendre cette partie pour l'introduction de NP-V. 
Tu as tort car les \'equations diff\'erentielles sur les couronnes ouvertes et l'anneau de Robba c'est vraiment 
le truc de Ch-Me. Pratiquement tout Ch-Me 3 et le 9/10 
de Ch-Me 4 sont sur les couronnes et/ou l'anneau de Robba.\\

Dans leur travaux, il n'y a qu'une 
section sur la cohomologie des courbes au sens que j'ai 
indiqu\'e ci plus haut (voisinage d'un point de Berkovich), c'est la section \cite[p.659-662]{Ch-Me-IV}.
Et, entre nous, les \'enonc\'es sont compr\'ehensibles, mais 
les preuves sont par endroit vraiment incompr\'ehensibles 
et manquent des morceaux importants ... j'avais 
demand\'e \`a Gilles Christol de me l'expliquer (je parle 
toujours de \cite[p.659-662]{Ch-Me-IV}), mais il a dit 
qu'il n'avait jamais compris ce paragraphe ecrit par 
Mebkhout... :-).

Plus tard, j'ai invit\'e Mebkhout \`a Grenoble (vers 2016-17 
je crois) et j'ai fini par comprendre que leur preuve 
coincide essentiellement avec la n\^otre. Je peux recuperer 
les infos si besoin, mais je crois me souvenir qu'ils citent 
un th\'eor\`eme sans donner la reference \`a un papier qui 
semble \^etre ce que nous on appelle le th\'eor\`eme de Quing 
Liu pour plonger le truc dans une courbe projective. Mais 
j'ai un souvenir, qu'ils appliquent un r\'esultat d'indice de 
\cite{Ch-Me-I} qui est demontr\'e pour la droite projective, 
mais pas en g\'en\'eral... Ensuite, ils appliquent des 
r\'esultats comme les notres, mais il y a des points 
critiques o\`u des morceaux manquent.\\

Par ailleurs, je me souviens notamment que dans 
certains endroits ils d\'emontrent des r\'esultats sur l'indice 
g\'en\'eralis\'e seulement dans le cas d'une base cyclique, 
alors que ÃÂ§a d\'epend de la base choisie.... par endroit ils 
prennent des polynomes diff\'erentielles \`a coefficients 
polynomiales et ensuite ils appliquent le resultat dans le 
cas g\'en\'eral sans se soucier du fait qu'il fallait g\'en\'eraliser 
leur r\'esultats au cas de syst\`emes diff\'erentielles \`a 
coefficients dans l'anneau de Robba et pas juste des 
polynÃÅ½mes... plein de choses comme ÃÂ§a...  
On peut aussi mentionner que leur resultats sont tous sur 
un corps spheriquement complet (mÃÂªme localemeent 
compact \`a vrai dire). Et toujours sous des conditions 
Liouville.
}

\comment{Je ne critique pas le fait que tu parles de CM, mais surtout la fa\c con dont tu le fais. Le d\'ebut du paragraphe parle de courbe \`a bonne r\'eduction, qui est essentiellement le cas orthogonal \`a celui qu'on traite ici. Idem pour la surconvergence. Je te laisse r\'e\'ecrire le paragraphe si tu veux le garder. Tu connais toutes ces r\'ef\'erences bien mieux que moi.

Pour la deuxi\`eme partie du paragraphe, \c ca ressemble en effet \`a une intro de NP V. Pas s\^ur que \c ca ait sa place ici.}
}\fi

We now describe more precisely the main results of the 
paper.

In Section~\ref{Secn1}, we collect some definitions and basic properties that we will use in the rest of the text~: pseudo-annuli (among which annuli and punctured disks), pseudo-triangulations, radii of convergence, Newton polygon, de Rham cohomology, etc. We take this opportunity to define a notion of local irregularity of a differential equation~$\Fs$, following Robba, for any good germ of segment~$b$ in a curve (in any residual characteristic). Up to some normalization factor, the irregularity~$\Irr_{b}(\Fs)$ is the derivative of the total height of the Newton polygon (cf. Definition \ref{def:generalizedirregularity}). 
We will later prove that it recovers the classical notion by Deligne-Malgrange-Ramis 
 in the case of 
formal power series (cf. 
Proposition \ref{Prop : Irr-form=Irr-x-1}).

Section~\ref{sec:indexpseudo-annuli} is devoted to 
Robba's notion of \emph{generalized index}, which is an 
invariant associated with operators satisfying the 
so-called Fredholm property. As for the irregularity, 
choosing as operator the given connection on~$\Fs$, we 
extend its scope and define it for any germ of 
segment~$b$ at the boundary of a pseudo-annulus~$C$ 
(in any residual characteristic). As we will see in 
\cite{NP-V}, this allows to define the generalized index 
on every good germ of segment of a quasi-smooth 
Berkovich curve. More precisely, we show that Robba's generalized index depends on some choices, and we refine it by introducing the notion of \emph{absolute index}~$\chiabs_{b}(\Fs)$ (cf. Definitions~\ref{def:Fredholmnabla} and~\ref{def:Fredholmnablageneral}), which we prove to be an intrinsic notion (cf. Proposition \ref{Prop : independence on the coordinate chiabs}). 

We then extend the existing results about generalized index and their relationship with irregularity. Let $C$ be a pseudo-annulus. Recall that, when the radii of~$\Fs$ are log-affine along the skeleton of~$C$, we may write $\Fs = \Fs^{\mathrm{Robba}}\oplus
\Fs^{<\mathrm{sol}}$, where $\Fs^{\mathrm{Robba}}$ is the Robba part of~$\Fs$, characterized by the fact that all its radii are equal to~1 on the skeleton of~$C$ (see \cite[Theorem~5.3.1]{NP-III})

\begin{proposition}[cf. Proposition \protect{\ref{Prop : chirel=Irr}}]
Assume that the radii of~$\Fs$ are log-affine along the skeleton of~$C$. Let $b$ be a germ of segment at the boundary of~$C$.
Then $\Fs^{<\mathrm{sol}}$ is Fredholm at~$b$ and its absolute index coincides with the irregularity:
\begin{equation}\label{eq : chiabs sol = Irr spns}
\chiabs_b(\Fs^{<\mathrm{sol}})\;=\;
\Irr_b(\Fs^{<\mathrm{sol}})\;=\;\Irr_b(\Fs)\;.
\end{equation}
\end{proposition}
The proof is largely inspired by the 
methods of Robba (in the rank one case, cf. \cite{RoIV}), and of Christol and Mebkhout (for 
solvable differential modules over the Robba ring, cf. \cite{Ch-Me-III}). 


We conclude the section by stating a result providing 
exact conditions for the finite dimensionality of the de 
Rham cohomology on a pseudo-annulus~$C$ (cf. Theorem \ref{Thm : index of finite opens}). It also 
includes a formula expressing the index for the de Rham 
cohomology in terms of the irregularities at the boundary.

\begin{theorem}[cf. Theorem \protect{\ref{Thm : index of finite opens}}]\label{THM 2}
Assume that the radii of~$\Fs$ are log-affine along the germs of segments~$b_{0}$ and~$b_{1}$ at the boundary of~$C$.
Then $\Fs$ has finite-dimensional de Rham cohomology over~$C$ if, and only if, $\Fs^{\mathrm{Robba}}$ is Fredholm at~$b_{0}$ and~$b_{1}$. 

In this case, assume moreover that $\chiabs_{b_0}(\Fs^{\mathrm{Robba}}) = -\chiabs_{b_1}(\Fs^{\mathrm{Robba}})$. Then, we have 
\begin{equation}
\chidr(C,\Fs)
\;=\;\Irr_{b_0}(\Fs)+\Irr_{b_1}(\Fs)\;.
\end{equation}
\end{theorem}

Using the fact that the radii of convergence are controlled by locally finite data
(cf. \cite{NP-I,NP-II,Kedlaya-draft}, see also 
\cite{RoIV, Pons, Ch-Dw,  DV-Balda, 
Balda-Inventiones} for prior results), and assuming some additional conditions on the exponents of the equation (a non Liouville condition, cf. Appendix~\ref{Liouville condition}, which implies a Fredholm property and the vanishing of some absolute indices), we are able to characterize the finite dimensionality of the de Rham cohomology in terms of the behavior of the radii of convergence.

\begin{corollary}[Corollary \protect{\ref{cor:cohomologypseudoannulusLiouville}}]\label{cor:HCFintro}\label{cor:cohomologypseudoannulusLiouville}
Assume that $\Fs$ is free of 
Liouville numbers along the germs of segments~$b_{0}$ and~$b_{1}$ at the boundary of~$C$.

Then, the following assertions are equivalent:
\begin{enumerate}[a)]
\item for each $i\ge 0$, $\Hdr^i(C,\Fs)$ is finite dimensional;
\item the total height of~$\Fs$ is log-affine along~$b_{0}$ and~$b_{1}$.
\end{enumerate}

Moreover, when these properties hold, we have 
\begin{equation}\label{eq:chidrpseudo-annulus}
\chidr(C,\Fs)
\;=\;\Irr_{b_0}(\Fs)+
\Irr_{b_1}(\Fs)\;.
\end{equation}
\end{corollary}

In Section~\ref{Disk-merom}, we investigate the case of a differential equation over an open disk~$D$ (or a slightly more general space), possibly with meromorphic singularities on a finite set of rigid points~$Z$. To be more precise about the definition of such an object, denote by $\O(D)[\ast Z]$ the ring obtained from~$\O(D)$ by inverting the non-zero functions vanishing at some point in~$Z$. A differential equation~$\Fc$ over~$D$ \emph{meromorphic at~$Z$} is then a locally free $\O(D)[\ast Z]$-module endowed with a connection. The de Rham cohomology spaces computed in this setting will be called \emph{meromorphic} and denoted by $\Hdr^i(D(*Z),\Fc)$ (cf. Section~\ref{Section : mero-alg} for the definitions).

One checks that the radii of convergence of such an~$\Fc$ are log-affine in the direction of any point $z\in Z$ (cf. Lemma \ref{lem:merosinglinear-1}). As a consequence, it is possible to define the irregularity~$\Irr_{z}(\Fc)$ of~$\Fc$ at~$z$. 

We obtain a characterization of the finite dimensionality of the meromophic de Rham cohomology (cf. Theorem \ref{Thm : Index meromorphic disk}). Set $Y := D - Z$ 
and denote by $\Fs$ the analytification of $\Fc$ (that is 
the restriction of $\Fc$ to $\Fs$). Let $\chi_{c}(Y)$ be 
the geometric \'etale Euler characteristic with compact support (which, in this case, is the opposite of the cardinality of $Z\otimes K^\textrm{alg}$ plus 
$1$). 

\begin{theorem}[Theorem 
\protect{\ref{Thm : Index meromorphic disk}}]\label{Thm. 003}
Assume that $\Fc$ is a free $\O(D)[\ast Z]$-module and that its radii of convergence are $\log$-affine 
along the germ of segment~$b$ at the open boundary 
of~$D$. Let $C_b$ be an open 
annulus at the open boundary 
of $D$, not intersecting $Z$. 
Then $\Fc$ has finite-dimensional meromorphic 
cohomology spaces $\Hdr^i(D(*Z),\Fc)$ if, and only if, 
the Robba part $\Fs_{|C_b}^{\mathrm{Robba}}$  
is Fredholm at $b$.

In this case, assume moreover that $\chiabs_{b}(\Fs_{|C_b}^{\mathrm{Robba}})=0$. Then, we have 
\begin{equation}
\label{eq : meromorphic index formula punctured disk-bis}
\chidr(D(*Z),\Fc)\;=\; 
\chi_c(Y)\cdot\mathrm{rank}(\Fc)-
\Irr_{Y}(\Fs)\;.
\end{equation}
\end{theorem}

In the previous statement, the term~$\Irr_{Y}(\Fc)$ is 
defined by means of the irregularities at the germ~$b$ at 
the boundary of~$D$ and at the points of~$Z$ (cf. 
\eqref{eq : global irr}). It is a special case of a more 
general notion of \emph{global irregularity}, that will 
systematically appear in~\cite{NP-V}. Theorem \ref{Thm. 003} is indeed the crucial result serving as a  basis for 
the finite dimensionality results of \cite{NP-V} for general 
quasi-smooth curves.

Let us add one specific word about meromorphic singularities in general. Let $z\in Z$. The completion of the local ring at~$Z$ is isomorphic to a ring of power series over a field, hence the differential equation~$\Fc$ induces a differential equation over a field of Laurent series. The latter is a classical setting where notions of Newton polygons and irregularities have already been defined (cf. \cite{Malgrange-Irreg, VS, Correspondance-Malgrange-Ramis}). While no direct relation seems to exist between the convergence and formal polygons, we prove that their derivatives (in the sense of Section \ref{Derived newton polygon at b0}) coincide (cf. Proposition \ref{Corollary : Irr^F=Irr}). It follows that the irregularities are the same (up to a sign).

An important feature of Theorem~\ref{Thm. 003} is that the only 
condition concerns the behavior of~$\Fc$ at the open 
boundary $b$ of the disk, and \emph{not at the points 
of~$Z$}. In contrast, the finite dimensionality of the \emph{analytic} cohomology 
$\Hdr^i(Y,\Fs)$ of $\Fs$ over $Y$ requires 
conditions at the 
points of~$Z$, as we will prove in \cite{NP-V} for global curves.

The relationship between meromorphic and analytic cohomology is the subject of Section~\ref{App : Some comparison results}. In particular, Theorem \ref{Thm. 003} will be used to obtain a comparison 
theorem between the analytic and meromorphic 
cohomologies for general quasi-smooth Berkovich curves, 
where no finite dimensionality of the 
cohomologies is required (cf. Corollary 
\ref{Coro : Mero = analif Liouville}). 

\begin{theorem}[Corollary \protect{\ref{Coro : Mero = analif Liouville}}]
Let $Y$ be a quasi-smooth $K$-analytic curve and $Z$ 
be a locally finite subset of $K$-rational points of~$Y$. 
Let $\Fc$ be a differential equation on $Y(*Z)$. 
Set $X:=Y-Z$ and $\Fs:=\Fc_{|X}$. 

For each $z\in Z$, let $D_z\subseteq Y$ be 
an open disk centered at $z$. 
We assume that for $z\neq z'$ one has 
$D_z\cap D_{z'}=\emptyset$. 
Denote by $b_{D_z}$ the open boundary of 
$D_z$ and by $b_z$ the 
germ of segment out of 
$z$.

Assume that, for all 
$z\in Z$, one has
\begin{enumerate}
\item $\Fc$ is free as $\O_{D_z}[*z]$-module;
\item if $D_z$ is a connected component of $Y$, then $\Fs$ has log-affine radii at $b_{D_z}$;
\item $\Fs$  
satisfies the assumptions of Theorem \ref{THM 2} on the pseudo-annulus $D_z-\{z\}$. 
\end{enumerate}

Then, for each $i \ge 0$, we have a natural isomorphism
\begin{equation}
\Hdr^i(Y(*Z),\Fc)\; \xrightarrow[]{\sim}\;
\Hdr^i(X,\Fs)\;.
\end{equation}
\end{theorem}

This result generalizes \cite{Balda-Comparaison, Chiarellotto-Comparison}  beyond analytifications of algebraic curves. Moreover, we remove the assumptions about the residual field (which is allowed to have any characteristic)
and we obtain an essential, yet technical, 
improvement regarding the assumptions at the points of $Z$. 
Namely, while in \cite{Balda-Comparaison,
Chiarellotto-Comparison} one finds the classical 
conditions of Clark \cite{Clark} on the formal exponents 
of $\Fc$ at the points of $Z$,\footnote{Which also 
implies finite dimensionality of the cohomologies, by their 
comparison theorem.} 
Corollary \ref{Coro : Mero = analif Liouville} has more 
precise assumptions involving Fredholm conditions 
instead.


The technical heart of Section~\ref{App : Some comparison results} consists of local comparison 
results between the formal, meromorphic and analytic de 
Rham cohomologies of a differential equation around a 
meromorphic singularity, that we deduce from Theorems~\ref{THM 2} and \ref{Thm. 003}.
To state them, let us 
consider an open disk~$D$ centered at~0 and a 
differential equation~$\Fc$ on~$D$ with a unique 
meromorphic singularity at~0. By restriction, it induces 
an analytic differential equation $\Fs := 
\Fc_{\vert D-\{0\}}$. We may associate several local objects to this 
situation (cf. \eqref{eq : setting merom}): 
\begin{itemize}
\item $K((T))$, the field of formal Laurent series;
\item  $K(\{T\})$, the field of meromorphic convergent Laurent series (series in $K((T))$ coming from functions on some disk that are meromorphic at~$0$);
\item  $\mathfrak{R}_0$, the Robba ring at~$0$ (series in $K((T))$ coming from functions on some open disk punctured at~$0$, hence possibly exhibiting an essential singularity at~$0$);
\item $\Os(C')$, the ring of functions of a sub-pseudo-annulus~$C'$ of $D-\{0\}$.
\end{itemize}
Accordingly, we define several differential equations: 
$\M := \Fc\otimes K((T))$ (formal case), 
$\Fc_0^\dag := \Fc\otimes K(\{T\})$ (local 
meromorphic case), 
$\Fs_0^\dag := \Fc\otimes \mathfrak{R}_{0}$ 
(local analytic case) and $\Fs_{|C'}$.
These different localizations of $\Fc$ 
are related by localization maps (cf. Diagram 
\eqref{eq: rings at 0}). 
The goal of Section~\ref{App : Some comparison results} is to compare 
the categories of such equations and their cohomologies. 
More precisely, we provide conditions for the following restriction maps (independently) to be 
isomorphisms: 
\begin{equation}
\Hdr^i(C',\Fs_{|C'})
\;\xleftarrow{\;\quad\;}\;
\Hdr^i(D(*0),\Fc)\;\xrightarrow{\;\quad\;}\;
\Hdr^i(K(\{T\}),\Fc_0^\dag)\;\xrightarrow{\;\quad\;}\;
\Hdr^i(K((T)),\M)
\end{equation}
(cf. Corollaries \ref{Corollary : hyp=> resiso mero-form}, \ref{Cor : comparison mero-formal K(t) and K((t))}, \ref{Comparison Rigid formal-1}, 
\ref{Comparison Rigid formal-2}, 
\ref{Cor : res=iso merodag anal-tre} and Theorem \ref{Theorem : restriction meromorphic}).
If $C'=D'-\{0\}$ for some open sub-disk of $D$, we also deal with the restriction $\Hdr^i(C',\Fs_{|C'})
\xrightarrow{\quad}\Hdr^i(\mathfrak{R}_0^\dag,\Fs_0^\dag)$ (cf. Theorem \ref{Theorem : restriction meromorphic}). Roughly speaking, these restriction maps are 
isomorphisms over the largest open disk $D$ on which 
the radii of $\Fc$ are all $\log$-affine. 
This precision about the disk where the isomorphism 
holds is another improvement of Clark's result 
\cite{Clark}.

An important application of these comparison 
results is the following. 
Recall that, in any of the above categories of differential 
modules, the $\Hdr^0$ of the internal 
$\Hom$ is the set of morphisms in that category. 
Therefore, the above equalities of $\Hdr^0$ can be used 
to prove that the restriction functors are 
\emph{fully faithful} on the pair of objects whose 
internal $\Hom$ satisfy the assumptions needed for the 
above comparison results. More precisely, we show that 
the restriction functors induce an equivalence 
on convenient sub-categories (cf. Corollaries 
\ref{Cor : descent turrittin K(T)},
\ref{Cor : equivalente D(*0)-form},
\ref{Cor : equivalente merodag-form} and
\ref{Cor : equivalente D(*0)-Robba}). 
These equivalences allow in particular to transport 
classification results from a locus to another inside the 
disk $D$. For instance, we prove that one has an 
equivalence of categories between a certain sub-category 
of formal differential modules over $K((T))$ and a 
sub-category of differential equations over the 
Robba ring at the open boundary of $D$
(cf. Corollary \ref{Cor : equivalente D(*0)-Robba}).

As an example, 
in Section \ref{Section : Descent Turritin-Balda} 
we show that the formal
Turrittin-Hukuhara-Levelt decomposition over $K((T))$ 
extends to $D(*0)$ for a certain disk $D$
which is, roughly speaking, the largest disk where the 
radii of convergence of our differential module $\Fc$ 
are $\log$-affine (cf. 
Corollary \ref{Corollary : descent Turrittin - CAN}). 
This improves a similar 
result by F.Baldassarri \cite{Balda-Turritin} 
who proved that the decomposition 
extends from $K((T))$ to $K(\{T\})$. While both proofs 
are based on the principle of using the equality between 
cohomology to deduce the fully faithfulness (which 
is well known, cf. for instance 
\cite[proof of Theorem 2]{GAGA}) 
the proof of Baldassarri is based on the equality  
$\Hdr^0(M,K((T)))=\Hdr^0(\Fc_0^\dag,K(\{T\}))$, 
which holds under Clark's condition on the formal 
exponents at $0$ (cf. \cite{Clark}). 
As we have explained, we have weaker 
assumptions (cf. Lemma \ref{Lemma : formal Liouv imply ch-me Liouv for Hom(F,M)}), 
and a stronger conclusion, which gives us some control on the disk where the decomposition extends, and holds not only over a $p$-adic field, but in arbitrary residual characteristic.

To state another application, recall that N.M.Katz proved that a differential module $\mathrm{N}$ over $K((T))$ extends 
to a differential module $\mathrm{Can}(\mathrm{N})$, called Katz's canonical 
extension, over the affine line punctured at~0. 
It has a meromorphic singularity at $0$ and a regular 
singularity at $\infty$ (cf. \cite{Katz-Can}).
In our setting, we prove that there exists an open 
disk $D'\subset D$ around $0$ on which the 
differential equation $\Fc$ is isomorphic to 
$\mathrm{Can}(\M)$ (cf. Corollary 
\ref{Corollary : descent Turrittin - CAN}). Again, we do 
that without specific assumption on the ground field, with 
general assumptions of Fredholm type generalizing 
those of Clark and with a control on the size of the disk~$D'$.


We also consider decompositions of the modules. Since Katz's canonical extension 
is a functor, the classical decomposition of 
$\M$ by the slopes of its classical formal Newton polygon 
(cf. \cite{VS} for instance) 
extends to a decomposition of $\mathrm{Can}(\M)$.
Recall that we also have a decomposition by 
the radii of convergence of $\Fs_0^\dag$ (cf. \cite{NP-III}). When an isomorphism 
$\Fc_{|D'}\xrightarrow{\;\sim\;}\mathrm{Can}(\M)$ as 
above exists, we compare the decompositions
and show that the convergent one is finer 
with respect to the formal one 
(cf. Corollary \ref{Cor : Formal VS conv deco}).

Appendix \ref{Liouville condition} is devoted to the notion of exponents of a differential equation. Roughly speaking, the exponents are certain 
elements of the ground field that are associated with the 
differential equation at some loci of the curve (singular 
points or germs of segments) and
whose \emph{type} controls the 
finite dimensionality of the absolute and generalized 
indexes. 

This technical notion was introduced by G.Christol and 
Z.Mebkhout for $p$-adic differential modules over an 
open annulus satisfying the Robba property, i.e. assuming 
that the radii of convergence of the solutions at the 
generic points are all maximal (cf. \cite{Ch-Me-II}). In 
analogy with Y.I.Manin's and Clark's theory around $0$ 
(cf. \cite{Man, Clark}), they prove that, if 
the exponent has non-Liouville differences, the 
differential module splits into rank one sub-quotients, 
has finite dimensional de Rham cohomology, and index~$0$. In particular, these assumptions are 
fulfilled by differential modules over the Robba ring with  
a Frobenius structure. 

In order to decrease the level of 
technicality required by Christol-Mebkhout's definition,
B.Dwork proposed an alternative 
(relatively easier and potentially not equivalent) notion 
of exponent which is the one uniformly used in literature 
since then (cf. \cite{Dwork-Exponents}).
More recently K.S.Kedlaya was able to obtain further 
simplifications and to generalize the theory over 
non-trivially valued 
base fields with arbitrary residual fields
\cite{Kedlaya-draft, Kedlaya-Shiho, Kedlaya-book-2}. 


In Appendix \ref{Liouville condition}, without dealing directly with 
the complexity of the problem, we recall the existing 
results and slightly improve them by providing some 
geometric features. We adapt the theory of 
exponents to our context in order to define a 
notion of differential equation \emph{free of Liouville 
numbers}, 
first over pseudo-annuli, then over good germs of 
segments (cf. Definitions \ref{Def : Free of LN over R} 
and \ref{def:NLgerm}).  
This condition implies in particular that the equation 
may be written locally as successive extensions of nice 
differential modules of rank~1, which has important 
consequences for the cohomology.  
Joint with Theorem~\ref{Thm. 003}, it implies Corollary~\ref{cor:HCFintro}.

%
%


\if{In particular, we show how to deal with 
the discrepancy of definitions existing 
between formal exponents at a meromorphic singularity 
and the $p$-adic ones (cf. \ref{}).
}\fi

\subsubsection*{Acknowledgments.}
We thank Yves Andr\'e, Francesco Baldassarri, Gilles 
Christol, Richard Crew, Kiran S. Kedlaya, Adriano 
Marmora, Nicola Mazzari, Zoghman Mebkhout, 
Bertrand To\"en, and Nobuo Tsuzuki for helpful 
discussions. 

\if{

\newpage

In this paper and in its sequel \cite{NP-V}, 
we approach the problem from another angle. 
We deal with general differential equations (with no extra 
conditions) over a quasi-smooth Berkovich curve 
in characteristic $0$ and we provide necessary and 
sufficient conditions for the finite dimensionality of their 
\emph{analytic} and 
\emph{meromorphic} de Rham cohomologies. 
The conditions for the finite dimensionality rely on the 
behavior of two quantities at the boundary of the curve:
\begin{itemize}
\item The exponents of the differential equation;
\item The radii of convergence of the Taylor 
solutions of the differential equation.
\end{itemize}
These are local conditions and they are the object of the 
present paper, while the global approach is dealth in \cite{NP-V}.

The variation of the radii in a neighborhood of the 
boundary is the key condition. In order to deal with the problem related to the discontinuity of the topological spaces in the ultrametric world, the language which has been showing up to be the more adapted is that of Berkovich spaces.
Indeed, it allows an appropriate description of 
these boundary conditions in relation to their 
\emph{continuity} properties and the behavior of their 
\emph{slopes} along the boundary. 
In \cite{NP-I,NP-II}, we have been able to show that the 
behavior of the radii of convergence along the Berkovich 
curve is \emph{continuous} and that it satisfies a strong 
\emph{finiteness condition} 
which will be a crucial for the finiteness of 
the de Rham cohomology.

....

We introduce new techniques allowing to split the 
differential equation at the boundary into a Robba's part 
characterized by the fact that the radii of convergence of 
the solutions are maximal, and a \emph{spectral non solvable} part characterized by   

....

In the Berkovich 
language, the results of P.Robba, G.Christol and 
Z.Mebkhout may be qualified as of local nature around a 
Berkovich point. 
It is tempting to try glue these local finiteness 
results. This was suggested by F.Baldassarri in a 
sequence of talks given between 2010 and 2012. 
The gluing process requires the continuity and the 
finiteness properties of the radii of \cite{NP-I,NP-II}, 
properties that were conjctured by F.Baldassarri in 
\cite{Balda-inventiones}.
Unfortunately, this is still not enough since the range of 
validity of the Christol-Mebkhout-Robba results 
meets that of rigid cohomology and 
it drastically restricts the class of differential equations 
for which they hold. Indeed, there are very few 
differential equations satisfying such specific conditions 
as the overconvergence and the existence of a 
Frobenius structure at every Berkovihc point. 
We then began a program in 2012 aiming to remove the 
assumptions of rigid cohomology and work with 
general ``\emph{local}''
 differential equations (without Frobenius nor 
overconvergence assumptions) and then 
glue them. This required in particular 
a decomposition result \cite{NP-III} which permits, 
in its local form, to split the equation into 
pieces that were locally either coming from rigid 
cohomology or with trivial de Rham cohomology groups.  
We completed that program at around 2014. 
However, there was another key point which made the 
final statement so obtained relatively unsatisfactory. 
While the continuity and finiteness properties of the radii 
hold for every differential equation, the conditions about 
the exponents coming from rigid cohomology 
was needed in the gluing process at every 
open of our covering. In the end the condition on the 
exponents appeared essentially at every point of the 
curve at which the radii were maximal. 
Conditions about the exponents are relatively involved 
and essentially hard to control.
Almost at the same time, we understood how to remove 
the condition about the exponents from the inside the 
curve, but this requires a more direct approach avoiding 
the gluing process suggested by F.Baldassarri. 
We then decided to do not publish our results in 
that form and we preferred to extend our approach 
reproving the results from another angle. 

The strategy we adopt here 
follows the following steps
\begin{itemize}
\item Using the finiteness property of the radii of 
convergence, we show how to 
remove the boundary point from the curve preserving 
the finite dimensionality of the de Rham cohomology.
\item By a paracompactness property of Berkovich 
curves, we may express our curve as an increasing union 
of \emph{finite} sub-curves (cf. Definition \ref{}).
We obtain then a limit formula expressing the 
de Rham cohomology over the 
total curve as the \emph{limit of the de Rham 
cohomologies} over a well chosen countable 
sequence of finite sub-curves 
at which the radii of the differential equation have a 
global finiteness property. Moreover, we prove that the 
dimensions of the cohomology groups stabilize.
\item In this way, we reduce to the situation where our 
curve admits an embedding into a projective curve and 
its complement is a finite disjoint union of disks. 
We obtain then an 
\emph{algebraizability result}, generalizing 
the one in \cite{Ch-Me-IV},\footnote{We recall that 
the algebrizability result in \cite{Ch-Me} is of local 
nature, it concerns the neighborhood of a Berkovich 
point, and it holds only for differential equations 
coming from rigid cohomology.} 
ensuring that our differential equation is the restriction of 
a differential equation defined on the whole 
projective curve having possibly a finite number of  
meromorphic singularities that are placed 
in the complement of our original curve.
\item 
We prove that the extended differential equation 
has a finite dimensional de Rham cohomology over the 
disks and their intersection with the finite curve. 
\item We may express the projective curve as the union of the curve and the disks. This forms
Using the Mayer-Vietoris long exact sequence, we 
may express the de Rham cohomology of the differential 
equation over the whole projective curve 
deduce the 
\end{itemize}

The conditions for the finite dimensionality of the de 
Rham cohomology in the global case, concern the 
behavior at the boundary 
of the curve of certain quantities. The first of them is a 
certain local index, called generalized index in 
\cite{Robba, Ch-Me-III, Ch-Me-IV} of them is  

For a general curve the finite dimensionality 

In this paper, we provide an exact condition for the finite 
dimensionality of the analytic and meromorhic 
de Rham cohomologies of a differential equation over disks and annuli. This

\comm{Reste ÃÂ  faire dans l'introduction :

1) Dire un mot sur le fait que l'hypothÃÂÃÂ¡se NL est vide si 
$p=0$. Dire aussi que elle remplace une hypothÃÂÃÂ¡se 
de ``\emph{cohomologie locale finie}'' ... \\

2) Dire un mot sur le fait que rÃÂÃÂ©lative compacitÃÂÃÂ© entraine 
la finitude cohomologique dans le cas surconvergent et 
qu'elle peut ÃÂÃÂªtre vue comme la raison qui entraine que la 
surconvergence marche ...

Avec des mots diffÃÂÃÂ©rents : la surconvergence peut 
s'encadrer dans la notion plus gÃÂÃÂ©nÃÂÃÂ©rale de rÃÂÃÂ©ltive 
compacitÃÂÃÂ©, qui assure la finitude dimensionnelle de la 
cohomologie de la mÃÂÃÂªme maniÃÂÃÂ¡re ... \\

3) Donner le contre-exemple au fait que finitude des 
graphes ÃÂÃÂ©quivaut ÃÂ  la finitude de la cohomologie ...\\

4) Parler dans l'introduction du Passage ÃÂ  la limite de la 
section 3, parler de la suite $x\mapsto\chi(x,S,\Fs)$ ... 
dire que ÃÂÃÂ§a met ensemble la complexitÃÂÃÂ© de la courbe (i.e.
$\chi_c(U_n)$) et la complexitÃÂÃÂ© de l'ÃÂÃÂ©quation (la somme 
des $\Irr_b(\Fs)$)...\\

}
\comment{Riscrivere gli enunciati !!! Sono cambiati 

Ajouter les references de Baldassarri.

Demander ÃÂ  Y.AndrÃÂÃÂ© de l'aide pour les references ...

...

...

...

....

...

}

\comm{Se rappeler de corriger les references ÃÂ  NP-I et NP-II, la numerotation a changÃÂÃÂ© aprÃÂÃÂ¡s publication.}
In this paper, we are interested in differential equations 
over $p$-adic quasi-smooth Berkovich curves. Under mild 
assumptions on the curves, we characterize the equations 
with \emph{finite-dimensional} de Rham cohomology. 
Moreover, we obtain a \emph{global index formula} 
relating the index of the differential equation to the Euler 
characteristic of the curve \textit{via} a \emph{global 
irregularity}, which we define in terms of the slopes of 
the global radii of convergence of the equation. 
Surprisingly enough these results were unknown even 
in the basic case of differential equations over disks or 
annuli.

The origins of this subject lie in the proof of the rationality of Zeta function by Bernard Dwork \cite{Dwork-zeta}. Although purely analytic at first sight, his proof was later shown to fit into a cohomological setting of de Rham type \cite{Robba-intro-naive-coh-Dwork}, \cite{Dw-hyp}. Since then, several actors have contributed to the development of the cohomological theories of de Rham 
type in the $p$-adic context. We invite the reader to consult the paper 
\cite{Ked-p-adic-cohomology} and related bibliography to 
have a general overview, 
and also the fundamental papers \cite{Adolphson}, \cite{Balda-Comparaison}, and \cite{RoIII}, \cite{RoIV} that are not mentioned there.

As firstly observed by 
Bernard Dwork (cf. \cite{Dw-I}, \cite{Dw-II}), and largely exploited by Philippe Robba (cf. \cite{RoIII}), the major tool 
for the study of $p$-adic differential equations is the radius of convergence of the differential equation. 
One of the crucial contributions of Robba was to relate the index of a $p$-adic differential operator with rational coefficients to the slopes of the radius of convergence by means of a Grothendieck-Ogg-Shafarevich formula 
(cf. \cite{RoIII} and \cite{RoIV}). The program indicated by Robba has been subsequently completed 
by Gilles Christol and Zoghman Mebkhout in the framework of rigid cohomology (cf. \cite{Ch-Me-I}, \cite{Ch-Me-II}, \cite{Ch-Me-III}, \cite{Ch-Me-IV}). 

From the point of view of rigid cohomology, differential equations are a category of coefficients for a good cohomological 
theory of algebraic varieties of characteristic $p>0$. The approach of this paper (and also of our former articles \cite{NP-I}, 
\cite{NP-II}, \cite{NP-III}) is different. We start with a quasi-smooth $K$-analytic Berkovich curve and deal with the de 
Rham cohomology of an \emph{arbitrary} locally free $\O_X$-module~$\Fs$ endowed with a connection $\nabla$. We call 
the couple $(\Fs,\nabla)$ a differential equation on~$X$. 


This category of differential equations is abelian, 
and it covers the class of differential equations 
studied in rigid cohomology. 
Namely the equations coming from rigid 
cohomology are often subjected to some conditions: 
\begin{enumerate}
\item They have \emph{overconvergent coefficients}. This means that $X$ is a compact curve
embedded into a projective curve $\overline{X}$ and that $\Fs$ is defined on an unspecified open neighborhood $U$ of $X$ in $\overline{X}$.

\item There are certain \emph{solvability} conditions that must be satisfied by their radii 
of convergence at the boundary of $X$.  
This means that the radii of convergence of their solutions are all maximal at these 
points, and hence on the whole curve~$X$. The only region where these radii may fail to be maximal is along  $U-X$ 
\textit{i.e.} along the ``\emph{overconvergent boundary}'' of~$X$.
This solvability condition often comes from the existence of a Frobenius action.
\end{enumerate}
We remove  both these assumptions from the picture, 
and we work with a general differential equation $\Fs$ 
on~$X$. 

\medbreak


As mentioned the most basic notion is that of radii of 
convergence (see \cite{Balda-Inventiones}, \cite{NP-I} and \cite{NP-II}). 
Roughly speaking, at every point $x\in X$, 
those radii correspond to the radii of convergence of 
Taylor solutions of the differential equation 
$\Fs$ around~$x$. 
To handle normalization issues, one needs to put some 
additional structure on the curve. Here, it will be a 
pseudo-triangulation~$S$: a locally finite subset 
of~$X$ whose complement in made of simple pieces 
like disks and annuli. (In the good cases, the choice of a 
pseudo-triangulation is not far from that of a semistable 
model as in~\cite{Balda-Inventiones}). We may use it 
to define a skeleton~$\Gamma_{S}$:  a locally finite 
subgraph of~$X$ that is a deformation retract of~$X$ 
in most cases. In this introduction, we leave aside the 
complications arising from the choice of~$S$.

Once the radii are defined, one may construct a geometric object (that is almost trivial 
under the above solvability condition): a controlling graph 
\begin{equation}
\Gamma_S(\Fs)\;\subset\; X
\end{equation}
inside the Berkovich curve~$X$ outside which the radii are locally constant. In \cite{NP-I} and \cite{NP-II}, we proved that 
$\Gamma_S(\Fs)$ is a locally finite connected 
graph and that the radii of~$\Fs$ are continuous 
functions on~$X$.

Recently, Francesco Baldassarri observed a link between 
the finiteness of the controlling graph and the finiteness 
of the de Rham cohomology.
Over the last years, he gave several 
talks on the subject and outlined a natural strategy to compute the cohomology: 
cutting the curve into simpler pieces according to the controlling graph, 
computing their cohomology by Christol-Mebkhout's theorems and 
putting everything back together thanks to a spectral sequence. 
He also explained that index formulas could be derived in this way, and 
that the local irregularities at a Berkovich point were to be interpreted 
as Laplacians of radii of convergence.

Several steps towards the completion of this program have been obtained recently in  
\cite{NP-I}, \cite{NP-II}, \cite{NP-III}, \cite{Kedlaya-draft}. 
One of the major technical difficulties 
was to prove the local finiteness of the 
controlling graph, obtained firstly in \cite{NP-I} and 
\cite{NP-II}, and then reproved in \cite{Kedlaya-draft} 
with a shorter proof based on similar methods.

In this paper we prove a necessary and sufficient 
criterion (which is fulfilled if the controlling graph is 
finite)
for the finiteness of the de Rham cohomology 
for arbitrary curves without boundary (under a mild finiteness 
assumption), and even curves with boundary under a non-solvability 
assumption. 
Moreover, we establish a Grothendieck-Ogg-Shafarevich formula of 
\emph{global nature} (cf. Theorems \ref{THM-INTRO-1}, 
\ref{THM-INTRO-2}, \ref{THM-INTRO-3} below), expressing the 
index of $\Fs$ in terms of the 
Euler characteristic $\chi_{c}(X)$ of $X$
and the \emph{global irregularity} of $\Fs$. The latter involves the slopes of the total height of the convergence Newton 
polygon at the open and closed boundary of $X$.

%
%
%

\medbreak

We now describe more specifically the contents of this paper. As firstly observed by Robba \cite{Ro-I} and Dwork-Robba 
\cite{Dw-Robba}, locally around a point, the differential equation~$\Fs$ splits into a part containing the solvable radii and 
another one containing the smaller ones (cf. \cite{NP-III}). 
In this decomposition, Robba observed (cf. \cite{Ro-I}) that the local de Rham cohomology of the non-solvable part is 
zero: ``\emph{locally, non-solvable differential equations have no cohomology}''.

On the other hand, to ensure finite-dimensionality, we need conditions at the boundary. As an example, the trivial equation 
over a closed disk has infinite dimensional de Rham cohomology. 
Indeed consider the derivation acting on the ring of functions over a closed disk. 
The formal primitives of such functions may fail to converge 
on the closed disk and the derivation actually has infinite dimensional cokernel. 

One solution consists in imposing overconvergence conditions, which more or less amounts to considering 
spaces with no boundary. One may also restrict the 
study to differential equations with no solvable radii at the boundary of $X$.

Now consider the equation~$\Fs$ from a global point of view on the 
Berkovich curve $X$. Even though its radii are not solvable at the boundary of $X$, its 
cohomology is highly non trivial. 
We prove in fact that there are local contributions to the cohomology arising from 
some other points where some of the radii are solvable.

\if{We prove in fact that its global index takes into account the local 
indexes of $\Fs$ at the boundary points of $\Gamma_S(\Fs)$ and at  
the points $x$ of $\Gamma_S$ where some of the radii are 
solvable and have a break at $x$. 
}\fi




\begin{def-intro}
We say that a pseudo-triangulation~$S$ of~$X$ is adapted to~$\Fs$ if the radii of~$\Fs$ are log-linear on the edges of $\Gamma_{S} - S$.
\end{def-intro}
The following result is the basic point on which our criterion is based (see Theorem \ref{THM-INTRO-3}). 
We refer to Theorem \ref{thm:finitetriangulationddc} for a more precise claim:
\begin{thm-intro}[cf. 
Theorem \ref{thm:finitetriangulationddc}]
\label{THM-INTRO-1}
Assume that the residue characteristic of~$K$ is not~2. Let $X$ be a connected quasi-smooth $K$-analytic Berkovich curve and
let $\Fs$ be a differential equation on $X$. Assume that there exists a non-empty finite pseudo-triangulation~$S$ of~$X$ such that
\begin{enumerate}
\item $\Fs$ is free of Liouville numbers along $\Gamma_S$;
\item $S$ is adapted to~$\Fs$;
\item the radii of $\Fs$ are spectral non-solvable at the points 
of the boundary $\partial X$ of $X$.
\end{enumerate}
Then the de Rham cohomology of $\Fs$ is finite-dimensional. 
\end{thm-intro}



Remark that the finiteness of $S$ implies that the curve $X$ has 
finite genus.
The non-Liouville condition is technical. By results of 
Christol and 
Mebkhout (see \cite{Ch-Me-III}, \cite{Ch-Me-IV}),
it ensures finite-dimensionality of the \emph{local cohomology}. 
The strategy of the proof of Theorem~\ref{THM-INTRO-1} basically 
consists in using this fact together with the finiteness of the 
triangulation in order to construct a \emph{finite} open covering 
of $X$ whose elements have finite cohomology. 
The result then follows from Mayer-Vietoris exact sequence.


Christol and Mebkhout's result is actually more precise: they prove a Grothendieck-Ogg-Shafarevich formula for the local index. We are also able to prove a global version of it.

%

\begin{thm-intro}[cf. Corollary \ref{cor:GOS}]\label{THM-INTRO-2}
Under the assumptions of Theorem \ref{THM-INTRO-1}, the index of $\Fs$ is expressed 
by the following formula of Grothendieck-Ogg-Shafarevich type:
\begin{equation}\label{intro GOS}
\chidr(X,\Fs)\;=\;\mathrm{rank}(\Fs)\cdot\chi_{c}(X) - 
\mathrm{Irr}_X(\Fs)\;.
\end{equation}
\end{thm-intro}

When~$K$ is algebraically closed, the quantity $\chi_{c}(X)$ is defined by
\begin{equation}\label{eq : calcul de l'indice -INTRO}
\chi_c(X)\;:=\;2-2g(X)-N(X)\;,
\end{equation}
where $g(X)$ is the genus of $X$ in the sense of \cite{Liu} 
and $N(X)$ is the number of germs of open segments in $X$ 
that are not relatively compact in $X$. We call it the ``\emph{open boundary}" of $X$ and denote it by~$\partial^o X$, as 
opposed to its ``\emph{closed boundary}'' $\partial X$.

The quantity $\mathrm{Irr}_X(\Fs)$ represents the global irregularity 
of $\Fs$. It is given by a sum of local terms. Some of these terms are 
the slopes of the radii at the open and closed boundaries of $X$.
The other local terms are related to the number of segments of 
$\Gamma_S$ that are incident upon the (closed) boundary of $X$ 
(cf. Definition \ref{Def : Global irreg}):
\begin{equation}\label{itro-04}
\mathrm{Irr}_X(\Fs)\;:=\;
\sum_{x\in\partial X} (
dd^c H_{S,r}(x,\Fs) + r \cdot \chi(x,S)) + 
\sum_{b\in \partial^o X}
\partial_bH_{\emptyset,r}(-,\Fs_{|\mathfrak{R}_b})\;,
\end{equation}
where $r$ is the rank of~$\Fs$, $H_{S,r}(x,\Fs)$ is the total height of the convergence Newton polygon, 
$\partial_b H_{\emptyset,r}(-,\Fs_{|\mathfrak{R}_b})$ is the slope along the direction $b$ 
of its restriction to the Robba ring $\mathfrak{R}_b$
and $dd^c H_{S,r}(x,\Fs)$ is its Laplacian at~$x$, \textit{i.e.} the sum of its slopes on all the directions out of~$x$. 
When~$K$ is algebraically closed, $\chi(x,S)$ is defined by
\begin{equation}
\chi(x,S) := 2 - 2g(x) - N_S(x)\;,
\end{equation} 
where~$N_S(x)$ is the number of germs of segments out of~$x$ that belong to~$\Gamma_S$. In general, one 
defines~$\chi(x,S)$ by summing up the contributions of the antecedents of~$x$ over a completed algebraic closure of~$K$.





If $X$ is relatively compact into a larger curve on which $\Fs$ is defined (for instance if~$X$ is compact and~$\Fs$ 
overconvergent), some assumptions are fairly automatic (see Corollary~\ref{Cor : relatively compact :: :: }). 
So we have a large class of equations that fulfill the theorem, namely those having over-convergent coefficients. 


In the general case, it is not possible to find a finite pseudo-triangulation that satisfy the hypotheses of the theorem, for instance if the radii of the equation present infinitely many breaks as one approaches the open boundary of~$X$. We are actually able to characterize the equations giving rise to finite cohomology.


%
%

Let us first give a few definitions. When~$K$ is algebraically closed and~$x$ is a point of type~2 or~3, we set
\begin{equation}
\chi(x,S,\Fs) := r\cdot\chi(x,S)+\sum_{b}\partial_bH_{S,r}(x,\Fs),
\end{equation}
where $b$ runs through the set of directions out of $x$ belonging to $\Gamma_S$. In general, one defines~$\chi(x,S,\Fs)$ by 
summing up the contributions of the antecedents of~$x$ over a completed algebraic closure of~$K$.

For every $y\in \partial X$, we denote by~$\Ds_{y}$ the set of connected components of $X - \{y\}$ that are isomorphic to virtual open disks with boundary~$y$. For such a virtual open disk~$D$, we denote by~$b_{D}$ the germ of segment out of~$y$ represented by~$D$. We set $\Ds_{\partial X} := \bigsqcup_{y\in \partial X} \Ds_{y}$.

An open pseudo-disk is a non-compact boundary-free 
contractible curve of genus~0 (see Definition~\ref{Def : 
Pseudo-disk}). This case can receive a direct treatment 
(see Proposition~\ref{prop:indexpseudodisk}). Remark 
that the projective analytic case reduces to the algebraic 
one, hence it is well-known.

%

\begin{thm-intro}
\label{THM-INTRO-3}
Assume that~$K$ is not trivially valued and that its 
residue characteristic is not~2. Let~$X$ be a connected 
quasi-smooth $K$-analytic Berkovich curve of finite 
genus that is neither a pseudo-disk nor a projective 
curve. Let~$\Fs$ be a differential equation on~$X$. 
Let~$S$ be a minimal triangulation of~$X$ that is 
adapted to~$\Fs$. Assume that
\begin{enumerate}
\item $\Fs$ is free of Liouville numbers along $\Gamma_S$;
\item the radii of $\Fs$ are spectral non-solvable at the points 
of the boundary $\partial X$ of $X$.
\end{enumerate}
Then the de Rham cohomology of~$X$ is finite-dimensional if, and only if, 
\begin{enumerate}
\item $\chi(-,S,\Fs) = 0$ at almost every point of~$S$;
\item for almost every~$D\in\Ds_{\partial X}$, we have $\partial_{b_{D}} H_{\emptyset,r}(-,\Fs_{|D}) = 0$,
\end{enumerate}
where~$r$ is the rank of~$\Fs$.
\end{thm-intro}

We refer to Corollary~\ref{Cor : MAIN COR} for a more precise statement.


The proof is a limit process originally coming from \cite{Ch-Me-III}. We write~$X$ as the union of a countable sequence of subspaces $(X_{n})_{n\ge 0}$ to which we can apply Theorem~\ref{THM-INTRO-1}. Under some extra hypotheses, we prove that, for $i=0,1$, we have
\begin{equation}
\Hdr^i(X,\Fs) \;=\; \varprojlim_n\Hdr^i(X_n,\Fs_{|X_n})\
\end{equation}	
(see Theorem~\ref{Thm . : X_n-->X}). A large part of the work is devoted to constructing a suitable sequence.


%

\medbreak

Along the paper, we also prove other results of independent interest. In particular, we deal with the technical problem of \emph{super-harmonicity} of the partial heights of the convergence Newton polygon. We know, since \cite{NP-I} and \cite{NP-III}, that there is potentially only a locally finite set $\C_{S,r}(\Fs)$ of pathological points in $X$ where super-harmonicity may fail. Here, we are able to prove super-harmonicity at those points too under some technical assumptions. More precisely, we have the following result, whose proof relies on Dwork's dual theory.

\begin{thm-intro}[cf. Theorem \ref{Thm : ddkgfuyg i separates ttthn ejl}]\label{THM-INTRO-4}
Let~$X$ be a quasi-smooth $K$-analytic curve.
Let $\Fs$ be a differential equation of rank $r$ on~$X$.
Let $x\in\C_{S,r}(\Fs)$, let $D_x$ be the virtual closed disk in 
$X-\Gamma_S$ with boundary $x$, and let $V=V_S(x,\Fs)$ be the union of all virtual open disks in $D_x$  
with boundary $x$ on which all the radii are constants. Assume that  
\begin{enumerate}
\item the canonical inclusion 
$\Hdr^0(D_x^\dag,\Fs)\subseteq \Hdr^0(V^\dag,\Fs)$ is an 
equality;
\item the radii of $\Fs$ are compatible with duals over $D_x$;
\item $\Fs$ is free of Liouville numbers at  $x$ (cf. 
Definition \ref{def. eqfree of LN}). 
\end{enumerate}
Then for all $i=1,\ldots,r$ the partial height 
$H_{S,i}(-,\Fs)$ is super-harmonic at $x$. 
\end{thm-intro}
We then deal with a result concerning the descent of the constants.
Indeed in order to apply Christol--Mebkhout's local cohomology results, as in the proof of Theorems~\ref{THM-INTRO-1} 
and~\ref{THM-INTRO-2}, one needs the base field~$K$ to be algebraically closed and maximally complete. To be able to 
obtain results that hold over any field, we had to be able to control extensions of scalars. In this direction, we proved the 
following result.

\begin{thm-intro}[cf. Corollary~\ref{cor:descent}]\label{THM-INTRO-5}
Let~$X$ be a quasi-smooth $K$-analytic curve. Let~$\Fs$ be a differential equation on~$X$. Let~$L$ be a complete non-trivially valued extension of~$K$. Assume that $\Hdr^1(X_{L},\Fs_{L})$ is finite-dimensional over~$L$. Then, we have natural isomorphisms 
\begin{equation}\label{eq:H^0H^1}
\Hdr^0(X,\Fs)\otimes_{K} L \;\xrightarrow{\;\sim\;}\; \Hdr^0(X_{L},\Fs_{L}) \quad\textrm{ and } \quad
\Hdr^1(X,\Fs)\otimes_{K} L \;\xrightarrow{\;\sim\;}\; \Hdr^1(X_{L},\Fs_{L}).
\end{equation}
In particular, $\Hdr^1(X,\Fs)$ is finite-dimensional over~$K$ 

Conversely, assume that~$K$ is not trivially valued and that $\Hdr^1(X,\Fs)$ is finite-dimensional over~$K$. Then, 
equation~\eqref{eq:H^0H^1} holds. In particular, $\Hdr^1(X_{L},\Fs_{L})$ is finite-dimensional over~$L$. 
\end{thm-intro}

\subsection*{Structure of the paper.}

In Section \ref{Secn1}, we provide basic definitions about curves and radii of convergence.

In Section \ref{Measure of the irregularity at a point}, 
we deal with the local cohomology of a differential 
equation at a Berkovich point. 
We generalize several local results to the 
non-solvable case. The finiteness of $\Gamma_S(\Fs)$ is 
systematically employed to reduce to the case of classical rigid 
cohomology over a tube $V_S(x,\Fs)$ canonically attached to $(x,\Fs)$, 
which is roughly the union of all the disks in $X$, with boundary point 
$x$, on which the radii of $\Fs$ are all constant. This section also contains Theorem~\ref{THM-INTRO-4} about super-harmonicity of the partial heights of the convergence Newton polygon.

In Section \ref{section : Global cohomology.}, we use the topological structure of the curve in order to globalize the results of Section~\ref{Measure of the irregularity at a point}. We prove Theorems~\ref{THM-INTRO-1}, \ref{THM-INTRO-2} and ~\ref{THM-INTRO-3} mentioned above.

Finally, Appendix~\ref{section:NF} is devoted to what we call ``\emph{Banachoid spaces}". They are spaces endowed with family of seminorms (which are really part of the data, unlike the case of Fr\'echet spaces). The theory is well-adapted to handle rings of global sections of sheaves on open sets. We spend some time defining and studying completed tensor products in this setting and prove results concerning extensions of scalars, like Theorem~\ref{THM-INTRO-5}.


\if{-----

-----

-----

-----

In this paper we give an exact description of the class of $p$-adic 
differential equations over Berkovich curves with \emph{finite 
dimensional} de Rham cohomology. 
Moreover, we obtain a \emph{global index formula} 
relating the index of the differential equation to its \emph{global 
irregularity}, that we define in term of the slopes of the global radii of 
convergence of the equation as defined in 
\cite{Balda-Inventiones}, \cite{Kedlaya-draft}, \cite{NP-I}, \cite{NP-II}, 
\cite{NP-III}. Surprisingly enough these results was unknown even in 
the basic case of differential equations over disks or annuli.

Major contribution to the finiteness of the de Rham cohomology in this 
context are 
\cite{Adolphson}, \cite{Dw}, \cite{Balda-Comparaison},
\cite{Ch-Me-I}, \cite{Ch-Me-II}, \cite{Ch-Me-III}, \cite{Ch-Me-IV}, 
\cite{Ro-I}, \cite{Ro-II}, \cite{RoIII}, \cite{RoIV}. 
For references concerning rigid cohomology related to the present 
paper we mention \cite{Crew-Finiteness}, and we refer to the introduction of 
\cite{Kedlaya-Finiteness}, and related bibliography. 

As firstly observed by 
Bernard Dwork (cf. \cite{Dw}), and largely exploited by Philippe Robba 
(cf. \cite{RoIII}), the major tool 
for the study of $p$-adic differential equations is the radius of 
convergence of the differential equation. 
One of the crucial contributions of Robba has been to relate 
the index of a $p$-adic differential operator with rational coefficients 
to the slopes of the radius of convergence, as a function, 
by means of a 
Grothendieck-Ogg-Shafarevich formula 
(cf. \cite{RoIII} and \cite{RoIV}). 
The program indicated by Robba 
has been subsequently completed 
by Gilles Christol and Zoghman Mebkhout in the framework of rigid 
cohomology
(cf. \cite{Ch-Me-I}, \cite{Ch-Me-II}, \cite{Ch-Me-III}, \cite{Ch-Me-IV}). 

From the point of view of rigid cohomology, differential 
equations are a category of coefficients for a good cohomological 
theory of an algebraic variety of characteristic $p>0$.

The approach of this paper (and also \cite{NP-I}, \cite{NP-II}, 
\cite{NP-III}) is different.   
We deal with the de Rham cohomology of \emph{any} locally free 
$\O_X$-module $\Fs$ endowed with a connection $\nabla$, where $X$ 
is a quasi-smooth $K$-analytic Berkovich curve. 
We call the couple $(\Fs,\nabla)$ a differential equation over $X$. 

This category of differential equations is abelian, 
and it covers the class of differential equations 
studied in the rigid cohomology. 
Namely the equations coming from rigid 
cohomology are subjected to some conditions: 
\begin{enumerate}
\item They always have \emph{overconvergent coefficients}. 
In the sense of Berkovich this means that $X$ is a compact curve
embedded into a 
projective curve $\overline{X}$, and that $\Fs$ is defined over an 
unspecified open neighborhood $U$ of $X$ in $\overline{X}$.

\item By the fact that they have a Frobenius action, 
they are subjected to certain \emph{solvability} conditions 
about their radii of convergence at the boundary of $X$. This means 
that the radii of convergence of their solutions are all maximal at these 
points, and hence on the whole curve $X$. 
The only region where these radii 
are possibly not maximal is along  $U-X$ i.e. along 
the ``overconvergent boundary'' of $X$.
\end{enumerate}
We remove  both these assumptions from the picture, 
and we work with 
general differential equations over $X$. 

These equations have a further geometrical 
datum which is (almost) trivial for equations coming from rigid 
cohomology: a controlling graph 
\begin{equation}
\Gamma_S(\Fs)\;\subset\; X
\end{equation}
inside the Berkovich curve $X$. Roughly speaking this is defined as the 
locus of points that do not belong to any open disk on which 
the radii of convergence are all constant functions.\footnote{In this 
introduction we remove from the picture the complications arising from 
the choice of the triangulation $S$. We only give the general image.}
In \cite{NP-I} and \cite{NP-II}
we have proved that $\Gamma_S(\Fs)$ is a locally finite connected 
graph such that $X-\Gamma_S(\Fs)$ is a disjoint union of open disks, 
and that the radii of convergence are all 
continuous functions on $X$ having the property that they are  
constant on the connected components of 
$X-\Gamma_S(\Fs)$.

Recently, Francesco Baldassarri established a link between the 
controlling graph and the de Rham cohomology, 
\smallcomment{conjecturing an equivalence between the finite 
dimensionality of the de} \phantom{ }%
\smallcomment{Rham cohomology and the finiteness of the 
controlling graph}. 
Over the last years, he gave several 
talks on the subject and outlined a natural strategy to obtain a proof: 
cutting the curve into simpler pieces according to the controlling graph, 
computing their cohomology by Christol-Mebkhout's theorems and 
putting everything back together thanks to a spectral sequence. 
He also explained that index formulas could be derived in this way, and 
that the local irregularities at a Berkovich point were to be interpreted 
as Laplacians of radii of convergence.

Several steps towards the completion of this program have been obtained recently in  
\cite{NP-I}, \cite{NP-II}, \cite{NP-III}, \cite{Kedlaya-draft}. 
The major technical difficulty was to prove the local finiteness of the 
controlling graph, obtained firstly in \cite{NP-I} and \cite{NP-II}, and 
then reproved in \cite{Kedlaya-draft} with a 
shorter proof based on similar methods.\footnote{We recall that we 
use the finiteness of controlling graph to prove the continuity of the 
radii (cf. \cite{NP-I}, \cite{NP-II}), and that the individual continuity 
of the first radius has been previously obtained by Lucia Di Vizio and 
Baldassarri \cite{DV-Balda}, and completed by Baldassarri 
\cite{Balda-Inventiones}.}
Another point is represented by the local decomposition theorems 
\cite{Dw-Robba}, \cite{Ch-Me-III}, 
\cite{Kedlaya-book}, \cite{Kedlaya-draft}, \cite{NP-III}.
A further important development is the correct 
definition of $\Gamma_S(\Fs)$ 
together with its fundamental properties 
\cite{NP-I}, \cite{NP-III}. Indeed the structure of $\Gamma_S(\Fs)$ 
carries with it several informations about the structure of $X$. 

\comment{
The conjecture of Baldassarri is false in general: there are differential 
equations with infinite skeleton having finite dimensional cohomology.}

In this paper we prove the precise relation between the finiteness of 
the controlling graph and the finiteness of the de Rham cohomology 
for arbitrary curves without boundary (under a mild finiteness 
assumption), and even with boundary under a non-solvability 
assumption. 
Moreover we establish a Grothendieck-Ogg-Shafarevich formula of 
\emph{global nature} (cf. Theorems \ref{THM-INTRO-1}, 
\ref{Intro -Thm.2}, \ref{intro-Thm : reciprocal} below), expressing the 
index of $\Fs$ in terms of a certain  
Euler-Characteristic $\chi(X)$ of $X$ of a \emph{topological 
nature}\footnote{Of course
it coincides with the index of the trivial equation over $X$, under 
appropriate conditions (e.g. if $X$ has no boundary, finite genus).}, 
and the \emph{global irregularity} of $\Fs$. This last 
involves the slopes of the height of the convergence Newton polygon 
along some segments at the boundary of 
$X$ (cf. \eqref{itro-04}), together with some further data related to 
the boundary of $X$.

\comment{POUR JÃÂÃÂ©rÃÂÃÅme : Ci plus haut j'ai ajoutÃÂÃÂ© une phrase en rouge 
sur Baldassarri. On a longuement discutÃÂÃÂ© sur ce pÃÂÃÂ©riode, donc je n'ai 
pas osÃÂÃÂ© le changer. Toutefois je sens le besoin de 
marquer le fait que Baldassarri a conjecturÃÂÃÂ© le faux. 
Ensuite j'explique que nous on corrige sa 
conjecture et on dÃÂÃÂ©montre le bon ÃÂÃÂ©noncÃÂÃÂ©.
Je sais qu'il va rÃÂÃÂ¢ler. Qu'en penses tu ?

Note que dans ce qu'on a ÃÂÃÂ©crit originairement on a ÃÂÃÂ©crit qu'il donne 
une idÃÂÃÂ©e de preuve, mais on a oubliÃÂÃÂ© de dire qu'est-ce qu'il voulait 
dÃÂÃÂ©montrer. Personnellement j'aimerais maintenant dire qu'il a 
conjecturÃÂÃÂ© finitude du graphe ÃÂÃÂ©quivaut ÃÂ  finitude dimensionnelle. 
Mais si tu vois une maniÃÂÃÂ¡re plus gentil de poser la chose, c'est mieux 
...
Je crois quand mÃÂÃÂªme que ÃÂ  ce point il faut dire les choses comme elle 
sont, sans trop supposer que les gens sachent ce qui s'est passÃÂÃÂ© ou ce 
que B. a conjecturÃÂÃÂ©. Dans 30 ans plus personne sera la pour raconter 
cette histoire et seulement les lignes qu'on va ÃÂÃÂ©crire la vont dÃÂÃÂ©finir la 
rÃÂÃÂ©alitÃÂÃÂ© ...}
We now enter more specifically in the context of this paper.

\if{As showed by Robba the variation of the radii (i.e. their slopes) 
along $\Gamma_S(\Fs)$ is highly related to the dimension of the local 
and global de Rham cohomology of $\Fs$. 
}\fi
In the case of differential 
equations coming from rigid cohomology 
the information is all contained 
at the ``overconvergent boundary'' of $X$. 
In the general case we have an extremely richer situation, because the 
graph $\Gamma_S(X)$, containing the information, is much more 
complex.

As firstly observed by Robba \cite{Ro-I} and Dwork-Robba 
\cite{Dw-Robba}, if we are in a neighborhood of 
a Berkovich point $x\in X$, then $\Fs$ splits locally at $x$, 
into a decomposition separating its solvable radii from the smaller 
one. In this decomposition  the 
local de Rham cohomology of the non solvable part  is zero: 
``\emph{non solvable differential equation locally do not have 
cohomology}'' (cf. \cite{Ro-I}). 
\if{This is, in fact, one of the motivations of the solvability assumption 
in rigid cohomology: ``\emph{non solvable differential equation locally 
do not have cohomology}''. }\fi

On the other hand overconvergence is unavoidable because  
solvable equations over a curve with boundary 
usually have infinite dimensional de Rham 
cohomology. As an example 
the derivation acting on the ring of functions 
over a closed disk has infinite dimensional cokernel, 
since the formal primitive of a function fails to converge on the closed 
disk.
But if one restricts  the study to differential equations with no solvable 
radii at the boundary of $X$, 
then overconvergence is unnecessary to have finite dimensionality. 
So we allow the boundary in our setting under a non solvability 
assumption on it. 

Now consider the equation $\Fs$ from a global point of view over the 
Berkovich curve $X$. 
Even though its radii are not solvable at the boundary of $X$, its 
cohomology is highly non trivial.
We prove in fact that its global index takes into account the local 
indexes of $\Fs$ at the boundary points of $\Gamma_S(\Fs)$, and at  
the points $x$ of $\Gamma_S(\Fs)$ where some of the radii are 
solvable and have some break at $x$. 
We are hence induced to give the following 
\begin{def-intro}[cf. Def. 
\ref{Def. Essentially finite graph}]
\label{intro-def ess fin}
We say that $\Gamma_S(\Fs)$ is \emph{essentially 
finite}  if it contains 
only finitely many points $x$ such that at least one of the following 
condition is realized:
\begin{enumerate}
\item some of the radii have a break 
at $x$,
\item $x\in \partial X$,
\item $x$ is a point of positive genus,
\item  is a bifurcation point of $\Gamma_S(X)$.
\end{enumerate} 
\end{def-intro}
We then obtain the following result relating 
the essential finiteness of $\Gamma_S(\Fs)$ 
to the finiteness of the dimension of the de Rham cohomology:
\begin{thm-intro}[cf. 
Thm. \ref{Thm : GLOBAL FINITENESS OF COHOMOLOGY}]
\label{THM-INTRO-1}
Le $X$ be  a quasi-smooth $K$-analytic Berkovich curve, and
let $\Fs$ be a differential equation over $X$. Assume that 
\begin{enumerate}
\item $\Fs$ is free of Liouville numbers along $\Gamma_S$,\smallcomment{dire chi ÃÂÃÂ¡ $\Gamma_S$?}
\item $\Gamma_S(\Fs)$ is essentially finite,
\item The radii of $\Fs$ are spectral non solvable at the points 
of the boundary $\partial X$ of $X$.
\end{enumerate}
Then the de Rham cohomology of $\Fs$ is finite dimensional. 
\end{thm-intro}
The Liouville condition means that the restriction of $\Fs$ to all 
germ of open annuli in $X$ (i.e. Robba ring) 
is free of Liouville numbers in the sense of 
Christol-Mebkhout (cf. \cite{Ch-Me-I}, \cite{Ch-Me-II}, \ldots). 
We prove that it is enough to test the Liouville condition 
on a locally finite family of annuli with skeletons in $\Gamma_S(\Fs)$ 
(cf. Lemma \ref{Lemma : NL is loc fin}).   

We notice that if $\Gamma_S(\Fs)$ is essentially finite, then it is 
\emph{topologically finite} (i.e. $\Gamma_S(\Fs)$ is a finite union of 
intervals), but not necessarily \emph{finite as a graph} (i.e. 
having a finite number of edges). 

This assumption moreover implies that $X$ is a curve with finite genus 
in the sense of Q. Liu \cite{Liu}, so that $X$ is either projective or 
\QS\!\!. The projective case is well known, so we focus on the \QS 
case.

\begin{thm-intro}[cf. 
Thm. \ref{Thm : GOS global}]\label{Intro -Thm.2}
Under the assumptions of Theorem \ref{THM-INTRO-1}, 
if $X$ is not projective, then the index of $\Fs$ is expressed 
by the following formula of Grothendieck-Ogg-Shafarevich type:
\begin{equation}\label{intro GOS}
\chidr(X,\Fs)\;=\;\mathrm{rank}(\Fs)\cdot\chi(X) - 
\mathrm{Irr}_X(\Fs)\;.
\end{equation}
\end{thm-intro}
In this formula we have
\begin{equation}\label{eq : calcul de l'indice -INTRO}
\chi(X)\;:=\;2-2g(X)-N(X)\;.
\end{equation}
where $g(X)$ is the genus of $X$ in the sense of \cite{Liu}, 
and $N(X)$ is a topological invariant of $X$ which roughly 
represents the maximal number of germs of open segments in $X$ 
that are not relatively compact in $X$. 
We call it the ``\emph{open boundary of $X$}'' 
(see section \ref{Global measure of the irregularity for equations 
with finite controlling graphs}).

The quantity $\mathrm{Irr}_X(\Fs)$ represents the global irregularity 
of $\Fs$. It is given by a sum of local terms: part of these terms are 
the slopes of the radii at the open an closed boundaries of $X$, here 
noted by  $\mathrm{seg}(\SF)$, in 
analogy with the index formula of Christol-Mebkhout. 
The other local terms are related to the number of segments of 
$\Gamma_S(\Fs)$ that are incident upon the (closed) boundary of $X$ 
(cf. Definition \ref{Def : Global irreg}):
\begin{equation}\label{itro-04}
\mathrm{Irr}_X(\Fs)\;:=\;
\Bigl(\sum_{x\in\partial X}
\chi(x,\SF)\Bigr)\cdot \mathrm{rank}(\Fs)-
\sum_{b\in \mathrm{seg}(\SF)}
\partial_bH_{\emptyset,r}(-,\Fs_{|\mathfrak{R}_b})\;.
\end{equation}
The proof of Theorem \ref{THM-INTRO-1} goes as follows. 
From the essential finiteness of $\Gamma_S(\Fs)$ 
we construct an \emph{finite} open covering of $X$, 
where the cohomology if finite dimensional, 
then we inductively apply a Mayer-Vietoris lemma to deduce the finite 
dimensionality of the global de Rham cohomology. 

The proof of Theorem \ref{Intro -Thm.2} 
results from an analysis of the 
$\check{\mathrm{C}}$ech resolution of $\Fs$, of such a covering.   
The only terms of the covering that contribute to the index are those 
open subsets that are small neighborhoods of the points having at least 
one of the properties listed in Definition \ref{intro-def ess fin}.\\

If $X$ is compact, or more generally if $X$ is relatively compact into a 
larger curve on which $\Fs$ is defined, the essential finiteness of 
$\Gamma_S(\Fs)$ is automatic (as a consequence of \cite{NP-I}, and 
\cite{NP-II}). In particular this is the case of rigid cohomology, by the 
overconvergence assumption. 
So we have a large class of equations that fulfill the theorem. 

In the general case, there are several differential equations for which 
$\Gamma_S(\Fs)$ is not essentially finite. 
The radii of these equations presents infinitely many breaks as one 
approaches the ``\emph{open boundary}'' of $X$. 
These equations do not admits a \emph{finite} covering on which 
the cohomology is finite, so a straightforward application 
of Mayer-Vietoris Lemma is not possible. 

\comment{However their cohomology can be finite dimensional. 
We obtain the following general statements providing the exact 
condition to have finite dimensionality: in the case of curves with 
topologically finite skeleton Theorem 
\ref{THM-INTRO-1} admits 
Theorem \ref{intro-Thm : reciprocal} as a reciprocal; 
while Theorem 
\ref{THM-INTRO : ???} if the general statement for general 
curves with arbitrary skeleton.}

As a converse of Theorem \ref{THM-INTRO-1} 
we provide criteria to prove that their 
cohomology is actually \emph{infinite dimensional}.

\begin{thm-intro}[cf. Thm. \ref{Thm : reciprocal}]
\label{intro-Thm : reciprocal}
Assume that $X$ is not 
projective, that it has finite genus $g(X)$, and that it admits a weak 
triangulation $S$ whose skeleton $\Gamma_S$ is topologically finite 
(cf. Def. \ref{Def top fin fin as gr}).
Let $\Fs$ be a differential equation free of Liouville numbers along $\Gamma_S$, \smallcomment{rÃÂÃÂ©-ÃÂÃÂ©crire? dire chi est $\Gamma_S$?}
with no solvable radii at the boundary $\partial X$ of $X$.
The following conditions are equivalent:
\begin{enumerate}
\item $\Gamma_S(\Fs)$ is essentially finite;
\item the de Rham cohomology of $\Fs$ is finite dimensional;
\item for all germ of segment $b$ at the open boundary of $X$, the 
radii of $\Fs$ have a finite number of breaks along $b$.
\end{enumerate}
\end{thm-intro}
\comment{Ajouter un thÃÂÃÂ©orÃÂÃÂ¡me gÃÂÃÂ©nÃÂÃÂ©ral ici ?}

The proof is a limit process coming from \cite{Ch-Me-III}.
We approach $X$ by a countable sequence 
$X_1\subseteq X_2\subseteq\cdots \subseteq X$ of affinoid domains in 
$X$, and we prove a limit formula
\begin{equation}
\chidr(X,\Fs)\;=\;\lim_n\chidr(X_n,\Fs_{|X_n})\;.
\end{equation}
More precisely we have 
\begin{equation}
\Hdr^i(X,\Fs)\;=\;
\varprojlim_n\Hdr^i(X_n,\Fs_{|X_n})
\end{equation}	
and the maps $\Hdr^i(X,\Fs)\to\Hdr^i(X_n,\Fs_{|X_n})$ are all 
surjective for all $n$ large enough. 
So the de Rham cohomology of 
$X$ is finite dimensional if and only if the sequence of dimensions of 
$\Hdr^i(X_n,\Fs_{|X_n})$ stabilizes for all $n$ large enough. \\

The proof of these results (even locally around a point) are not 
straightforward consequence of the finite dimensionality 
of rigid cohomology. Indeed essential ingredients are 
the finiteness results of \cite{NP-I}, \cite{NP-II}, that provide the 
existence of affinoid with prescribed properties, and the decomposition 
of \cite{NP-III}. 

\subsection*{Structure of the paper.}
In section \ref{Secn1}, 
we provide basic definitions about curves, 
and radii of convergence.

In section \ref{Measure of the irregularity at a point} 
we deal with the local cohomology of a differential 
equation at a Berkovich point. 
We generalize several local results to the 
non solvable case. The finiteness of $\Gamma_S(\Fs)$ is 
systematically employed to reduce to the case of classical rigid 
cohomology over a tube $V_S(x,\Fs)$ canonically attached to $(x,\Fs)$, 
which is roughly the union of all the disks in $X$, with boundary point 
$x$, on which the radii of $\Fs$ are all constant. 

In section \ref{Measure of the irregularity at a point} 
we also deal with the technical problem of 
\emph{super-harmonicity} of the partial heights of the 
convergence Newton polygon. 
We know since \cite{NP-I} and 
\cite{NP-III}, that there are potentially a 
locally finite set $\C_{S,r}(\Fs)$ of pathological points in $X$ such that 
\emph{super-harmonicity} fails. 
The super-harmonicity at these points is important because 
it improves the \emph{global bound} on the size of $\Gamma_S(\Fs)$, 
obtained in \cite[Cor. 7.2.5]{NP-III} 
(cf. Cor. \ref{COR: BOUND}). We here obtain such a super-harmonicity 
at the pathological points under some technical assumptions. More 
precisely we have the following result, whose central point of the proof 
is Dwork's dual theory :

\begin{thm-intro}[cf. Thm. \ref{Thm : ddkgfuyg i separates ttthn ejl}]
Assume that the residual field of $K$ has characteristic $p>0$.
Let $\Fs$ be a differential equation over $X$ or rank $r$.
Let $x\in\C_{S,r}(\Fs)$, let $D_x$ be the closed disk in 
$X-\Gamma_S$ with boundary $x$, and let $V=V_S(x,\Fs)$. 
Assume that  
\begin{enumerate}
\item the canonical inclusion 
$\Hdr^0(D_x^\dag,\Fs)\subseteq \Hdr^0(V^\dag,\Fs)$ is an 
equality;
\item the radii of $\Fs$ are compatible with duals;
\item $\Fs$ is free of Liouville numbers at  $x$ (cf. 
Def. \ref{def. eqfree of LN}). 
\end{enumerate}
Then for all $i=1,\ldots,r$ the partial height 
$H_{S,i}(-,\Fs)$ is super-harmonic at $x$. 
\end{thm-intro}

In section \ref{section : Global cohomology.} 
we deal with global cohomology. \\
}\fi
\medskip

\if{
Following an original idea of Dwork, if a differential equation has two 
solutions with different radii of convergence, 
then it should correspond to a decomposition of the equation.
Decomposition theorems are the central tool in the classification of 
$p$-adic differential equations.  
As an example they are the first step to 
obtain the $p$-adic local monodromy theorem \cite{An}, 
\cite{Ked}, \cite{Me}. 

The main contributions to the decomposition results 
are due to Robba \cite{Ro-I}, \cite{Robba-Hensel-original},
\cite{Robba-Hensel},
Dwork-Robba \cite{Dw-Robba}, 
Christol-Mebkhout \cite{Ch-Me-III}, \cite{Ch-Me-IV}, Kedlaya 
\cite{Kedlaya-book}, \cite{Kedlaya-draft}. 
We also recall \cite[p. 97-107]{Correspondance-Malgrange-Ramis}, 
\cite{Levelt}, and \cite{Ramis-Devissage-Gevrey} for the 
decomposition by the 
\emph{formal} slopes of a differential equation 
over the field of power series $K((T))$ (cf. Section \ref{Formal differential equations}).

There are very few examples of global nature 
of such decomposition theorems by the radii 
(mainly on annuli or disks), and all with restrictive 
assumptions (as an example see \cite[Ch.12]{Kedlaya-book}, 
\cite[5.4.2]{Kedlaya-draft}). As a matter of fact, a large part of the 
literature is devoted to the following two cases: 
differential equations defined over a germ of punctured disk, or 
over the Robba ring. In the language of Berkovich curves this 
corresponds respectively to a germ of segment out 
of a rational point, or a germ of segment out of a point of type $2$ or 
$3$. Except in those two situations, there is a lack of results.

We present here a general 
decomposition theorem of a global nature in the framework of 
Berkovich smooth curves, that works without any technical assumption 
(solvability, exponents, Frobenius, harmonicity, ...).

A theoretical point of a crucial importance is the definition itself 
of the radii of convergence. 
The former definition of radii relates them to 
the spectral norm of the connection. 
This (partially) fails in Berkovich geometry 
because there are solutions converging more than the natural bound 
prescribed by the spectral norm.
%
Here we deal with a more geometrical definition of the radii due to F. Baldassarri in 
\cite{Balda-Inventiones}, following the ideas of \cite{Dv-Balda}. 
He improves the former definition by introducing 
\emph{over-solvable radii}, 
i.e. radii that are larger than the spectral bound (cf. 
\eqref{eq : spectral - solvable - oversolvable}).  
These radii are not intelligible in terms of spectral norm even if the 
point is of type $2$, $3$, or $4$. 
Moreover he ``\emph{normalizes}'' the usual spectral radii of 
convergence, with respect to a semi-stable formal model of the curve. 
He is also able to prove the continuity of the 
smallest radius. 
In \cite{NP-II} we introduce in that 
picture the notion of \emph{weak triangulation} as a substitute of 
Baldassarri's semi-stable model. 
In fact 
such a semi-stable 
model produces a (non weak) triangulation 
(see \cite{Duc} for instance). 
This permits to generalize the definition to a larger class of curves, and 
it has the advantage for us of being completely within the framework of 
Berkovich curves. 
Exploiting this point of view, we have proved in 
\cite{NP-I} and \cite{NP-II} that there exists a 
locally finite graph outside which the radii are all locally constant, 
as firstly conjectured by Baldassarri.  
This proves that there are 
relatively  few numerical invariants of the equation encoded in the 
radii of convergence. The finiteness theorem have been 
proved in \cite{NP-II} by recreating the notion of \emph{generic disk} 
in the framework of Berkovich curves, as in the very original point of 
view of Bernard Dwork and Philippe Robba. 
The global decomposition theorem presented here
enlarges the picture, it makes evident that Baldassarri's idea for the 
definition of the radii is the good one, and gives to it a more operative 
meaning.\\
\if{
Here we add another result to the picture by providing a global 
decomposition theorem that makes evident that Baldassarri's 
setting is the good one, and gives to it a more operative meaning.
}\fi

\if{
With this in mind in this paper we obtain our global 
decomposition by \emph{augmenting} and then \emph{gluing} a local 
decomposition over the local ring $\O_{X,x}$ by the radii. 
Such a decomposition generalizes that obtained by 
Dwork and Robba in a Berkovich neighborhood of a point of the affine 
line \cite{Dw-Robba}. \\
}\fi

We come now come more specifically into the content of the paper. Let 
$(K,|.|)$ be a complete ultrametric valued field of characteristic $0$. 
Let $X$ be a quasi-smooth $K$-analytic curve, in the sense of 
Berkovich theory\footnote{Quasi-smooth means that 
$\Omega_{X}$ is locally free, see \cite[2.1.8]{Duc}. This 
corresponds to the notion called ``rig-smooth'' in the rigid analytic 
setting.}. We assume without loss of generality that $X$ is connected. 
Let $\Fs$ be a locally free $\O_X$-module of finite rank $r$ 
endowed with an integrable connection $\nabla$. 

In \cite{NP-II} we  
explained how to associate to each point $x\in X$ the so called  
\emph{convergence Newton polygon of $\Fs$ at $x$}. Its 
slopes are the logarithms of the radii of 
convergence $\R_{S,1}(x,\Fs)\leq\cdots\leq\R_{S,r}(x,\Fs)$ of a 
conveniently chosen basis of solutions of $\Fs$ at $x$. 
Here $S$ is a weak triangulation.

Following \cite{NP-I}, we then define for all 
$i\in\{1,\ldots,r=\mathrm{rank}(\Fs)\}$ a 
\emph{locally finite} graph $\Gamma_{S,i}(\Fs)$, as the locus of 
points that do not admit a virtual open disk in $X-S$ 
as a neighborhood on which 
$\R_{S,i}(-,\Fs)$ is constant. 

We say that the index $i$ separates 
the radii (globally over $X$) if for all $x\in X$ one has 
\begin{equation}
\R_{S,i-1}(x,\Fs)\;<\;\R_{S,i}(x,\Fs)\;.
\end{equation}
\begin{thm-intro}[cf. Theorem \ref{MAIN Theorem} and Proposition 
\ref{Prop. : independence on S}]
\label{Intro : Thm : Main}
If the index $i$ separates the radii of $\Fs$, 
then there exists a sub-differential 
equation
$(\Fs_{\geq i},\nabla_{\geq i})\subseteq(\Fs,\nabla)$ of rank 
$r-i+1$ together with an exact sequence
\begin{equation}\label{eq : sequence intro}
0\to\Fs_{\geq i}\to\Fs\to\Fs_{<i}\to 0
\end{equation}
such that for all $x\in X$ one has 
\begin{equation}
\R_{S,j}(x,\Fs)\;=\;\left\{\begin{array}{lcl}
\R_{S,j}(x,\Fs_{<i})&\textrm{ if }&j=1,\ldots,i-1\\
\R_{S,j-i+1}(x,\Fs_{\geq i})&\textrm{ if }&j=i,\ldots,r\;.\\
\end{array} \right.
\end{equation}
Moreover $\Fs_{\geq i}$ is independent on $S$, in the sense that
if $i$ separates the radii of $\Fs$ with respect to another weak 
triangulation $S'$, then the resulting sub-object $\Fs_{\geq i}$ is the 
same.
\end{thm-intro}
In section \ref{An explicit counterexample.} 
we provide an explicit example where \eqref{eq : sequence intro} does 
not split. 
In section \ref{Conditions to have a direct sum decomposition} 
we provide criteria to guarantee that $\Fs_{\geq i}$ 
is a direct summand of $\Fs$. More precisely we have the following
\begin{thm-intro}[cf. Theorems 
\ref{Thm : 5.7 deco good direct siummand} and 
\ref{Thm : criterion self direct sum}]
\label{Intro : Thm : direct sum}
In each one of the following two situations $\Fs_{\geq i}$ is a direct 
summand of $\Fs$:
\begin{enumerate}
\item The index $i$ separates the radii of $\Fs^*$ and of $\Fs$, and  
$(X,S)$ is either different from a virtual open disk with empty 
triangulation, or, if $X$ is a virtual open disk $D$ with empty 
triangulation, there exists a point $x\in D$ such that 
one of $\R_{\emptyset,i-1}(x,\Fs)$ or $\R_{\emptyset,i-1}(x,\Fs^*)$ 
is spectral at $x$.

\item One has $\Bigl(\Gamma_{S,1}(\Fs)\cup\cdots\cup
\Gamma_{S,i-1}(\Fs)\Bigr)\subseteq\Gamma_{S,i}(\Fs)$.
\end{enumerate}
In both cases $(\Fs^*)_{\geq i}$ is isomorphic to 
$(\Fs_{\geq i})^*$, and it is 
also a direct summand of $\Fs^*$.
\end{thm-intro}
In section \ref{An operative description of the controlling graphs} 
we provide conditions to describe the controlling graphs, and in 
particular to fulfill the assumptions of 
Thm. \ref{Intro : Thm : direct sum}. 

Condition ii) of Thm. \ref{Intro : Thm : direct sum} implies i), 
and it has the advantage that it involves 
only the radii of $\Fs$. Nevertheless i) is more natural, and general, 
by the following reason. If one is allowed to choose 
arbitrarily the weak triangulation, then to guarantee the existence of 
$\Fs_{\geq i}$ one has to choose it as small as possible, and 
the same is true for condition i). 
On the other hand to fulfill ii) one is induced to choose it quite large. 
For more precise statements see Remarks 
\ref{Remark : changing tr F_S,I} and \ref{Remark : bad condition}.
%
%

As a corollary we obtain the following:

\begin{cor-intro}[cf. \protect{Cor. \ref{Prop : Crossing points shsh}}]\label{cor-intro : filtration intro}
Let $\Fs$ be a differential equation over $X$. 
There exists a locally finite subset $\mathfrak{F}$ of $X$ such that, if 
$Y$ is a connected component of $X-\mathfrak{F}$, then:
\begin{enumerate}
\item For all $i< j$ one has either 
$\R_{S,i}(y,\Fs)=\R_{S,j}(y,\Fs)$ for all $y\in Y$, or 
$\R_{S,i}(y,\Fs)<\R_{S,j}(y,\Fs)$ for all $y\in Y$. 
\item Let $1=i_1<i_2<\ldots<i_h$ be the indexes separating the 
global radii $\{\R_{S,i}(-,\Fs)\}_i$ over $Y$. Then one has a filtration
\begin{equation}\label{eq: intro: - filtration}
0\;\neq\;(\Fs_{|Y})_{\geq i_h}\;\subset\;
(\Fs_{|Y})_{\geq i_{h-1}}\;\subset\;
\cdots\;\subset\;(\Fs_{|Y})_{\geq i_1}\;=\;\Fs_{|Y}\;,
\end{equation}
such that the rank of $(\Fs_{|Y})_{\geq i_k}$ is $r-i_k+1$ and its 
solutions at each point of $Y$ are the solutions of $\Fs$ with radius 
larger than $\R_{S,i_k}(-,\Fs)$.
\end{enumerate}
\end{cor-intro}
Note that we did not endow $Y$ with a weak triangulation, 
and that the radii of the Corollary 
\ref{cor-intro : filtration intro} are 
those of $\Fs$ viewed as an equation over $X$. 
In section \ref{Crossing points and global decomposition.} 
we also provide conditions making \eqref{eq: intro: - filtration} 
a graduation.

A direct corollary of the above results is the Christol-Mebkhout 
decomposition over the Robba ring \cite{Ch-Me-III} 
(cf. Section \ref{Some interesting particular cases (annuli).}). 
Another corollary is the following 
classification result:

\begin{thm-intro}[cf. Cor. \ref{Cor Ell curves}]
Assume that $X$ is either 
a Tate curve, 
or that $p\neq 2$ and $X$ is an elliptic curve with good reduction. 
Consider a triangulation $S$ of $X$ formed by an individual point. 
Let $\Fs$ be a differential equation over $X$ of rank $r$. Then
\begin{enumerate}
\item For all $i=1,\ldots,r$ the radius $\R_{S,i}(-,\Fs)$ 
is a constant function on $X$, and $\Gamma_{S,i}(\Fs)=\Gamma_S$;
\item One has a direct sum decomposition as 
\begin{equation}
\Fs\;=\;\bigoplus_{0<\rho\leq 1}\Fs^{\rho}\;,
\end{equation}
where $\R_{S,j}(-,\Fs^{\rho})=\rho$ for all $j=1,\ldots,
\mathrm{rank}\,\Fs^{\rho}$.
\end{enumerate}
\end{thm-intro}

The proof of the existence of $\Fs_{\geq i}$ 
(cf. Theorem \ref{MAIN Theorem}) is obtained as follows. 
Firstly, in section \ref{Robba-deco}, we prove a local decomposition
theorem by the spectral radii for differential modules over 
the field $\H(x)$, of a point $x$ of type $2$, $3$, or $4$ of a 
Berkovich curve. This generalizes to curves the classical  
Robba's classical decomposition theorem \cite{Ro-I}, originally proved 
for points of type $2$ or $3$ of the affine line.
Then, in section \ref{Dwork-Robba (local) decomposition}, 
we descends that decomposition to $\O_{X,x}\subseteq \H(x)$ 
(i.e. to a neighborhood of $x$ in $X$). 
This is a generalization to curves of  Dwork-Robba's 
decomposition result \cite{Dw-Robba}, originally proved for points of 
type $2$ or $3$ of the affine line. 
The language of generic disks introduced in \cite{NP-II} 
permits to extend smoothly these proofs to curves, 
up to minor implementations.

Such local decompositions are also present in \cite{Kedlaya-draft}, 
where one makes everywhere a systematic use of the spectral norm 
of the connection, as in \cite{Kedlaya-book}. 
Methods involving spectral norms work thank to the Hensel 
factorization theorem of \cite{Robba-Hensel}, 
and \cite[Lemme 1.4]{Ch-Dw}. 
We presents here the same local decomposition results as a 
consequence of the original, and more geometric, techniques of 
\cite{Ro-I} and \cite{Dw-Robba}.

Robba's and Dwork-Robba's local decompositions take in account 
only spectral radii, because over-solvable radii are not invariant by 
localization. 
For this reason Dwork-Robba's local decompositions 
(at the points of type $2$, $3$, and $4$) do not glue, 
and they do not give the global decomposition.
We then \emph{augment} the Dwork-Robba decomposition 
of the stalk $\Fs_x$ by taking in account the decomposition 
of the trivial submodule of $\Fs_x$ 
coming from the existence of solutions converging in a disk 
containing $x$, i.e. taking in account over-solvable radii. 
This augmented decomposition of $\Fs_x$ 
glues without any obstructions, 
and it provides the existence of $\Fs_{\geq i}$. 

In Theorem \ref{Intro : Thm : Main} we actually 
use the continuity of the radii $\R_{S,i}(-,\Fs)$ (cf. proof of Proposition 
\ref{Prop : extended by continuity}). 
Implicitly we also use the 
local finiteness of $\Gamma_{S,i}(\Fs)$, 
since for $i\geq 2$ the continuity is an indirect 
consequence of the finiteness (cf. \cite{NP-I} and \cite{NP-II}).\\

In the second part of the paper, we provide an operative description of
the controlling graphs, that is essential, for instance, to fulfill the
assumptions of the above theorems.
 
As a consequence we obtain, in section
\ref{section : explicit bound},
a bound on the number of edges of $\Gamma_{S,i}(\Fs)$.
We prove that this number is controlled by the knowledge of the slopes
of the radii at a certain \emph{locally finite} family of points, that are
roughly speaking those where the super-harmonicty fails.

If $X$ is a smooth geometrically
connected projective curve, then we find \emph{unconditional} 
bounds, in the sense that they depend only on the geometry 
of the curve and the rank of $\Fs$, not on the equation itself.

In particular, when $i=1$, we prove the following neat result:
\begin{thm-intro}[cf. Cor. \ref{cor:RS1edge}]\label{Thm-intro: key point}
Let $X$ be a smooth geometrically connected projective curve
of genus~$g\ge 1$. Let~$\NE_S$ be the number of edges
of the skeleton $\Gamma_{S}$ of the weak triangulation.
Then the number of edges of $\Gamma_{S,1}(\Fs)$ is at most
\begin{equation}
\NE_S + 4r(g-1)\;.
\end{equation}
\end{thm-intro}
Similar bounds are derived if $i\geq 2$.
The key point in the proof of Theorem \ref{Thm-intro: key point} is
the control of the \emph{locus of failure}
of the super-harmonicity property of
the partial heights of the convergence Newton polygon, as in
\cite{NP-I}.
In particular, in section \ref{Weak super-harmonicity of partial heights},
we use the local part of the decomposition theorem to reprove a formula of \cite[Thm. 5.3.6]{Kedlaya-book} describing
the Laplacian of the partial heights of the convergence
Newton polygon (cf. Thm. \ref{Prop : Laplacian}).
We extend it by taking into account solvable and
over-solvable radii.

We are also able to relate the failure of the locus of super-harmonicity 
to the existence of Liouville numbers. We obtain this by interpreting 
the index theorems of Christol Mebkhout \cite{Ch-Me-IV}, as a 
control about the default of harmonicity. 

In that picture, the solvable part of the equation at $x$ coincides with 
an overconvergent isocristal over a certain neighborhood of $x$. 
And its over convergent cohomology results related to the slopes of 
the radii along the germs of segments out of $x$, as in very original 
spirit of \cite{RoIV}.




\if{
Other results of this paper: 
\begin{itemize}
\item In section \ref{An operative description of the controlling graphs} 
we provide an operative description of 
the  controlling graphs. In particular we give explicit conditions to fulfill 
point ii) of Theorem \ref{Intro : Thm : direct sum}.

\item In section \ref{Weak super-harmonicity of partial heights}
we use the local part of the decomposition theorem to give another 
proof of a formula of \cite[Thm. 5.3.6]{Kedlaya-book} describing 
the Laplacian of the partial heights of the convergence 
Newton polygon (cf. Thm. \ref{Prop : Laplacian}). 
Our formula is a generalization since it takes in account solvable and 
over-solvable radii. 
This is a key point for Theorem \ref{Thm-intro: key point} below.

\item In section 
\ref{section : explicit bound}
we obtain a bound on the number of edges of $\Gamma_{S,i}(\Fs)$. 
This is controlled by the knowledge of the slopes of the radii at a 
locally finite family of points. 
In particular if $i=1$, and if $X$ is a smooth geometrically 
connected projective curve, then we have the following \emph{unconditional} 
bound, which is independent on the equation, in terms 
of the geometry of the curve and the rank of $\Fs$:
\end{itemize}
\begin{thm-intro}[cf. Cor. \ref{cor:RS1edge}]\label{Thm-intro: key point}
Let $X$ be a smooth geometrically connected projective curve 
of genus~$g\ge 1$. Let~$\NE$ be the number of edges 
of the skeleton $\Gamma_{S}$ of the weak triangulation. 
Then the number of edges of $\Gamma_{S,1}(\Fs)$ is at most 
\begin{equation}
\NE + 4r(g-1)\;.
\end{equation}
More precisely, the graph $\Gamma_{S,1}(\Fs)$ is build by attaching at 
most $r(2g-2)$ trees to~$\Gamma_{S}$, with at most $r(2g-2)$ edges 
in total.
\end{thm-intro}
\begin{itemize}

\item In section \ref{An explicit counterexample.} 
we give explicit examples of the incompatibility 
of over-solvable radii to exact sequence and duals.

\item We discuss the link with Grothendieck-Ogg-Shafarevich formula 
in the setting of Christol-Mebkhout, for solvable modules over a dagger disk.
We showing in \eqref{GOS-changed} that, the intrinsic radii describe directly 
$h^1(\M)$, while usual Christol-Mebkhout radii describe an index 
$h^0(\M)-h^1(\M)$.  

\item In the appendix \ref{Comparison with the general definition of radii} 
we discuss a more general definition of the radii 
and its relation with possibly non semi-steble models of $X$. 
We essentially show that the framework of weak triangulation is the 
more convenient one.
\end{itemize}
}\fi

\subsubsection*{NOTE.} 
This is a first draft, containing a maximum number of details. 
We plan to reduce its volume in a next version. 
We shall also improve the bounds of section 
\ref{section : explicit bound}, 
with further developments.


\subsection*{Structure of the paper.}
In Section \ref{Secn1} we recall some notations and basic results. 
In Section \ref{Radii and filtrations by the radii} 
we define the radii and their graphs together with 
their elementary properties.
In Section \ref{Robba-deco} we give Robba's local decomposition 
by the spectral radii over $\H(x)$, for a point of type $2$, $3$, or $4$. 
In Section \ref{Dwork-Robba (local) decomposition} 
we descend that decomposition to the local ring 
$\O_{X,x}\subseteq\H(x)$ following Dwork-Robba's original 
techniques.
In Section \ref{Proof of Main Th} we obtain the global decomposition 
Theorem \ref{MAIN Theorem}, and the criteria 
to have a direct sum decomposition (cf. Theorems 
\ref{Thm : 5.7 deco good direct siummand} and 
\ref{Thm : criterion self direct sum}).
In Section \ref{An operative description of the controlling graphs} 
we provide an operative description of the 
graphs $\Gamma_{S,i}(\Fs)$, together with the control of the failure 
of super-harmonicity using Christol-Mebkhout index theorems. 
In section \ref{section : explicit bound}
 we obtain the bound on the number of their edges, and 
the classification results for the elliptic curves (cf. Corollary 
\ref{Cor Ell curves}).
In Section \ref{An explicit counterexample.} we provide explicit 
counterexamples of the basic pathologies of over-solvable radii 
(incompatibility with duality, and exact sequence, a link between 
super-harmonicity property and presence of Liouville numbers, by 
means of the Grothendieck-Ogg-Shafarevich formula). 
In Appendix \ref{Comparison with the general definition of radii} we
discuss the definition of the radius.

}\fi

\subsubsection*{Acknowledgments.}
We thank Yves AndrÃÂÃÂ©, Francesco Baldassarri, Gilles 
Christol, Richard Crew, Kiran S. Kedlaya, Adriano 
Marmora, Nicola Mazzari, Zoghman Mebkhout, Bertrand 
ToÃÂÃÂ«n, and Nobuo Tsuzuki for helpful discussions.

\comm{Dans l'introduction j'aimerais expliquer plus prÃÂÃÂ©cisement les choses suivantes :

1) Si l'ÃÂÃÂ©quation $\Fs$ est finiment contrÃÂÃÅlÃÂÃÂ©e, alors 
$g(X)$ et $N(X)$ sont finis. Du coup on peut considerer 
$\chi_c(X)=2-2g(X)-N(X)$.\\

Par contre ce qu'on fait dans la section aprÃÂÃÂ¡s c'est de 
donner une condition nÃÂÃÂ©cÃÂÃÂ©ssaire et suffisante pour la 
finitude de la cohomologie, quand $g(X)$ est fini, mais 
$N(X)$ n'est pas forcement fini.\\

Dire que pour ces courbes (genre fini, mais squelette 
pas forcement fini) on a quand mÃÂÃÂªme une formule 
d'indice de la maniÃÂÃÂ¡re suivante : il existe un certain 
ouvert connexe de $X$ (qu'on peut calculer 
explicitement si l'on connait 
la fonction hauteur totale) tel que $g(U)=g(X)$ et 
$N(U)<+\infty$, tel que $\chi(U,\Fs)=\chi(X,\Fs)$.
(Notamment dans la preuve $U$ contient tous les points 
"critiques", et son squelette est contenu dans celui de 
$X$. Comme $U$ est ouvert, il contient toute l'ÃÂÃÂ©toile 
des germes de segments sortants de chacun de ses 
points, et cela devrait montrer que en prenant 
l'intersection avec $U$ on ne modifie pas les quantitÃÂÃÂ©s 
$\chi(S,x,\Fs)$).\\

J'aimerais ensuite dire juste un mot sur le fait que le 
passage ÃÂ  la limite pour les courbes de genre infini ne 
marche pas en gÃÂÃÂ©nÃÂÃÂ©ral essentiellement ÃÂ  cause de 
certaines conditions de densitÃÂÃÂ© ...\\

2) J'aimerais prÃÂÃÂ©ciser que en rang un on a ÃÂÃÂ©quivalence 
entre graphe contrÃÂÃÅlant fini et cohomologie finie. Mais 
qu'en rang supÃÂÃÂ©rieur ce n'est plus vrai.\\

3) J'aimerais changer NP-III en changeant tout de 
triangulation faible ÃÂ  pseudo-triangulation.

Et ensuite mieux dÃÂÃÂ©finir ce que c'est qu'un graphe ... \\

4) Dans la section des ÃÂÃÂ©quations mÃÂÃÂ©romorphes, je dois 
amÃÂÃÂ©liorer la remarque \ref{Rk : Irr-form=Irr-x} en 
ajoutant la preuve du fait que tout le polygone formel 
coincide avec le polygone derivÃÂÃÂ© du polygÃÂÃÅne de 
convergence. Pas seulement l'irrÃÂÃÂ©gularitÃÂÃÂ©. La preuve 
coincide avec celle de NP-III ... on decompose par les 
pentes formelles et par les pentes du rayon en $T=0$, 
et on constante qu'elles sont les mÃÂÃÂªmes quand le rayon 
est petit, et quand c'est pas petit on fait push-forward 
... comme d'hab ... pour la partie de Robba on dÃÂÃÂ©duit 
par contrapposÃÂÃÂ©e. C'est ÃÂ  dire que d'abord on montre 
que si une des deux pentes (celle formelle ou celle du 
polygone des converges) est positive, alors l'autre l'est 
aussi et elles coincident. Donc si un module a une des 
deux pentes qui est nulle au voisinage de  $|T|=0$, il 
ne peut pas avoir l'autre pente positive... \\

5) Parler de l'expression de l'indice comme un integral,  
et expliquer les hypothÃÂÃÂ¡ses techniques comme dans l'e-
mail de Kedlaya ...}

}\fi

\section{Definitions and notations}\label{Secn1}

Here and for the rest of the text, we fix an ultrametric complete valued field $(K,\va)$ of characteristic~0. We denote by~$\wti K$ its residue field and by~$p$ be the characteristic exponent of the latter (either~1 or a prime number). We fix an algebraic closure~$K^\alg$ of~$K$. The absolute value~$\va$ on~$K$ extends uniquely to it and we denote it identically. We denote by $(\wKa,\va)$ its completion.

We also 
set  $\omega:=\liminf_n|n!|^{1/n}$ 
(the radius of convergence of the exponential series). One has
\begin{equation}\label{eq : OMEGA}
\omega\;=\;\left\{
\begin{array}{ll}
1&\textrm{ if the valuation of $K$ is trivial on 
$\mathbb{Q}$}, \\
|p|^{\frac{1}{p-1}}&\textrm{ if the valuation of $K$ is 
$p$-adic on $\mathbb{Q}$.}
\end{array}
\right.
\end{equation}

In the whole paper, we will use the definitions and 
notation from~\cite{NP-IV}. 
In particular for the terminologies concerning Berkovich 
curves and differential equations. 

From now on, we fix a quasi-smooth $K$-analytic curve~$X$ and a 
differential equation $(\Fs,\nabla)$ over~$X$, that is a 
coherent $\O_X$ module $\Fs$ together with an 
(integrable) connection $\nabla:
\Fs\to\Fs\otimes\Omega^1_X$. 
By \cite[Proposition 1.0.2]{NP-III}, $\Fs$ is automatically a locally free 
$\O_X$-module of finite rank.

\subsection{Radii of convergence}

We recall in this section the main definitions about the 
radii of convergence of $\Fs$.

Recall that a non-empty connected $K$-analytic space is called a \emph{virtual open disk} (resp. \emph{annulus}) if it becomes isomorphic to a disjoint 
union of open disks over~$\wKa$ (resp. if it becomes a 
disjoint union of open annuli whose orientations are 
preserved by $Gal(\wKa/K)$, 
cf. \cite[3.6.32 and 3.6.35]{Duc}).
Following~\cite[5.1.8]{Duc}, we may now define the analytic skeleton of an analytic curve. 

\begin{definition}[Analytic skeleton]\label{def:analyticskeleton}
We call \emph{analytic skeleton} of an analytic curve 
$X$ the set of points that have no neighborhoods 
isomorphic to a virtual open disk. We usually denote it by $\Gamma_X$.
\end{definition}

Notice that the analytic skeleton~$\Gamma_{D}$ of a 
virtual open disk~$D$ is empty.

\begin{definition}[Open pseudo-annulus]
\label{Def : Pseudo-annulus}
Assume that~$K$ is algebraically closed. We say that a connected quasi-smooth $K$-analytic curve~$C$ is an 
\emph{open pseudo-annulus} if
\begin{enumerate}
\item it has no boundary;
\item it contains no points of positive genus;
\item its analytic skeleton~$\Gamma_{C}$ is an open segment.
\end{enumerate}  
We call~$\Gamma_{C}$ the skeleton of~$C$.

If~$K$ is arbitrary, we say that a connected quasi-smooth $K$-analytic curve~$C$ is an 
\emph{open pseudo-annulus} if $C \ho_{K} \widehat{K^\alg}$ is a disjoint union of open pseudo-annuli and if Gal($K^\alg/K$) preserves the orientation of their skeletons. In this case, one can check that~$C$ has no boundary, contains no points of positive genus and that its analytic skeleton~$\Gamma_{C}$ is an open segment. We call $\Gamma_{C}$ the skeleton of~$C$.
\end{definition}

\begin{remark}
Obviously, open annuli and virtual open annuli are 
open pseudo-annuli. Moreover, if $K$ is spherically 
complete, it can be shown that an open pseudo-annulus 
is always an increasing union of open virtual-annuli 
\cite[Lemme 1.1.34]{NP-IV}. 
\label{Rk : pseudo-annuli}
We also observe that, by \cite[Prop. 3.2]{Liu}, 
if~$K$ is non-trivially valued, 
algebraically closed and maximally complete, then any 
open pseudo-annulus~$C$ may be embedded into the 
affine line. We deduce that, in this case, either~$C$ is an 
open annulus or it may be written in the form 
$C=Y-\{y\}$, where $Y$ is either the affine line or an 
open disk and $y\in Y$ is a rational point.
\end{remark}

\begin{definition}[Pseudo-triangulation]
A pseudo-triangulation of~$X$ is a locally finite subset $S \subset X$, formed by points of type 2 or 3, such that every connected component of $X - S$ is a virtual open disk or an open pseudo-annulus. 

The skeleton $\Gamma_{S}$ of a pseudo-triangulation~$S$ is the union of~$S$ with the skeletons of the connected components of $X-S$ that are open pseudo-annuli.
\end{definition}

\begin{remark}
It follow from the definition that $X-\Gamma_S$ is a 
disjoint union of virtual open annuli.
\end{remark}

From now on, we fix a pseudo-triangulation~$S$ 
on~$X$. 
\if{Several notions  and results of this paper 
will be of global 
nature. However, we will mostly be 
interested in the case where~$X$ is an open pseudo-
annulus endowed with the empty pseudo-triangulation 
(cf. Section~\ref{sec:indexpseudo-annuli}) and in the 
case where~$X$ is a punctured open disk (hence an open pseudo-annulus again) endowed with the empty pseudo-triangulation (cf. Section~\ref{Disk-merom}). 
}\fi

The pseudo-triangulation~$S$ on~$X$ is the datum we will use to normalize the radii of convergence of~$(\Fs,\nabla)$ at the points of~$X$. Let us recall quickly how to do it.

Let~$x \in X$ and let~$r$ be the rank of~$\Fs$ at~$x$. 
Let~$L$ be a complete algebraically closed valued field 
containing isometrically the field of the point $\Hs(x)$ (cf. \cite{Ber}). 
In this case, there 
exists an $L$-rational point~$x'$ of~$X_{L}$ whose image in $X$ is $x$. 
By \cite{NP-II} the pseudo-triangulation~$S$ of~$X$ 
induces canonically a pseudo-triangulation~$S_{L}$ 
of~$X_{L}$. Since $x'$ is $L$-rational, it is not contained 
in $\Gamma_{S_{L}}$ and we call~$D$ the connected 
component of~$X_{L} - \Gamma_{S_{L}}$ 
containing~$x'$. It is an open disk and we identify it to 
some standard open disk~$D(0,R)^-=\{|T|<R\}\subset 
\mathbb{A}^{1,\mathrm{an}}_L$ by an unspecified 
isomorphism sending~$x'$ to~$0$. Denote by $r(x)$ 
the radius of the point $x$ with respect to the coordinate $T$ (cf. \cite[p.78]{Ber}), 
it can be shown that $r(x)$ is the 
radius of the largest open disk $D(0,r(x))^-$ whose 
image in $X$ via the above morphisms $D(0,r(x))^-\to D\subset X_L\to X$ is reduced to the 
individual point $\{x\}$. 

The pull-back $(\Fs_{L},\nabla_{L})$ of $(\Fs,\nabla)$ 
to~$X_{L}$ is still a module with connection of 
rank~$r$. For $i\in \{1,\dotsc,r\}$, denote by~$\Rc_{i}'\leq R$ 
the radius of the largest open sub-disk of $D^-(0,R)$
centered at~$0$ on which~$(\Fs_L,\nabla_L)$ admits at 
least $r-i+1$ linearly independent solutions. 
By the Cauchy existence theorem of solutions, $\Rc_i'$ is 
strictly positive for all $i$.

\begin{definition}
For $i\in \{1,\dotsc,r\}$, the $i^\textrm{th}$ radius of convergence of~$(\Fs,\nabla)$ at~$x$ is
\begin{equation}
\Rc_{S,i}(x,(\Fs,\nabla)) \;:=\; 
\frac{\Rc'_{i}}{R} \;\in\; ]0,1]\;.
\end{equation}
It is independent of the choices made.

The total height of the Newton polygon of~$(\Fs,\nabla)$ at~$x$ is
\begin{equation}
H_{S,r}(x,(\Fs,\nabla)) \;=\; 
\prod_{i=1}^r \Rc_{S,i}(x,(\Fs,\nabla))\;.
\end{equation}

The \emph{convergence Newton polygon} of $(\Fs,\nabla)$ at 
$x\in X$ is the polygon whose $i$-th slope is 
$\ln(\R_{S,i}(x,(\Fs,\nabla)))$ (cf. Definition \ref{Def: Conv NP} and \cite{NP-I,NP-II,NP-III}).

We say that $\Rc_{S,i}(x,(\Fs,\nabla))$ is 
\begin{itemize}
\item[$\bullet$] a \emph{spectral} radius if $\Rc_{S,i}(x,(\Fs,\nabla)) \leq r(x)/R$;
\item[$\bullet$] a \emph{solvable}  radius if $\Rc_{S,i}(x,(\Fs,\nabla))=r(x)/R$;
\item[$\bullet$] an \emph{over-solvable} radius if 
$\Rc_{S,i}(x,(\Fs,\nabla))>r(x)/R$.\footnote{Notice that 
the points $x$ belonging to $\Gamma_S$ satisfy by 
definition $r(x)=R$. Hence, at those points all the radii 
are spectral. Over-solvable radii exist only at points $x$ 
outside $\Gamma_S$.}
\end{itemize}
In the rest of the text, we will often write~$\Fs$ instead 
of~$(\Fs,\nabla)$. This should lead to no confusion.

If $X$ is a disjoint union of pseudo-annuli 
and virtual open disks endowed 
with the empty pseudo-triangulation $S=\emptyset$, then we sometimes write 
\begin{equation}
\R_{i}(x,\Fs)\;:=\;\R_{\emptyset,i}(x,\Fs)\;\quad\textrm{ and }\quad 
H_{i}(x,\Fs):=H_{\emptyset,i}(x,\Fs)\;.
\end{equation}
\end{definition}

The main result of \cite{NP-I,NP-II} may be stated as follows. For each $i\in \{1,\dotsc,r\}$, the map $\Rc_{S,i}(-,\Fs)$ is continuous and piecewise 
log-affine with rational slopes on every germ of segment 
out of a point $x\in X$. Moreover, 
$\Rc_{S,i}(-,\Fs)$ is locally constant outside a 
locally finite subgraph of~$X$. We denote by
\begin{equation}\label{eq : controlling graph}
\Gamma_{S}(\Fs)
\end{equation}
the smallest subgraph of~$X$ containing~$\Gamma_{S}$ outside which all the maps $\Rc_{S,i}(-,\Fs)$, for $i\in\{1,\dotsc,r\}$, are all locally constant. We call $\Gamma_{S}(\Fs)
$ the \emph{total controlling graph of $(\Fs,\nabla)$}.

\subsection{Irregularity at a good germ of 
segment}\label{Section : Local irreg}

In this section, we introduce the crucial notion of local irregularity of the differential equation $\Fs$. The definition involves the slope of the total height of the convergence Newton polygon along certain germs of segments in $X$. 

The importance of this notion relies on the fact that, 
under appropriate conditions, it controls the finite 
dimensionality of the de Rham cohomology groups of 
$\Fs$. 


Recall that we have define a notion of germ of segment in~$X$ in~\cite{NP-IV}. 
Note that, according to our conventions, a 
germ of segment out of a point is always oriented out 
of that point, while a germ of segment at the open 
boundary of~$X$ is always oriented towards the interior 
of~$X$.

\begin{definition}\label{def:goodgerm}
A germ of segment in $X$ is said to be good if it may 
be represented by the skeleton of an open 
pseudo-annulus contained in $X$.
\end{definition}

\begin{remark}
\label{Remark : good properties}
\begin{enumerate}
\item Every good germ has finite degree (in the sense 
of \cite[Definition 1.1.20]{NP-IV}).
\item Every germ of segment out of a point is good.
\item A germ of segment it good if, and only if, it admits 
a representative on which the map $x\mapsto \deg(x)$ 
is constant (see \cite[Lemma 1.1.28]{NP-IV}).
\item A relatively compact germ of segment is good 
(see \cite[Lemma 1.1.35]{NP-IV}).
\end{enumerate}
\end{remark}

Christol and Mebkhout gave a definition of 
irregularity for solvable differential modules 
over a germ of open annulus (see 
\cite[D\'efinition~8.3-8]{Ch-Me-III} and 
\cite[Section~2.1]{Ch-Me-IV}). We here extend this 
definition to the case of a 
germ of pseudo-annulus and to 
possibly non-solvable differential equations whose radii 
are all log-affine. 



\begin{definition}[Log-affine total height over a germ of segment]\label{def:logaffinetotalheight}
Let~$b$ be a good germ of segment in~$X$. Let~$r$ be the rank of~$\Fs$ around~$b$. Let $i\in \{1,\dotsc,r\}$. We say that~$\Fs$ has log-affine $i^\textrm{th}$ radius (resp. total height) along~$b$ if there exists an open pseudo-annulus~$C$ in~$X$ whose skeleton (suitably oriented) represents~$b$ such that the $i^\textrm{th}$ radius function $\Rc_{\emptyset,i}(-,\Fs_{|C})$ (resp. the total height function $H_{\emptyset,r}(-,\Fs_{|C})$) is log-affine on~$\Gamma_{C}$ (where $C$ is endowed with 
the trivial pseudo-triangulation).
\end{definition}


\begin{remark}\label{rem:logafftrivval}
If~$K$ is trivially valued, then all the radii of~$\Fs$ are log-affine along any good germ of segment (see \cite{NP-IV}). In particular, $\Fs$ has log-affine total height along any good germ of segment.
\end{remark}

\begin{definition}[Irregularity over a germ of 
segment]
\label{def:generalizedirregularity}
\label{def:Irrgermpseudoannulus}
Let~$b$ be a good germ of segment in~$X$ on which~$\Fs$ has log-affine total height. Let~$C$ 
be an open pseudo-annulus whose 
skeleton~$\Gamma_C$ (suitably oriented) 
represents~$b$ and such that the total 
height function $H_{\emptyset,r}(-,\Fs_{|C})$ 
is log-affine on~$\Gamma_{C}$, where 
$r = \rk(\Fs_{|C})$. We define the irregularity of~$\Fs$ 
along~$b$ as
\begin{equation}\label{eq : irr_b def}
\Irr_{b}(\Fs)\;:=\;-\deg(b)\cdot \partial_b 
H_{\emptyset,r}(-,\Fs_{|C}) \;\in\; \mathbb{Z}\;.
\end{equation}
\end{definition}

\begin{remark}\label{rk : Irr_b is indep on S}
In the above definition, $C$ is considered with the \emph{empty 
pseudo-triangulation} and the slope 
$\partial_b H_{\emptyset,r}(-,\Fs_{|C})$ is 
considered \emph{after} localization to~$C$. 
In particular, it need not be equal to 
$\partial_b H_{S,r}(-,\Fs)$, 
where $S$ is a pseudo-triangulation of $X$ (cf. 
\cite[Lemma 2.8.4]{NP-IV}). Notice also that the 
definition is intrinsically associated with the restriction of 
$\Fs$ to $b$, in the sense that it  
does not depend on the choice of the 
triangulation $S$ of $X$, nor on the choice of 
a specific pseudo-annulus~$C$ representing $b$ (cf. 
\cite[Proposition 2.8.2]{NP-IV}).
\end{remark}


The fact that the irregularity is an 
integer comes from 
\cite[Proposition 2.3.6]{NP-IV}. It can 
be \emph{negative}.

\begin{lemma}
Let~$b$ be a good germ of segment in~$X$ on which~$\Fs$ has log-affine total height. Let~$L$ be 
a complete valued extension of~$K$. 
Let $c_{1},\dotsc,c_{t}$ be the preimages of~$b$ 
in~$X_{L}$. Then, the~$c_{i}$'s are good, $\Fs_{L}$ has log-affine total height on them and we have
\begin{equation}
\Irr_{b}(\Fs) \;=\; \sum_{i=1}^t \Irr_{c_{i}}(\Fs_L)\;.
\end{equation}
\hfill$\Box$
\end{lemma}

\begin{remark}
We will see that, 
if $\Fs$ has a meromorphic singularity at a 
rational point $x$ and $b_x$ is the germ of segment 
out of $x$ (oriented out of $x$), then $\Fs$ has 
$\log$-affine radii on~$b_{x}$ (cf. 
Lemma \ref{lem:merosinglinear-1}). 
We will also prove that, in this case, 
the irregularity $\Irr_{b_x}(\Fs)$ just defined 
coincides with \emph{the opposite} of the formal 
irregularity of $\Fs$ at $x$ viewed as a differential 
equation over the field of power series $K((T-x))$ 
as defined in \cite{Ramis-Devissage-Gevrey}, 
\cite[p.110]{Deligne-Reg-Sing}, 
\cite{Malgrange-Irreg} and 
\cite{Correspondance-Malgrange-Ramis} (cf. 
Proposition \ref{Prop : Irr-form=Irr-x-1}). 
See Section \ref{Disk-merom} for more details.

In the context of \cite{Ch-Me-IV}, our irregularity  
coincides with that defined therein. It is the Robba's 
notion of irregularity \cite{RoIV}. 
It is useful to clarify here that there is a discrepancy of 
signs between the irregularity 
in the formal case around a rational point of 
\cite{Malgrange-Irreg} and that of \cite{Ch-Me-IV}. 
Indeed, in both papers, the authors made the choice of 
signs which makes non negative their local irregularities. 
However, in general the irregularity can be negative and 
there is no way to establish uniform 
orientations on the germs of segments in $X$ 
in order to make the two conventions of sign agree. We 
eventually chose the sign convention of \cite{Ch-Me-IV}.
\end{remark}

There are explicit examples of differential equations over 
an open annulus $C$ whose all radii have infinitely many 
breaks along the skeleton $\Gamma_C$  
(and hence the controlling graph $\Gamma_S(\Fs)$ has 
infinitely many bifurcations along $\Gamma_C$) and 
whose total height is however log-affine 
(cf. \cite[Section 15.4]{HDR}). The following 
proposition shows that a similar phenomena cannot 
happens for the total heights of the sub-quotients of a 
given equation.

\begin{proposition}\label{Prop : Irr_b ex seq}
Let $0\to\Fs_1\to\Fs_2\to\Fs_3\to 0$ be an exact 
sequence of differential equations on a quasi-smooth $K$-analytic curve~$X$. 
Let $b$ be a good germ of segment in $X$. 
Then $\Fs_2$ has log-affine total height 
along $b$ if and only 
if so have $\Fs_1$ and $\Fs_3$. In this case, one has
\begin{equation}
\Irr_b(\Fs_2)\;=\;\Irr_b(\Fs_1)+\Irr_b(\Fs_3)\;.
\end{equation}
\end{proposition}
\begin{proof}
Definition \ref{def:generalizedirregularity} of 
irregularity at a germ of segment involves only spectral 
radii because in \eqref{eq : irr_b def} 
one localizes before computing the slope. 
In particular, it follows from the spectral definition of 
the radii (cf. \cite[Definition 9.8.1]{Kedlaya-book} 
or \cite[Section 4.2]{NP-I}), 
that the localized radii of $\Fs_2$ are the 
union with multiplicity 
of those of $\Fs_1$ and $\Fs_3$ . In particular the localized 
total height of $\Fs_2$ is the logarithmic sum of 
those of $\Fs_1$ and $\Fs_3$.
Moreover, by \cite[Theorem 3.9, item iii)]{NP-I}, 
the localized total heights \eqref{eq : irr_b def} 
are log-concave along $b$, and hence 
the total height of $\Fs_2$ is log-affine 
along $b$ if, and only if, so are the total heights of 
$\Fs_1$ and $\Fs_3$. 
\end{proof}

\subsection{Polygons and their derivatives on curves.}
\label{Section polygons and derivatives}

Let $n\le m \in \Z$. Let $v_{n},v_{m} \in \ERRE$ and, 
for each $i\in \{n+1,\dotsc,m-1\}$, let $v_{i} \in \ERRE 
\cup\{\pm\infty\}$. If all the $v_i$'s are different from 
$-\infty$, we define the Newton polygon 
of the set 
\begin{equation}
V\; =\; \{(i,v_{i}) \mid n\le i\le m\}\;
\end{equation}
 to be the biggest convex function 
\begin{equation}\label{eq: NPV def-1}
N\!P_{V}\; \colon \;[n,m] \;\xrightarrow{\quad}\; \ERRE
\end{equation}
such that, for each  $i\in \{n,\dotsc,m\}$, we have
\begin{equation}
v_{i} = +\infty \;\textrm{ or } \; 
v_{i} \ge N\!P_{V}(i)\;.
\end{equation}

If all the $v_i$'s are different from $+\infty$, we define 
analogously the inverted Newton polygon as the smallest 
concave function $I\!N\!P_{V} \colon [n,m] \to \ERRE$ 
such that for each $i\in \{n,\dotsc,m\}$, we have
\begin{equation}\label{eq : INP def}
v_{i} = -\infty \;\textrm{ or } \; 
v_{i} \le I\!N\!P_{V}(i)\;.
\end{equation}

The functions~$N\!P_{V}$ and $I\!N\!P_{V}$ are continuous on~$[n,m]$ and 
affine on each $[i,i+1]$. 
\begin{definition}\label{Def : slope, height, break-1}
Let $N$ denote either the polygon $N\!P_V$ or 
$I\!N\!P_V$. For each $i\in \{1,\dotsc,m-n\}$, 
we call $i^\textrm{th}$ slope of the polygon the 
slope of the function $N$ on $[n+i-1,n+i]$. We usually 
denote it by
\begin{equation}
s_{i}\;:=\;\frac{d}{du}N(u)\;,\quad u\in[n+i-1,n+i]\;.
\end{equation}

The total height of the polygon is 
defined as $\sum_{i=1}^{m-n}s_i$. 

The polygon is said to have a break at $k\in\{1,\ldots,m-n-1\}$ if 
$s_k\neq s_{k+1}$.
\end{definition}

Note that, conversely, given a non-decreasing (resp. non-increasing) sequence of real numbers $s_1\leq s_2\leq\ldots\leq s_{m-n}$ (resp. $s_1\geq s_2\geq\ldots\geq s_{m-n}$) and a real number~$v_n$, we can define a Newton 
polygon (resp. inverted Newton polygon) as the unique continuous function on $[n,m]$ 
that takes the value~$v_n$ at~$n$ 
and that is affine of slope~$s_{i}$ over 
$[n+i-1,n+i]$, for each $i\in \{1,\dotsc,m-n\}$. 

\begin{notation}
In the sequel, if no specific data are mentioned, we 
assume that $n=0$ and $v_0=0$. In this case the 
polygon and the inverted polygon are determined by the 
sequence of their slopes.
\end{notation}



%
%

Let $x$ be a point of~$X$ and denote by~$r$ the rank of~$\Fs$ at~$x$.



\begin{definition}[Convergence Newton polygon]
\label{Def: Conv NP}
The \emph{convergence Newton polygon 
$N\!P_S(x,\Fs):[0,r]\to\mathbb{R}$ of~$\Fs$ 
at~$x$} is the polygon on $[0,r]$ satisfying
\begin{enumerate}
\item $N\!P_S(x,\Fs)(0)=0$;
\item for each $i\in \{1,\dotsc,r\}$, the $i^\textrm{th}$ slope of the poygon $N\!P_S(x,\Fs)$ is 
\begin{equation}\label{eq : s_i = log(R_i) newton pol-1}
s_i(x)\;:=\;\ln(\R_{S,i}(x,\Fs))\leq 0\;.
\end{equation}
\end{enumerate}
In order to simplify the notation, we sometimes set 
\begin{equation}\label{eq : v_i = NP(i) newton pol-1}
v_i(x)\;:=\;NP_S(x,\Fs)(i) \in \ERRE\;.
\end{equation}
\end{definition}
For expository reasons, we anticipate here the  
Definition of the Convergence Newton polygon for 
differential equations with meromorphic singularities 
(cf. Section \ref{Section : mero-alg})
\begin{definition}
Let $\Fc$ be a differential equation on $X$ 
with meromorphic singularities on a locally finite set of 
rigid points $Z\subset X$ (cf. Definition \ref{def:meromorphicconnection}). 
We define the convergence Newton polygon of $\Fc$ as that of $\Fs:=
\Fc_{X-Z}$. More precisely, the curve $X-Z$ admits a 
minimal pseudo-triangulation $S'$ containing 
$S$ \footnote{Indeed, notice 
that for every $z\in Z$ the connected component of 
$X-\Gamma_S$ containing $z$ is a virtual open disk 
$D_z$ whose relative boundary  in $X$ is a point $x_z$ 
that lies in $\Gamma_S$. 
For all $z\in Z$ let $I_z$ be the segment joining $z$ to 
$x_z\in\Gamma_S$. 
Then, one has 
$\Gamma_{S'}=\Gamma_S\cup(\cup_{z\in Z}I_z)$. 
Moreover $S'$ is obtained from $S$ by adding the points 
$\{x_z\}_{z\in Z}$ and the bifurcation points of 
$\Gamma_{S'}$ that are not in $\Gamma_S$.} 
and for all $x\in X-Z$ and all $i\in\{1,\ldots,r\}$ we set
\begin{eqnarray}
\R_{S,i}(x,\Fc)&\;:=\;&
\R_{S',i}(x,\Fs)\;,\\
NP_S(x,\Fc)&\;:=\;&
NP_{S'}(x,\Fs)\;.
\end{eqnarray}
\end{definition}
Let $b$ be a good germ of segment in $X$ on which the 
radii \eqref{eq : s_i = log(R_i) newton pol-1}
are all $\log$-affine.\footnote{This means that  the 
functions $s_i$ are affine. 
Notice that this is equivalent to the affinity 
of all the functions $v_i(x):=\sum_{j=1}^is_{j}(x)$.} 
There are essentially two different ways to construct a 
polygon as a derivative of the convergence Newton 
polygon along~$b$. The first takes the derivative along $b$ of the functions $v_i(x)$ (cf. \eqref{eq : v_i = NP(i) newton pol-1}), while the second takes the derivatives 
along $b$ of the slopes $s_i(x)$ (cf. 
\eqref{eq : s_i = log(R_i) newton pol-1}):
\begin{enumerate}
\item We consider both the \emph{polygon} 
and the \emph{inverted polygon}
associated with the set 
\begin{equation}
\label{eq : inverted polygon derivative at b-1}
\{(i,\partial_bv_i)\;,\;0\leq i\leq r\}\;.
\end{equation}
\item We consider the family 
$(\partial_b(s_i))_{i=1,\ldots,r}$ of derivatives 
along $b$ of the slopes \eqref{eq : s_i = log(R_i) 
newton pol-1} of the convergence Newton polygon. In general 
the sequence $(\partial_b(s_1),\ldots, \partial_b(s_r))$ is 
not monotonous. Therefore, 
we consider a permutation 
of the indexes $\sigma$ (resp. $\sigma'$) that turns 
it into non-decreasing (resp. non-increasing) order 
\begin{eqnarray}
\partial_b(s_{\sigma(1)})\;\leq\;
\partial_b(s_{\sigma(2)})\;\leq\;
\ldots\;\leq\; \partial_b(s_{\sigma(r)})\quad\;
\label{eq : partial_b s_sigma i-incr-1}\\
\textrm{(resp. }\partial_b(s_{\sigma'(1)})\;\geq\;
\partial_b(s_{\sigma'(2)})\;\geq\;
\ldots\;\geq\; \partial_b(s_{\sigma'(r)})\textrm{)}\;,
\label{eq : partial_b s_sigma i-1}
\end{eqnarray} 
and we consider the polygon on $[0,r]$ that takes the value 
$0$ at $0$ and that has 
\eqref{eq : partial_b s_sigma i-incr-1} 
(resp. \eqref{eq : partial_b s_sigma i-1}) as slopes.
\end{enumerate}

\begin{remark}\label{rk : total height derived = -Irr-1}
The four polygons defined above all have
$-\Irr_b(\Fs)$ as total height. Indeed, 
their total height is the derivative of 
$v_r=\sum_{i=1}^rs_i$ which is the 
total height of the convergence Newton polygon.
\end{remark}
\begin{remark}
We here illustrate  
with an example the necessity of  
considering both the polygon and the inverted polygon.  

Let $x\in X$. The finiteness of the radii 
(cf. \cite{NP-I, NP-II})
shows that for almost all directions $b$ out of $x$ the 
radii are all constant, therefore both the polygon and the 
inverted polygon associated with the set 
$\{(i,\partial_bv_i)\;,\;0\leq i\leq r\}$ are the zero functions on 
$[0,r]$. There are a finite number of directions out of 
$x$ on which the polygons may be non zero.

Assume now that $x$ lies inside an open 
segment  $]z,y[\subseteq X$ where the radii are 
log-affine. Let $b_z$ and $b_y$ be the 
germs of segments out of $x$ directed towards 
$z$ and $y$ respectively. Then 
\begin{equation}
\partial_{b_z}v_i\;=\;-\partial_{b_y}v_i\;.
\end{equation}

This shows that in order to have the same polygon on 
$b_z$ and $b_y$ we have to consider the polygon 
associated with the set 
$\{(i,\partial_{b_z}v_i)\;,\;0\leq i\leq r\}$ and the inverted 
polygon associated with the set 
$\{(i,\partial_{b_y}v_i)\;,\;0\leq i\leq r\}$.
The inversion of the orientation turns the polygon into 
the inverted polygon, and it seems unnatural to fix a 
choice.
\end{remark}

\subsection{de Rham cohomology}

Consider the complex of sheaves
\begin{equation}
\mathcal{E}(\Fs)^\bullet \colon (\cdots\to0\to\Fs\stackrel{\nabla}{\to}\Omega^1_X\otimes\Fs\to0\to\cdots)\;,
\end{equation}
where $\Fs$ is placed in degree~$0$ and $\Omega^1_X\otimes\Fs$ in degree $1$. The cohomology of~$\Fs$ (resp. the hypercohomology of 
$\mathcal{E}(\Fs)^{\bullet}$) will be denoted by $H^i(X, \Fs)$ (resp. 
$\mathbb{H}^i(X, \mathcal{E}(\Fs)^{\bullet})$).


\begin{remark}
In our situation, $X$ is a $K$-analytic curve, hence has topological dimension~1. It follows that $H^i(X,\Fs)=0$ for $i\ge 2$ and that $\mathbb{H}^i(X, \mathcal{E}(\Fs)^{\bullet}) = 0$ for $i\ge 3$ (by a spectral sequence argument). 
\end{remark}


\begin{definition}\label{Def. coh de de Rham}
The de Rham cohomology groups $\Hdr^i(X,\Fs)$ 
of~$\Fs$ are the hypercohomology groups 
$\mathbb{H}^i(X,\mathcal{E}(\Fs)^{\bullet})$ of the 
complex
$\mathcal{E}(\Fs)^{\bullet}$:
\begin{equation}
\Hdr^i(X,\Fs)\;:=\;
\mathbb{H}^i(X,\mathcal{E}(\Fs)^{\bullet})\;.
\end{equation}

We say that $\Fs$ has \emph{finite index} 
if $\Hdr^i(X,\Fs)$ is finite-dimensional for all degrees $i \in \Z$. 
In this case we denote by $\hdr^i(X,\Fs)$ the dimension 
of the $K$-vector space~$\Hdr^i(X,\Fs)$ and set 
\begin{align}\label{eq : chidr def}
\chidr(X,\Fs) &\;:= \;\sum_{i}(-1)^i\cdot \hdr^i(X,\Fs)\\
& \;= \;\  \hdr^0(X,\Fs) - \hdr^1(X,\Fs) + \hdr^2(X,\Fs)\;.
\end{align}
We call $\chidr(X,\Fs)$ the index of $\Fs$.
\end{definition}

\begin{lemma}[Mayer-Vietoris]
\label{Lemma : frfreign}
Let~$U$ and~$V$ be two open subsets of~$X$ such 
that $X=U\cup V$. Let $\mathcal{E}^\bullet$ be a 
complex of sheaves of groups over~$X$. We have the 
Mayer-Vietoris long exact sequence
\begin{equation}\label{eq :MVseq}
\cdots\to 
\mathbb{H}^{i-1}(U\cap V,\mathcal{E}^\bullet) \to 
\mathbb{H}^i(X,\mathcal{E}^\bullet) \to
\mathbb{H}^i(U,\mathcal{E}^\bullet)\oplus 
\mathbb{H}^i(V,\mathcal{E}^\bullet) \to 
\mathbb{H}^i(U\cap V, \mathcal{E}^\bullet) \to 
\mathbb{H}^{i+1}(X,\mathcal{E}^\bullet) \to \cdots
\end{equation}

In particular, if, for all $i\in\Z$,  the spaces $\mathbb{H}^i(U,\mathcal{E}^\bullet)$, $\mathbb{H}^i(V,\mathcal{E}^\bullet)$ and $\mathbb{H}^i(U\cap V,\mathcal{E}^\bullet)$ are finite-dimensional, then, for all $i\in\Z$, the space $\mathbb{H}^i(X,\mathcal{E}^\bullet)$ is finite-dimensional too. \hfill $\Box$
\end{lemma}

\subsubsection{Properties.}
In some cases, the de Rham cohomology may be computed in a simple way. Let us first recall a definition.


\begin{definition}[Cohomologically Stein]\label{def:cohomStein}
The curve~$X$ is said to be \emph{cohomologically Stein} if, for every coherent sheaf~$\Fs$ on~$X$ and every $q\ge 1$, we have
\begin{equation}
H^q(X,\Fs)=0.
\end{equation}
\end{definition}

Classical examples of cohomologically Stein curves include disks, pseudo-disks, annuli, pseudo-annuli, etc. More generally, by \cite[Corollary 4.6]{Banachoid}, every quasi-smooth curve with no proper connected components is cohomologically Stein. We also recall that, on quasi-Stein curves, coherent sheaves are generated by their global sections (see \cite[Corollary 4.8]{Banachoid}) and that the global sections functor induces an equivalence between the category of locally free sheaves of bounded rank and the category of projective $\Os(X)$-modules of finite type (see \cite[Corollary 4.11]{Banachoid}).

\begin{lemma}\label{Lemme : bon quasi-Stein}
If~$X$ is cohomologically Stein, then we have
\begin{equation}
\Hdr^0(X,\Fs) \;=\; \Ker (\nabla \colon \Fs(X) \to \Omega^1(X)\otimes_{\O(X)}\Fs(X))
\end{equation}
and
\begin{equation}
\Hdr^1(X,\Fs) \;=\; \Coker (\nabla \colon \Fs(X) \to  \Omega^1(X)\otimes_{\O(X)}\Fs(X))\;,
\end{equation}
and $\Hdr^i(X,\Fs)=0$ for all $i\neq 0,1$. \hfill$\Box$
\end{lemma}

\begin{lemma}
If~$X$ has finitely many connected components, then $\Hdr^0(X,\Fs)$ is finite-dimensional.
\end{lemma}
\begin{proof}
If~$X$ is cohomologically Stein, then the result follows 
from Lemma~\ref{Lemme : bon quasi-Stein} 
and~\cite[Lemma~1.2.10 i)]{NP-IV}.

If~$X$ is not cohomologically Stein, then we may cover it by two open subsets~$U$ and~$V$ that are cohomologically Stein. The result then follows from the previous case together with the Mayer-Vietoris exact sequence (see Lemma~\ref{Lemma : frfreign}).
\end{proof}

Let us recall a descent statement for de Rham cohomology that will be used several times in the paper.

\begin{theorem}[\protect{\cite[Corollary~4.14]{Banachoid}}]\label{thm:descent}
Let~$L$ be a complete valued extension of~$K$. Assume that there exists $M\in \{K,L\}$ such that~$M$ is not trivially valued and $\Hdr^1(X_{M},\Fs_{M})$ is finite-dimensional. 
Then, $\Hdr^1(X,\Fs)$ and $\Hdr^1(X_{L},\Fs_{L})$ are both finite-dimensional and we have natural 
isomorphisms 
\begin{equation}\label{eq:H^0H^1}
\Hdr^0(X,\Fs)\otimes_{K} L \simto 
\Hdr^0(X_{L},\Fs_{L}) 
\quad\textrm{ and }\quad 
\Hdr^1(X,\Fs)\otimes_{K} L \simto 
\Hdr^1(X_{L},\Fs_{L}).
\end{equation}
\hfill$\Box$
\end{theorem}

For later use, we record here some surjectivity result in de Rham cohomology.

\begin{lemma}[\protect{\cite[Lemma~4.15]{Banachoid}}]\label{Lemma : H^1 surjectif}
Assume that~$X$ has no proper connected component. Let~$W$ be an analytic domain of~$X$ such that the restriction map $\Os(X) \to \Os(W)$ has dense image. Assume that there exists a complete non-trivially valued extension~$L$ of~$K$ such that $\Hdr^1(W_{L},(\Fs_{L})_{|W_{L}})$ is finite-dimensional. 
Then, the map
\begin{equation}\label{eq : surjectivity of H^1}
\Hdr^1(X,\Fs)\;\xrightarrow{\quad}\;
\Hdr^1(W,\Fs_{|W})
\end{equation}
is surjective.
\hfill$\Box$
\end{lemma}

To finish this section, we provide some conditions to ensure that we 
have trivial cohomology. 

\begin{proposition}
\label{Prop : H^i=0 if not solvablegfz}
Assume that we are in one of the following two situations.

Situation~1:
\begin{enumerate}
\item $X$ is an open pseudo-annulus; 
\item all the radii of $\Fs$ are log-affine along~$\Gamma_{X}$ and strictly smaller than~1.
\end{enumerate}

Situation~2:
\begin{enumerate}
\item $X$ is a virtual open disk;
\item all the radii of~$\Fs$ are constant on~$X$ and strictly smaller than~1.
\end{enumerate}
Then, for all $i$, we have
\begin{equation}
\Hdr^i(X,\Fs)\;=\;0\;.
\end{equation}
\end{proposition}
\begin{proof}
Let $r$ be the rank of~$\Fs$. First note that, in situation~1, all the radii of $\Fs$ are locally constant on $X\setminus \Gamma_{X}$ by \cite[Corollary~6.2.28]{NP-III}.

We proceed in two steps.

\smallbreak

\emph{Step 1: } Assume that~$\Omega^1_{X}$ is free.

If~$\Fs$ admits a global non-zero solution on~$X$, then 
for all $x\in X$ its restriction converges on the whole 
maximal disk associated with $x$. Therefore 
$\Rc_{r}(x,\Fs) = 1$, which contradicts the assumptions. 
We deduce that $\Hdr^0(X,\Fs)=0$.


Since~$X$ is cohomologically Stein and~$\Omega^1_{X}$ is free, 
by Lemma~\ref{lem:HdrquasiSteinOmegafree}, we 
have $\Hdr^1(X,\Fs) = \Hdr^1(\Fs(X),\nabla)$. By 
\cite[Lemma~5.3.3 and Remark~5.3.4]{Kedlaya-book}, 
we deduce that 
$\Hdr^1(X,\Fs)=\mathrm{Ext}^1(\Fs^\ast(X),\O(X))$. 
We will now prove that this last group is~0 by proving 
that every exact sequence of differential modules
\begin{equation}\label{eq : seq yon splpot}
0\to\O(X)\to E\to\Fs^*(X)\to 0
\end{equation}
splits. Since~$X$ is cohomologically Stein, the coherent sheaves on~$X$ are 
generated by their global sections, hence such a 
sequence induces an exact sequence of differential 
equations 
\begin{equation}\label{eq:sheaves}
0 \to \Os_{X} \to \Es \to \Fs^\ast \to 0\;.
\end{equation} 
It is enough to prove that this last sequence splits.

Let us first prove that all the radii of~$\Fs^\ast$ are 
everywhere strictly smaller than~1. We recall that 
spectral non-solvable radii are invariant under duality. 
In situation~1, the radii of~$\Fs$ are alll spectral non-
solvable on~$\Gamma_{X}$ by assumption, hence so 
are the radii of~$\Fs^\ast$. By continuity of the radii, 
and transfer principle on disks (cf. \cite[Proposition 4.2]{NP-I}), we deduce that the radii of $\Fs^*$ are strictly smaller than~1 on every point of $X$. In situation~2, by assumption, there exists a point of~$X$ on which the radii are all spectral non-solvable and the same reasoning applies.





We have now proven that all the radii of~$\Fs^\ast$ are everywhere strictly smaller than~1. We also know that all the radii 
of~$\O$ are constant and equal to~1. By~\cite[Proposition~2.10.5]{NP-III}, the family of radii of~$\Es$ is the union (with 
multiplicities) of those of~$\Fs^*$ and those of~$\O_X$. By~\cite[Theorem~5.4.11]{NP-III}, we deduce that~$\Es$ splits 
into~$\Es_{<r+1}$ and~$\Es_{\geq r+1}$. By uniqueness, we have $\Es_{\ge r+1} = \Os_{X}$ and $\Es_{< r+1} = 
\Fs^\ast$. The result follows.

\smallbreak

\emph{Step 2:} The general case.

By Theorem~\ref{thm:descent}, it is enough to prove the result after extending the scalars. Hence, we may assume that~$K$ is algebraically closed and maximally complete and that $|K| = \ERRE_{\ge0}$.

If we are in situation~2, then~$X$ is a disk 
and~$\Omega^1_{X}$ is free by \cite{Lazard}, so we 
can conclude by step 1.

Let us now assume that we are in situation~1. Recall that there exists a deformation retraction $r \colon X \to \Gamma_{X}$ and that, for each compact interval~$I$ in~$\Gamma_{X}$, the inverse image $r^{-1}(I)$ of~$I$ is a closed annulus with analytic skeleton~$I$. Moreover, the ring of global functions of a closed annulus being principal, the restriction of~$\Omega^1_{X}$ to such a space is free. We deduce that, for each open relatively compact interval~$J$, the inverse image $r^{-1}(J)$ of~$J$ is an open annulus and that the restriction of~$\Omega^1_{X}$ to it is free. 

Let us now write $\Gamma_{X} = J_{1} \cup J_{2}$, where $J_{1}$ and $J_{2}$ are disjoint unions of open relatively compact intervals. Note that $J_{1} \cap J_{2}$ is also a disjoint union of open relatively compact intervals. Set $U_{1} := r^{-1}(J_{1})$ and $U_{2} := r^{-1}(J_{2})$.  It follows from the previous discussion and from step~1 that, for $W \in \{U,V,U\cap V\}$, we have $\Hdr^0(W,\Fs) = \Hdr^1(W,\Fs) = 0$. We conclude by the Mayer-Vietoris long exact sequence. 
\end{proof}

\begin{corollary}\label{Cor : trivial over disk with constant radii}
Assume that $X$ is a virtual open disk and that the radii of~$\Fs$ are all constant over~$D$. Then $\Fs$ has finite index on~$D$ and we have $\Hdr^1(D,\Fs)=0$.

\end{corollary}
\begin{proof}
By Theorem~\ref{thm:descent}, we may extend the scalars and assume that~$K$ is algebraically closed. The radii of~$\Fs$ are all constant on~$D$. If they are all strictly smaller than~1, then we are in situation~2 of Proposition \ref{Prop : H^i=0 if not solvablegfz} and the result holds.

Otherwise, let $j\in \{1,\dotsc,r\}$ be the smallest index such that $\Rc_{j}(-,\Fs) = 1$ on~$D$. By \cite[Theorem~5.3.1]{NP-III}, there exists a sub-object $\Fs_{\ge j}$ of~$\Fs$ of rank~$r-j+1$ such that, for each $x\in D$, we have
\[\begin{cases}
\forall i\in \{1,\dotsc,r-j+1\},\ \Rc_{i}(x,\Fs_{\ge j}) =1;\\
\forall i\in \{1,\dotsc,j-1\},\ \Rc_{i}(x,\Fs/\Fs_{\ge j}) = \Rc_{i}(x,\Fs_{|D}) < 1.
\end{cases}\]
By Proposition \ref{Prop : H^i=0 if not solvablegfz}, we have $\Hdr^1(D,\Fs/\Fs_{\ge j})=0$. Since all the radii of the module $\Fs_{\ge j}$ are equal to~1, it is a finite sum of trivial modules. Since the usual derivation is surjective on~$\Os(D)$, we have $\Hdr^1(D,\Os) = 0$, hence $\Hdr^1(D,\Fs_{\ge j})=0$. The result now follows by writing the cohomology long exact sequence associated to the short exact sequence $0 \to \Fs_{\ge j} \to \Fs \to \Fs/\Fs_{\ge j} \to 0$.
\end{proof}

\subsubsection{Differential modules}

We can also define Rham cohomology in the setting of differential modules.

\begin{definition}[General definition of index]
\label{Def: index general u}
Let $V$ be a $K$-vector space and let $u:V\to V$ be a $K$-linear map. We say that~$u$ has finite index if $\Ker(u)$ and $\Coker(u)$ are finite-dimensional 
$K$-vector spaces. In this case, we define the 
\emph{index} of~$u$ as 
\begin{equation}
\chi(V,u)\;:=\;
\dim_K \Ker(u) - \dim_K \Coker(u)\;.
\end{equation}
\end{definition}

\begin{lemma}[Additivity of index]
\label{Lemma : additivity of Index}
Let 
\begin{equation}
\xymatrix{
0\ar[r]&V_1\ar[d]^{u_1}\ar[r]&
V_2\ar[d]^{u_2}\ar[r]&
V_3\ar[d]^{u_3}\ar[r]&0\\
0\ar[r]&V_1\ar[r]&V_2\ar[r]&V_3\ar[r]&0}
\end{equation}
be a commutative diagram of $K$-vector spaces in 
which the horizontal sequences are both exact. If two among $u_1$, $u_2$, $u_3$ have finite index, 
then so has the third. In this case, we have
\begin{equation}
\chi(V_2,u_2)\;=\;
\chi(V_1,u_1)+\chi(V_3,u_3)\;.
\end{equation}
\end{lemma}
\begin{proof}
This follows from the exact sequence 
$0\to\Ker(u_1)\to\Ker(u_2)\to\Ker(u_3)\to
\mathrm{Coker}(u_1)\to
\mathrm{Coker}(u_2)\to
\mathrm{Coker}(u_3)\to 0$ given by the snake lemma.
\end{proof}

\begin{definition}[de Rham cohomology of a differential module]\label{Def: index t}
Let $A$ be a $K$-algebra together with a derivation $d \colon A\to A$ 
such that $K=\Ker(d)$. Let $(\M,\nabla)$ be a differential module over~$(A,d)$, \textit{i.e.} an $A$-module~$\M$ 
equipped with a connection $\nabla \colon \M\to\M$ satisfying the Leibniz rule:
\begin{equation}
\forall a\in A, m\in M,\ \nabla(am)=d(a)m+a\nabla(m).
\end{equation} 
If~$\nabla$ has finite index, we set 
\begin{eqnarray}
\Hdr^0(\M,\nabla)&:=&\Ker(\nabla),\\
\Hdr^1(\M,\nabla)&:=&\Coker(\nabla),\\
\chidr(\M,\nabla)&\;:=\;& \chi(\M,\nabla)\;.
\end{eqnarray}

\end{definition}


\begin{lemma}\label{lem:HdrquasiSteinOmegafree}
If~$X$ is cohomologically Stein and if the sheaf~$\Omega_{X}^1$ is free of rank~1, 
then Definitions~\ref{Def. coh de de Rham} and~\ref{Def: index t} can be made to agree. Namely, choose a global differential form on~$X$ that is a basis of $\Omega_{X}^1(X)$ and consider the associated derivation~$d$. Then, the differential equation~$(\Fs,\nabla)$ induces a differential 
module $(\Fs(X),\nabla)$ over $(\Os(X),d)$ and, for $i=0,1$, we have
\begin{eqnarray}
\Hdr^i(X,\Fs)&\;=\;&\Hdr^i(\Fs(X),\nabla)\;,\\
\qquad\qquad\qquad\chidr(X,\Fs)&=&\chidr(\Fs(X),\nabla)\;.
\end{eqnarray}
\hfill$\Box$
\end{lemma}



\begin{remark}
\label{rk : gen index depends on the derivation}
The index $\chi(\Fs(X),\nabla)$ 
\emph{depends on the chosen derivation of 
$\O(X)$}. Namely if $fd$, 
is another derivation, with $f\in \O(X)$, then 
$(\Fs(X),f\nabla)$ 
is a differential module over $(\O(X),fd)$. Hence 
with the notations of Definition 
\ref{Def: index general u} we have
\begin{equation}
\chi(\Fs(X),f\nabla)\;=\;
\chi(\Fs(X),f)+
\chi(\Fs(X),\nabla)\;.
\end{equation}
The equality $\chi(\Fs(X),\nabla)=\chidr(X,\Fs)$ holds if 
and only if $d$ generates $\Omega^1_X$.
\end{remark}

\subsection{Meromorphic de Rham cohomology.}\label{Section : mero-alg}
In this section, we introduce definitions and 
basic results on meromorphic 
differential equations.

Let~$P$ be a quasi-smooth $K$-analytic curve. Let~$Z$ be a locally finite subset of rigid points of~$P$. Set  
\begin{equation}
Y \;:=\; P - Z\;.
\end{equation}
We denote by $j \colon Y \hookrightarrow P$ the 
associated open immersion. We denote by $\Os_{P}[*Z]$ the sheaf 
of meromorphic functions on~$P$ that are holomorphic 
on~$Y$ (hence have poles at worst on~$Z$). Recall that it is the sheaf on~$P$ associated to the presheaf whose ring of sections on an analytic domain~$U$ of~$P$ is the localization of~$\Os_{P}(U)$ by the subset of its elements that do not vanish outside~$Z$.



We now define meromorphic connections following \cite[Chapter~5]{HTT-Dmodules} (which itself borrows from~\cite{Deligne-Reg-Sing}).

\begin{definition}\label{def:meromorphicconnection}
Let~$\Fc$ be a locally free $\Os_{P}[*Z]$-module  of finite rank on~$P$. A \emph{meromorphic connection on~$\Fc$ with poles on~$Z$}  is a $K$-linear map
\begin{equation}
\nabla \colon \mathcal{F} \to  \Omega^1_{P}
\otimes_{\Os_{P}}\mathcal{F} 
\end{equation}
that satisfies the Leibniz rule: for every open 
subset~$U$ of~$P$ and every $f\in \Os_{P}[*Z](U)$ 
and $s\in \mathcal{F}(U)$, we have\footnote{Remark that it is enough to require that~\eqref{eq:Leibniz} holds for $f\in \Os_{P}(U)$.}
\begin{equation}\label{eq:Leibniz}
\nabla(fs) = df \otimes s + f\nabla s.
\end{equation}
We also say that the pair $(\Fc,\nabla)$ is a 
\emph{differential equation on $P(*Z)$} or a \emph{(meromorphic) differential equation on $P$ with poles on~$Z$}. 
As usual, morphisms of differential equations 
$\varphi \colon 
(\mathcal{F},\nabla) \to (\mathcal{F}',\nabla')$ are 
morphisms of $\Os_{P}[*Z]$-modules that are compatible 
with the connections.

%
%
\end{definition}

\begin{definition}
\label{definition  : merocoh}
Let~$(\Fc,\nabla)$ be a differential equation 
on~$P(*Z)$. 
The de Rham cohomology groups 
\begin{equation}
\Hdr^i(P(*Z),(\Fc,\nabla))
\end{equation}
of~$(\Fc,\nabla)$ are the hypercohomology groups of 
the complex
\begin{equation}
\cdots\to0\to\Fc \xrightarrow[]{\nabla} 
\Omega^1_P\otimes_{\Os_{P}}\Fc\to0\to\cdots,
\end{equation}
where $\Fc$ is placed in degree~0 and 
$\Omega_P^1\otimes_{\Os_{P}}\Fc$ in degree~1. 

As usual, we will often suppress~$\nabla$ 
from the notation when it is clear from the context.
\end{definition}


The notation $\Fs$ will be used to indicate the 
restriction of $\Fc$ to $Y$:
%
\begin{equation}
\Fs\;:=\;\Fc_{|Y}\;.
\end{equation}
This operation gives rise to a functor
\begin{equation}\label{eq : res-mero-to-an section 1}
\xymatrix{
\ar[d]^{\Fc\;\mapsto\;\Fc_{|Y}\;=\;\Fs}
\{\textrm{Differential equations 
on }P(*Z)\}\\
\{\textrm{Analytic differential equations on }Y\}}
\end{equation}
and a canonical morphism between the 
cohomology groups
\begin{eqnarray}\label{eq : H^i kjed}
\Hdr^i(P(*Z),\Fc)&\;\to\;&\Hdr^i(Y,\Fs)\;. 
\end{eqnarray}

When we do not mention poles, we understand that the 
connection is holomorphic: $Z=\emptyset$ 
and~$\mathcal{F}=\Fs$ is a genuine analytic  
differential equation over $Y=P$. 

\begin{remark}\label{rk : mero=anal over U}
If $U$ is an open subset of $P$ such that 
$U\cap Z=\emptyset$, the restriction of the sheaf 
$\O_P[*Z]$ to $U$ coincides by definition with $\O_U$. 
Hence, over $U$, we have 
the usual \emph{analytic} cohomology:
\begin{equation}
\Hdr^i(U(*Z),\Fc_{|U})\;=\;\Hdr^i(U,\Fs_{|U})\;.
\end{equation}
\end{remark}

\begin{lemma}
\label{Lemma : val tri == equal}
If the valuation of $K$ is trivial, for each open subset~$U$ of~$P$ 
one has 
\begin{equation}
\O_P[*Z](U)\;=\;\O_{P-Z}(U-Z)\;.
\end{equation}
In particular, the restriction functor 
\eqref{eq : res-mero-to-an section 1} is an 
equivalence of categories.
\end{lemma}
\begin{proof}
The equality can be tested on a basis of open subsets of 
$P$. 
By Remark \ref{rk : mero=anal over U}, it is enough to 
test the equality on a basis of open neighborhoods of a 
point $z\in Z$, which is given by the set of 
virtual open disks containing $z$. 
Therefore, we can assume that $P$ is a virtual open disk
and that $Z$ is reduced to an individual rigid point~$z$. 
Up to replacing $K$ by a finite extension, we can assume 
that $P$ is an open disk and that $z$ is a $K$-rational 
point. 
In this case the equality $\O_{P-\{z\}}=\O_P[*z]$ 
follows from an explicit computation 
(cf. Section \ref{Remark : index annulus trivial valuation} 
for more details). 
\end{proof}

\if{

.....

.....

.....

.....

.....

\subsection{Meromorphic de Rham cohomology.}\label{Section : mero-alg}
In this section, we introduce definitions and 
basic results on meromorphic 
differential equations.

Let~$P$ be a quasi-smooth $K$-analytic curve. Let~$Z$ be a locally finite subset of rigid points of~$P$.

%
%
%
%

Set  
\begin{equation}
Y \;:=\; P - Z\;.
\end{equation}
We denote by $j \colon Y \hookrightarrow P$ the 
associated open immersion. We denote by $\Os_{P}[*Z]$ the sheaf 
of meromorphic functions on~$P$ that are holomorphic 
on~$Y$ (hence have poles at worst on~$Z$). Recall that it is the sheaf on~$P$ associated to the presheaf whose ring of sections on an analytic domain~$U$ of~$P$ is the localization of~$\Os_{P}(U)$ by the subset of its elements that do not vanish outside~$Z$.



We now define meromorphic connections following \cite[Chapter~5]{HTT-Dmodules} (which itself borrows from~\cite{Deligne-Reg-Sing}).

\begin{definition}\label{def:meromorphicconnection}
Let~$\Fc$ be a locally free $\Os_{P}[*Z]$-module  of finite rank on~$P$. A \emph{meromorphic connection on~$\Fc$ with poles on~$Z$}  is a $K$-linear map
\begin{equation}
\nabla \colon \mathcal{F} \to  \Omega^1_{P}
\otimes_{\Os_{P}}\mathcal{F} 
\end{equation}
that satisfies the Leibniz rule: for every open 
subset~$U$ of~$P$ and every $f\in \Os_{P}[*Z](U)$ 
and $s\in \mathcal{F}(U)$, we have\footnote{Remark that it is enough to require that~\eqref{eq:Leibniz} holds for $f\in \Os_{P}(U)$.}
\begin{equation}\label{eq:Leibniz}
\nabla(fs) = df \otimes s + f\nabla s.
\end{equation}
We also say that the pair $(\Fc,\nabla)$ is a 
\emph{differential equation on $P(*Z)$} or a \emph{(meromorphic) differential equation on $P$ with poles on~$Z$}. 
As usual, morphisms of differential equations 
$\varphi \colon 
(\mathcal{F},\nabla) \to (\mathcal{F}',\nabla')$ are 
morphisms of $\Os_{P}[*Z]$-modules that are compatible 
with the connections.

%
%
\end{definition}

\begin{definition}
\label{definition  : merocoh}
Let~$(\Fc,\nabla)$ be a differential equation 
on~$P(*Z)$. 
The de Rham cohomology groups 
\begin{equation}
\Hdr^i(P(*Z),(\Fc,\nabla))
\end{equation}
of~$(\Fc,\nabla)$ are the hypercohomology groups of 
the complex
\begin{equation}
\cdots\to0\to\Fc \xrightarrow[]{\nabla} 
\Omega^1_P\otimes_{\Os_{P}}\Fc\to0\to\cdots,
\end{equation}
where $\Fc$ is placed in degree~0 and 
$\Omega_P^1\otimes_{\Os_{P}}\Fc$ in degree~1. 

As usual, we will often suppress~$\nabla$ 
from the notation when it is clear from the context.
\end{definition}


The notation $\Fs$ will be used to indicate the 
restriction of $\Fc$ to $Y$:
%
\begin{equation}
\Fs\;:=\;\Fc_{|Y}\;.
\end{equation}
This operation gives rise to a functor
\begin{equation}\label{eq : res-mero-to-an section 1}
\xymatrix{
\ar[d]^{\Fc\;\mapsto\;\Fc_{|Y}\;=\;\Fs}
\{\textrm{Differential equations 
on }P(*Z)\}\\
\{\textrm{Analytic differential equations on }Y\}}
\end{equation}
and a canonical morphism between the 
cohomology groups
\begin{eqnarray}\label{eq : H^i kjed}
\Hdr^i(P(*Z),\Fc)&\;\to\;&\Hdr^i(Y,\Fs)\;. 
\end{eqnarray}

When we do not mention poles, we understand that the 
connection is holomorphic: $Z=\emptyset$ 
and~$\mathcal{F}=\Fs$ is a genuine analytic  
differential equation over $Y=P$. 

\begin{remark}\label{rk : mero=anal over U}
If $U$ is an open subset of $P$ such that 
$U\cap Z=\emptyset$, the restriction of the sheaf 
$\O_P[*Z]$ to $U$ coincides by definition with $\O_U$. 
Hence, over $U$, we have 
the usual \emph{analytic} cohomology:
\begin{equation}
\Hdr^i(U(*Z),\Fc_{|U})\;=\;\Hdr^i(U,\Fs_{|U})\;.
\end{equation}
\end{remark}

\begin{lemma}
\label{Lemma : val tri == equal}
If the valuation of $K$ is trivial, for each open subset~$U$ of~$P$ 
one has 
\begin{equation}
\O_P[*Z](U)\;=\;\O_{P-Z}(U-Z)\;.
\end{equation}
In particular, the restriction functor 
\eqref{eq : res-mero-to-an section 1} is an 
equivalence of categories.
\end{lemma}
\begin{proof}
The equality can be tested on a basis of open subsets of 
$P$. 
By Remark \ref{rk : mero=anal over U}, it is enough to 
test the equality on a basis of open neighborhoods of a 
point $z\in Z$, which is given by the set of 
virtual open disks containing $z$. 
Therefore, we can assume that $P$ is a virtual open disk
and that $Z$ is reduced to an individual rigid point~$z$. 
Up to replacing $K$ by a finite extension, we can assume 
that $P$ is an open disk and that $z$ is a $K$-rational 
point. 
In this case the equality $\O_{P-\{z\}}=\O_P[*z]$ 
follows from an explicit computation 
(cf. Section \ref{Remark : index annulus trivial valuation} 
for more details). 
\end{proof}

Over a cohomologically Stein space, it is natural to expect that one can read off the properties of a sheaf on its global sections. We prove comparison results in this direction. 

Using the fact that cohomology and tensor products commute with filtered direct limits, the classical results on coherent sheaves may be generalized this way.

\begin{proposition}\label{prop:ABqc}
Let~$P$ be a quasi-smooth $K$-analytic curve with no proper connected component. Let~$\Fc$ be a sheaf of $\Os_{P}$-modules that is a filtered direct limit of coherent sheaves. Then, for every $g\ge 1$, we have 
$H^q(P,\Fc)=0$
and~$\Fc$ is generated by its global sections. 
\hfill$\Box$ 
\end{proposition}

\begin{corollary}\label{cor:eqcatqc}
In the setting of Proposition~\ref{prop:ABqc}, the global section functor $\Fc \mapsto \Fc(P)$ sets up an equivalence between the category of $\Os_{P}$-modules that are filtered direct limits of coherent sheaves and the category of $\Os(P)$-modules.
\hfill$\Box$
\end{corollary}

It is now straightforward to prove that the global sections functor induces various an equivalence between the category of locally free $\Os_{P}[\ast Z]$-modules of bounded rank and the category of projective $(\Os_{P}[\ast Z])(P)$-modules of finite type by following the classical arguments, as in the proof of \REF{} for instance. One may also include the datum of a connection.


\begin{corollary}\label{cor:eqcatStein}
Let~$P$ be a quasi-smooth $K$-analytic curve with no 
proper connected component. Let~$Z$ be a 
locally finite subset of rigid points of~$P$. 

The global section functor $\Fc \mapsto \Fc(P)$ sets up an equivalence between the category of locally free $\Os_{P}[\ast Z]$-modules of bounded rank endowed with a meromorphic connection and the category of projective $(\Os_{P}[\ast Z])(P)$-modules of finite type endowed with a connection.
\hfill$\Box$
\end{corollary}

\if{\subsubsection{Localization of the partial heights.}

Let~$X$ be a quasi-smooth $K$-analytic curve. Let $(\Fs,\nabla)$ be a differential equation on~$X$. The slopes of the partial heights are related to the local 
irregularities (cf. Section \ref{Section : Local irreg}).
We quickly recall a result of \cite{NP-III} explaining 
how the Laplacian modifies 
when we localize to an elementary tube. 

\begin{notation}\label{Def : D_b -2}
Let~$x\in X$ be a point of type~2 or~3. If~$b$ is a germ of segment out of~$x$ that does not belong to~$\Gamma_{S}$, then the connected component of~$X-\{x\}$ that corresponds to~$b$ is a virtual open disk. We denote it by~$D_{x,b}$, or~$D_{b}$ if no confusion may arise. 

If~$x\notin\Gamma_{S}$, since $X-\Gamma_S$ is a disjoint union of open pseudo-disks, there exists a unique virtual closed disk inside $X-\Gamma_{S}$ with boundary~$x$. We denote it by~$D_{x}$. There exists a unique germ of segment out of~$x$ that does not belong to~$D_{x}$. We denote it by~$b_{x,\infty}$, or by~$b_{\infty}$ if no confusion may arise.
\end{notation}

We use $h^0(D_{b},\Fs)$ (resp. $h^0(D_{x}^\dag,\Fs)$) to denote the dimension of the $K$-vector spaces of solutions of~$\Fs$ on~$D$ (resp. overconvergent solutions of~$\Fs$ on~$D_{x}$). 

Recall that a germ of segment out of a point is oriented away from that point.

\begin{lemma}[\smallcomment{\cite[.????.]{NP-III}}]\label{Lemma : Irr local sopp}
Let $x\in X$ be a point of type~2 or~3. Set $r := \rk(\Fs_{x})$. Let $b$ be a germ of segment out of $x$ and let $C_b$ 
be a virtual open annulus that is a section of $b$. Then, the following equalities hold.
\begin{enumerate}
\item If $x\in\Gamma_{S}$ and $b\subseteq \Gamma_{S}$, then we have 
$\partial_bH_{\emptyset,r}(x,\Fs_{|C_b})=
\partial_bH_{S,r}(x,\Fs)$;
\item If $x\in\Gamma_{S}$ and $b\nsubseteq\Gamma_S$ or if $x\notin\Gamma_{S}$ and $b\ne b_{\infty}$, then we have
\begin{equation}\label{eq : changing radii sptbq99}
\partial_bH_{\emptyset,r}(x,\Fs_{|C_b})\;=\;
\partial_bH_{S,r}(x,\Fs)-h^0(D_b,\Fs)+r\;;
\end{equation}
\item If $x\notin \Gamma_{S}$ and $b = b_{\infty}$, then, we have
\begin{equation}
\qquad\qquad\qquad\qquad\partial_{b}H_{\emptyset,r}(x,\Fs_{|C_b})=
\partial_{b}H_{S,r}(x,\Fs)+h^0(D_x^\dag,\Fs)-r\;.\qquad\qquad\Box
\end{equation}
\end{enumerate}
\end{lemma}
\if{\begin{proof}
All the statements are deduced from 
Proposition \ref{Prop : immersion}.

\comment{Ajouter une r\'ef\'erence \`a NP 3.}

i) It is immediate. 

ii) Restricting from~$X$ to~$D_b$ leaves the slopes of the radii along~$b$ unchanged. Restricting to~$C_b$ causes the slopes of the radii $\R_{S,i}(-,\Fs)$ that are spectral on~$b$ to increase by~1, whereas it leaves unchanged (and equal to~0) the slopes of the radii that are oversolvable on~$b$. Since the oversolvable radii correspond to the solutions of~$\Fs$ on~$D_{b}$, the result follows.

iii) As above, restricting from~$X$ to~$D_x^\dag$ leaves the slopes of the radii along~$b$ unchanged. The argument continues as before except that the slopes of the spectral radii decrease by~1 this time.
\end{proof}
}\fi

\begin{corollary}[\smallcomment{\cite[.????.]{NP-III}}]\label{Prop: localization of ddc H to a tube}
Let $x$ be a point of type $2$ or $3$ and let~$V$ be an elementary tube centered at~$x$ that is adapted to~$\Fs$. Set $r := \rk(\Fs_{x})$. Then the following results hold.
\begin{enumerate}
\item If $x\notin\Gamma_S$, then we have
\begin{equation}
dd^cH_{\{x\},r}(x,\Fs_{|V^\dag})\;=\;
dd^cH_{S,r}(x,\Fs) 
- r\cdot \chi_{c}(V^\dag)
+ h^0(D_x^\dag,\Fs) 
- \sum_{
\sm{b\in\sing{x,V}\\
b\neq b_\infty}} h^0(D_b,\Fs)\;.  
\end{equation}

\item If $x\in\Gamma_S$, then we have
\begin{equation}
dd^cH_{\{x\},r}(x,\Fs_{|V^\dag})\;=\;
dd^cH_{S,r}(x,\Fs) 
+ r\cdot (N_V(x)-N_S(x))
- \sum_{
\sm{b\in\sing{x,V}\\
b\nsubseteq\Gamma_S}} h^0(D_b,\Fs)\;.
\end{equation}
\end{enumerate}
\hfill$\Box$
\end{corollary}

}\fi

}\fi

\section{The analytic index theorem over open 
pseudo-annuli}\label{sec:indexpseudo-annuli}

In this section, we consider a differential equation over an 
open pseudo-annulus with log-affine radii and we 
provide a necessary and sufficient condition for the finite 
dimensionality of its de Rham cohomolgy. 
This condition is expressed by the finiteness of 
certain \emph{absolute indexes} at the open boundary of 
the pseudo-annulus. 
Absolute indexes are normalized versions of the 
\emph{generalized indexes} introduced by Robba 
\cite{RoIII,RoIV} and also 
exploited by Christol and Mebkhout 
\cite{Ch-Me-III}. 
\if{Robba's generalized indexes are the indexes of certain 
truncations of the connection that are no 
more differential operators. }\fi
We improve their approach by the introduction of an 
intrinsic definition of absolute index which is independent 
of the coordinate, of the choice of the derivation and of 
the basis of the differential module. %
\if{The usual finiteness condition in literature is the so-called 
\emph{Liouville condition} on the exponents (and 
their differences) introduced by G.~Christol and 
Z.~Mebkhout in \cite{Ch-Me-II}, and refined in 
\cite{Dwork-Exponents},
\cite{Kedlaya-book}, \cite{Kedlaya-draft}. 
The Liouville condition guarantees 
that a differential equation with the Robba property over 
the annulus has zero index. 
More precisely, the Liouville condition implies that a 
Robba equation splits into rank one differential modules 
for which the computations are quite explicit. 

However, Liouville conditions 
do not seem minimal for the finite 
dimensionality of the cohomology. 
In practice, the minimal property that one really needs is 
the vanishing of the generalized indexes. This is the 
property $\Fin$ that we introduce in this section, and 
that will be used in the sequel of the paper. It is 
conjecturally equivalent to the Liouville condition, 
but it is much less technical.
}\fi
More specifically, this section generalizes the analogous 
results of \cite{Ch-Me-III} where similar results 
were obtained with some restrictions, in particular about 
the ground field 
$K$. As an example, the investigation of the case 
where $K$ is trivially valued permits to establish the link 
with the classical theory of differential equations  over a 
field of power series and leads to new proofs of
some major well known results in that context (cf. 
Section \ref{Remark : index annulus trivial valuation}).

We use as a reference the book of 
P.~Schneider~\cite{Schneider}. There, the base field is 
always assumed to be non-trivially valued and spherically 
complete, which explains why this hypothesis appears in 
several technical statements. Nevertheless, we have been 
able to remove it from the major statements thanks to a 
descent result from~\cite{Banachoid} (see Theorem~\ref{thm:descent}).

A complete and self-contained exposition turned out to 
be necessary in this section 
because several technical parts do not 
admit analogous accurate statements in \cite{Ch-Me-IV} 
and will be generalized or used as central tools in Section 
\ref{Disk-merom}, where we will obtain similar results 
for differential equations with some meromorphic 
singularities over an open disk. 
Indeed, the case of a disk with some meromorphic 
singularities is a central tool in the proof of our main 
index results in \cite{NP-V}. 
\if{For this reason, we though that a 
complete and self contained exposition was necessary in 
this section.
}\fi

%
%

\subsection{Generalized indexes of vector spaces}
\label{Ind-genind}



Let $\mu \in \NN$ and let $V_0, \dotsc, V_\mu$ be vector spaces over~$K$. Denote by~$V$ their direct sum. For every $k\in\{0,\dotsc,\mu\}$, we denote the canonical injections and projections by 
\begin{equation}\label{eq:injproj}
i_k\;:\;V_k\hookrightarrow V\;\qquad\textrm{ and }\;\qquad 
p_k\;:\;V\to V_k
\end{equation}
and the corresponding projector by 
\begin{equation}\label{eq : u_k (def)}
\pi_k\;:=\;i_k\circ p_k\;:\;V\to V\;.
\end{equation}

Recall the notion of operator of finite index from Definition~\ref{Def: index general u}.

\begin{definition}\label{Def: genindgeneralclassical}
Let~$f$ be an endomorphism of~$V$. For $k\in\{0,\dotsc,\mu\}$, set $f_k:=p_k\circ f\circ i_k$:
\begin{equation}\label{eq : truncated op u_k}
f_k\;:\;V_k\xrightarrow{\;i_k\;} V\xrightarrow{\;f\;}
V\xrightarrow{\;p_k\;} V_k\;.
\end{equation}
We say that~$f$ has finite $k$-generalized index if $f_{k}$ has finite index. In this case, the index of the 
map~$f_{k}$ is denoted by 
\begin{equation}
\chi_k^{\mathrm{gen}}(V,f)\;:=\;
\chi(V_k,f_k)\;.
\end{equation}
and called the $k$-generalized index of~$f$.
\end{definition}


We will also use a similar notion when each of the~$V_{i}$'s carries a family of seminorms~$v_{i}$ that makes it a normoid Fr\'echet space (see \cite{Banachoid}). By \cite{Banachoid}, there exists a family of seminorms~$v$ on their direct sum~$V$ that makes it a Banachoid space, hence a normoid Fr\'echet space, and we will consider bounded endomorphisms of~$(V,v)$. In the following, we will always use this setting implicitly when writing that~$f$ is a bounded endomorphism of~$V$.

\begin{definition}\label{Def: genindgeneral}
Let~$f$ be a bounded endomorphism of~$V$. It is said to be Fredholm if it is topologically strict and has finite index.


For $k\in\{0,\dotsc,\mu\}$, we say that~$f$ is $k$-Fredholm if $f_{k}$ is 
Fredholm.

\end{definition}

\begin{remark}\label{rem:Fredholmnontriv}
Assume that~$K$ is not trivially valued. Let~$V$ be a Fr\'echet space and let~$f$ be a continuous endomorphism of~$V$. Let~$v$ be a family of seminorms on~$V$ that induces its topology. Then, by \cite{Banachoid}, $f$ induces a bounded endomorphism of $(V,v)$. Moreover, by \cite{Banachoid}, if~$f$ has finite cokernel, then it is topologically strict.

In other words, in the case where~$K$ is not trivially valued, Fredholm operators are nothing but operators of finite indexes.

\end{remark}
%


\begin{lemma}[Descent]
\label{Lemma : descent of chigen}
Let~$f$ be an endomorphism of~$V$ (resp. a bounded endomorphism of~$V$). Let~$L$ be an extension of~$K$ (resp. a complete valued extension of~$K$). Denote by $f_{L}$ the endomorphism $f\otimes \mathrm{Id}_L$ of $V_{L} := V\otimes_{K}L$ (resp. $V_{L} := V\ho_{K}L$). 

Then, $f$ has finite index (resp. is Fredholm) if, and only if, $f_{L}$ has finite index (resp. is Fredholm) and, in this case, we have 
\begin{equation}\label{eq:chiVuL}
\chi(V,f)\;=\;
\chi(V_L,
f_{L})\;.
\end{equation}

For $k\in\{1,\dotsc,\mu\}$, set $V_{k,L} := V_k\otimes_{K}L$ (resp. $V_{k,L} := V_k\ho_{K}L$). The space $V_{L}$ is naturally isomorphic to the direct sum of the $V_{k,L}$'s (resp. in the category of normoid Fr\'echet spaces). For $k\in\{1,\dotsc,\mu\}$, $f$ has finite generalized $k$-index (resp. is $k$-Fredholm) with respect to $V = \bigoplus_{i} V_{i}$ if, and only if, $f_{L}$ has finite generalized $k$-index (resp. is $k$-Fredholm) with respect 
to $V_{L}=\bigoplus_{i} V_{L,i}$ and, in this case, we have 
\begin{equation}\label{eq:chigenVuL}
\chi^{\mathrm{gen}}_k(V,f)\;=\;
\chi^{\mathrm{gen}}_k(V_L,
f_{L})\;.
\end{equation}
\end{lemma} 
\begin{proof}
In the classical case, the first part is a consequence of the exactness of the tensor product $\cdot \otimes_{K} L$. In the normoid Fr\'echet case, it follows from \cite{Banachoid}. 

In the classical case, it is well-known that tensor products commutes with direct sums, hence $V\otimes_{K}L = \bigoplus_{i}V_i\otimes_{K}L$. In the normoid Fr\'echet case, by \cite{Banachoid}, 
completed tensor products commute with colimits, hence direct sums. It follows that we also have a direct sum decomposition $V\ho_{K}L = \bigoplus_{i}V_i\ho_{K}L$. The rest of the result is a consequence of the first part.
\end{proof}

Using Remark~\ref{rem:Fredholmnontriv}, we can rephrase the Fredholm property in terms of finite dimensionality of the cokernel after extension of scalars.

\begin{corollary}\label{cor:Fredholmcokernel}
Let~$f$ be a bounded endomorphism of~$V$. The following conditions are equivalent.
\begin{enumerate}
\item $f$ is Fredholm;
\item for each  complete valued extension~$L$ of~$K$, $f_{L}$ has finite-dimensional kernel and cokernel; 
\item there exists a complete valued extension~$L$ of~$K$ with non-trivial valuation such that $f_{L}$ has finite-dimensional kernel and cokernel. 
\end{enumerate}
Similar results hold for $k$-Fredholm.
\hfill$\Box$
\end{corollary}

Fredholm morphisms satisfy a two out of three principle 
(see \cite[Lemma~22.1]{Schneider}).


\begin{lemma}[Composition]\label{lem:composition}
Let~$f$ and~$g$ be endomorphisms of~$V$ (resp. bounded endomorphisms of~$V$). If two among $f,g,f\circ g$ are of finite index (resp. are Fredholm), then so is the third. In this case, we have
\begin{equation}\label{eq:composition}
\chi(f\circ g) = \chi(f) + \chi(g).
\end{equation}
\end{lemma}
\begin{proof}
In the classical case, the claim follows from the two exact sequences 
\begin{equation}
0 \to \Ker(g) \to \Ker(f\circ g) \xrightarrow[]{g} \Im(g)\cap \Ker(f) \to 0
\end{equation}
and
\begin{equation}
0 \to \Ker(f)/(\Ker(f)\cap\Im(g)) \to \Coker(g) \xrightarrow[]{f} \Im(f)/\Im(f\circ g) \to 0\;.
\end{equation}

In the normoid Fr\'echet case, by Lemma~\ref{Lemma : descent of chigen}, we may extend the scalars in order to assume that~$K$ is not trivially valued. By Remark~\ref{rem:Fredholmnontriv}, this allows us to forget about topological strictness in the definition of Fredholm operators and we are reduced to the classical case.
\end{proof}

A general notion of \emph{compact} operator is defined in \cite[Section 16]{Schneider}, with the restriction that the base field is required to be non-trivially valued and maximally complete. At the moment, no definition of compact operator is available outside this case and we propose to get around this difficulty by extending the 
scalars. Thanks to Lemma~\ref{Lemma : descent of chigen}, this is harmless for proving index theorems.



\begin{definition}[Potentially compact operator]
We say that a bounded endomorphism~$c$ of~$V$ is potentially compact if there exists a complete non-trivially 
valued maximally complete extension~$L$ of~$K$ such that the endomorphism of $V\ho_{K} L$ induced by~$c$ is compact.
\end{definition}

\begin{remark}\label{rem:compositioncompact}
Let~$c,f,g$ be bounded endomorphisms of~$V$ with~$c$ potentially compact. Then $f\circ c\circ g$ is potentially compact. Actually, this holds with ``potentially compact'' replaced by ``compact'' when the base field is not trivially valued and maximally complete by \cite[Remark~16.7]{Schneider}. Our statement is a straightforward consequence.
\end{remark}

%
%
%

\begin{proposition}[Compact perturbation]
\label{Prop : Fred --> compact perturbation}
Let~$f$ and~$c$ be bounded endomorphisms of~$V$. Assume that~$c$ is potentially compact.

Then $f$ is Fredholm if, and only if, $f+c$ is Fredholm and, in this case, we have
\begin{equation}
\chi(V,f)=\chi(V,f+c)\;.
\end{equation}
\end{proposition}
\begin{proof}
Assume that~$f$ is Fredholm. By Lemma~\ref{Lemma : descent of chigen}, we can 
extend the scalars in order to assume that~$K$ is not trivially valued and maximally complete and that~$c$ is compact.

%

By \cite[Corollary 22.11]{Schneider}, 
$f$ is invertible modulo compact operators: there 
exists continuous endomorphisms $g,c',c''$ of~$V$ with~$c'$ and~$c''$ compact such that $f\circ g=1+c'$ and 
$g\circ f=1+c''$. 

Now write $(f+c)\circ g=f\circ g+c\circ g=1+c'+c \circ g$. 
By \cite[Remark 16.7]{Schneider}, $c'+c\circ g$ is compact 
and, by \cite[Corollary 22.9]{Schneider}, 
$1+c'+c\circ g$ is Fredholm with index equal to $0$. Since 
$\Coker((f+c)\circ g)$ maps surjectively onto 
$\Coker(f+c)$, the latter is finite-dimensional.


Similarly, $g\circ (f+c)$ is Fredholm with 
zero index. Since 
$\Ker(f+c)$ is contained in  $\Ker(g\circ(f+c))$, we 
deduce that $\Ker(f+c)$ is finite-dimensional and, finally, that $f+c$ is Fredholm (see Remark~\ref{rem:Fredholmnontriv}).

By \cite[Corollary 22.9]{Schneider}, $1+c' = f\circ g$ is Fredholm with index~0. By Lemma~\ref{lem:composition}, we deduce that~$g$ is Fredholm and that $\chi(V,f)=-\chi(V,g)$. 
Using the same argument with $1+c'+c\circ g = (f+c)\circ g$, we find
\begin{equation}
\chi(V,f+c)=-\chi(V,g) = \chi(V,f).
\end{equation}

To prove the converse, we can add the compact 
operator $-c$ to $f+c$ and deduce that $f$ is Fredholm 
from the fact that $f+c$ is Fredholm.
%
%
\end{proof}


%


We now want to establish an analogue of Lemma~\ref{lem:composition} for generalized indexes. As the reader may expect, this property fails for 
general endomorphisms, hence we will focus on a class of endomorphisms satisfying the following compactness property. It will be automatically satisfied by connections
(see Proposition \ref{Proposition : compass}).



\begin{definition}\label{Def. Compactness property}
We say that a bounded endomorphism~$f$ of~$V$ 
satisfies the \emph{compactness property} if, for every $k,s \in \{0,\dotsc,\mu\}$ with $s\neq k$, the operator $\pi_s \circ u \circ\pi_k$ is potentially compact.
\end{definition}


\begin{lemma}\label{Lemma : additivity gen ind}
Let~$f$ and~$g$ be bounded endomorphisms of~$V$. 
Assume that at least one among~$f$ and~$g$ satisfies the compactness property of 
Definition~\ref{Def. Compactness property}. 

Let $k\in\{0,\dotsc,\mu\}$. If two 
among $f,g,f\circ g$ are $k$-Fredholm, then so is the third.
In this case, we have
\begin{equation}\label{eq: additivity index gen }
\chi_{k}^{\mathrm{gen}}(V, f\circ g)\;=\;
\chi_{k}^{\mathrm{gen}}(V, f)\;+\;
\chi_{k}^{\mathrm{gen}}(V, g)\;.
\end{equation}
\end{lemma}
\begin{proof}
By Proposition~\ref{Prop : Fred --> compact perturbation}, it is enough to prove that $(f\circ g)_k-f_k\circ g_k$ is a potentially compact 
endomorphism of~$V_k$.

Consider the bounded endomorphism of~$V$ defined by 
$\alpha:=\pi_k\circ f\circ g\circ\pi_k - \pi_k\circ f\circ\pi_k\circ g\circ\pi_k.$
Since $(f\circ g)_k-f_k\circ g_k = p_k\circ\alpha\circ i_k$, by Remark~\ref{rem:compositioncompact}, it is enough to prove that~$\alpha$ is potentially compact. Observing that $g=\sum_k\pi_k\circ g$, we can write 
$\alpha = \sum_{s\neq k}\pi_k\circ f\circ \pi_s\circ g\circ\pi_k$ and potential compactness now follows from Remark~\ref{rem:compositioncompact} again.
\end{proof}

\begin{proposition}\label{Prop : chi=sumchigen}
Let~$f$ be a bounded endomorphism of~$V$ satisfying the compactness property 
of Definition~\ref{Def. Compactness property}. 

Then, $f$ is Fredholm if, and only if, for every $k\in\{0,\dotsc,\mu\}$, $f$ is $k$-Fredholm. In this case, we have 
\begin{equation}
\chi(V,f)\;=\;
\sum_{k=0}^\mu\chi_k^{\mathrm{gen}}(V,f)\;.
\end{equation}
\end{proposition}
\begin{proof}
We have $f=
\sum_{k,s}\pi_k \circ f\circ\pi_s=
\sum_{k}\pi_k \circ f\circ\pi_k+
\sum_{k\neq s}\pi_k \circ f\circ \pi_s$. 
By assumption, for all $k\neq s$, the operator 
$\pi_k \circ f\circ \pi_s$ is potentially compact. By Proposition \ref{Prop : Fred --> compact 
perturbation}, $f$ is Fredholm if, and only if, $\sum_{k}\pi_k \circ f\circ\pi_k$ is and, in this case, they have the same index. The result now follows from the equality $\sum_{k}\pi_k \circ f\circ\pi_k=\bigoplus_k f_k$.
\end{proof}

\subsection{Generalized indexes for connections over
open pseudo-annuli.}
\label{Generalized indexes at the open boundaries}


\subsubsection{Standard pseudo-annuli.}
\label{section : Standard pseudo-annuli}
We fix a coordinate~$T$ on $\mathbb{A}^{1,\an}_{K}$. The definitions that follow will depend on it.

\begin{definition}\label{Def : standard pseudo-annulus}
A $K$-analytic space~$C$ is said to be a standard open 
pseudo-annulus over~$K$ if it is isomorphic to an open 
subset of~$\AfK$ of the form $\{r_1<|T|<r_2\}$, where 
$r_1,r_2$ are elements of $[0,\infty]$ such that 
$r_1<r_2$.
%
\end{definition}

Let~$C=\{r_1<|T|<r_2\}$, with $r_1,r_2\in[0,\infty]$ and $r_{1}<r_{2}$, be a standard open
pseudo-annulus over~$K$. 
We consider the open pseudo-disks (cf. \cite[Definition 1.1.8]{NP-IV})
\begin{equation}
D_0:=\{|T|<r_2\}\;,
\end{equation}
and
\begin{equation}
D_1:=\{|T|>r_1\}\cup\{+\infty\}\;,
\end{equation}
that is the complement of $\{|T| \le r_{1}\}$ 
in~$\mathbb{P}^{1,\an}_{K}$. 
Remark that $C = D_{0}\cap D_{1}$. 
For $k=0,1$, we denote by~
\begin{equation}
b_k
\end{equation}
the germ of segment at the open boundary of~$D_{k}$.

%
We have
\begin{eqnarray}
\O(D_0)&\;=\;&\{\sum_{n\geq 0}a_nT^n\;,\;
a_n\in K\;,\;
\lim_{n\to+\infty}|a_n|\rho^n=0\;,\;\forall\rho<r_2\}\\
\O(D_1)&\;=\;&\{\sum_{n\leq 0}a_nT^n\;,\;
a_n\in K\;,\;
\lim_{n\to-\infty}|a_n|\rho^n=0\;,\;\forall\rho>r_1\}
\end{eqnarray}
and
\begin{equation}
\O(C)\;=\;\{\sum_{n\in\Z}a_nT^n\;,\;
a_n\in K\;,\;\lim_{n\to\pm\infty}|a_n|\rho^n=0\;,\;\forall\rho \in ]r_{1},r_2[\}\;.
\end{equation}

Let us endow~$\O(C)$ with an admissible normoid Fr\'echet structure, for instance the structure associated to an affinoid covering of~$C$ (see \cite{Banachoid}). 



Let $m,n\in\mathbb{Z}$ with $m<n$. The spaces $T^{m}\O(D_1)$, $K\cdot T^{m+1}$, \dots, $K\cdot T^{n-1}$ and $ T^n\O(D_0)$ of~$\Os(C)$ are all closed subspaces of~$\Os(C)$, hence naturally inherit structures of normoid Fr\'echet spaces. We have a direct sum decomposition in the category of 
normoid Fréchet spaces
\begin{equation}\label{eq: O(C) deco O(D_0)+O(D_1)}
\O(C)\;=\;T^{m}\O(D_1)\oplus 
\Bigl(\bigoplus_{s=m+1}^{n-1}K\cdot T^s\Bigr)
\oplus T^n\O(D_0)\;.
\end{equation}



Let~$V$ be a free $\O(C)$-module of finite rank~$r$ and let~$j$ be an $\O(C)$-linear isomorphism
\begin{equation}
\label{eq : def j}
j\;:\;V\xrightarrow{\;\sim\;} \O(C)^r\;.
\end{equation}
The normoid Fr\'echet structure on~$\O(C)$ induces one on~$\O(C)^r$, hence also one on~$V$ \textit{via}~$j$. 

\begin{remark}
The choice of~$j$ does not really affect the normoid structure on~$V$ in the sense that another choice would produce an equivalent structure. In particular, the bounded endomorphisms of~$V$ remain the same. 

\end{remark}






The direct sum decomposition~\eqref{eq: O(C) deco O(D_0)+O(D_1)} induces a direct sum decomposition of~$V$ in the category of normoid Fr\'echet spaces. We denote by 
\begin{equation} \label{eq : V_0 tt V_1}
V_1 \;:=\; j^{-1}(T^m\O(D_1)^r) \qquad 
\textrm{ and } \qquad V_0 \;:=\; 
j^{-1}(T^{n}\O(D_0)^r)
\end{equation}
the extremal summands of this decomposition.

\begin{definition}
\label{Def : }
Let~$f$ be a bounded endomorphism of~$V$. For $k\in\{0,1\}$, we say that~$f$ is Fredholm at~$b_{k}$ if~$f_{k}$ is Fredholm (see Definition~\ref{Def: genindgeneral}). In this case, we set
\begin{equation}
\chi^{\mathrm{gen}}_{b_k}(V,f)\;:=\;
\chi^{\mathrm{gen}}_k(V,f)\;=\;\chi(V_k,f_k)
\end{equation}
and call this quantity the 
generalized index of~$f$ at~$b_k$. 
\end{definition}

%

%

The notation $\chi^{\mathrm{gen}}_{b_k}$ is justified by the following lemma, which shows that this index is stable by restriction to a sub-annulus of $C$ containing $b_k$. 

\begin{lemma}\label{Lemma : local index --t-}
Let~$k \in \{0,1\}$. Let $C'\subseteq C$ be an open sub-pseudo-annulus such that
$b_{k}\subseteq\Gamma_{C'}$. Endow~$\O(C')$ with the normoid Fr\'echet structure induced by that on~$\Os(C)$ and $V':=V\otimes_{\O(C)}\O(C')$ with the one given by the tensor product.


Let $f$ and $f'$ be bounded endomorphisms of $V$ 
and $V'$ respectively, 
commuting with the natural restriction 
$\rho_k:x \in V\mapsto x\otimes 1\in  V'$. Then, $f$ is Fredholm at~$b_{k}$ if, and only if, $f'$ is and, in this case, we have
\begin{equation}
\chi^{\mathrm{gen}}_{b_{k}}(V,f)\;=\;
\chi^{\mathrm{gen}}_{b_{k}}(V',f')\;.
\end{equation}
\end{lemma}
\begin{proof}
The result follows immediately form the fact that both $i_k$ 
and $p_k$ commute with $\rho_k$.
\end{proof}


The definition of 
$\chi^{\mathrm{gen}}_{b_k}(V,f)$ seems to
depend on the choices of~$n$, $m$, $j$ and the embedding of~$C$ 
into~$\AfK$ (\textit{i.e.} a coordinate on~$C$). We will now show that it does not depend on~$n$, $m$ and~$j$. In the case of connections, we will show later that it is also independent of the choice of a coordinate on~$C$ (see Proposition~\ref{Prop : independence on the coordinate chiabs}).


\begin{lemma}\label{lem:b0b1}
Let~$f$ be a bounded endomorphism of~$V$. For $b\in\{b_{0},b_{1}\}$, the definition of $\chi^{\mathrm{gen}}_b(V,f)$ is 
independent of the choices of $n$ and $m$. 

Moreover, if~$f$ satisfies the compactness property of Definition~\ref{Def. Compactness property}, 
then $f$~is Fredholm if, and only if, it is Fredholm at~$b_0$ and~$b_1$ and, in 
this case, one has
\begin{equation}
\chi(V,f)\;=\;
\chi^{\mathrm{gen}}_{b_0}(V,f) 
+\chi^{\mathrm{gen}}_{b_1}(V,f)\;.
\end{equation}
\end{lemma}
\begin{proof}
Let $m',n'\in\mathbb{Z}$ with $m'<n'$. We may assume that $m\le m'$. In this case, $T^m \O(D_{1})$ is a subspace of $T^{m'}\O(D_{1})$ of finite codimension and the projection of~$\O(C)$ onto~$T^m \O(D_{1})$ factors through $T^{m'}\O(D_{1})$. This injection and this projection are both Fredholm and their indexes are opposite. It now follows from Lemma~\ref{lem:composition} that~$f$ is Fredholm at~$b_{1}$ with respect to the decomposition associated to~$n$ and~$m$ if, and only if, it is with respect to the decomposition associated to~$n'$ and~$m'$ and that, in this case, the generalized indexes at~$b_{1}$ coincide. The result for~$b_{0}$ is proved the same way.

The second part of the claim follows from Proposition~\ref{Prop : chi=sumchigen} and the fact that any operator on a 
finite-dimensional vector space is Fredholm with zero index.
%
\end{proof}


\begin{lemma}
\label{Lem : gen ind of a constant}
Let~$f$ be a bounded 
endomorphism of~$V$ that is $\O(C)$-linear. 
Identify $\det(f)$ with an element of $\O(C)$. 
Let $b\in \{b_{0},b_{1}\}$. Then, $f$ is Fredholm at~$b$ if, and only if, 
$\det(f)$ has no zeros on some 
sub-pseudo-annulus of~$C$ containing~$b$. 
Moreover, in this case, 
the generalized index of~$f$ at~$b$ 
is equal to 
\begin{equation}\label{eq : chirel = -partial_b}
\chi^{\mathrm{gen}}_{b}(V,f)\;=\;
-\partial_{b}(\det(f))\;.
\end{equation}
In particular, $\chi^{\mathrm{gen}}_{b}(V,f)$ does not 
depend on the coordinate $T$ of $C$.
\end{lemma}
\begin{proof}
By Lemma~\ref{Lemma : descent of chigen}, we may extend the scalars and assume that~$K$ is non-trivially valued and maximally complete. In this case, the result follows from 
\cite[Proposition 8.2-6]{Ch-Me-III} (see also Remark~8.9-10 at the end of Section~8.2 of~\textit{ibid.}) or \cite[Section 6]{Tsu-swan}. 
\end{proof}

We now prove that the generalized indexes are 
independent of the choice of~$j$ and that they are 
compatible with exact sequences of $\O(C)$-modules
(cf. Lemmas \ref{Lemma : independence on j} and 
\ref{Lemma : additivity on exact sequence}). Since we use Lemma~\ref{Lemma : additivity gen ind}, we assume that~$f$ satisfies the compactness 
property of Definition~\ref{Def. Compactness property}.

\begin{lemma}\label{Lemma : independence on j}
Let~$f$ be a bounded endomorphism of~$V$ that satisfies the compactness 
property of Definition \ref{Def. Compactness property}. Let $b\in \{b_{0},b_{1}\}$. Then, the fact that~$f$ is Fredholm at~$b$ does not depend on the choice of~$j$ and, when it is, the value of the generalized index 
$\chi_{b}^{\mathrm{gen}}(V,f)$ does not depend on the choice of~$j$ either.

\end{lemma}
\begin{proof}
%
Let~$g$ be an $\Os(C)$-linear bounded automorphism of~$V$. By Lemma~\ref{Lem : gen ind of a constant}, it is Fredholm at~$b$. By Lemma~\ref{Lemma : additivity gen ind}, $f$ is Fredholm at~$b$ if, and only if, $g \circ f\circ g^{-1}$ is and, in this case, we have
\begin{equation}
\chi_{b}^{\mathrm{gen}}(V,g \circ f\circ g^{-1})
=
\chi_{b}^{\mathrm{gen}}(V,g)+
\chi_{b}^{\mathrm{gen}}(V, f)-
\chi_{b}^{\mathrm{gen}}(V,g)=
\chi_{b}^{\mathrm{gen}}(V,f).
\end{equation} 
The result follows.
%
\end{proof}


\begin{lemma}
\label{Lemma : additivity on exact sequence}
Let $0\to U\xrightarrow{\varphi} V\xrightarrow{\psi}
W\to0$ be an exact 
sequence of finite free $\O(C)$-modules. 
Let $f_U,f_V,f_W$ be   
endomorphisms of $U,V,W$ respectively 
commuting with the maps~$\varphi$ and~$\psi$. Assume 
that $f_U,f_V,f_W$ satisfy the compactness property 
of Definition \ref{Def. Compactness property}.

Let $k\in\{0,1\}$. If two operators among $f_V,f_U,f_W$ are Fredholm at $b_{k}$, then so is the third. 
In this case, we have
\begin{equation}
\chi^{\mathrm{gen}}_{b_{k}}(V,f_V)\;=\;
\chi^{\mathrm{gen}}_{b_{k}}(U,f_U)+
\chi^{\mathrm{gen}}_{b_{k}}(W,f_W)\;.
\end{equation}
\end{lemma}
\begin{proof}
Choose a basis of~$V$ that is obtained by putting together a basis of~$U$ and a lift of a basis of~$W$ and consider the corresponding $\O(C)$-linear isomorphisms $j_{U},j_{V},j_{W}$. Recall that the result does not depend on thoses choices by Lemma \ref{Lemma : independence on j}. 

The exact sequence $0\to U\to V\to W\to0$ now induces an exact sequence $0\to U_{k} \to V_{k}\to W_{k}\to0$ that commutes with the truncated 
operators~$f_{k}$ and~$g_{k}$ (see \eqref{eq : truncated op u_k}). The claim then follows from Lemma~\ref{Lemma : additivity of Index}.
%
%
%
%
\end{proof}

\subsubsection{Absolute generalized indexes for connections over standard pseudo-annuli.}


We now want to modify the definition of generalized index 
in order to obtain a notion that is intrinsic in the case of connections. 



As before, we consider a standard open pseudo-annulus $C=\{r_1<|T|<r_2\}$ with $0 \le r_1 < r_2 \le + \infty$. We retain the notations of the previous section.

%
%
%
%
%
%

\begin{remark}

Let  $b\in\{b_0,b_1\}$. Then $d/dT$ and $T\,d/dT$ are Fredholm at~$b$ and we have
\begin{equation}
\chi^{\mathrm{gen}}_b(\O(C),d/dT)\;=\;\left\{
\begin{array}{rcl}
1&\textrm{ if }b=b_0,\medskip\\
-1&\textrm{ if }b=b_1.
\end{array}\right.
\qquad\quad
\chi^{\mathrm{gen}}_b(\O(C),Td/dT)\;=\;0\;.
\end{equation}
If $d:\O(C)\to\O(C)$ is a continuous derivation, there 
exists $h\in\O(C)$ such that $d = h \, d/dT$. 
By Lemma \ref{Lemma : additivity gen ind}, $d$ is Fredholm at $b$ if, and only if, so is the multiplication by~$h$ (see Lemma~\ref{Lem : gen ind of a constant}). 
\end{remark}


%


Let $\nabla:\Fs\to\Omega_C^{1}\otimes\Fs$ 
be a differential equation over $C$ of rank $r$. The 
$\Os(C)$-module~$\Omega^1(C)$ is free of rank~1 
with basis~$dT$. This induces an isomorphism 
$\ell \colon \Omega^1(C) \simto \Os(C)$ given by
$\ell(f(T)dT)=f(T)$.


Let $d:\O(C)\to\O(C)$ be a continuous derivation. There exists $h\in\O(C)$ such that $d = h \, d/dT$. Denote by $\ell_d:\Omega^1(C)\to\O(C)$ the $\O(C)$-linear map obtained by composing~$\ell$ with the multiplication by~$h$. We set
\begin{equation}\label{eq : nabla(d)}
\nabla(d)\;:=\;(\ell_d\otimes 1)\circ\nabla\;.
\end{equation}

The map $\nabla(d):\Fs(C)\to\Fs(C)$ provides $\Fs(C)$ 
with a 
structure of $(\O(C),d)$-differential module. 
Note that, for every $g\in\O(C)$, we have $\nabla(gd)=g\nabla(d)$.
If the derivation 
$d$ is clear, we often drop the symbol $d$ and write  
$\nabla$ instead of $\nabla(d)$.


\medbreak

From now on, we assume that~$\Fs$ is free of finite rank~$r$. We identify~$\Fs(C)$ with~$\Os(C)^r$ and endow~$\Fs(C)$ with a structure of normoid Fr\'echet space accordingly.

There exists a matrix 
$G(T)\in M_r(\O(C))$ such that the map $\nabla(d)$ is 
given by 
\begin{equation}
d-G(T)\;:\;\O(C)^r\to\O(C)^r\;.
\end{equation}

We denote by $\O(C)\langle d/dT\rangle$ the ring of 
differential polynomials with coefficients in $\O(C)$: 
the elements of $\O(C)\langle d/dT\rangle$ are abstract
sums $\sum_{i=0}^n f_i\circ (d/dT)^i$ (so that we have an 
isomorphism of abelian  groups 
$\O(C)\langle d/dT\rangle = \bigoplus_{i\geq 0}\O(C)\cdot(d/dT)^i $) and the multiplication law 
$\circ$ is the unique one
satisfying the properties 
$f\circ (d/dT)=(d/dT)\circ f +df/dT$, for all $f\in\O(C)$.

\begin{proposition}\label{Proposition : compass}
Let~$f$ be an endomorphism of~$\O(C)^r$ given by the multiplication by 
an $r\times r$ matrix with coefficients in 
$\O(C)\langle d/dT\rangle$. Then, $f$ is a bounded endomorphism of~$\O(C)^r$ that satisfies the 
compactness property of Definition~\ref{Def. Compactness property}. 

In particular, for every bounded derivation~$d$, the endomorphism~$\nabla(d)$ of~$\Fs(C)$ defined above satisfies the 
compactness property of Definition~\ref{Def. Compactness property}. 
\end{proposition}
\begin{proof}
We may extend the scalars in order to assume that~$K$ is not trivially valued and maximally complete.

The result now follows from~\cite[Proposition 
8.2-2]{Ch-Me-III} (see also Remark~8.9-10 at the end of Section~8.2 of~\textit{ibid.}). 
\end{proof}


%
%
%
%
%


\begin{definition}\label{def:Fredholmnabla}
Let  $b\in\{b_0,b_1\}$. We say that $(\Fs,\nabla)$ is Fredholm at~$b$ if $\nabla(d/dT) : \Fs(C) \to\Fs(C)$ is. In this case, we define the \emph{absolute generalized index} of $(\Fs,\nabla)$ at $b$ as
\begin{equation}
\chiabs_{b}(\Fs,\nabla)\;:=\;
\chi^{\mathrm{gen}}_b(\Fs(C),\nabla(d/dT))-
\chi^{\mathrm{gen}}_b(\O(C)^r,d/dT)\;.
\end{equation}
We often remove~$\nabla$ from the notation when it is clear from the context.


\end{definition}

The claim below shows that, in this definition, we could have chosen any other bounded derivation instead of~$d/dT$, as soon as it is Fredholm at the boundary of the annulus.

%


\begin{lemma}
\label{Lemma: Def of chiabs chigen-chigen d}
Let  $b\in\{b_0,b_1\}$. Let $d:\O(C)\to\O(C)$ be a bounded derivation that is Fredholm at~$b$. Then $(\Fs,\nabla)$ is Fredholm at~$b$ if, and only if, $\nabla(d): \Fs(C) \to\Fs(C)$ is. Moreover, in this case, we have 
\begin{equation}
\chiabs_{b}(\Fs,\nabla)\;:=\;
\chi^{\mathrm{gen}}_b(\Fs(C),\nabla(d))-
\chi^{\mathrm{gen}}_b(\O(C)^r,d)\;.
\end{equation}

\end{lemma}
\begin{proof}
Let $d:\O(C)\to\O(C)$ be a bounded derivation. There exists $h\in \O(C)$ such that $d=h\,d/dT$. By 
Lemma~\ref{Lemma : additivity gen ind} and Proposition 
\ref{Proposition : compass}, since~$d$ is Fredholm at~$b$, so is~$h$. Moreover, we have
\begin{equation}
\chi^{\mathrm{gen}}_b(\O(C)^r,hd/dT) \;=\; 
\chi^{\mathrm{gen}}_b(\O(C)^r,h) + 
\chi^{\mathrm{gen}}_b(\O(C)^r,d/dT)\;.
\end{equation}

Similarly, $\nabla(hd/dT) = h\nabla(d/dT)$ is Fredholm at~$b$ if, and only if, $\nabla(d/dT)$ is and, in this case, we have
\begin{equation}
\chi^{\mathrm{gen}}_b(\Fs(C),\nabla(hd/dT)) \;=\; \chi^{\mathrm{gen}}_b(\Fs(C),h) + \chi^{\mathrm{gen}}_b(\Fs(C),\nabla(d/dT))\;.
\end{equation}
Finally, by Lemma \ref{Lemma : independence on j}, one has 
$\chi^{\mathrm{gen}}_b(\Fs(C),h)=
\chi^{\mathrm{gen}}_b(\O(C)^r,h)$ and the result follows.
%
%
\if{
The independence on $j$ follows from 
Lemma \ref{Lemma : independence on j} too.

Now let $t$ be another coordinate on $C$ respecting the 
orientation of $C$, then $t = T(1+a)$, where $a\in\O(C)$ 
verifies $|a|(x_\rho)<\rho$ 
for all $r_1<\rho<r_2$. It follows that $d/dT = d/dt$

\comm{Terminer la preuve ...}
}\fi
\end{proof}

\begin{remark}\label{Remark : chiabs=chigen}
Let $b\in\{b_0,b_1\}$. If~$(\Fs,\nabla)$ is Fredholm at~$b$, then, since $\chi^{\mathrm{gen}}_b(\O(C),T\frac{d}{dT})=0$, one 
has
\begin{equation}
\chiabs_b(\Fs,\nabla)\;=\;
\chi^{\mathrm{gen}}_b(\Fs(C),\nabla(T\frac{d}{dT}))\;.
\end{equation}
In particular, the absolute generalized index has all 
the properties of the usual generalized index. 
This is in fact the choice of derivation of \cite{Ch-Me-III}.
\end{remark}

\subsubsection{Push-forward by standard ramification 
of irregularities and absolute indexes.}
\label{Frobenius push-forward.}

In this section, we study the behavior of 
\emph{generalized indexes} and 
\emph{irregularities} of a connection
under push-forward by \emph{standard ramification}, 
\emph{i.e.} by a finite étale morphism of the form by $T\mapsto T^n$. 

Let $C=\{r_1<|T|<r_2\}$, with $0 \le r_1 < r_2 \le + \infty$, be a standard open pseudo-annulus.
Let $n\geq 1$. We denote by $\varphi_n$ 
the endomorphism of 
$\mathbb{A}^{1,\mathrm{an}}_K$ that raises to the 
$n^\textrm{th}$ power. It induces a finite étale morphism 
between~$C$ and the standard open pseudo-annuli
\begin{equation}
C^n:=\varphi_n(C) = \{r_1^n<|T|<r_2^n\}\;.
\end{equation} 
Let 
\begin{equation}
\varphi_n^\sharp\; \colon \;\O(C^n)\xrightarrow{\quad}\O(C)
\end{equation}
be the induced map on the rings of functions. For clarity, we will denote by 
$\wti{T}$ and $T$ the coordinate functions on~$C$ 
and $C^n$ respectively. For every $f(T) \in 
\O(C^n)$, we have $\varphi^\sharp_n 
(f(T))=f(\widetilde{T}^{n})$. 
\if{We denote by $b_{0}^n$ and $b_{1}^n$ the 
germs of segments at the open boundary of 
$D_0(C^n)=\{|T|<r_2^n\}$ and $D_1(C^n)=\{|T|>r_1^n\}\cup\{\infty\}$ respectively. }\fi

\begin{proposition}[\protect{\cite[Section 5.3]{NP-I}}]
\label{Prop : slope of push-f-1}
Let $b$ be a germ of segment in $\Gamma_C$. Denote by~$b^n$ its image in $\Gamma_{C^n}$. 
Let $\Fs$ be a differential equation over 
$C$. Then $\Fs$ has log affine radii along $b$ if, and 
only if, $(\varphi_n)_*\Fs$ has log-affine radii 
along $b^n$. In this case, one has
\begin{equation}\label{eq : temp hhh}
\Irr_{b^n}((\varphi_n)_*\Fs)\;=\;
\Irr_{b}(\Fs)\;.
\end{equation}
\if{
Let $p$ denote the characteristic of the residual 
field $\widetilde{K}$ of $K$, and set
\begin{equation}
a(d,K)\;=\;\left\{
\begin{array}{ll}
1&\textrm{ if $p=0$}\\
d/p^{v_p(d)}&\textrm{ if $p>0$}
\end{array}\right.
\end{equation}
where $v_p(d)$ denotes the 
exponent of $p$ in the factorization of $d$ into prime 
numbers.

Let $b$ be a germ of segment in $\Gamma_C$. 
Let $\Gs$ (resp. $\Fs$) be a differential equation over 
$C^d$ (resp. $C$) with log-affine radii along 
$b^d$ (resp. $b$), then $\varphi_d^*(\Gs)$ 
(resp. $(\varphi_d)_*(\Fs)$) has log-affine radii 
along $b$ (resp. $b^d$) and one has
\begin{equation}\label{eq : temp hhh}
\Irr_{b}(\varphi_d^*(\Gs))\;=\;a(d,K)\cdot
\Irr_{b^d}(\Gs)\;,
\qquad
\Irr_{b^d}((\varphi_d)_*(\Fs))\;=\;
\Irr_{b}(\Fs)\;.
\end{equation}
}\fi
\end{proposition}
\begin{proof}
A factorization of $n$ into prime numbers 
corresponds to a factorization of 
the corresponding morphisms $\varphi_n$, therefore 
we can assume that $n$ is a prime number.
Let $p$ be the characteristic of the residual field 
$\widetilde{K}$ of $K$. If $p=0$, or 
if $n$ is prime to $p$, the claim follows from 
\cite[Lemmas 3.22 and 3.23]{NP-II}. 
If $n=p>0$, the claim 
follows from \cite[Chapter 10]{Kedlaya-book} (cf. also
\cite[Section 3]{NP-I}).
\end{proof}

We now focus on the absolute index. 
In the following, we use notation \eqref{eq : nabla(d)}.

\begin{corollary}\label{Cor : push-f of inde gen-1}
Let $b$ be a germ of segment in $\Gamma_C$. Denote by~$b^n$ its image in $\Gamma_{C^n}$. 
Let $\Fs$ be a differential equation over~$C$. Then $\Fs$ has finite generalized index 
at $b$ if, and only if, its push-forward $(\varphi_n)_*\Fs$ 
has finite generalized index at $b^n$. In this case, we have
\begin{align}
\chiabs_{b^n}
((\varphi_n)_*\Fs)
\;=\;
\chiabs_{b}(\Fs)\;.
\end{align}
\end{corollary}
\begin{proof}
A simple computation shows that, for every 
$h(T) \in \O(C^n)$, we have
\begin{equation}\label{eq:hnabla}
\Bigl(\frac{h(\widetilde{T}^n)}{n\widetilde{T}^{n-1}}\cdot 
\frac{d}{d\widetilde{T}}\Bigr)\circ\varphi_n^\sharp
\;=\;
\varphi_n^\sharp \circ \Bigl(h(T)\frac{d}{dT}\Bigr)\;.
\end{equation}
We set $d:=T\frac{d}{dT}$ and 
$\widetilde{d}:=
\frac{\widetilde{T}}{n}
\cdot 
\frac{d}{d\widetilde{T}}$. It follows from 
\eqref{eq:hnabla} that the push-forward 
$(\varphi_n)_*\Fs$ is nothing but $\Fs$ seen as an 
$\O(C^n)$-module via $\varphi_n^\sharp$ and endowed 
with the connection
\begin{equation}
((\varphi_n)_*\nabla)(d)\;=\;\nabla(\widetilde{d})\;.
\end{equation}
By Remark \ref{Remark : chiabs=chigen}, we have 
$\chiabs_b(\Fs)=
\chi^{\mathrm{gen}}_{b}(\nabla(\widetilde{d}))$
and 
$\chiabs_{b^n}((\varphi_n)_*(\Fs))=
\chi^{\mathrm{gen}}_{b^n}(\nabla(\widetilde{d}))$. 
We now compute 
$\chi^{\mathrm{gen}}_{b}(\nabla(\widetilde{d}))$ and 
$\chi^{\mathrm{gen}}_{b^n}(\nabla(\widetilde{d}))$ 
as in Definition \ref{Def : }.
We have 
\begin{equation}
\O(C)\;=\;
\bigoplus_{i=0}^{n-1}
\widetilde{T}^i\varphi_n^\sharp(\O(C^n))\;.
\end{equation}
The same is true for the decompositions
\begin{equation}\label{eq : dededert}
\widetilde{T}^k\O(D_0(C))\;=\;
\bigoplus_{i=0}^{n-1}
\widetilde{T}^{i+k}\varphi_n^\sharp(\O(D_0(C^n)))
\quad 
\textrm{ and }
\quad
\widetilde{T}^m\O(D_1(C))\;=\;
\bigoplus_{i=0}^{n-1}\widetilde{T}^{m+i} 
\varphi_n^\sharp(\O(D_1(C^n)))\;.
\end{equation}
It follows that the truncations 
\eqref{eq : truncated op u_k} that we consider in 
Definition \ref{Def : } before and after 
push-forward actually are the same. 
In particular, the generalized indexes, at 
$b$ and $b^n$, are equal before and after push-forward.
\end{proof}
\if{
In the next section we also need the following 
\begin{proposition}
Let $\Gs$ be a differential equation over 
$C^n$ with log-affine radii along $\Gamma_{C^n}$. 
Then, $\varphi_n^*\Gs$ also has log-affine radii along  
$\Gamma_C$. 
Moreover, for all $i=1,\ldots,r=\mathrm{rank}(\Gs)$, 
for all  $x\in\Gamma_C$ one has
\begin{equation}
\R_{\emptyset,i}(x,\varphi_n^*\Gs)\;=\;
\left\{\begin{array}{lll}
\R_{\emptyset,i}(\varphi_n(x),\Gs)&\textrm{ if }&\;;\\
\R_{\emptyset,i}(\varphi_n(x),\Gs)&\textrm{ if }&\;.\\
\end{array}\right.
\end{equation}
\end{proposition}
\begin{proof}
\end{proof}
}\fi
\subsubsection{Absolute generalized indexes for connections over general open pseudo-annuli.}
\label{section : GENIND}

Let $C$ be an open pseudo-annulus. Denote by~$b_{0}$ and~$b_{1}$ the germs of segments at the open boundary of~$C$.

Let~$\Omega$ be a spherically complete and algebraically 
closed field extension of~$K$ such that $|\Omega|=
\mathbb{R}_{\geq 0}$. By~\cite[Proposition~3.2]{Liu}), $C_\Omega$ may 
be identified with an analytic domain of 
$\mathbb{P}^{1,\mathrm{an}}_\Omega$. 
It is hence a finite disjoint union 
$C_\Omega=C_1\sqcup\ldots\sqcup C_n$ of standard 
open pseudo-annuli over~$\Omega$ (see Definition 
\ref{Def : standard pseudo-annulus}). 

Let $i\in\{1,\ldots,n\}$. For every $k\in\{0,1\}$, there 
exists a unique germ of segment in~$C_{i}$ 
over~$b_{k}$. We denote it by~$b_{k,i}$. The open 
boundary of~$C_{i}$ contains exactly the two germs of 
segments~$b_{0,i}$ and~$b_{1,i}$.

Let $(\Fs,\nabla)$ be a differential equation on~$C$. The space~$C_{\Omega}$ being the disjoint union of the~$C_{i}$'s, we have a direct sum decomposition of $\Omega$-vector spaces
\begin{equation}
(\Fs_{\Omega})_{|C}\;=\;\bigoplus_{i=1}^n \Fs_{i}\;,
\end{equation}
where $\Fs_{i} := (\Fs_\Omega)_{|C_{i}}$. For every 
$i\in\{1,\dotsc,n\}$, denote by~$\nabla_{i}$ the 
connection on~$\Fs_{i}$ induced by~$\nabla$. 
Since~$K$ is maximally complete, by~\cite{Lazard}, 
$\Fs_i$ is a free $\O_{C_i}$-module. 


We fix a coordinate~$T_{1}$ on~$C_1$ with respect to which we will compute the generalized indexes at~$b_{0,1}$ and~$b_{1,1}$ (see Section~\ref{section : Standard pseudo-annuli}). By~\cite[Corollary~2.20]{NP-II}, for every $i\in \{2,\dotsc,n\}$, there exists a continuous $K$-linear automorphism of~$\Omega$ that sends~$C_{1}$ to~$C_{i}$. We fix such an automorphism~$\sigma_{i}$.The image of~$T_{1}$ by~$\sigma_{i}$ is a coordinate~$T_{i}$ on~$C_{i}$ with respect to which we will compute the generalized indexes at~$b_{0,i}$ and~$b_{1,i}$. 

With these choices, given $i,j\in \{1,\dotsc,n\}$ and $k\in\{0,1\}$, $\Fs_i$ is Fredholm at~$b_{k,i}$ if, and only if,  
$\Fs_{j}$ is Fredholm at $b_{k,j}$ and, in this case, we have 
\begin{equation}
\chiabs_{b_{k,i}}(\Fs_i)\;=\;
\chiabs_{b_{k,j}}(\Fs_j)\;.
\end{equation}

%




\begin{definition}\label{def:Fredholmnablageneral}

Let $k\in\{0,1\}$. We say that $(\Fs,\nabla)$ is Fredholm at~$b_{k}$ if, for 
some (or equivalently every) $i \in \{1,\dotsc,n\}$, $\Fs_i$ is Fredholm at $b_{k,i}$.  
In this case, we define the absolute generalized index 
of~$(\Fs,\nabla)$ at~$b_{k}$ as
\begin{equation}\label{eq : hpanwts}
\chiabs_{b_{k}}(\Fs)\;:=\;
\sum_{i=1}^n\chiabs_{b_{k,i}}(\Fs_i)\;.
\end{equation}
\end{definition}

%
%


\begin{remark}\label{rem:Omega}
By Lemma~\ref{Lemma : descent of chigen}, the 
property of being Fredholm at~$b_{k}$ and, in this case, the value of 
$\chiabs_{b_{k}}$ remain unchanged if we 
replace~$\Omega$ by a bigger extension and use the 
coordinates induced by the $T_{i}$'s. In particular, 
if~$C$ is a standard open pseudo-annulus and if~$\Fs$ is 
free, then Definition~\ref{def:Fredholmnablageneral} 
agrees with Definition~\ref{def:Fredholmnabla} (for suitable coordinates). 
\end{remark}


\begin{lemma}\label{lem:b0b1chidr}
The following assertions are equivalent.
\begin{enumerate}
\item $\Fs$ is Fredholm at~$b_0$ and~$b_1$;
\item there exists a complete valued extension~$L$ of~$K$ with non-trivial valuation such that~$\Fs_{L}$ has finite-dimensional de Rham cohomology over~$C_{L}$;
\item for every complete valued extension~$L$ of~$K$, $\Fs_{L}$ has finite-dimensional de Rham cohomology over~$C_{L}$.
\end{enumerate}
Moreover, when these conditions are satisfied, for every complete valued extension~$L$ of~$K$, we have
\begin{equation}\label{eq:Fredholmnabla}
\chidr(C_{L},\Fs_{L})\;=\;
\chiabs_{b_0}(\Fs)+\chiabs_{b_1}(\Fs)\;.
\end{equation}
\end{lemma}
\begin{proof}
By Lemma \ref{lem:b0b1}, Remark~\ref{Remark : chiabs=chigen} and Proposition \ref{Proposition : compass}, $\Fs$ is Fredholm at~$b_{0}$ and~$b_{1}$ if, and only if, for some (or equivalently all) $i\in \{1,\dotsc,n\}$, $\Fs_{i}$ has finite-dimensional de Rham cohomology over~$C_{i}$ if, and only if, $\Fs_{\Omega}$ has finite-dimensional de Rham cohomology over~$C_{\Omega}$. Moreover, in this case, we have 
\begin{align}
\chidr(C_{\Omega},\Fs_{\Omega}) &= \sum_{i=1}^n \chidr(C_{i},\Fs_{i})\\
&= \sum_{i=1}^n \chiabs_{b_{0,i}}(\Fs_{i})+\chiabs_{b_{1,i}}(\Fs_{i})\\
&= \chiabs_{b_0}(\Fs)+\chiabs_{b_1}(\Fs).
\end{align}

It follows that i) implies ii), that iii) implies i) and that~\eqref{eq:Fredholmnabla} holds for $L=\Omega$.

By Theorem~\ref{thm:descent}, ii) implies iii) and the value of~$\chidr(C_{L},\Fs_{L})$ is independent of~$L$. The result follows.
\end{proof}

\begin{remark}
Having developed the theory, we will actually be able to prove that, under natural assumptions on~$(\Fs,\nabla)$, $\chiabs_{b_{k,i}}(\Fs,\nabla)$ is independent of the coordinate chosen on~$C_{i}$ (see Proposition~\ref{Prop : independence on the coordinate chiabs}).
\end{remark}


From now on, we will not mention the field~$\Omega$ nor the coordinates~$T_{i}$ when speaking about being Fredholm at a germ of segment. Still, different choices may lead to different definitions and we advise the reader to use those results with care. We often use implicitly the natural choices: for instance, when we pass from a pseudo-annulus to a smaller one, we use the same coordinates.

\if{\begin{lemma}
Let $C'\subseteq C$ be an inclusion of open 
pseudo-annuli such that 
$\Gamma_{C'}\subset\Gamma_C$. 
Let $(b_0,b_1)$ and $(b_0',b_1')$ be the germs of 
segments at the open boundary of $C$ and $C'$ 
respectively, and we assume that, for $i=0,1$, $b_i$ 
and $b_i'$ have the same orientation.
 
Let $\Fs$ be a differential equation over $C$. Consider 
the following assertions
\begin{enumerate}
\item $\Fs$ is Fredholm at $b_0,b_1,b_0',b_1'$ and one 
has 
\begin{equation}
\chiabs_{b_0}(\Fs)\;=\;
-\chiabs_{b_1}(\Fs)\;,
\qquad
\chiabs_{b_0'}(\Fs_{|C'})\;=\;
-\chiabs_{b_1'}(\Fs_{|C'})\;;
\end{equation}
\item $\Fs$ has finite dimensional de Rham cohomology 
over $C$ and $C'$, and 
\begin{equation}
\chidr(C,\Fs)\;=\;\chidr(C',\Fs_{|C'})\;=\;0\;.
\end{equation}
\end{enumerate}
Then i) implies ii). 

If now $K$ is non trivially valued, then ii) also implies i), 
and moreover, for $i=0,1$, the restriction maps
\begin{equation}
\Hdr^i(C,\Fs)\;\xrightarrow{\;\;\sim\;\;}\;
\Hdr^i(C',\Fs_{|C'})\;
\end{equation}
are isomorphisms.
\end{lemma}
\begin{proof}
The first part follows from Remark 
\ref{rem:chiabsindependence}. 
For $i=0$ the restriction map 
\eqref{eq : restriction isomorphic} is injective, 
while for $i=1$ it is surjective by 
Lemma~\ref{Lemma : H^1 surjectif}. 
The result now follows from the assumptions 
$\chidr(C,\Fs)=0$ and $\chidr(C',\Fs_{|C'})=0$.
\end{proof}
}\fi

\subsection{Index formula over an open pseudo-annulus.}

\label{Section : DE LA on annuli}
Let~$C$ be an open pseudo-annulus. Let~$\Fs$ be a differential 
equation of rank~$r$ on~$C$. Recall that, by \cite[Corollary 
6.2.28]{NP-III}, if all the radii of~$\Fs$ are log-affine radii along $\Gamma_C$, then all those radii are locally constant outside $\Gamma_C$.
In particular, the radii are separated over~$C$ if, and only if, they 
are separated along $\Gamma_C$.

\begin{definition}[Robba property]
\label{Def. index i satisfies the robba property}
We say that an index $i\in\{1,\dotsc,r\}$ satisfies the 
\emph{Robba property} if 
$\R_{i}(x,\Fs)=1$ for all $x$ 
in the skeleton~$\Gamma_{C}$ of~$C$.
\end{definition}


\begin{definition}\label{Def.: Rrobba part}
Assume that all the radii of~$\Fs$ are log-affine along~$\Gamma_C$. If some index satisfies the Robba property, we denote by~$i_{R}$ the smallest that does. By \cite[Theorem~5.3.1]{NP-III}, there exists a unique sub-object $\Fs^{\mathrm{Robba}}$ of~$\Fs$ of rank~$r-i_{R}+1$ all of whose indexes satisfy the Robba property.

If none of the indexes satisfy the Robba property, we set $i_{R}:=r+1$ and $\Fs^{\mathrm{Robba}}:= 0$.

\end{definition}

\begin{remark}
In the setting above the quotient $\Fs/\Fs^{\mathrm{Robba}}$ is a differential equation of rank~$i_{R}-1$ and, for each $x\in C$ and $j\in\{1,\dotsc,i_{R}-1\}$, we have
\begin{equation}  
\R_{j}(x,\Fs/\Fs^{\mathrm{Robba}})=\R_{j}(x,\Fs) < 1\;.
\end{equation}
This follows from \cite[Theorem~5.3.1]{NP-III} in the first case and it is obvious in the second.
\end{remark}

%
%

\begin{remark}
If the radii of $\Fs$ are not log-affine 
along~$\Gamma_{C}$, it is not clear whether an analogue of~$\Fs^{\mathrm{Robba}}$ exists. 
%
\end{remark}

\begin{proposition}
\label{Cor : Coh M = Coh MRobbaanna}
Assume that all the radii of~$\Fs$ are log-affine along~$\Gamma_C$. Then $\Fs$ has finite index if, and only if, 
$\Fs^{\mathrm{Robba}}$ has. Moreover, in this case, for $i=0,1$, we have
\begin{equation}
\Hdr^i(C,\Fs^{\mathrm{Robba}})=
\Hdr^i(C,\Fs)\;,\quad
\chidr(C,\Fs^{\mathrm{Robba}})=
\chidr(C,\Fs)\;.
\end{equation}
\end{proposition}
\begin{proof}
By Proposition~\ref{Prop : H^i=0 if not solvablegfz}, Situation~1, for all~$i$, we have $\Hdr^i(X,\Fs/\Fs^{\mathrm{Robba}})=0$. The result now follows by writing the cohomology long exact sequence associated to the short exact sequence $0 \to \Fs^{\mathrm{Robba}} \to \Fs \to \Fs/\Fs^{\mathrm{Robba}} \to 0$.

\end{proof}

%
%

\subsubsection{Absolute generalized index and 
irregularity.}
\if{\begin{definition}[Condition $\Fingen$]
\label{Def : Condition FIN}
Let $C$ be an open pseudo-annulus, let $b$ be a germ of 
segment at the open boundary of $C$, 
and let~$\Fs$ be a 
differential equation on~$C$. We say that $\Fs$ satisfies 
the condition $\Fingen_b$ if
\begin{enumerate}
\item $\Fs$ has log-affine radii along $b$; 
\item $\Fs^{\mathrm{Robba}}$ has finite generalized index at $b$, and
\begin{equation}\label{eq : chi_b defi FIN}
\chiabs_{b}(
\Fs^{\mathrm{Robba}})\;=\;0\;.
\end{equation}
\if{
\item for all sub-pseudo-annulus $C'\subseteq C$ such 
that $\Gamma_{C'}\subseteq\Gamma_C$ one has
\begin{eqnarray}
\chi^{\mathrm{gen}}_{b'_0}(
\Fs^{\mathrm{Robba}}(C'),\nabla)&\;=\;&+\mathrm{rank}(\Fs^{\mathrm{Robba}})\;,\medskip\\
\chi^{\mathrm{gen}}_{b'_1}(\Fs^{\mathrm{Robba}}(C'),\nabla)&\;=\;&-\mathrm{rank}(\Fs^{\mathrm{Robba}})\;.
\end{eqnarray}
where $b_0',b_1'$ are the two germs of segments at the 
open boundary of $C'$ respecting the convention 
\ref{Notation : b_0=+infty b_1=0}.
\item for all inclusion $C'\subseteq C$ of open 
pseudo-annuli such that 
$\Gamma_{C'}\subseteq\Gamma_C$, then,
for $i=0,1$, the natural restriction
\begin{equation}\label{eq : res C to C' log-affine}
\Hdr^i(C,\Fs)\;\xrightarrow{\;\sim\;}\;
\Hdr^i(C',\Fs_{|C'})
\end{equation}
is an isomorphism.
}\fi
\end{enumerate}
\end{definition}
}\fi

\begin{lemma}\label{Lemma : restrictions iso}
Let~$C$ be an open pseudo-annulus and let~$\Fs$ be a differential equation on~$C$. Let~$C'$ be an open sub-pseudo-annulus of~$C$ with $\Gamma_{C'}\subseteq\Gamma_C$.  Assume that there exists a complete valued extension~$L$ of~$K$ with non-trivial valuation such that the equations~$\Fs_{L}$ and $(\Fs_{L})_{|C_{L}'}$ have finite-dimensional de Rham cohomology and $\chidr(C_{L},\Fs_{L}) = \chidr(C_{L}',(\Fs_{L})_{|C_{L}'})$.
Then, for each $i\in\{0,1\}$, the restriction map
\begin{equation}\label{eq : restriction isomorphic}
\Hdr^i(C,\Fs)\simto\Hdr^i(C',\Fs_{|C'})
\end{equation}
is an isomorphism.
\end{lemma}
\begin{proof}
For $i=0$ the restriction map 
\eqref{eq : restriction isomorphic} is injective, 
while for $i=1$ it is surjective by Lemma~\ref{Lemma : H^1 surjectif}. By Theorem~\ref{thm:descent}, we have 
$\chidr(C,\Fs) = \chidr(C',\Fs_{|C'})$
and the result follows.
\end{proof}
\begin{remark}
Theorem \ref{Thm : index of finite opens} will 
provide conditions that ensure that the assumptions of 
Lemma \ref{Lemma : restrictions iso} are satisfied.
\end{remark}
We state here a result that is related to 
Lemma \ref{Lemma : restrictions iso} and that 
will be useful later on.
\begin{lemma}
\label{Lemma : H^i(X,F)=H^i(X',F)}
Let~$X$ be a quasi-smooth $K$-analytic curve. Let~$Z$ 
be a locally finite subset of rigid points of~$X$ and let 
$\Fc$ be a differential equation on $X$ with 
meromorphic singularities on $Z$. Let $X'$ be an analytic domain of~$X$ and let $C := \bigsqcup_i C_{i}$ be a 
disjoint union of open pseudo-annuli of~$X$. Assume 
that 
\begin{enumerate}
\item $X'\cup C=X$;
\item $Z\cap C=\emptyset$;
\item for each $i$, $C_i':=X'\cap C_i$ is an open 
pseudo-annulus in $C_i$ such that 
$\Gamma_{C_i'}\subseteq\Gamma_{C_i}$;
\item for each $i$ and each $k=0,1$, the restriction map 
$\Hdr^k(C_i,\Fc_{|C_i})\to\Hdr^k(C_i',\Fc_{|C_i'})$ is 
an isomorphism.
\end{enumerate} 
Then, for each $k\ge 0$, the restriction map 
\begin{equation}
\Hdr^k(X(*Z),\Fc)\;\xrightarrow{\;\sim\;}\;
\Hdr^k(X'(*Z),\Fc_{|X'})\;
\end{equation}
is an isomorphism.
\end{lemma}
\begin{proof}
Let $C':=\bigsqcup_i C_i'$. 
Then $X=X'\cup C$ and $C\cap X'=C'$. By ii) we have 
$\Hdr^k(C(*Z),\Fc_{|C})=\Hdr^k(C,\Fc_{|C})$ and 
$\Hdr^k(C'(*Z),\Fc_{|C'})=\Hdr^k(C',\Fc_{|C'})$. 
The Mayer-Vietoris sequence then gives 
\begin{equation}
\cdots \to\Hdr^{i-1}(C',\Fc_{|C'})
\to
\Hdr^i(X(*Z),\Fc)\to
\Hdr^i(X'(*Z),\Fc_{|X'})
\oplus\Hdr^i(C,\Fc_{|C})\to
\Hdr^i(C',\Fc_{|C'})\to\cdots
\end{equation}
By iv) the restriction map 
$\Hdr^i(C,\Fc_{|C})\to
\Hdr^i(C',\Fc_{|C'})$ is an isomorphism.
The claim follows.
\end{proof}

\begin{lemma}\label{lem:epsbb'}
Let~$C$ be an open pseudo-annulus and let~$\Fs$ be a differential equation on~$C$. Assume that there exists a complete valued extension~$L$ of~$K$ with non-trivial valuation such that, for every open pseudo-annulus $C'\subseteq C$ with $\Gamma_{C'}\subseteq\Gamma_C$, the equation~$(\Fs_{L})_{|C_{L}'}$ has finite-dimensional de Rham cohomology over~$C_{L}'$ and $\chidr(C_{L}',(\Fs_{L})_{|C_{L}'})=0$.

Then, for every germs of segments~$b$ and~$b'$ in~$\Gamma_{C}$, every open sub-pseudo-annuli~$C_{b}$ and~$C_{b'}$ of~$C$ whose open boundaries contains~$b$ and~$b'$ respectively, $\Fs_{|C_{b}}$ and $\Fs_{|C_{b'}}$ are Fredholm at~$b$ and~$b'$ respectively and we have
\begin{equation}
\chiabs_b(\Fs_{|C_{b}})\;=\; \eps(b,b') \, \chiabs_{b'}(\Fs_{|C_{b'}})\;,
\end{equation}
where $\eps(b,b')$ is equal to~1 (resp.~-1) if~$b$ and~$b'$ have the same orientation (resp. opposite orientations).
\end{lemma}
\begin{proof} 
Let~$b$ and~$b'$ be two germs of segment in~$\Gamma_{C}$ and let~$C_{b}$ and~$C_{b'}$ be as in the statement. 

Let~$b_{0}$ (resp.~$b_{1}$) be the germ of segment at the open boundary of~$C$ whose orientation is opposite to (resp. equal to) that of~$b$. There exist an open sub-pseudo-annulus~$C'_b$ of~$C$ whose open boundary is~$\{b_{0},b\}$. Note that we necessarily have $\Gamma_{C'_{b}}\subseteq\Gamma_C$. By Lemma~\ref{lem:b0b1chidr} and Lemma~\ref{Lemma : local index --t-}, $\Fs_{|C'_{b}}$ is Fredholm at~$b$ and we have
\begin{equation}
\chiabs_{b}(\Fs_{|C'_{b}}) = - \chiabs_{b_{0}}(\Fs_{|C'_{b}}) = \chiabs_{b_{0}}(\Fs)\;.
\end{equation}
By Lemma~\ref{Lemma : local index --t-}, the same results hold with~$C'_{b}$ replaced by~$C_{b}$. 

Similarly, we show that $\Fs_{|C_{b'}}$ is Fredholm at~$b'$ and that $\chiabs_{b'}(\Fs_{|C_{b'}})$ is equal to $\chiabs_{b_{0}}(\Fs)$ if~$b'$ has the same orientation as~$b$ or to $\chiabs_{b_{1}}(\Fs)$ if~$b'$ has the opposite orientation. By Lemma~\ref{lem:b0b1chidr} again, we have
\begin{equation}
\chiabs_{b_{0}}(\Fs) + \chiabs_{b_{1}}(\Fs) = 0
\end{equation}
and the result follows. 
\end{proof}

\begin{proposition}\label{Prop : chirel=Irr}
Let $C$ be an open pseudo-annulus, let $b$ be a germ of 
segment at the open boundary of $C$ and let~$\Fs$ be a 
differential equation on~$C$ with log-affine radii along~$\Gamma_{C}$. 



Let 
\begin{equation}\label{eq:decomposition}
\Fs\;=\;
\Fs^{\mathrm{Robba}}\oplus
\Fs^{<\mathrm{sol}}
\end{equation}
be the decomposition of $\Fs$ into its Robba part and its spectral 
non-solvable part. Then $\Fs^{<\mathrm{sol}}$ is Fredholm at $b$ and one has
\begin{equation}\label{eq : chiabs sol = Irr spns}
\chiabs_b(\Fs^{<\mathrm{sol}})\;=\;
\Irr_b(\Fs^{<\mathrm{sol}})\;=\;\Irr_b(\Fs)\;.
\end{equation}
In particular, the following conditions are equivalent:
\begin{enumerate}
\item $\Fs^{\mathrm{Robba}}$ is Fredholm at~$b$ and 
\begin{equation}\label{eq:chiabsRobba0}
\chiabs_{b}(\Fs^{\mathrm{Robba}})\;=\;0\;;
\end{equation}
\item $\Fs$ is Fredholm at~$b$ and 
\begin{equation}\label{eq : chiabs=Irr}
\chiabs_{b}(
\Fs)\;=\;
\Irr_{b}(\Fs)\;.
\end{equation}
\end{enumerate}
\end{proposition}

\begin{proof}

By Definition~\ref{def:Fredholmnablageneral}, we may assume that~$K$ is algebraically closed, spherically complete with $|K|=\mathbb{R}_{\geq 0}$ and that~$C$ is a standard open pseudo-annulus $C=\{r_1<|T|<r_2\}$ with $0 \le r_{1} < r_{2} \le +\infty$. 
 In this case, the equivalence of i) and ii) follows from 
\eqref{eq : chiabs sol = Irr spns}, by additivity of 
generalized index and irregularity.

%
%

We are then reduced to prove 
\eqref{eq : chiabs sol = Irr spns}, therefore 
we can assume that $\Fs = \Fs^{<\sol}$.

Let us first assume that the residual field of $K$ has 
positive characteristic $p>0$.
Thanks to Proposition~\ref{Prop : H^i=0 if not solvablegfz}, Lemma~\ref{lem:epsbb'} applies, so we can shrink $C$ and assume that there exists $r<1$ such 
that all the radii of $\Fs$ are smaller than $r$ at all points of 
$\Gamma_C$ (\textit{i.e.} the radii are not approaching~$1$ at the 
open boundary of $C$). 
This implies that the push-forward by the Frobenius morphism~$\varphi_{p}$ reduces 
the radii uniformly on $\Gamma_C$, 
hence that there exists $n>0$ 
such that  all the radii of $(\varphi_p)_*^n(\Fs)$
are all strictly smaller than the radius of Young's disk \cite{Young} at each point of $\Gamma_C$.
Corollary \ref{Cor : push-f of inde gen-1} and Proposition 
\ref{Prop : slope of push-f-1} show that the irregularity 
and the generalized index behave in the 
same way by Frobenius push-forward.

Replacing $\Fs$ by $(\varphi_p)_*^n(\Fs)$, we can assume that the radii of $\Fs$ are all strictly smaller than the 
radius of Young's disk \cite{Young} at each point 
of~$\Gamma_C$. Shrinking~$C$ again and 
decomposing~$\Fs$ by the radii, 
we can assume that all the radii of 
$\Fs$ are equal to the same 
log-affine function along $\Gamma_C$.



Since the determinant of Katz's base change matrix 
(see~\cite{Katz-cyclic-vect}) has a finite number of 
zeros on every closed annulus (this is independent on 
the residual characteristic), by restricting 
$C$ again, we can assume that $\Fs$ admits a cyclic basis over 
$C$.

For these reasons 
we can assume that 
\begin{enumerate}
\item[a)] $\Fs$ is a cyclic module;
\item[b)] the radii $\R_{i}(-,\Fs)$ of $\Fs$ are 
all equal to the same $\log$-affine function on 
$\Gamma_C$;
\item[c)] the radii $\R_{i}(-,\Fs)$ of $\Fs$ are 
all strictly smaller than Young's bound 
along~$\Gamma_{C}$.
\end{enumerate}

If the residual field of $K$ has characteristic $p=0$, we 
can again assume that a), b), c) hold. The proof is 
essentially the same, with the difference that
assumption c) is automatically verified 
(i.e. Frobenius push-forward in not needed). 
Indeed, if $p=0$, Young's bound is $1$, so the radii are 
automatically all strictly smaller than it because 
$\Fs=\Fs^{<\sol}$. 


Under a), b), c), the irregularity and the generalized 
index are 
explicitly intelligible in a cyclic basis. 
From now on the proof is the same in all residual 
characteristics. 

Set $d:=\frac{d}{dT}$, and denote by $x_\rho$ the 
sup-norm on the annulus $\{|T|=\rho\}$. We recall that 
the norm of $d$ as an operator on $\O(C)$ endowed 
with the norm $x_\rho$ is given by 
\begin{equation}\label{eq : d=rho^-1}
|d|(x_\rho)\;=\;\sup_{f\in\O(C)-\{0\}}
\frac{|d(f)|(x_\rho)}{|f|(x_\rho)}\;=\;\rho^{-1}
\;=\;|T|^{-1}\;.
\end{equation}

Consider a cyclic basis 
$(v,\nabla(v),\nabla^2(v),\ldots,\nabla^{r-1}(v))$ 
of $\Fs$. In such a basis the differential 
equation $\Fs$ is associated to an operator
\begin{equation}
P(T,d)\;=\;\sum_{k=0}^ra_k(T)d^k\;,
\end{equation}
where $a_r=1$. More precisely 
we have $\nabla=d-G(T)$, where $G(T)$ is the 
companion matrix of the polynomial $P(T,d)$:
\begin{equation}
G(T)\;=\;\left(\begin{small}
\begin{array}{c|ccccc}
0&&&&\\
\vdots&&Id&&\\
0&&&&\\
\hline
-a_0&-a_1&\cdots&&-a_{r-1}
\end{array}
\end{small}
\right)\;.
\end{equation}

Since the radii are all small, the convergence Newton 
polygon of $\Fs$ coincides with the spectral Newton 
polygon of $P(T,d)$ (cf. \cite[Proposition 4.11]{NP-I} or 
also \cite[Théorème 6.2]{Astx}).\footnote{Recall that the 
convergence Newton polygon is the polygon whose 
slopes are $\{\log(\R_{i}(x_\rho,\Fs))\}_{i=1,
\ldots,r}$, while the spectral Newton polygon of $P(T,d)$ 
is defined by considering the convex hull of the 
$x_\rho$-values of the coefficients of $P(T,d)$, cf. 
\cite[Proposition 4.11]{NP-I}.}

Since the radii are all equal to the same log-affine 
function, the Young's spectral Newton polygon associated 
to $P(T,d)$ has no breaks, hence the radii are directly related to the norm of the 
constant term $a_0(T)\in\O(C)$ (note that $a_0(T)$ is 
not zero since, by construction, $\Fs$ has no trivial 
submodules). More precisely, Young's theorem ensures 
that for all $k=1,\ldots,r-1$ and all $\rho\in]r_1,r_2[$, 
one has
\begin{equation}\label{eq : YOUNG}
|a_0|(x_\rho)^{1/r} \;\;>\;\; 
\max(\;|d|(x_\rho)\;,
\;|a_k|(x_\rho)^{\frac{1}{r-k}}\;)\;,
\end{equation}
and that for all $i=1,\ldots, r$
\begin{equation}\label{eq : radii small index proof}
\R_{i}(x_\rho,\Fs)\;=\;
\omega\cdot|a_0|(x_\rho)^{-\frac{1}{r}}\rho^{-1}\;,
\footnote{The 
presence of $\rho^{-1}$ in \eqref{eq : radii small index proof} 
is due to the fact that we consider 
normalized radii instead of spectral radii as in 
\cite[Proposition 4.11]{NP-I}.}
\end{equation}
where $\omega$ is either $1$ or $p$ if the characteristic 
of the residual field  $\widetilde{K}$ is $0$ or $p>0$ 
respectively (cf. \eqref{eq : OMEGA}). 
In particular, since the radii are assumed to be log-affine along 
$\Gamma_C$, $a_0$ has no zeros in $C$. We also deduce that 
\begin{equation}\label{eq : Irr_b = partial_b a_0}
\Irr_b(\Fs)\;=\;-r\partial_b|T|-\partial_b\,|a_0(T)|\;.
\end{equation}


%

We now compute the absolute index. 
First of all, by Lemma~\ref{lem:epsbb'}, 
we can replace $b$ by another germ of segment in 
$\Gamma_C$ with the same orientation. 
We then pick $\gamma\in]r_1,r_2[$ and 
assume that $b$ is a germ of segment out of 
$x_{\gamma}$. 

Let $\beta\in K$ such that 
\begin{equation}
|\beta|\;=\;|a_0|(x_{\gamma})^{1/r}\;.
\end{equation}
We now replace the cyclic basis of $\Fs$
by the following more convenient one 
\begin{equation}
(v,\beta^{-1}\nabla(v),\beta^{-2}\nabla^2(v),\ldots,
\beta^{1-r}\nabla^{r-1}(v))\;.
\end{equation} 
In this basis, the connection is given by $d-G_\beta(T)$ 
where
\begin{equation}
G_\beta(T)\;=\;\beta\cdot
\left(\begin{small}
\begin{array}{c|ccccc}
0&&&&\\
\vdots&&Id&&\\
0&&&&\\
\hline
-\beta^{-r}a_0&-\beta^{-r+1}a_1&\cdots&&-\beta^{-1}a_{r-1}
\end{array}
\end{small}
\right)\;,
\end{equation}
and its inverse is given by 
\begin{equation}
G_\beta(T)^{-1}\;=\;\beta^{-1}
\cdot
\left(\begin{small}
\begin{array}{cccc|c}
-\beta\frac{ a_{1}}{a_0} &&\cdots&
-\beta^{r-1}\frac{a_{r-1}}{a_0}&
-\beta^r\frac{1}{a_0}\\
\hline
&&&&0\\
&&Id&&\vdots\\
&&&&0
\end{array}
\end{small}
\right)\;.
\end{equation}
The reason of these choices is that, in this situation, 
we have (cf. \eqref{eq : YOUNG})
\begin{eqnarray}
|G_\beta|(x_\gamma)&\;=\;&
|\beta|\cdot\max(1,|\beta|^{-r}|a_0|(x_\gamma),
|\beta|^{1-r}|a_1|(x_\gamma),\ldots,
\beta^{-1}|a_{r-1}|(x_\gamma))\;=\;|\beta|\;,\medskip\\
|G_\beta^{-1}|(x_\gamma)&\;=\;&
|\beta|^{-1}\cdot
\max(1,|\beta|^{r}\frac{1}{|a_0|(x_\gamma)},
|\beta|^{1}\frac{|a_1|(x_\gamma)}{|a_0|(x_\gamma)},\ldots,
\beta^{r-1}
\frac{|a_{r-1}|(x_\gamma)}{|a_0|(x_\gamma)})\;=\;
|\beta|^{-1}\;.\label{eq : G_beta^-1}
\end{eqnarray}

By Proposition \ref{Prop : H^i=0 if not solvablegfz}, $\Hdr^0(C,\Fs)$ and $\Hdr^1(C,\Fs)$ are finite-dimensional. In other words, the endomorphism $Td-TG_{\beta}$ of~$\Os(C)^r$ is Fredholm, therefore, by Proposition \ref{Prop : chi=sumchigen}, it is Fredholm at~$b$. By Remark \ref{Remark : chiabs=chigen}, we have 
$\chi^{\mathrm{gen}}_b(Td-TG_\beta)= \chiabs_b(\Fs)$.

We now write
\begin{equation}
Td-TG_\beta\;=\;(Td(TG_\beta)^{-1}-1)\circ(TG_\beta).
\end{equation}
By Lemma \ref{Lemma : additivity gen ind}, we have 
\begin{equation}
\chi^{\mathrm{gen}}_b(Td-TG_\beta)\;=\;
\chi^{\mathrm{gen}}_b(Td\circ(TG_\beta)^{-1}-1)+
\chi^{\mathrm{gen}}_b(TG_\beta)\;.
\end{equation}

The generalized index of the last term is obtained 
from Lemma~\ref{Lem : gen ind of a constant} and 
\eqref{eq : Irr_b = partial_b a_0}:
\begin{equation}
\chi^{\mathrm{gen}}_b(TG_\beta)\;=\;
-\partial_b(\mathrm{det}(TG_\beta))\;=\;
-\partial_b|T^r\cdot\beta^{1-r} a_0(T)|\;=\;
-r\partial_b|T|-\partial_b|a_0(T)|\;=\;
\Irr_b(\Fs)\;.
\end{equation}


With the notations of Definition \ref{Def : }, we consider 
a 
sub-annulus of $C_b=\{r_1'<|T|<r_2'\}$ 
$C$ having $b$ at its open boundary and 
we fix $k\in\{0,1\}$ in order that $b=b_{k}$. 
To conclude the proof, we will show that the truncated 
operator 
$p_k\circ (Td\circ(TG_\beta)^{-1}-1)\circ i_k$ (cf. 
\eqref{eq : truncated op u_k}) is invertible. This will imply that 
$\chi^{\mathrm{gen}}_b(Td\circ(TG_\beta)^{-1}-1)=0$.

Let $D_b$ be the disk in 
$\mathbb{P}^{1,\textrm{an}}_K$ having $b$ 
at its open boundary. It contains $C_b$. 

Let $\rho\in]r_1',r_2'[$. Endowing~$\Os(C_{b})$ and~$\Os(D_{b})$ with the norm~$x_{\rho}$, the inclusion $i_{k} \colon \Os(D_{b}) \to \Os(C_{b})$ is isometric and the projection $p_{k} \colon \Os(C_{b}) \to \Os(D_{b})$ is a contraction. We deduce that 
\begin{equation}
|p_k\circ (Td\circ(TG_\beta)^{-1})\circ i_k|(x_{\rho})
\leq
|Td\circ(TG_\beta)^{-1}|(x_{\rho})\;.
\end{equation}
 

On the other hand, it follows from 
\eqref{eq : d=rho^-1}
and \eqref{eq : G_beta^-1} that
\begin{equation}
|Td\circ(TG_\beta)^{-1}|(x_{\gamma})
\;\leq\;
|Td|(x_{\gamma})\cdot
|(TG_\beta)^{-1}|(x_{\gamma})
\;=\;
|(TG_\beta)^{-1}|(x_{\gamma})
\;=\;
|(TG_\beta)|(x_{\gamma})^{-1}\;<\;1\;,
\end{equation}
where the last strict inequality follows from 
\eqref{eq : d=rho^-1} and \eqref{eq : YOUNG} .

By continuity, the strict inequality holds for all $\rho$ 
sufficiently close to $\gamma$. Since the boundary of~$C_b$ in~$C$ contains~$x_\gamma$, up to shrinking~$C_b$, we 
can assume that it holds for all $\rho\in]r_1',r_2'[$. 
It follows that 
$|p_k\circ(Td\circ(TG_\beta)^{-1})\circ i_k|
(x_\rho)<1$ for all $\rho\in]r_1',r_2'[$.
We deduce that 
$p_k\circ(Td\circ(TG_\beta)^{-1}-1)\circ i_k=
p_k\circ(Td\circ(TG_\beta)^{-1})\circ i_k-1$ is invertible.
The claim follows.
\end{proof}


\begin{corollary}
\label{Cor : first index over a virtual disk}
\label{lem:chiD=0}
Let~$D$ be an open pseudo-disk inside~$X$ (cf. \cite[Definition 1.1.8]{NP-IV}). 
Let~$b_{D}$ be the germ of segment at infinity 
on~$D$ (oriented towards the interior of~$D$). 
Let~$\Fs$ be a differential equation over $D$
that is Fredholm at~$b_D$ and whose radii are all 
log-affine along $b_D$. Let $C$ be an 
open pseudo-annulus in $D$ containing~$b_D$ 
such that the radii of $\Fs$ are log-affine on 
$\Gamma_{C}$.

Then $\Fs$ has finite-dimensional de Rham cohomology 
over $D$ and we have
\begin{eqnarray}
\chidr(D,\Fs)
&\;=\;&
\chiabs_{b_D}(\Fs)+r\;,\label{eq : chidr(D,F)=chiabs}\\
&\;=\;&
\chiabs_{b_D}(\Fs_{|C}^{\mathrm{Robba}})+
\Irr_{b_D}(\Fs)+r\;,\label{eq : chi-2---gtfr}\\
&=&
\chiabs_{b_D}(\Fs_{|C}^{\mathrm{Robba}})+
\hdr^0(D,\Fs)-\partial_{b_D}H_{r}(-,\Fs)\footnotemark\;,\label{eq : chi-2---gtfr-2}
\end{eqnarray}
\footnotetext{Note that the 
slope of the total height of the convergence Newton 
polygon $\partial_{b_D}H_{r}(-,\Fs)$ 
is taken before localization to~$C$.}
where $r:=\rk(\Fs_{|D})$.
In particular, if 
$\chiabs_{b_D}(\Fs_{|C}^{\mathrm{Robba}})=0$, we 
obtain
\begin{equation}\label{eq : h^1=partial_bH_r}
\hdr^1(D,\Fs)\;=\;\partial_{b_D}H_{r}(-,
\Fs)\;.
\end{equation} 
If, moreover, the radii of~$\Fs$ are spectral non-solvable on~$b_{D}$, then we have $\chidr(D,\Fs) \le 0$ and $\chidr(D,\Fs)=0$ if, and only if, $\partial_{b_{D}} H_{r}(-,\Fs) = 0$.
%
\end{corollary}
\begin{proof}
All the quantities are stable by scalar extension. 
Therefore we can assume that 
$K$ is spherically complete, algebraically closed, and that 
$|K|=\mathbb{R}_{\geq 0}$,  that $D$ is a disk 
defined by $D=\{|T|<r_2\}$ and that $C_b=\{r_1<|T|<r_2\}$. We are then in the context of Definition 
\ref{Def : }, which we now apply with $D=D_0$, 
$n=0$, and $m=-1$. 
The truncation $p_0\circ\nabla(d/dT)\circ i_0$ 
(cf. \eqref{eq : truncated op u_k}) of the connection 
$\nabla(d/dT):\Fs(C)\to\Fs(C)$ coincides with 
our original connection $\nabla(d/dT):\Fs(D)\to\Fs(D)$. 
Hence $\chi^{\mathrm{gen}}_{b_D}(\Fs(C),
\nabla(d/dT))=\chi(\Fs(D),
\nabla(d/dT))=\chidr(D,\Fs)$. We deduce that
\begin{align}
\chiabs_{b_D}(\Fs) &= \chi^{\mathrm{gen}}_{b_D}(\Fs(C),\nabla(d/dT)) - \chi^{\mathrm{gen}}_{b_D}(\Os(C),d/dT)\\
&=\chidr(D,\Fs) - \chidr(D,\Os)\\
& =\chidr(D,\Fs) -r\;,
\end{align}
hence \eqref{eq : chidr(D,F)=chiabs} holds.
Formulas 
\eqref{eq : chi-2---gtfr}, 
\eqref{eq : chi-2---gtfr-2}
and 
\eqref{eq : h^1=partial_bH_r} now follow 
immediately from Proposition \ref{Prop : chirel=Irr} and 
\cite[Lemma 2.8.4]{NP-III}.

If the radii are all spectral non solvable at $b_D$, we 
have $\Hdr^0(D,\Fs)=0$. 
Indeed any global solution of $\Fs$ on $D$ generates 
an over-solvable radius along $b_D$. 
Hence $\chidr(D,\Fs)=-\hdr^1(D,\Fs)\leq 0$. The last 
assertion then follows from 
\eqref{eq : h^1=partial_bH_r}.
\end{proof}

\subsubsection{Invariance of the absolute index.} We 
now prove that the absolute index is an intrinsic notion.
The invariance actually follows from Proposition 
\ref{Prop : chirel=Irr} in the spectral non solvable case,
and from Proposition \ref{prop:deformation} in the 
solvable case. 

\begin{proposition}[Invariance of the absolute index]
\label{Prop : independence on the coordinate chiabs}
Let $C$ be an open pseudo-annulus and let $b$ be a 
germ of segment at the open boundary of 
$C$. Let $\Fs$ be a finite differential equation over $C$ 
whose radii are all log-affine along $b$. 

Let $\sigma:C\simto C$ be an 
automorphism of $C$. Then
\begin{equation}\label{eq : chiphimenouno}
\chiabs_b(\Fs)\;=\;
\chiabs_{\sigma^{-1}(b)}(\sigma^*\Fs)\;.
\end{equation}
In particular, the absolute generalized index 
$\chiabs_b(\Fs)$ is independent of~$\Omega$ and of the coordinates~$T_{i}$ (see Section 
\ref{section : GENIND}).
\end{proposition}
\begin{proof}
By Definition~\ref{def:Fredholmnablageneral} we may assume that~$K$ is algebraically closed, spherically complete with $|K|=\mathbb{R}_{\geq 0}$ and that~$C$ is a standard open pseudo-annulus $C=\{r_1<|T|<r_2\}$ with $0 \le r_{1} < r_{2} \le +\infty$. 
An inversion of the orientation is possible only if
$(r_1,r_2)=(0,\infty)$ or $r_1,r_2\notin\{0,\infty\}$. If $\sigma(b)\neq b$, we consider the 
automorphism $\psi$ defined by
\begin{equation}
\psi(T):=
\left\{\begin{array}{lcl}
T^{-1}&\textrm{if}& r_1=0,\;r_2=+\infty\\
a_1\cdot a_2\cdot T^{-1}&\textrm{if}&r_1,r_2\in]0,+\infty[
\end{array}\right.
\end{equation}
where $a_i\in K$, $|a_i|=r_i$ for $i=1,2$. Equality \eqref{eq : chiphimenouno} 
is satisfied by $\psi$ and $\psi^{-1}$, who only exchange the roles of $D_0$ and $D_\infty$ in the 
definition of $\chiabs_b$. Now $\sigma':=\sigma\circ\psi^{-1}$ preserves the 
orientation. Since $\sigma=\sigma'\circ\psi$, it is 
enough to prove the assertion for $\sigma'$. In other 
words, we can assume that $\sigma$ preserves the 
orientation of $C$, \textit{i.e.} $\sigma(b)=b$.

Since $\sigma(T)$ is invertible in $\O(C)$, it can be 
written as
\begin{equation}
\sigma(T)\;=\;q(T^d+h(T))
\end{equation}
with $q\in K^\ast$ (and $|q|=1$ if $(r_{1},r_{2}) \ne (0,+\infty)$) and $h(T)\in\O(C)$ satisfying $|h(x_\rho)|<\rho^d$ for all 
$\rho\in]r_1,r_2[$ (where $x_\rho$ denotes as usual the 
sup-norm on the annulus $\{|T|=\rho\}$). Since~$\sigma$ is an automorphism, we have $d=\pm 1$, and even $d=1$, since it preserves the orientation.

As before, equality \eqref{eq : chiphimenouno} is satisfied by the multiplications by~$q$ and~$q^{-1}$, because they induce isomophisms of~$D_{0}$ and~$D_{\infty}$. Replacing~$\sigma$ by $\sigma(q^{-1}\, \wc)$, we can assume that~$q=1$. 
%
%
%

In this case, the endomorphism~$\sigma$ stabilizes the sub-pseudo-annuli of~$C$ whose skeletons are included in~$\Gamma_{C}$. Therefore, by Lemma \ref{Lemma : local index --t-}, we can replace~$C$ by a smaller sub-pseudo-annulus 
containing $b$, and assume that~$\Fs$ 
has log-affine radii along $\Gamma_{C}$.

%
%

By decomposing~$\Fs$ by the radii and using Lemma~\ref{Lemma : additivity on exact sequence}, 
we can assume that the radii of~$\Fs$ are all equal to the same log-affine
affine function. We then distinguish two situations: 
$\Fs = \Fs^{<\sol}$ and 
$\Fs= \Fs^{\mathrm{Robba}}$. 


If $\Fs = \Fs^{<\sol}$, by Proposition~\ref{Prop : chirel=Irr}, we have 
$\chiabs_b(\Fs)=\Irr_b(\Fs)$, hence the claim follows 
from the fact that the radii are preserved by $\sigma$ 
(see \cite[Lemma 3.23]{NP-II}).

Assume now that $\Fs= \Fs^{\mathrm{Robba}}$. 
In this case Proposition \ref{prop:deformation}
shows that we have an isomorphism 
$\sigma_2^*(\Fs)\cong\Fs$. 
Equality \eqref{eq : chiphimenouno} then follows from
Lemma \ref{Lemma : independence on j}.

\medbreak

The last part of the statement follows from the first and from Remark~\ref{rem:Omega}, using the fact that two extensions of~$K$ may always be embedded in a common extension~$\Omega$ that is spherically complete, algebraically closed and with $|\Omega| = \ERRE_{\ge0}$.
\end{proof}


With this result at hand, we can now extend our 
definition of absolute generalized index to a more general 
setting.

\begin{definition}
Let $X$ be a quasi-smooth $K$-analytic curve. Let $b$ be a good germ of segment in 
$X$ (cf. Definition~\ref{def:goodgerm}) and let $C$ be an open pseudo-annulus in $X$ 
having $b$ at its open boundary. Let $\Fs$ be a differential 
equation on $X$ with log-affine radii along~$b$. We say that~$\Fs$ is Fredhlom at~$b$ if~$\Fs_{|C}$ is and, in this case, we define the \emph{absolute 
index of $\Fs$ at $b$} by
\begin{equation}
\chiabs_{b}(\Fs) := \chiabs_{b}(\Fs_{|C}).
\end{equation}
\end{definition}

From what we have done, we know that the definition is 
independent of all the choices involved at the different 
steps of the construction: $C$ (see 
Lemma~\ref{Lemma : local index --t-}), $\Omega$ (see 
Remark~\ref{rem:Omega}), the coordinates~$T_{i}$ 
(see Proposition~\ref{Prop : independence on the coordinate chiabs}) and the isomorphism~$j$ (see 
Lemma~\ref{Lemma : independence on j}).

\subsubsection{Log-affine radii.}
We now are ready to prove the index theorem over an 
annulus.

\begin{theorem}[Index theorem]
\label{Thm : index of finite opens}
Let $C$ be an open pseudo-annulus and let $b_0$ and 
$b_1$ be the germs of segment at its open boundary. 
Let $\Fs$ be a differential equation over $C$ whose 
radii are all log-affine along $b_0$ and $b_1$. 

For $b\in\{b_{0},b_{1}\}$, let~$C_{b}$ be an open 
sub-pseudo-annulus of~$C$ containing $b$ 
such that $\Gamma_{C_{b}} \subseteq \Gamma_{C}$ 
and such that~$\Fs$ has log-affine radii 
along~$\Gamma_{C_{b}}$.  

Then, the following assertions are equivalent:
\begin{enumerate}
\item for each $b\in\{b_{0},b_{1}\}$, $\Fs_{|C_b}^{\mathrm{Robba}}$ is Fredholm at~$b$ and we have
\begin{equation}\label{eq:chiabs01}
\chiabs_{b_0}(\Fs_{|C_{b_0}}^{\mathrm{Robba}}) 
\;=\;
-\chiabs_{b_1}(\Fs_{|C_{b_1}}^{\mathrm{Robba}}) 
\;;
\end{equation}
\item $\Fs$ has finite-dimensional de Rham cohomology over~$C$ and we have
\begin{equation}\label{eq:chidrIrr}
\chidr(C,\Fs)
\;=\;\Irr_{b_0}(\Fs)+\Irr_{b_1}(\Fs)\;.
\end{equation}
\end{enumerate}

%
\end{theorem}
\begin{proof}
By Lemma~\ref{lem:b0b1chidr} and Proposition~\ref{Prop : chirel=Irr}, i) implies~ii) and, if $K$ is not trivially valued, then ii) implies~i).

In the case where $K$ is trivially valued, it follows from Corollary~\ref{Cor. H^i(C)=H^i(C') restr} that~i) is always satisfied, hence~ii) is alway satisfied too. 

\end{proof}

For further reference, let us extract from the preceding proof the statement in the trivially valued case. Note also that, by Remark~\ref{rem:logafftrivval}, the assumption that the radii are log-affine is always satisfied in this case.

\begin{theorem}\label{thm:indexannulustrivialvaluation}
Assume that~$K$ is trivially valued.
Let $C$ be an open pseudo-annulus and let $b_0$ and 
$b_1$ be the germs of segment at its open boundary. 
Let $\Fs$ be a differential equation over $C$. 

Then, the radii of~$\Fs$ are log-affine along~$b_{0}$ and~$b_{1}$, $\Fs$ has finite-dimensional de Rham cohomology over~$C$ and we have
\begin{equation}
\chidr(C,\Fs)
\;=\;\Irr_{b_0}(\Fs)+\Irr_{b_1}(\Fs)\;.
\end{equation}
\hfill$\Box$
\end{theorem}

\subsubsection{Non-log-affine radii.}



In order to generalize Theorem~\ref{Thm : index of finite opens} to the case where the radii are not necessarily log-affine at the boundary, we introduce conditions on germs of segments.

\begin{definition}[Condition $\Fin_b^+$]
\label{Def: Fin+}
Let $X$ be a quasi-smooth $K$-analytic curve and 
let~$\Fs$ be a differential equation on~$X$. Let $b$ be 
a good germ of segment in $X$. 
We say that \emph{$\Fs$ satisfies $\Fin_b^+$} if there exists an open 
pseudo-annulus~$C_{b}$ whose skeleton represents~$b$ such that
\begin{enumerate}[i)]
\item $\Fs$ has log-affine radii along~$C_{b}$; 
\item $\Fs_{|C_{b}}^\mathrm{Robba}$ is Fredholm at~$b$;
\item 
$\chiabs_{b}(\Fs_{|C_{b}}^\mathrm{Robba})=0$.
\end{enumerate}
\end{definition}

\begin{definition}[Condition $\Fin_b$]
\label{Def: Fin_b}
Let $X$ be a quasi-smooth $K$-analytic curve and 
let~$\Fs$ be a differential equation on~$X$. Let $b$ be 
a good germ of segment in $X$. 
We say that \emph{$\Fs$ satisfies $\Fin_b$} if, for each open 
pseudo-annulus~$C$ whose skeleton represents~$b$, 
there exists an open sub-pseudo-annulus $C'\subseteq C$ with 
$\Gamma_{C'}\subseteq\Gamma_C$ (possibly 
not representing $b$) such that~$\Fs$ satisfies $\Fin_{b'}^+$ where~$b'$ denotes the germ of segment at the open boundary of~$C'$ with the same orientation as~$b$.
\end{definition}

\begin{remark}
If~$K$ is trivially valued or, more generally, if~$K$ has residue characteristic~0, then, by Corollary~\ref{Cor. H^i(C)=H^i(C') restr}, each good germ of segment in~$X$ satisfies~$\Fin_{b}$.
\end{remark}

\begin{corollary}\label{cor:cohomologypseudoannulus}
Let $C$ be an open pseudo-annulus and let $b_0$ and 
$b_1$ be the germs of segment at its open boundary. 
Let $\Fs$ be a differential equation over $C$. Assume that $\Fs$ satisfies~$\Fin_{b_{0}}$ and~$\Fin_{b_{1}}$.

The following assertions are equivalent:
\begin{enumerate}[i)]
\item the total height of~$\Fs$ is log-affine along~$b_{0}$ and~$b_{1}$;
\item for each $i\ge 0$, $\Hdr^i(C,\Fs)$ is finite dimensional.
\end{enumerate}

Moreover, when they hold, we have 
\begin{equation}\label{eq:chidrpseudo-annulus}
\chidr(C,\Fs)
\;=\;\Irr_{b_0}(\Fs)+
\Irr_{b_1}(\Fs)\;.
\end{equation}
\end{corollary}

%
%
\begin{proof}
%
%
%

Let us first assume that~$K$ is not trivially valued. By assumption, there exists an increasing sequence $(C_{n})_{n\ge 0}$ of relatively compact open pseudo-annuli of~$C$ with $\Gamma_{C_{n}} \subseteq \Gamma_{C}$ such that
\begin{equation}\bigcup_{n\ge 0} C_{n} = C\end{equation}
and, for each $n\ge 0$, $\Fs$ satisfies $\Fin_{b_{n,0}}$ and~$\Fin_{b_{n,1}}$, where $b_{n,0}$ (resp.~$b_{n,1}$) denotes the germ of segment at the open boundary of~$C_{n}$ with the same orientation as~$b_{0}$ (resp.~$b_{1}$).
 By Theorem~\ref{Thm : index of finite opens},
 for each $n\ge 0$, $\Fs_{|C_{n}}$ has finite dimensional de Rham cohomology and we have 
\begin{equation}
\chidr(C_{n},\Fs_{|C_{n}}) = \Irr_{b_{n,0}}(\Fs) + \Irr_{b_{n,1}}(\Fs).
\end{equation}
By log-concavity of the radii (see \cite[Theorem 3.9, 
iii)]{NP-I}), 
the sequences $(\Irr_{b_{n,0}}(\Fs))_{n\ge 0}$ and $(\Irr_{b_{n,1}}(\Fs))_{n\ge 0}$ are both non-increasing. In particular, the sequence $(\chidr(C_{n},\Fs_{|C_{n}}))_{n\ge 0}$ converges if, and only if, both sequences $(\Irr_{b_{n,0}}(\Fs))_{n\ge 0}$ and $(\Irr_{b_{n,1}}(\Fs))_{n\ge 0}$ converge, which is equivalent to saying that~$\Fs$ has log-affine total height along~$b_{0}$ and~$b_{1}$ or, in other words, that~$\Fs$ has well-defined irregularity. The result now follows from \cite{Banachoid}.

Let us now assume that~$K$ is trivially valued. By Remark~\ref{rem:logafftrivval}, the radii of~$\Fs$ are log-affine along~$b_{0}$ and~$b_{1}$, hence assertion~i) holds. 

Let~$L$ be a complete valued extension of~$K$ with non-trivial valuation. The space~$C_{L}$ is a finite disjoint union of open pseudo-annuli $C_{1},\dotsc,C_{m}$. For each $j\in \{1,\dotsc,m\}$, denote by~$b_{j,0}$ (resp. $b_{j,1}$) the germ of segment of~$C_{i}$ over~$b_{0}$ (resp. $b_{1}$). For each $j\in \{1,\dotsc,m\}$, the radii of~$\Fs_{L}$ are log-affine along~$b_{j,0}$ and~$b_{j,1}$, hence it follows from the result in the non-trivially valued case that, for each $i\ge 0$, $\Hdr^i(C_{j},\Fs_{L})$ is finite dimensional and that we have
\begin{equation}
\chidr(C_{j},\Fs_{L}) = \Irr_{b_{j,0}}(\Fs_{L}) + \Irr_{b_{j,1}}(\Fs_{L}).
\end{equation}
It now follows from Theorem~\ref{thm:descent} that assertion~ii) holds, as well as~\eqref{eq:chidrpseudo-annulus}.
\end{proof}

\subsection{Index formula for Robba rings}


Let~$X$ be a quasi-smooth $K$-analytic curve. Let $b$ be a good germ of segment in $X$.

\begin{definition}\label{Def.: Robba ring at b}
We call \emph{Robba ring at $b$} the ring 
\begin{equation}\label{eq : Def : Rorigb}
\mathfrak{R}_b\;:=\;\varinjlim_{
C}\O(C)\;,
\end{equation} 
where $C$ runs through the family of open pseudo-annuli 
whose open boundary contains~$b$. 
\end{definition}

Let $\Fs$ be a differential equation over $X$. The restriction~$\Fs_{|\mathfrak{R}_b}$ of $\Fs$ to 
$\mathfrak{R}_b$ is a locally 
free $\mathfrak{R}_b$-module 
of finite rank endowed with a connection. 

\begin{lemma}\label{Lemma : H^i Robba=lim}
For $i=0,1$, we have a natural isomorphism of $K$-vector spaces
\begin{equation}
\varinjlim_{C}\Hdr^i(C,\Fs_{|C})
\;\simto\;\Hdr^i(\mathfrak{R}_b,\Fs_{|\mathfrak{R}_b})\;,
\end{equation}
where~$C$ runs through the set of 
pseudo-annuli whose open boundary contains~$b$.

\end{lemma}
\begin{proof}
The result for~$\Hdr^1$ follows from the fact that cokernels and colimits commute. Since the colimit in~\eqref{eq : Def : Rorigb} is filtered and kernels commute with filtered colimits, the result also holds for~$\Hdr^0$.
\end{proof}

\begin{corollary}
\label{Cor : zero index over the robba ring}
Let  $C_1\supseteq C_2\supseteq\cdots$ 
be a decreasing 
sequence of open pseudo-annuli all of them having~$b$ in their open boundary and such that 
$\bigcap_nC_n=\emptyset$. 
Consider the following conditions.
\begin{enumerate}
\item $\Fs$ has log-affine radii along $\Gamma_{C_1}$, for each $n\ge 1$, 
$\Fs_{|C_n}$ has finite-dimensional de Rham cohomology and 
\begin{equation}
\chidr(C_n,\Fs_{|C_n})=0\;.
\end{equation}
\item There exists a complete valued extension~$L$ of~$K$ with non-trivial valuation such that, for each $n\ge 1$, 
$\Fs_{|C_{n}}$ and $(\Fs_{L})_{|(C_n)_{L}}$ have finite-dimensional de Rham cohomologies and 
\begin{equation}
\chidr(C_n,\Fs_{|C_n})\;=\;
\chidr(C_{n+1},\Fs_{|C_{n+1}})\;.
\end{equation}
\end{enumerate}
Then i) implies ii). Moreover, if ii) holds, then for $i=0,1$, 
the cohomology group 
$\Hdr^i(\mathfrak{R}_b,\Fs_{|\mathfrak{R}_b })$ 
is finite-dimensional and, for each $n\geq 1$, the natural maps
\begin{equation}\label{eq: chi_n=chi_n+1 Robba}
\Hdr^i(C_n,\Fs_{|C_n})\;\xrightarrow{\;\sim\;}\; \Hdr^i(C_{n+1},\Fs_{|C_{n+1}})\;\xrightarrow{\;\sim\;}\;
\Hdr^i(\mathfrak{R}_b,\Fs_{|\mathfrak{R}_b})
\end{equation}
are isomorphisms.
\end{corollary}
\begin{proof}
Assume that i) holds. If $K$ is not trivially valued, then
ii) holds with~$L=K$. Assume that~$K$ is trivially valued. Let~$L$ be a complete non-trivially valued extension of~$K$. Let~$n\ge 1$. Then, by Corollary~\ref{cor:cohomologypseudoannulusLiouville}, $(\Fs_{L})_{|(C_n)_{L}}$ has finite-dimensional de Rham cohomology and we have $\chidr((C_n)_{L},\Fs_{|(C_n)_{L}}) = 0$. In particular, ii) holds.


Let us now assume that ii) holds. Let $n\ge 1$. For $i=0,1$, denote by $R_{n}^i:\Hdr^i(C_n,\Fs_{|C_n})\to\Hdr^i(C_{n+1},\Fs_{|C_{n+1}})$ the restriction map. For $i=0$ it is injective, while for $i=1$ it is surjective by Lemma \ref{Lemma : H^1 surjectif}. A dimension argument combined with~\eqref{eq: chi_n=chi_n+1 Robba} implies that $R^0_{n}$ and $R^1_{n}$ are isomorphisms. The claim now follows from Lemma~\ref{Lemma : H^i Robba=lim}.
\end{proof}

\subsection{Formal differential equations.}
\label{Remark : index annulus trivial valuation}
\label{Rk : formal gen indexes}


In this section, we assume that the valuation of~$K$ is trivial and we describe more explicitly the 
consequences of Theorem~\ref{thm:indexannulustrivialvaluation}. We 
remind that in this case the radii are automatically 
log-affine along all good germ of segment in $X$ (cf. Remark \ref{rem:logafftrivval}).  


Let $C:=\{r_1<|T|<r_2\}$, 
with $0 \le r_1 < r_2 \le  +\infty$ be a standard open pseudo-annulus over~$K$. Let~$\Fs$ be a differential equation over~$C$.

%

First of all, we recall that the triviality of the valuation 
implies that
\begin{equation}\label{eq : O(C)=K((T))}
\O(C)\;=\;\left\{
\begin{array}{lll}
K((T))&\textrm{ if }&0\leq r_1<r_2\leq 1\;;\\
K[T,T^{-1}]&\textrm{ if }&0\leq 
r_1<1<r_2\leq+\infty\;;\\
K((T^{-1}))&\textrm{ if }&1\leq r_1<r_2\leq+\infty\;.
\end{array}
\right.
\end{equation}
If $0\leq r_1< r_2\leq 1$, we deduce that 
every analytic function over~$C$ 
naturally extends to the whole punctured disk 
$\{0<|T|<1\}$ and is bounded on all sub-annuli of the 
form $\{r<|T|<1\}$, with $r>0$. It follows that the 
differential equation~$\Fs$ also extends to the whole 
punctured disk. 


A similar phenomenon occurs in the other cases. Therefore, we can assume that $\Fs$ is defined either 
on $\{0<|T|<1\}$, $\{0<|T|<+\infty\}$, or $\{1<|T|<+\infty\}$. 



In these three cases, the functions of $\O(C)$ are 
bounded in the neighborhood of the Gauss point~$x_{0,1}$,
hence it is a classical fact that it has a meaning 
to consider the restriction of functions of $\O(C)$ 
to the generic disk $D(x_{0,1})$, 
and therefore also the restriction of 
any differential equation $\Fs$ over $C$ to $D(x_{0,1})$ 
\cite{Ch}.\footnote{For 
instance if $t$ is an indeterminate over $K$ 
we can consider the field extension $K\to K((t))$, and 
endow $K((t))$ with the absolute value 
$|\sum a_n t^n|=\sup|a_n|$. It follows that $t$ is a 
Dwork generic point for $x_{0,1}$ 
(cf. \cite[Notation 2.1.5]{NP-III}), 
and any bounded function on $C$ converges over the 
generic disk $D(x_{0,1})$. 
More concretely we have a ring homomorphism 
$\O(C)\to \O(\{|T-t|<1\})$ defined by 
$f(T)\mapsto\sum_{n\geq 0}f^{(n)}(t)(T-t)^n/n!$, that 
commutes with $d/dT$.

In more geometric terms, since the absolute value associated to the Gauss point $x_{0,1}$ belongs to the spectrum of the ring of bounded functions, one may define a generic disk over this point by the usual base-change techniques (see \cite[Section~2.2]{NP-II}).} 
In particular it has a 
meaning to consider the radii of~$\Fs$ at~$x_{0,1}$.

\if{
\comment{\`A quoi sert le paragraphe qui suit~? \c Ca devrait d\'ej\`a figurer dans la remarque \ref{rem:logafftrivval} et de toute fa\c con \^etre utilis\'e dans la preuve du th\'eor\`eme \ref{thm:indexannulustrivialvaluation}. Je pense qu'on peut enlever le paragraphe.}
}\fi
It follows easily from the triviality of the absolute value of $K$ that we have $\R_{i}(x_{0,1},\Fs)=1$ for all $i$. Indeed, the solutions converge at least with radius one at the Dwork generic point. This fact is known as \emph{solvability} at $x_{0,1}$. In the case where 
$C=\{0<|T|<\infty\}$ 
it implies that the radii can only have a break at 
$x_{0,1}$. 
More precisely, for every extension $L/K$ and 
every point $x\in C(L)$, one verifies that 
the restriction of $\Fs$ to the maximal disk 
$D(x,S)$ is trivial.
Hence the controlling graph of $\Fs$ is given by 
$\Gamma_{\emptyset}(\Fs)=\Gamma_C$ 
and the radii are all $\log$-affine along the 
segments $]0,x_{0,1}]$ and 
$[x_{0,1},\infty[$. Indeed the radii along 
$\Gamma_C$ can be computed explicitly by the formal 
Newton polygon in a cyclic basis 
(see \cite[Section 5.7]{NP-III}, 
 the 
proof is similar to that of 
Propositions \ref{lem:merosinglinear-1} and 
\ref{Prop : chirel=Irr}). 


Since the residue field of $K$ has characteristic $0$, by Corollary~\ref{Cor. H^i(C)=H^i(C') restr}, the equivalent conditions of Proposition 
\ref{Prop : chirel=Irr} 
and Corollary~\ref{cor:cohomologypseudoannulus} hold. More explicitly, in the case where $C=\{0<|T|<1\}$, 
we can identify $\Fs(C)$ to $K((T))^r$ and 
$\nabla$ to a $K$-linear endomorphism of $K((T))$. 
Denote by $b_0$ and $b_1$ the germs of segment at 
the open boundary of $C$ as in 
Section \ref{section : Standard pseudo-annuli}.  
Then with the notation \eqref{eq : V_0 tt V_1}
we have $V_0:=K[[T]]^r$ and 
$V_1:=T^{-1}K[T^{-1}]^r$, and Proposition 
\ref{Prop : chirel=Irr} gives
\begin{eqnarray}
\label{eq : Irr_b_0^F=chigen}
\Irr_{b_0}^F(\Fs)\;=\;\chiabs_{b_0}(\Fs)
&\;:=\;&\chi(K[[T]]^r,p_0\circ\nabla(T\frac{d}{dT})
\circ i_0)
\;,\\
\Irr_{b_1}^F(\Fs)\;=\;\chiabs_{b_1}(\Fs)
&\;:=\;&\chi(T^{-1}K[T^{-1}]^r,p_1\circ\nabla(T\frac{d}{dT})
\circ i_1)\;,
\end{eqnarray}
where the superscript $^F$ stands 
for ``\emph{Formal}" 
and indicates that we are computing the slopes of the 
radii with respect to the trivial valuation on~$K$. In particular, since the radii are log-affine along 
$]0,x_{0,1}]$, one has
\begin{equation}\label{eq : Irr^F_0=Irr^F_1}
\Irr_{b_0}^F(\Fs)\;=\;-\Irr_{b_1}^F(\Fs)\;.
\end{equation}

In this situation, Theorem~\ref{thm:indexannulustrivialvaluation} gives 
another proof of the classical index theorem over 
$K((T))$ of Deligne-Malgrange \cite{Malgrange-Irreg}: 
if $G(T)\in 
M_r(K((T)))$, then, for the connection $\nabla := d/dT-G(T)$ on $K((T))^r$, we have
\begin{equation}\label{S3eq : chidr=0 formal}
\chi(K((T))^r,\nabla)\;=\;0\;.
\end{equation}

\subsection{Some remarks about the boundary 
conditions}
\label{section : some situations FIn}


We now state some useful remarks. In the following 
$\Fs$ is a differential equation on a quasi-smooth 
$K$-analytic curve $X$, $b$ is a good germ of segment in $X$ and $C$ is an open pseudo-annulus in $X$ representing $b$ such that the radii of $\Fs$ 
are all log-affine along $\Gamma_C$.
 
In this section, we deal with the conditions $\Fin_b$ and $\Fin_b^+$ (cf. Definitions \ref{Def: Fin+} and \ref{Def: Fin_b}), 
the Liouville conditions on the exponents, and 
the absolute indexes. By definition, to study them, it is 
harmless to extend the base-field. As a consequence, we 
will assume that~$K$ is spherically complete, 
algebraically closed 
and that $|K|=\mathbb{R}_{\geq 0}$. 

Over such a field, it follows from~\cite{Liu} (cf. also \cite[Remark 1.1.7]{NP-IV}) 
that any pseudo-annulus is isomorphic to a 
standard one
\begin{equation}\label{eq : isom ... C standard}
C\;\cong\;\{r_1<|T|<r_2\}\;,
\qquad r_1,r_2\in[0,+\infty]\;.
\end{equation}
In this section, we will assume that~$C$ has the latter 
form with the standard orientation.

\if{\comm{Au lieu de faire une longue liste de remarques 
qui paraissent éparpillés, j'ai fait des sous-sections avec 
des titres, ça me semble plus clair.}

\subsubsection{Three cases where $\Fin_C$ is 
automatic.}

It is known that 
condition $\Fin_C$ is automatically fulfilled 
in the following situations :
\begin{enumerate}
\item[(a)] the characteristic of the residual field 
$\widetilde{K}$ of $K$ is $0$;
\smallskip

\item[(b)] the characteristic of $\widetilde{K}$ is 
positive, the Christol-Mebkhout exponent of 
$\Fs^{\mathrm{Robba}}$ 
is non-Liouville and has non-Liouville differences (cf.  
Corollary \ref{Cor. H^i(C)=H^i(C') restr});\smallskip

\item[(c)] $\Fs^{\mathrm{Robba}}$ is an
extension of rank one differential 
modules defined by equations $\frac{d}{dT}-g(T)$ 
whose exponent $\mathrm{res}(g)\in\mathbb{Z}_p$ is 
non-Liouville. \smallskip

\end{enumerate}
Indeed, it is known that if (a) or 
(b) hold, then $\Fs$ satisfies (c) (cf. Theorems 
\ref{Thm : Kedlaya exponent char 0} 
and \ref{Thm : deco in rk 1 Ch-Me Robba}). 

Condition (c) implies condition $\Fin_C$ because 
the absolute indexes are additive on exact sequences, 
therefore we are reduced to working with rank one 
differential equations, in which case the computation of 
the absolute index can be achieved directly 
(see for instance \cite[Théorème 11.3.2]{Ch-Ro} or 
\cite[Section 4.19]{Ro-I}).

\subsubsection{The exact condition for the finiteness 
of the cohomology does not need $\End(\Fs)$.}

\comm{J'ai changé un peu et 
ajouté des précisions importantes au 
commentaire suivant (je dis que la condition Liouville 
n'est pas optimale sur une couronne, mais qu'elle est 
essentielle pour l'algébricité).}

If the characteristic of $\widetilde{K}$ is positive 
and if $\Fs^{\mathrm{Robba}}$ is extension 
of rank one modules $\Gs_i:\frac{d}{dT}-g_i(T)$, 
$i=1,\ldots,r$
the exact condition for the finite dimensionality of the 
cohomology is that for all $i$ the exponent
$\mathrm{res}(g_i)\in\mathbb{Z}_p$ is non-Liouville.
In particular, condition (b) is stronger than condition (c). 

\comm{J'ai modifié un peu le point suivant}

To show this, we firstly notice that 
a rank one equation of type Robba over $C$ is always 
isomorphic to a differential 
equation $\Ns(\lambda)$ defined by an equation 
of the form $T\frac{d}{dT}(Y)=\lambda\cdot Y$, with 
$\lambda\in \mathbb{Z}_p$ (cf. for instance 
\cite[Lemma 1.3, Propositions 1.1, 1.2]{Rk1} and 
Corollary \ref{Cor. H^i(C)=H^i(C') restr}).

Now, let $b$ be a germ of segment in $\Gamma_C$. For 
$\lambda\in\mathbb{Z}_p$ we have 
$\type_b(\lambda)=1$ (cf. Definition \ref{Def : type_b}) 
if, and only if, $\Ns(\lambda)$ is Fredholm at $b$, and in 
this case we have 
$\chiabs_b(\Ns(\lambda))=0$ (cf. Lemma 
\ref{Lemma : Liouville iff index O}). 

If $b'$ is another germ of 
segment in $\Gamma_C$ oriented as $b$, 
then, by definition, $\type_{b'}(\lambda)=
\type_{b}(\lambda)$, therefore 
$\Ns(\lambda)$ is Fredholm at $b$ if, and only if,  
it is Fredholm at $b'$. 
However, if $b'$ has opposite orientation with respect to 
$b$, we may have $\type_{b'}(\lambda)<1$.

This proves the claim.\\

The fact that the exponent of $\Fs^{\mathrm{Robba}}$ 
has non Liouville differences is actually used by 
Christol-Mebkhout to decompose 
$\Fs^{\mathrm{Robba}}$ into rank one 
pieces (cf. Theorem 
\ref{Thm : deco in rk 1 Ch-Me Robba}), 
but it is not a necessary condition 
for the finite dimensionality 
of the cohomology of $\Fs$.

However, if $\Fs^{\mathrm{Robba}}=0$, and if 
the differential module $\End(\Fs)$ has log-affine radii 
along $\Gamma_C$, then $\End(\Fs)$ 
can have a non trivial Robba part 
whose nature seems relatively unrelated to the fact that 
$\Fs$ itself decomposes into rank one pieces. 
In this case we will see in section 
\ref{section : Finite differential equations are 
algebraic} that the condition \medskip

\begin{enumerate}
\item[$(b')$] the exponent of  
$\End(\Fs)^{\mathrm{Robba}}$ is non Liouville and has 
non Liouville differences
\medskip
\end{enumerate}
implies the essential algebraicity of $\Fs$ which is 
one of the key points of our index theorems in Section 
\ref{section : 7}. If $\Fs$ is 
not of type Robba $(b')$ is essentially unrelated to 
condition $(c)$.

....

....

....

....
}\fi
\begin{enumerate}
%
%

\item It is known that 
condition $\Fin_b$ is automatically fulfilled 
in the following situations :

\begin{enumerate}
\item[(a)] the characteristic of the residual field 
$\widetilde{K}$ of $K$ is $0$;
\smallskip

\item[(b)] the characteristic of $\widetilde{K}$ is 
positive, the Christol-Mebkhout exponent of 
$\Fs^{\mathrm{Robba}}$ 
is non-Liouville and has non-Liouville differences (cf.  
Corollary \ref{Cor. H^i(C)=H^i(C') restr});\smallskip

\item[(c)] $\Fs^{\mathrm{Robba}}$ is an
extension of rank one differential 
modules defined by equations $\frac{d}{dT}-g(T)$ 
whose residue $\mathrm{res}(g)\in\mathbb{Z}_p$ is 
non-Liouville.\footnote{Notice 
that a rank one equation of type 
Robba over $C$ is always isomorphic to a differential 
equation $\Ns(\lambda)$ defined by an equation 
of the form $T\frac{d}{dT}(Y)=\lambda\cdot Y$, with 
$\lambda\in \mathbb{Z}_p$ (cf. for instance \cite[Lemma 1.3, Propositions 1.1, 1.2]{Rk1} and Corollary \ref{Cor. H^i(C)=H^i(C') restr}).} \smallskip

\end{enumerate}
Indeed, it is known that if (a) or 
(b) hold, then $\Fs$ satisfies (c) (cf. Theorems 
\ref{Thm : Kedlaya exponent char 0} 
and \ref{Thm : deco in rk 1 Ch-Me Robba}). 
Moreover, condition (c) implies condition $\Fin_b$ because 
the absolute indexes are additive on exact sequences, 
therefore we are reduced to working with rank one 
differential equations, in which case the computation of 
the absolute index can be achieved directly 
(see for instance \cite[Théorème 11.3.2]{Ch-Ro} or 
\cite[Section 4.19]{Ro-I}).

\item If the characteristic of $\widetilde{K}$ is positive 
and if $\Fs^{\mathrm{Robba}}$ is extension 
of rank one modules $\{\Gs_i:\frac{d}{dT}-g_i(T)\}_{i=1,\ldots,r}$, 
the exact condition for the finite dimensionality of the 
cohomology is that, for each $i$, the exponent
$\mathrm{res}(g_i)\in\mathbb{Z}_p$ is non-Liouville.
In particular, condition (b) is stronger than condition (c). 
The fact that the exponent of $\Fs^{\mathrm{Robba}}$ 
has non Liouville differences is actually used by 
Christol-Mebkhout to decompose 
$\Fs^{\mathrm{Robba}}$ into rank one 
pieces, but it is not a necessary condition 
for the finite dimensionality 
of the cohomology of $\Fs$.

However, such strong conditions on the exponent 
have an interest for the following 
reason. If $\Fs^{\mathrm{Robba}}=0$, and if 
the differential module $\End(\Fs)$ has log-affine radii 
along $\Gamma_C$, then $\End(\Fs)$ 
can have a non trivial Robba part 
whose nature seems relatively unrelated to the fact that 
$\Fs$ itself decomposes into rank one pieces. 
In this case we will see in \cite{NP-V} that the condition \medskip

\begin{enumerate}
\item[$(b')$] the exponent of  
$\End(\Fs)^{\mathrm{Robba}}$ is non-Liouville and has 
non-Liouville differences
\medskip
\end{enumerate}

implies the essential algebraicity of $\Fs$ which is 
one of the key points of our index theorems. 
If $\Fs$ is 
not of type Robba $(b')$ is essentially unrelated to 
condition $(c)$.


\item Denote by $\Ns(e)$ the differential module 
associated with the equation 
$T\frac{d}{dT}(Y)=e\cdot Y$, $e\in K$ 
(cf. Section \ref{Section : N(e)}).
For 
$\lambda\in\mathbb{Z}_p$ we have 
$\type_b(\lambda)=1$ (cf. Definition \ref{Def : type_b}) 
if, and only if, $\Ns(\lambda)$ is Fredholm at $b$, and in 
this case we have 
$\chiabs_b(\Ns(\lambda))=0$ (cf. Lemma 
\ref{Lemma : Liouville iff index O}). 

If $b'$ is another germ of 
segment in $\Gamma_C$ oriented as $b$, 
then, by definition, $\type_{b'}(\lambda)=
\type_{b}(\lambda)$, therefore 
$\Ns(\lambda)$ is Fredholm at $b$ if, and only if,  
it is Fredholm at $b'$. 
However, if $b'$ has opposite orientation with respect to 
$b$, we may have $\type_{b'}(\lambda)<1$.

\item There are no examples where 
$\chiabs_{b}(\Fs^{\mathrm{Robba}})$ 
(resp. $\chi(C,\Fs^{\mathrm{Robba}})$) 
is non-zero but finite. 
%
To our knowledge, the only example of irreducible Robba 
module seems to be \cite{Ch-Irrobba} where a 
Robba module of rank two was considered and it might be interesting to compute the absolute index in that example.

\item If (a) (resp. (b), (c)) holds, then it still  holds 
for the sub-quotients of  $\Fs$. In this 
case, $\Fin_b$ and $\Fin_b^+$ pass to sub-quotients 
as well.
In the general case it is unknown whether 
conditions $\Fin_b$ and $\Fin_b^+$ 
pass to sub-quotients.\footnote{To prove this fact a 
natural strategy would be the following.
Using Lemma 
\ref{Lemma : additivity on exact sequence}, we may 
argue in a similar way as the proof of 
\cite[Lemma 1.4.10]{NP-V}. 
But this fails because the proof of 
\cite[Lemma 1.4.10]{NP-V} uses in an essential 
way the fact that the space of solutions of a differential 
equation (i.e. the kernel of the connection) 
is always finite dimensional,
while in our context the kernels of the corresponding 
truncated operators \eqref{eq : truncated op u_k} 
are not necessarily finite dimensional. 
A further argument is needed.} 
However, there are no examples where this 
fails.

We also recall that the condition $\chidr(C,\Fs)=0$ 
passes to the sub-quotients of $\Fs$ 
(cf. \cite[Lemma 1.4.10]{NP-V}).

%
%

\item Condition $\Fin_b$ is satisfied by $\Fs$ 
if the radii of $\Fs$ are all spectral non-solvable along 
$\Gamma_C$, because in this case 
$\Fs^{\mathrm{Robba}}=0$. 

Condition~$\Fin_b$ is also satisfied in the following situations:
\begin{enumerate}
\item[(A)] Let $D$ be a disk in $X$ let $b$ be the germ 
of segment at the open boundary of $D$. If the radii of 
$\Fs$ are all constant along $b$, 
then they are constant on the whole $D$ by 
\cite[Lemma 4.3.12]{Kedlaya-draft}. 
Therefore, $\Fs^{\mathrm{Robba}}$ is trivial and (c) is 
satisfied. It follows that $\Fs$ satisfies $\Fin_b^+$ and $\Fin_b$.

\item[(B)] Let $x\in X$. Condition~$\Fin_b^+$ holds at 
all germs of segment $b$ out of $x$ that are not in the 
controlling graph $\Gamma_S(\Fs)$. 
Indeed, the connected component of 
$X-\Gamma_S(\Fs)$ containing $b$ is an open disk 
where (A) applies.
In particular, $\Fin_b^+$ and $\Fin_b$ may possibly fail only 
if $b\in\Gamma_S(\Fs)$.

\item[(C)] With the same notations as in (B), if all the 
radii are spectral non-solvable at $x$, then condition 
$\Fin_b$ and $\Fin_b^+$ holds for all germs of 
segment $b$ out of $x$. Indeed, if $C_b$ is an open pseudo-annulus in $X$ having $b$ in its open boundary and such that the radii of $\Fs$ are all log-affine along  
$\Gamma_{C_b}$, then the Robba part of 
$\Fs_{|C_b}$ is $0$. 
\end{enumerate}

\item The conditions of being Fredholm at $b$, the 
conditions $\Fin_b$ and $\Fin_b^+$, 
and the Liouville conditions on the 
exponents of $\Fs$ and $\End(\Fs)$, 
are not stable by tensor product nor by internal $\Hom$. For instance, with the notations of 
i)-(c) one has $\Ns(\lambda)\otimes\Ns(\lambda')
=\Ns(\lambda+\lambda')$ and 
$\Hom(\Ns(\lambda),\Ns(\lambda'))
=\Ns(\lambda'-\lambda)$, but, if $\lambda$ and 
$\lambda'$ are non-Liouville, then 
$\lambda+\lambda'$ and $\lambda'-\lambda$ are 
not necessarily non-Liouville. 
\item It follows from Lemma 
\ref{Lemma : additivity on exact sequence} that if 
$0\to\M\to\mathrm{E}\to\N\to0$ is an exact sequence
where $\M$ and $\N$ are both Fredholm at $b$ 
(resp. both satisfy $\Fin_b^+$), then so does 
$\mathrm{E}$. This seems not true in general for the 
condition $\Fin_b$, because the choice of the annulus 
$C'$ appearing in Definition 
\ref{Def: Fin_b} may be different for $\M$ and $\N$. 
If we can chose the same annuli~$C'$ for~$\M$ 
and~$\N$ in Definition \ref{Def: Fin_b} then $\mathrm{E}$ also satisfies $\Fin_b$.
\end{enumerate}


Let us now turn back to the case of an arbitrary field~$K$, but assume that it has residue characteristic~0. Note that this covers the case where~$K$ is trivially valued. Then i)-(a) has many important consequences such as the fact that a differential equation over an open pseudo-annulus whose radii are log-linear radii at the boundary always has finite-dimensional de Rham cohomology (see Theorem~\ref{Thm : index of finite opens}). Similarly, many statements that we had in the previous sections can be simplified because their assumptions are automatically satisfied. As an example, let us rewrite Lemma~\ref{Lemma : restrictions iso} in this setting.

\begin{lemma}\label{lem:restrictionisotrivialvaluation}
Assume that~$K$ has residue characteristic~0 (which holds for instance if~$K$ is trivially valued). Let~$C$ be an open pseudo-annulus and let~$\Fs$ be a differential equation on~$C$ with log-affine radii at the open boundary of $C$. Let~$C'$ be an open sub-pseudo-annulus of~$C$ with $\Gamma_{C'}\subseteq\Gamma_C$ such that $\chidr(C,\Fs) = \chidr(C',\Fs_{|C'})$. Then, for each $i\in\{0,1\}$, the restriction map
\begin{equation}
\Hdr^i(C,\Fs)\simto\Hdr^i(C',\Fs_{|C'})
\end{equation}
is an isomorphism.\hfill$\Box$
\end{lemma}

\section{Differential equations over an open disk 
with a meromorphic singularity}
\label{Disk-merom}

In this section, we provide an index formula 
for differential equations over open  
\emph{disks} with some meromorphic singularities. 

Such differential equations arise naturally in several contexts, but the study of their 
\emph{meromorphic cohomology} from a global point of view had not been carried out so far. Some important classical developments are due to 
Clark \cite{Clark} and Baldassarri \cite{Balda-Turritin} 
and concern differential equations over a 
\emph{germ of punctured disk}, i.e. differential 
equations over the field of convergent Laurent power 
series $K(\{T\})$ (cf. Definition 
\ref{eq : K(T) def union}). In this context, their 
results are stated under the crucial assumption that the 
formal exponents at $0$ are non-Liouville and/or have 
non-Liouville differences (the link with their 
results is given in Appendix 
\ref{App : Some comparison results}).

The novelty of this section is precisely the fact that the 
exact necessary and sufficient condition for the finite 
dimensionality of the \emph{meromorphic} 
de Rham cohomology over the disk is not a Liouville 
condition and does not arise at $0$ (contrary to 
what the results of Clark might suggest). 
The exact condition arises at the 
open boundary of the disk, and it consists precisely in the fact that 
the equation is Fredholm at the open boundary of the 
disk (not at the meromorphic singularities).

\subsection{Setting.}
\label{section : meromorphic Settings punctured disk}
For the whole Section~\ref{Disk-merom}, we fix a 
positive real number~$0<r_{D}\leq+\infty$ and set 
\begin{eqnarray}
D&\;:=\;&\{|T|<r_D\}\;;\\
C&\;:=\;&\{0<|T|<r_D\}\label{eq : C mero}
\end{eqnarray}
We set
\begin{eqnarray}
b_D&\;:=\;&\textrm{the germ of segment at the open 
boundary of }D\;,\label{eq : b_1}\\
b_0&:=&\textrm{the germ of segment out of }0\;;
\label{eq : b_2}
\end{eqnarray}

Denote by $K(\{T\})$ the field of convergent Laurent series. Every element of $K(\{T\})$ can be 
written as 
$\sum_{n\geq n_{0}} a_n\,T^n$,
with $n_{0} \in \Z$, for every $n\ge n_{0}$, $a_n\in K$ 
and the power series $\sum_{n\geq 0} |a_n|\, T^n$ 
has a positive radius of convergence. 
In other words, if $D'$ runs in the set of open disks 
centered at~$0$, we have
\begin{equation}\label{eq : K(T) def union}
K(\{T\})\;=\;\varinjlim_{D'}\O(D')[T^{-1}]
\;=\;\bigcup_{D'}\O(D')[T^{-1}]\;.
\end{equation}

We also consider the Robba ring at $0$
\begin{equation}\label{eq : Robba_0}
\mathfrak{R}_0\;:=\;
\mathfrak{R}_{b_0}\;=\;
\varinjlim_{C'}\O(C')\;=\;
\bigcup_{C'}\O(C')\;,
\end{equation} 
where $C'$ runs in the set of pseudo-annuli of the form 
$\{0<|T|<r\}$ with $r>0$.

The intersection of $K(\{T\})$ with $\O(C)$ in 
$\mathfrak{R}_0$ is the ring $\O(D)[T^{-1}]$ obtained 
from the ring $\O(D)$ of analytic functions on $D$ by 
inverting $T$.

We have a diagram of inclusions of rings
\begin{equation}\label{eq: rings at 0}
\xymatrix{
\O(D)[T^{-1}]\ar@{}[r]|{\subset} 
\ar@{}[d]|{\cap}& K(\{T\}) \ar@{}[r]|{\subset}\ar@{}[d]|{\cap}& K((T))\\
\O(C)\ar@{}[r]|{\subset}&\mathfrak{R}_0&
}
\end{equation}
where all the inclusions commute with $d/dT$.

\begin{remark}
\label{Section 4, Rk : valuation trivial then all ring are equal}
Notice that if the valuation of $K$ is trivial we have 
$\O(D)[T^{-1}]=\O(C)$ and 
$\mathfrak{R}_0=K(\{T\})=K((T))$; if moreover 
$r_D\leq 1$ then all the above rings coincide.\medskip
\end{remark}

In accordance with 
Section~\ref{Section : mero-alg}, 
the notation $D(*0)$ indicates the 
disk $D$ with structure sheaf $\O_{D}[*0]$ 
formed by analytic functions on $D-\{0\}$ with a 
meromorphic pole at~$0$.
The ring of global sections  
of $\O_{D}[*0]$ is $\O(D)[T^{-1}]$.
%

According to Definition 
\ref{def:meromorphicconnection}, a (meromorphic) 
differential equation $\Fc$ over $D(*0)$ 
is a sheaf of locally free $\Os_{D}[*0]$-modules 
of finite rank on~$D$ endowed with a connection. 
By abuse of notation 
we often denote by $\Fc$ the 
$\O(D)[T^{-1}]$-module of its global sections. This does not cause any trouble because of the following result.

\begin{proposition}\label{prop:eqcatStein}
The global section functor sets up an equivalence between the category of locally free $\Os_{D}[\ast 0]$-modules of finite rank (resp. endowed with a meromorphic connection) and the category of projective $\O(D)[T^{-1}]$-modules of finite type (resp. endowed with a connection).
\hfill$\Box$
\end{proposition}
\begin{proof}
Such a result is classical if one replaces $\Os_{D}[\ast 0]$ and $\O(D)[T^{-1}]$ by $\Os_{D}$ and $\O(D)$ respectively (see \cite[Corollary~4.11]{Banachoid} for instance). The proof relies on Kiehl's Theorem: for every coherent sheaf~$\Fs$ of $\O_{D}$-modules, one has $H^1(D,\Fs) = 0$ and $\Fs$ is generated by its global sections. Note that Kiehl's Theorem immediately extends to $\O_{D}$-modules that are filtered direct limits of coherent sheaves. Arguing as in the classical case then gives the result we want. The version with connections follows immediately. We refer to \cite[Corollary~1.7.15]{NP-V} for more details.
\end{proof} 

For a  differential equation $\Fc$ over 
$D(*0)$ we set
\begin{eqnarray}\label{eq : setting merom}
\Fs&\;:=\;&\Fc\otimes_{\O_D[*0]}\O_C\;=\;
\Fc_{|C}\;,\\
\Fc^\dag_0&\;:=\;&
\Fc\otimes_{\O_D[*0]} K(\{T\})\;,
\label{eq : Fcdag0=calFotimes K(T)}\\
\M&\;:=\;&
\Fc\otimes_{\O_D[*0]}K((T))\;,
\label{eq : M=calFotimes K((T))}\\
\Fs^\dag_0&\;:=\;&
\Fc\otimes_{\O_D[*0]}\mathfrak{R}_0\;.
\label{eq : Fsdag0=calFotimes Robba_0}
\end{eqnarray}

\if{
We now deal with meromorphic and analytic de Rham 
cohomologies. We first state some elementary lemmas.
\begin{lemma}
Let $C'$ be any open annulus 
in $C$ such that $\Gamma_{C'}\subset\Gamma_C$. 
Assume that $\Fs_{|C'}$ verifies the Liouville condition at the 

Then the natural map
\begin{eqnarray}
\Hdr^i(D(*0),\Fc)&\;\xrightarrow{\quad}\;&
\Hdr^i(C',\Fs_{|C'})
\end{eqnarray}
is injective for $i=0$ and surjective for $i=1$.
\end{lemma}
\begin{proof}
The injectivity for $i=0$ 
follows immediately from the fact that 
the map $\O(D)[T^{-1}]\to\O(C')$ is injective (cf. 
\cite[Lemma 1.2.8]{NP-III}).

The surjectivity for $i=1$
follows  $\O(D)[T^{-1}]$ is dense in $\O(C')$, together 
with the fact that $\Hdr^1(C',\Fs_{C'})$ is finite 
dimensional, hence separated (cf. \ref{}).
\end{proof}
}\fi

\subsection{Log-affinity of the radii at $0$.} 
\label{section : Log-affinity of the radii at 0}
We begin by showing that 
the radii of $\Fs$ are all log-affine around $0$. 

For 
$\rho \in ]0,r_{D}[$, denote by~$x_{\rho}$ as usual the point unique point in the Shilov boundary
of the annulus $\{|T| = \rho\}$.

If, in a cyclic basis, $\Fs$ is given by a 
differential operator $\Ls$, there is a (spectral Newton) 
polygon associated with $\Ls$ 
whose slopes are related to the radii of the solutions by 
Young's theorem (cf. \cite{Young} and 
\cite[Proposition 4.11]{NP-I} for a setting more closely 
to ours). 
The notation of \cite[Proposition 4.11]{NP-I} 
is adapted to spectral radii \cite[(4.4)]{NP-I}, therefore 
we here reformulate the 
statement with respect to the 
normalized radii $\R_{i}(x_\rho,\Fs)$ 
(cf. \cite[Definition 2.3.1]{NP-III}
).

\begin{definition} 
\label{Def : NP of an operator}
Let 
$\Ls := \sum_{i=0}^r g_{r-i}(T)\cdot (d/dT)^i$ with 
$g_{0}=1$ and, for all $j\in \{1,\dotsc,r\}$, 
$g_{j} \in \Os(C)$. Denote by $\Fs$ the differential 
equation associated with~$\Ls$.
Define the spectral Newton polygon $NP(x_\rho,\Ls)$ 
of~$\Ls$ at~$x_{\rho}$ as the Newton polygon of the 
set (cf. \eqref{eq: NPV def-1})
\begin{equation}\label{eq : spectral NP}
\Bigl\{\Bigl(k, -\ln\left((\rho/\omega)^{k}\cdot
|g_{k}(x_{\rho})|\right)\Bigr) \mid 
0\le k \le r\Bigr\}\;.
\end{equation}
Denote by $s_i^{\Ls}(x_\rho)$ the $i$-th slope of this 
polygon (cf. Definition \ref{Def : slope, height, break-1}). 
Set, as usual (cf. \eqref{eq : s_i = log(R_i) newton pol-1})
\begin{equation}
s_i(x_\rho)\;:=\;
\ln(\R_{i}(x_\rho,\Fs))\;.
\end{equation}
\end{definition}

It follows from the definition that 
\begin{equation}
s_1^{\Ls}(x_\rho)\;=\;\min_{1\le k\leq r}
\frac{-\ln\Bigl((\rho/\omega)^k|g_k(x_\rho)|\Bigr)}{k}\;.
\end{equation}

Recall the definition of $\omega$ (cf. 
\eqref{eq : OMEGA}).
\begin{proposition}[\protect{\cite{Young}, \cite[Proposition 
4.11]{NP-I}}]\label{Prop : Young normalized}
We maintain the notations of Definition 
\ref{Def : NP of an operator}. One has  
$s_i(x_\rho)<\ln(\omega)$ if, 
and only if, $s_i^{\Ls}(x_\rho)<\ln(\omega)$ 
and, in this case,
\begin{equation}
\qquad\qquad
s_i^{\Ls}(x_\rho)\;=\;s_i(x_\rho)\;.
\qquad\qquad\Box
\end{equation}
\end{proposition}

\begin{lemma}\label{lem:merosinglinear-1}
There exists $\eps\in \mathopen{]}0,1\mathclose{[}$ 
such that the radii 
$\R_{i}(-,\Fs)$ are all 
$\log$-affine on the segment 
$\{x_{\rho} \mid 0 < \rho < \eps\}$.
\end{lemma}
\begin{proof}
By assumption~$\Fc$ is a locally free $\Os_{D}[*0]$-module, hence, by shrinking~$C$ around~0, we may assume that $\Fc$ is free. By shrinking again, we may assume that it has a cyclic basis where it is given by a differential operator~$\Ls$. This operation does not affect the radii of~$\Fs$. The advantage is that we can now read the radii in terms of the coefficients of~$\Ls$. 

Set $r:=\rk(\Fs)$. For all $i \in \{1,\dotsc,r\}$, the function $\log(\rho)\mapsto \log(H_{i}(x_{\rho},\Fs))$ is concave along $\of{]}{-\infty,\log(r_{D})}{[}$ and bounded by~0. We deduce that either $H_{i}(-,\Fs)$ 
is constant on $\of{]}{0,x_{\varepsilon'}}{[}$, for some $\eps' \in\of{]}{0,r_{D}}{]}$, or we have 
$\lim_{\log(\rho)\to -\infty}\log(H_{i}(x_{\rho},\Fs))=-\infty$.

We now proceed by induction on $i$ to show that the sequence of slopes of $\log(\rho)\mapsto \log(H_{i}(x_{\rho},\Fs))$ is constant around $0$. For $i=1$, if we have $\lim_{\log(\rho)\to -\infty}\log(H_{1}(x_{\rho},\Fs))=-\infty$, then for $\rho$ close enough to~0, 
the first radius $\R_{1}(x_{\rho},\Fs) = H_{1}(x_{\rho},\Fs)$ is smaller than~$\ln(\omega)$, hence explicitly intelligible in terms of the coefficients of the operator~$\Ls$, by Proposition~\ref{Prop : Young normalized}.
Since the coefficients lie in $K(\{T\})$, they have only finitely many slopes along $\of{]}{0,x_{r_{D}}}{]}$ and the result follows.

Now, assume inductively that for all $j \in \{1,
\dotsc,i-1\}$ the radii $\R_{j}(-,\Fs)$ are 
log-affine on some $\of{]}{0,x_{\varepsilon'}}{[}$. 
Since $\log(H_{i}(-,\Fs))$ is 
concave, then so is 
\begin{equation}
\log(\Rc_{i}(-,\Fs)) = \log(H_{i}(-,\Fs)) - \sum_{j=1}^{i-1} \log(\Rc_{j}(-,\Fs))
\end{equation} 
and we can use the same argument as above. 
\end{proof}

\begin{corollary}\label{cor:meromorphiclogaffine}
Let~$P$ be a quasi-smooth $K$-analytic curve, let~$z$ be a rigid point in~$P$ and let~$\Gc$ be a meromorphic differential equation on~$P$ with poles at~$z$. Set $Y := P -\{z\}$ and denote by~$b_{z}$ the germ of segment out of the point~$z$, seen as a germ of segment at the open boundary of~$Y$. Then all the radii of convergence of $\Gs := \Gc_{|Y}$ are log-affine along~$b_{z}$ (in the sense of Definition~\ref{def:logaffinetotalheight}).
\hfill$\Box$
\end{corollary}

\begin{definition}\label{def:Irrmeromorphic-1}
In the setting of Corollary~\ref{cor:meromorphiclogaffine}, we define the irregularity of~$\Gc$ 
at~$z$ to be 
\begin{equation}\label{eq : formal irreg-1}
\Irr_{z}(\Gc) \;:=\; \Irr_{b_{z}}(\Gs) \;\in \;\Z
\end{equation}
(see Definition~\ref{def:Irrgermpseudoannulus}).
\end{definition}

\begin{corollary}\label{Cor : algebraic implies finite}
Let $\Yk$ be a smooth connected algebraic curve over~$K$ and let $\Fk$ be an 
algebraic differential equation on $\Yk$. Then, the 
analytified equation $\Fk^{\an}$ has log-affine radii 
at the open boundary of $\Yk^{\an}$.\hfill$\Box$
\end{corollary}
\begin{proof}
Consider a compactification~$\Yk'$ of~$\Yk$, 
\textit{i.e.} a smooth connected projective curve 
over~$K$ containing~$\Yk$ as an open subset. Then, 
$\Fk$ extends to a differential equation~$\Fk'$ 
on~$\Yk'$ with meromorphic poles on $\Zk := \Yk'-\Yk$. 
The germs of segment at infinity of~$\Yk^\an$ 
correspond bijectively to the germs of segment 
in~${\Yk'}^\an$ out of the points of~$\Zk^\an$. By 
Corollary~\ref{cor:meromorphiclogaffine},
the differential 
equation~$\Fk^\an$ has log-affine radii along those 
germs. 
\end{proof}

\subsection{Derived Newton polygon at $b_0$.}
\label{Derived newton polygon at b0}
In Section~\ref{Section polygons and derivatives}, we explained that there is no natural way 
to define a derivative of the convergence Newton polygon 
on a germ of segment. We here show that there is a 
natural definition at the germ of segment out of a point 
of type $1$ or out of a meromorphic singularity.


\begin{lemma}\label{lem:c+pt-1}
Let $n\le m \in \Z$. For each $i\in\{n,\dotsc,m\}$, let 
$c_{i} \in \ERRE\cup\{+\infty\}$ and $p_{i} \in \ERRE$. 
Assume that $c_{n},c_{m}<+\infty$. Let $t\in \ERRE$. 
For each $i\in\{n,\dotsc,m\}$, set 
$v_{i} := c_{i} + p_{i} t$ and denote by $N(t)$ the 
Newton polygon of the set 
$\{(i,v_{i}(t)) \mid n\le i\le m\}$. Denote its slopes by 
\begin{equation}
s_{1}(t) \;\le\; \dotsb \;\le\; s_{m-n}(t)\;.
\end{equation}

For each $i\in\{n,\dotsc,m\}$, set 
$\partial_{b_0}v_{i} := -\infty$
 if 
$c_{i}=+\infty$ and 
\begin{equation}
\partial_{b_0}v_{i}\; :=\;\frac{d}{dt}(v_{i}(t))\;=\; p_{i}
\end{equation}
otherwise. 
Denote by 
\begin{equation}
\partial_{b_0}N\;:\;[n,m]\;\xrightarrow{\quad}\;\mathbb{R}\;
\end{equation}
the inverted Newton 
polygon of the set $\{(i,\partial_{b_0}v_{i}) 
\mid n\le i\le m\}$. Denote its slopes by 
\begin{equation}
s^{b_0}_{1} \;\ge\; \dotsb \;\ge\; s^{b_0}_{m-n}\;.
\end{equation}

Then, there exists $t_{0}\in \ERRE$, such that, for each 
$j\in\{1,\dotsc,m-n\}$, the map 
$t \in \mathopen{]}-\infty, t_{0} \mathclose{[} 
\mapsto s_{j}(t)$ is affine with slope~$s^{b_0}_{j}$:
\begin{equation}
\frac{d}{dt}(s_j(t))\;=\;s_j^{b_0}\;.
\end{equation}
\end{lemma}
\begin{proof}
Let us prove the result by induction on~$m-n$. 
If $m-n = 0$, there is no slope and the result is trivial.

Assume that $m>n$. 
Set $I := \{i\in\{n+1,\dotsc,m\} \mid c_{i} < +\infty\}$. 
Let $t\in \ERRE$. The first slope of~$N(t)$ is 
\begin{equation}
s_{1}(t) = \min_{i \in I} \left(\frac{(c_{i}-c_{n}) + (p_{i}-p_{n})t}{i-n} \right).
\end{equation}

Set $P := \max_{i \in I} \left(\frac{p_{i}-p_{n}}{i-n} \right)$ and $I_{p} := \{i\in I \mid (p_{i}-p_{n})/(i-n) = P\}$. Set $C := \min_{i\in I_{p}} \left(\frac{c_{i}-c_{n}}{i-n}\right)$. It is easy to check that there exists $t_{0}\in \ERRE$ such that, for each $t<t_{0}$, we have 
\begin{equation}
s_{1}(t) \;=\; C+Pt\;.
\end{equation}

Similarly, the first slope of~$\partial_{b_0}N$ is 
\begin{equation}
s^{b_0}_{1} = \max_{i \in I} \left(\frac{p_{i}-p_{n}}{i-n} \right) = P,
\end{equation}
hence the result holds for the first slope.

\medbreak

Let $t<t_{0}$. The polygon~$N(t)$ passes through the point $(n+1,c_{n}+p_{n}t + s_{1}(t))$. Set $\bar v_{n+1} := c_{n}+p_{n}t + s_{1}(t) = (c_{n}+C)+(p_{n}+P)t$. For $i\in\{n+2,\dotsc,b\}$, set $\bar v_{i}=v_{i}$. Denote by $\bar N(t)$ the Newton polygon of the set $\{(i,\bar v_{i}(t)) \mid n+1\le i\le m\}$. By construction, its slopes $\bar s_{1}(t) \le \dotsb \le \bar s_{m-n-1}(t)$ are exactly $s_{2}(t) \le \dotsb \le s_{m-n}(t)$.

Similarly, the polygon~$\partial_{b_0}N$ passes through the point $(n+1,p_{n} + s^{b_0}_{1})$. 
Set $\bar v^{b_0}_{n+1} := p_{n} + s^{b_0}_{1} = p_{n}+P$. For $i\in\{n+2,\dotsc,m\}$, set $\bar v^{b_0}_{i}=\partial_{b_0}v_{i}$. Denote by $\bar N^{b_0}$ the inverted Newton polygon of the set $\{(i,\bar v^{b_0}_{i}) \mid n+1\le i\le m\}$. By construction, its slopes $\bar s^{b_0}_{1} \ge \dotsb \ge \bar s^{b_0}_{m-n-1}$ are exactly $s^{b_0}_{2}\ge \dotsb \ge s^{b_0}_{m-n}$.

Moreover, the polygon~$\bar N^{b_0}$ is exactly the inverted Newton polygon associated with the Newton polygon~$\bar N(t)$ by the construction of the statement. By induction, there exists $t_{1} \le t_{0}$ such that, for each $j\in\{1,\dotsc,m-n-1\}$, the map $t \in \mathopen{]}-\infty, t_{1} \mathclose{[} \mapsto \bar s_{j}(t)$ is affine with slope~$\bar s^{b_0}_{j}$. This concludes the proof.
\end{proof}

The following result is a direct consequence of Lemmas \ref{lem:merosinglinear-1} 
and \ref{lem:c+pt-1}.

\begin{proposition}
\label{Prop : Two derivatives coincide at b_0}
Let $Z$ be a locally finite set of rigid points in $X$. 
Let $\Fc$ be a meromorphic differential equation on 
$X$ with poles in $Z\subset X$.
Let $x\in X$ be a point of type $1$ and let $b_x$ be the 
germ of segment out of $x$ (as usual 
oriented away from $x$). 

Let $C_{x}$ be an open pseudo-annulus whose skeleton represents~$b_{x}$. In the following, we compute the radii of convergence on~$C_{x}$. Let $r$ be the rank of~$\Fs$ on~$C_{x}$.


Denote by 
\begin{equation}
s_1^{b_x}\; \ge\; \cdots \;\ge\; s_{r}^{b_{x}}
\end{equation} 
the slopes of 
the \emph{inverted} polygon 
associated to the set
\begin{equation}
\{(i,\partial_{b_{x}} N\! P(-,\Fs)(i))\;,\;0\leq i\leq r\}\;,
\end{equation}
where $N\! P(y,\Fs)$ denotes the convergence Newton polygon of~$\Fs$ at~$y$ (see Definition~\ref{Def: Conv NP}).


Then, for each $i\in\{1,\dotsc,r\}$, we have
\begin{equation}
\partial_{b_x} \ln(\R_{i}(-,\Fs)) \;=\;s_i^{b_x}\;.
\end{equation}
In particular, the sequence 
$(\partial_{b_x}(\ln(\R_{1}(-,\Fs))),\ldots,\partial_{b_x}(\ln(\R_{r}(-,\Fs))))$ is 
non-increasing and non-negative.


\hfill$\Box$
\end{proposition}


Proposition \ref{Prop : Two derivatives coincide at b_0}
shows that the two ways to construct a derived 
polygon from the convergence Newton polygon that we 
have evoked in Section~\ref{Section polygons and derivatives} 
are actually the same. We are then 
allowed to formulate without ambiguity 
the following definition.

\begin{definition}[Derivative of $N\!P(-,\Fs)$]
\label{Def : derivative of the polygon at b_0}
Let $Z$ be a locally finite set of rigid points in $X$. 
Let~$\Fc$ be a meromorphic differential equation on 
$X$ with poles in $Z\subset X$.
Let $x\in X$ be a point of type~$1$ and let $b_x$ be the 
germ of segment out of $x$ (as usual 
oriented away from $x$). 
We call \emph{derivative of the convergence Newton 
polygon of $\Fc$ at $b_x$} the inverted polygon 
associated with the family 
\eqref{eq : inverted polygon derivative at b-1}. 
\end{definition}

\begin{remark}\label{Remark : affine partial< implies <}
We maintain the notation of 
Definition \ref{Def : derivative of the polygon at b_0}.
Assume that the derivative of the convergence Newton 
polygon of $\Fc$ at $b_x$ has a break at 
$k\in\{1,\ldots,r-1\}$. Note that the point $x$ must then belong to~$Z$ because the radii are 
constant around a type~1 point that is not a singularity of the differential equation.

By Proposition~\ref{Prop : Two derivatives coincide at b_0}, 
we $\partial_{b_x}(s_k) > \partial_{b_x}(s_{k+1})$. We 
also have $s_{k}<s_{k+1}$ over $b_x$, because the 
functions $s_k=\ln(\R_{S,k}(-,\Fs))$ 
are log-affine functions in the neighborhood of~$x$ (which corresponds to~$-\infty$ in the logarithmic coordinate). Note that, for each~$y$ in some segment 
representing~$b_{x}$, the convergence Newton 
polygon $N\!P_S(y,\Fs)$ has a break at $k$ too.
\end{remark}


\subsection{Analytic vs. formal 
irregularities.}
\label{Section : ANVSFORMAL irr}
We now compare formal and analytic irregularities at 
$0$. 

%

We maintain the notations of Section 
\ref{section : meromorphic Settings punctured disk}. 
In particular $\M=\Fc\otimes K((T))$.

We have seen in Section  
\ref{Remark : index annulus trivial valuation} that, if 
$K$ is endowed with the trivial valuation, then 
$K((T))$ coincides with the ring $\O(C')$ 
of analytic functions over an annulus 
$C'=\{0<|T|<r'\}$, with $r'\in ]0,1[$. 
By Theorem~\ref{thm:indexannulustrivialvaluation}, the 
index of $\M$ on $K((T))$ exists and it is always zero.
On the other hand, Proposition \ref{Prop : chirel=Irr} 
computes the absolute index of $\M$ at a germ of 
segment $b$ in the skeleton of the formal annulus 
$\{0<|T|<1\}$ 
(cf. Section \ref{Rk : formal gen indexes}):
\begin{equation}
\chiabs_{b}(\M)\;=\;
\Irr_{b}^{F}(\M)\;,
\end{equation}
where the superscript~$F$ (which stands 
for ``Formal'') indicates that we are computing 
the slope with respect to the trivial valuation on $K$.


It is natural to ask whether the 
the irregularity $\Irr_0(\Fs)$ from Definition~\ref{def:Irrmeromorphic-1} coincides with the formal 
irregularity $\Irr_0^F(\M)$ (from Definition~\ref{def:Irrmeromorphic-1} again but applied to~$K$ endowed with the trivial valuation).
We prove that this is indeed the case.

\bigbreak

More specifically, if $K$ is trivially valued, we denote by $b_0^F$ the germ 
of segment out of $0$ (as usual the superscript $^F$ 
stands for ``formal'').

Propositions \ref{Prop : Two derivatives coincide at b_0} 
and  \ref{Prop : Young normalized} applies to both $\Fc$ 
and $\M$. The latter is a $K((T))$-differential 
module (cf. Remark 
\ref{Section 4, Rk : valuation trivial then all ring 
are equal}). Now, there is another 
polygon classically associated with $\M$: the 
\emph{formal Newton polygon} that is defined in 
\cite[Section 5.7]{NP-III}. 
The following result provides a 
link between all the polygons.
\if{Thanks to Proposition 
\ref{Prop : Two derivatives coincide at b_0} 
we may speak without ambiguity about the inverted 
Newton polygons of $\Fs$ and $\M$ along $b_0$ and 
$b_0^F$ respectively (cf. Definition 
\ref{Def : derivative of the polygon at b_0}). 
The following proposition shows that they coincide.}\fi
\begin{proposition}\label{Corollary : Irr^F=Irr}
\label{Prop : Irr-form=Irr-x-1}
We maintain the notations of Section 
\ref{section : meromorphic Settings punctured disk}. 
The following polygons coincide:
\begin{enumerate}
\item the derivative of the convergence Newton 
polygon of $\Fc$ at $b_0$ (cf. Definition \ref{Def : derivative of the polygon at b_0});
\item the derivative of the convergence Newton 
polygon of $\M$ at $b_0^F$ (cf. Definition 
\ref{Def : derivative of the polygon at b_0} with $K$ 
trivially valued).
\end{enumerate}
Moreover 
\begin{enumerate}
\item[iii)] if $p_1\geq p_2\geq \cdots\geq p_r$ are the 
slopes of the above polygons, then the formal Newton 
polygon of $\M$ 
coincides with the polygon $N:[0,r]\to \mathbb{R}$ 
having slopes $p_1'\leq p_2'\leq\cdots\leq p_r'$, where 
$p_i':=p_{r-i+1}$ (i.e. reordered in increasing order) 
and such that $v_r=0$.
\end{enumerate}

We summarize these properties 
by saying that the derivative of the convergence Newton 
polygon along $b_0$ is independent of the valuation of 
$K$. In particular, 
\begin{enumerate}
\item[iv)] we have the equality
\begin{equation}\label{eq : Irr_0=Irr_0^F}
\Irr_{b_0}(\Fs)\;=\;\Irr_{b_0^F}(\M)\;.
\end{equation}
Moreover, $\Irr_{b_0^F}(\M)$ coincides 
with the opposite of the formal irregularity 
$i_0(\M)$ of $\M$ 
defined in
\cite{Ramis-Devissage-Gevrey} and 
\cite{Correspondance-Malgrange-Ramis}
(cf. \cite[Section 5.7]{NP-III}).
\item[v)] If the formal Newton polygon of $\M$ has a 
break at $k$, then so has the convergence Newton 
polygon $NP(y,\Fs)$, for all $y$ in some 
segment representing $b_0$.
\end{enumerate}
\end{proposition}
\begin{proof}
The claim is invariant by restriction of the radius of $D$.  
Therefore, we may assume 
that~$\Fc$ is associated with a differential operator 
$\Ls := \sum_{i=0}^r g_{r-i}(T) (d/dT)^i$ with 
coefficients in $\O(*D)$ as in Proposition 
\ref{Prop : Young normalized}. 
Denote by $s_i^{\Ls}(x_\rho)$ 
(resp. $s_i^{L}(x_\rho)$) the $i$-th slope of this 
polygon at~$x_{\rho}$.
\if{ and set $p_i^{\Ls}:=\partial_{b_0}(s_i^{\Ls})$ 
(resp. $p_i^{L}:=\partial_{b_0^F}(s_i^{L})$). Define 
$s_i^{\Ls,b_0}$ (resp. $s_i^{L,b_0^F}$) as in 
Proposition \ref{Prop : Two derivatives coincide at b_0}.
}\fi
%
By Proposition \ref{Prop : Young normalized}, 
$s_i(x_\rho):=
\ln(\R_{i}(x_\rho,\Fs))<\ln(\omega)$ if 
and only if $s_i^{\Ls}(x_\rho)<\ln(\omega)$ 
and in this case
\begin{equation}
s_i^{\Ls}(x_\rho)\;=\;
\ln(\R_{i}(x_\rho,\Fs))\;=\;s_i(x_\rho)\;.
\end{equation}
The same holds for $L$ and $\M$. Namely, if we set 
$s_i^F(x_\rho):=\ln(\R_{i}(x_\rho,\M))$, 
then $s_i^F(x_\rho)<\ln(\omega)=0$ if and only if 
$s_i^L(x_\rho)<\ln(\omega)=0$ and 
in this case one has $s_i^{\Ls}(x_\rho)\;=\;
\ln(\R_{i}(x_\rho,\M))\;=\;s_i^F(x_\rho)$. 

This correspondence of slopes holds only for slopes that 
are less than $\ln(\omega)$. In our case this is enough 
because the slopes $s_i$ (resp. $s_i^F$) 
that are larger than $\ln(\omega)$ along $b_0$ 
(resp. $b_0^F$) are constant functions in a 
neighborhood of $-\infty$. Indeed, by Lemma 
\ref{lem:merosinglinear-1} the functions $s_i$ and 
$s_i^F$ are affine along in a neighborhood of $-\infty$ 
and, by definition of the radii, they are also 
less than or equal to~$0$. It follows that if $s_i$ (resp. $s_i^F$) is not constant over 
$b_0$ (resp. $b_0^F$), then it has to tend to $-\infty$ 
as $\rho$ approaches $0$. In particular, it is less than 
$\ln(\omega)$ in a neighborhood of $-\infty$ 
(see the proof of Lemma \ref{lem:merosinglinear-1} for 
more details).

Now, denote by $v_i^\Ls$ (resp. $v_i^{L}$) the 
$i$-th partial height of $NP(x_\rho,\Ls)$ 
(resp. $NP(x_\rho,L)$) as in Definition 
\ref{Def : NP of an operator}. For $i=0,\ldots,r$, 
we denote by  $s_i^{\Ls,b_0}$ (resp. 
$s_i^{L,b_0^F}$) the $i$-th slope of the Newton 
polygon associated with the set 
$\{(i,\partial_{b_0}v_i^\Ls\;,\; 0\leq i\leq r)\}$ 
(resp. $\{(i,\partial_{b_0}v_i^L)\;,\; 0\leq i\leq r\}$).

With this notation, to prove the equality of polygons 
as in i) and ii) it is enough to show that for all $i$ one has $\partial_{b_0}s_i^{\Ls}=\partial_{b_0}s_i^{L}$. 
By Proposition \ref{lem:c+pt-1}, 
this is equivalent to showing that for all $i$ one has
\begin{equation} \label{eq : s_i =s_i^F for L}
s_i^{\Ls,b_0}\;=\;
s_i^{L,b_0^F}\;.
\end{equation}

Now, this follows from the following remark. For each non-zero
function $f(T)=\sum_{i} a_i T^i\in\O(D)[T^{-1}]$ and each $\rho>0$ close enough to $0$, one has 
\begin{equation}
|f|(x_\rho)\;=\; |a_{v_T(f)}| \, \rho^{v_T(f)}\;,
\end{equation}
where $v_T(f)=\min(i,a_i\neq 0)$ is the $T$-adic 
valuation of $f$. 
This equality is true for any valuation of~$K$. 
In particular, the derivative $\partial_{b_0}(f)$ is 
independent of the valuation of $K$.

It follows that for all $i=1,\ldots,r$ one has 
$\partial_{b_0}(g_i)=\partial_{b_0^F}(g_i)$ and 
therefore \eqref{eq : s_i =s_i^F for L} holds.
This proves the coincidence of the two polygons in i) and 
ii).

Item iii) follows from \cite[Section 5.7]{NP-III}. 
Now, equality \eqref{eq : Irr_0=Irr_0^F} follows readily 
from the coincidence of the polygons in i) and ii) because 
the irregularity is nothing but the opposite of their total 
height. Analogously, the 
equality between $\Irr_{b_0^F}(\M)$ and $-i_0(\M)$ 
follows from iii). Indeed, since  
the total heights of the polygons in ii) and of the formal 
Newton polygon in iii) coincide, 
the claim follows from 
Remark \ref{rk : total height derived = -Irr-1} and the 
fact that the total height of the formal 
Newton polygon is by definition $i_0(\M)$.

Finally, v) follows from Remark 
\ref{Remark : affine partial< implies <}.
\end{proof}

\begin{remark}
\label{Remark : no relations between the radii}
In general, there is no direct relationship between 
$NP(-,\Fs)$ and $NP(-,\M)$. 
For instance, let $a\in K$ and let $\Fc$ be the equation 
$T\frac{d}{dT}(y)=ay$. If the valuation of $K$ is 
trivial, then the formal radii of $\M$ are uniformly equal to 
$1$ along the (formal) segment $]0,+\infty[$. Now, if the 
valuation of $K$ is not trivial on $\mathbb{Z}$, and if the residue characteristic of~$K$ is~$p$, then the radii of $\Fs$ 
along the segment $]0,+\infty[$ are constant and equal 
to $|p|^{\frac{1}{p-1}}\cdot
\liminf_s|a(a-1)(a-2)\cdots(a-s+1)|^{-1/s}$ 
(cf. \cite[Lemma 1.4]{NP-I}). This radius is not equal to 
$1$, i.e. non-maximal, 
if and only if $a\notin\mathbb{Z}_p$, 
(cf. \cite[Proposition 7.3, Chapter IV]{DGS}). 
\end{remark}

\subsection{Index of a differential equation with meromorphic singularities
on an open pseudo-disk.}
\label{Section 4.4: index D(*Z)}
\if{\comment{Plus bas, je traite le cas de plusieurs singularit\'es. Je garde le cas d'une seule ci-dessous tant que tu n'as pas v\'erifi\'e. Dans la version avec une singularit\'e, il y a des commentaires par rapport \`a la tienne ; je ne les ai pas gard\'es dans l'autre.}

We maintain the notations of Section 
\ref{section : meromorphic Settings punctured disk}. 

In order to deal with index theory, 
rather than Definition \ref{def:meromorphicconnection}, 
we consider differential modules over
the differential ring $(\O(D)[T^{-1}], d/dT)$. By 
Proposition 
\ref{prop:eqcatStein} this is not restrictive. 
\if{
The following Proposition gives more precisions.
\begin{proposition}
\label{PROP : proj implies presque loc free}
Let $A$ be a finitely generated 
$\O(D)[T^{-1}]$-module together with a 
connection, and let $r$ be the dimension of 
$A\otimes K((T))$. Then 
\begin{enumerate}
\item for all (closed or open) strict sub-disk 
$E\subset D$ containing $0$ the restriction 
$A_{|E(*0)}:=
A\otimes_{\O(D)[T^{-1}]}\O(E)[T^{-1}]$ 
is a free $\O(E)[T^{-1}]$-module of rank $r$.
\item for all open pseudo-annulus 
(resp. closed annulus) 
$C'\subseteq C=D-\{0\}$ such that 
$\Gamma_{C'}\subseteq\Gamma_C$
the restriction 
$A_{|C'}:=A\otimes_{\O(D)[T^{-1}]}\O(C')$ 
is a locally free $\O(C')$-module of rank $r$.
\end{enumerate}
In particular $A$ defines a (meromorphic) differential 
equation $\Fc$ on $D(*0)$. Moreover, the following 
conditions are equivalent : 
\begin{enumerate}
\item[iii)] $\Fc$ is free as a sheaf of $\O_D[*0]$-modules; 
\item[iv)] $A$ is a free $\O(D)[T^{-1}]$-module; 
\item[v)] there exists an open annulus $C'\subseteq D$ 
containing the open boundary $b_D$ such that $A_{|C'}$ 
is free.
\end{enumerate}
In this case $A=\Fc(D)$. 
\end{proposition}
\begin{proof}
If $E$ is a closed disk, $\O(E)$ 
is a principal ideal ring (PID), whose ideals 
are generated by polynomials. The same is true for 
$\O(E)[T^{-1}]$. Therefore, it has no ideals 
stable by the action of $d/dT$, and the proof of
\cite[Proposition 9.1.2]{Kedlaya-book} shows that 
$\Fc_{|E(*0)}=\Fc\otimes\O(E)[T^{-1}]$ has no 
torsion. It is hence free. Its rank is $r$ because 
$\Fc_{|E(*0)}\otimes K((T)) = 
\Fc\otimes K((T))$.

If $E$ is a strict open sub-disk of $D$, we can first 
localize to a closed sub-disk of $D$ containing $E$, and 
the claim follows.

Claim ii) follows similarly from 
\cite[Proposition 9.1.2]{Kedlaya-book}. 
The local rank is again $r$ because the restriction to an 
open sub-pseudo-annulus (resp. closed sub-annulus) of 
$C'$ preserves the local rank, therefore we can assume 
that $C'$ is contained in some closed sub-disk $E$ 
of $D$ containing $0$ and we have
$\Fc\otimes\O(C')=\Fc_{|E(*0)}\otimes\O(C')$.
\end{proof}

%
%
%
%
\begin{remark}
In the following Theorem we denote by $\Fc$ a free
$\O(D)[T^{-1}]$-module together with a connection. 
The reason is that Proposition 
\ref{PROP : proj implies presque loc free} furnishes a 
one to one correspondence between differential 
equations over $D(*0)$ that are free 
as $\O_D[*0]$-modules and free 
$\O(D)[T^{-1}]$-modules with connection. 
\end{remark}
}\fi

The main goal of this section is the following result. 

\
\comm{J'ai ajouté l'équation 
\eqref{eq : index form meronjhbnh} dans le Théorème.}


\begin{theorem}
\label{Thm : Index meromorphic disk}
Assume that 
\begin{enumerate}
\item $\mathcal{F}$ is a free 
$\O(D)[T^{-1}]$-module;
\item the radii of $\Fs$ are all $\log$-affine 
along the germ of segment $b_D$ at the open 
boundary of $D$.
\end{enumerate}
Let $C_{b_D}$ be an open pseudo-annulus of $C := D-\{0\}$ 
containing $b_D$ such that $\Fs$ has log-affine radii 
along $\Gamma_{C_{b_D}}$.\smallskip

Then, the following conditions are equivalent:
\begin{enumerate}
\item[(a)] $\Fc$ has finite-dimensional meromorphic 
cohomology groups $\Hdr^i(D(*0),\Fc)$;
\item[(b)] $\Fs_{|C_{b_D}}^{\mathrm{Robba}}$ is Fredholm at $b_D$.
\end{enumerate}
In this case we have
\begin{equation}\label{eq : index form meronjhbnh}
\chidr(D(*0),\Fc)\;=\;
\chiabs_{b_{D}}(\Fs_{|C_{b_D}}^{\mathrm{Robba}}) 
+\Irr_{b_D}(\Fs) + \Irr_0(\Fs)\;.
\end{equation}

In particular, the following assertions are equivalent:
\begin{enumerate}
\item[(c)] one has 
(cf. Proposition \ref{Prop : chirel=Irr})
\begin{equation}
\label{eq : assumption iii) Liouville punctured disk}
\chiabs_{b_D}(
\Fs_{|C_{b_D}}^{\mathrm{Robba}})\;=\;0\;;
\end{equation}
\item[(d)] the index formula holds:
\begin{equation}
\label{eq : meromorphic index formula punctured disk}
\chidr(D(*0),\Fc)\;=\;\Irr_0(\Fs)+\Irr_{b_D}(\Fs)\;=\;-\mathrm{Irr}_C(\Fs)\;.
\end{equation}
\end{enumerate}
%
%
\end{theorem}
\begin{proof}
Since $\Omega^1_{D}$ and $\mathcal{F}$ are both 
free, we can identify the connection
\begin{equation}\label{eq:nablaan}
\nabla : \Fc(D) \;\to\; \Fc(D) \otimes_{\Os(D)} 
\Omega^1_{D}(D)
\end{equation} 
with an endomorphism
\begin{equation}\label{eq:nablaglobal}
\nabla:\O(D)[T^{-1}]^r\;\to\;
\O(D)[T^{-1}]^r
\end{equation}
of the form $Td/dT-G(T)$, with 
$G(T)\in M_r(\O(D)[T^{-1}])$.

Since~$D$ is quasi-Stein, 
coherent sheaves have no higher 
cohomology on it. This holds in particular for~$\Omega^1_{D}$. Since~$\mathcal{F}$ may be written as a direct limit of coherent 
sheaves, the same result holds for it. A spectral 
sequence argument now shows that 
$\Hdr^0(D(*0),\Fc)$ and $\Hdr^1(D(*0),\Fc)$ coincide 
respectively with the kernel and the cokernel of the 
morphism~\eqref{eq:nablaglobal}. We now focus on the 
latter.

In analogy with the analytic case (cf. Theorem 
\ref{Thm : index of finite opens}), 
we shall compare the generalized indexes of~$\nabla$ 
with respect to the decomposition 
$\O(D)[T^{-1}]=(T^{-1}K[T^{-1}])\oplus\O(D)$ (see Definition~\ref{Def: genindgeneralclassical}) and 
the index $\chidr(D(*0),\Fc)$. We will call generalized 
index of~$\Fc$ at~$b_{0}$ (resp.~$b_{D}$) the 
generalized index corresponding to the factor 
$T^{-1}K[T^{-1}]$ (resp. $\O(D)$) and denote it by 
$\chi^{\mathrm{gen}}_{b_{0}}(\Fc)$ (resp. 
$\chi^{\mathrm{gen}}_{b_{D}}(\Fc)$). We explain 
below 
that we have already studied those generalized indexes.

\comm{Je ne suis pas d'accord qu'on ait besoin de 
nouvelles définitions d'indices généralisés ici ... 

tout peut être dit en parlant des 
indices absolus qui existent déjà ...

En plus, pour être cohérent avec la notation 
$\chi^{\mathrm{gen}}$
il faut faire paraitre la connexion et le choix d'une 
dérivation dans la notation. 

NOTE : Comme on a choisi $Td/dT$ comme dérivation, 
on doit avoir que $\chi^{\mathrm{gen}}=\chiabs$, mais 
si tu définit les choses comme tu l'a fait, et comme le 
$\chiabs$ vit sur un corps sphériquement complet, 
algebriquement clos, tel que 
$|K|=\mathbb{R}_{\geq 0}$ la question se pose de 
savoir si les deux indices 
$\chi^{\mathrm{gen}}$ et $\chiabs$ sont 
encore égaux dans ce contexte ... plutôt que de rentrer la 
dedans je propose de parler uniquement de ce qui existe 
déjà... et de mettre en commentaire le philosophie 
générale, sans donner des véritables définitions...

Si tu est d'accord donc, je supprimerai à partir de la 
phrase "We will call generalized index index of $\Fc$ ..." 
ok ?

Tu peux parler sans ambiguité de l'indice généralisé 
de $\nabla(Td/dT)$ par rapport à cette décomposition. 
Ca ça rentre dans la définition 
\ref{Def: genindgeneralclassical}.}

%

First, remember that we have investigated the 
generalized indexes of~$\nabla$ acting on~$\Fs(C)$ 
with respect to a decomposition of the 
form~\eqref{eq: O(C) deco O(D_0)+O(D_1)}. Let us 
consider this decomposition with~$m=-1$ and~$n=0$. 
We have $D_{0} = D$ and the generalized indexes 
of~$\Fc$ and~$\Fs$ at~$b_{D}$ coincide, in the sense 
that one exists if, and only if, the other does and that, in 
this case, they have the same value. Indeed, the two 
endomorphisms of~$\Os(D)^r$ induced by~$\nabla$ by 
formula~\eqref{eq : truncated op u_k} already coincide. 

\comm{J'ai remplacé $\O(D)$ par $\O(D)^r$.}

In particular, by Remark~\ref{rem:Fredholmnontriv} and 
Proposition~\ref{Prop : chirel=Irr}, if~$K$ is not trivially 
valued, then~$\Fc$ has finite generalized index 
at~$b_{D}$ if, and only if, 
$\Fs_{|C_{b_D}}^{\mathrm{Robba}}$ is Fredholm at 
$b_D$ and, in this case, by 
Remark~\ref{Remark : chiabs=chigen}, we have
\begin{equation}\label{eq:chigenbDIrr}
\chiabs_{b_{D}}(\Fc) = \chiabs_{b_{D}}(\Fs_{|C_{b_D}}^{\mathrm{Robba}}) + \Irr_{b_{D}}(\Fs).
\end{equation}
If~$K$ is trivially valued, then item i)-(a) of Section~\ref{section : some situations FIn} ensures that $\Fs_{|C_{b_D}}^{\mathrm{Robba}}$ is Fredholm at $b_D$ (and that its generalized index is~0), hence that~$\Fc$ has finite generalized index at~$b_{D}$. By Proposition~\ref{Prop : chirel=Irr}, \eqref{eq:chigenbDIrr} still holds.

%
%

Second, recall that we denote by~$M$ the differential equation induced by~$\Fc$ on~$K((T))$. We see it as a differential equation on the punctured open unit disk over~$K$ endowed with the trivial valuation. The direct sum~\eqref{eq: O(C) deco O(D_0)+O(D_1)} then reads $K((T)) = T^{-1} K[T^{-1}] \oplus K[[T]]$. By arguments similar to those above, we show that the generalized index of~$\Fc$ and~$M$ at~$b_{0}$ coincide. 

\comm{Ici il y a 2 $b_0$, car on a 2 espaces. Il faudrait 
reformuler car on ne comprends pas comment $\Fc$ 
pourrait exister  en $b_0$ en valuation triviale, car il vit sur un autre espace  ... ou comment 
$\M$ pourrait se restreindre à $b_0$ en valuation non 
triviale alors qu'il existe seulement formellement.}

In particular,  Section~\ref{Rk : formal gen indexes} ensures that this generalized index always exist and is equal to $\Irr_0^F(\M)$ (see~\eqref{eq : Irr_b_0^F=chigen} and Remark~\ref{Remark : chiabs=chigen}). 
By Proposition~\ref{Prop : Irr-form=Irr-x-1}, we have
\begin{equation}\label{eq:chigen0DIrr}
\chi^\mathrm{gen}_{b_{0}}(\Fc) = \Irr_0(\Fs).
\end{equation}
\comm{Il n'y a pas d'erreurs, mais comme je le disais, la 
notation de $\chi^\mathrm{gen}$ a besoin de faire 
paraitre la connexion et la dérivation.

Si tu veux un simbôle compact, il faut mettre le $\chiabs$ ici.}
%
%
%

The problem is that, in order to apply 
Propositions \ref{Prop : chi=sumchigen} and 
\ref{Prop : Fred --> compact perturbation}, we need to define a Fr\'echet topology on 
$\O(D)[T^{-1}]$ for which it is the topological sum of 
$\O(D)$ and $T^{-1}K[T^{-1}]$, where the latter is 
considered as a Fréchet with respect to the trivial 
valuation of $K$. A topology with these properties 
seems inexistent. 
Therefore we need an \textit{ad hoc} argument for this 
situation.

\comment{En fait, je viens de me rendre compte que \c{c}a ne sert \`a rien de d\'evelopper la th\'eorie pour tout~$q$. On pourrait tout \'ecrire avec $q=0$, non~?}
\comm{Oui, c'est vrai... le $q>0$ arrive de ma première 
redaction où je ne savait pas encore completement ce 
que j'allais démontrer}
For $q\geq 1$, we consider the following 
spaces:
\begin{eqnarray}
U_q&:=&KT^{-q}\oplus\cdots\oplus K T^{-1}\;;\\
K((T))_q&\;:=\;&U_q\oplus K[[T]] \;=\; T^{-q} K[[T]]\;;\\
\O(D)[T^{-1}]_q&\;:=\;&U_q\oplus\O(D) \;=\; T^{-q} \Os(D)\;.
\end{eqnarray}
We have $T^{-1}K[T^{-1}]=\bigcup_qU_q$,  
$K((T))=\bigcup_qK((T))_q$ and 
$\O(D)[T^{-1}]=\bigcup_q\O(D)[T^{-1}]_q$. 

The map $\nabla=Td/dT-G(T)$ acts naturally on 
$(\O(D)[T^{-1}])^r$ and on $K((T))^r$ but in general 
it does not stabilize $(\O(D)[T^{-1}]_q)^r$ nor $(K((T))_q)^r$.

Let $q\ge 1$. Let $N$ be a positive integer such that the map 
$T^N\nabla$ stabilizes $(\O(D)[T^{-1}]_q)^r$ and 
$(K((T))_q)^r$. Since the multiplication by $T$ is 
invertible on both $\O(D)[T^{-1}]$ and $K((T))$, $\nabla$ has finite index exactly when $T^N \nabla$ has and, in this case, we 
have
\begin{eqnarray}
\chi(\O(D)[T^{-1}]^r,T^N\nabla)&\;=\;&
\chi(\O(D)[T^{-1}]^r,\nabla)\;; \label{eq:ODTNnabla}\\
\chi(K((T))^r,T^N\nabla)&\;=\;&
\chi(K((T))^r,\nabla)\;. \label{eq:K((T))TNnabla}
\end{eqnarray}

We now consider the following commutative diagram in 
which the arrows are exact:
\begin{equation}\label{eq : diagram key point D(*0)}
\xymatrix{0\ar[r]&(\O(D)[T^{-1}]_q)^r\ar[r]
\ar@{}[d]|{\cap}&
\O(D)[T^{-1}]^r\ar[r]\ar@{}[d]|{\cap}&
\frac{\O(D)[T^{-1}]^r}{(\O(D)[T^{-1}]_q)^r}\ar[r]\ar@{=}[d]&0\\
0\ar[r]&(K((T))_q)^r\ar[r]&
K((T))^r\ar[r]&\frac{K((T))^r}{(K((T))_q)^r}\ar[r]&0
}.
\end{equation}
The map $T^N\nabla$ acts on all the terms and 
commutes with the maps of the diagram. 
We now compute its indexes on 
$(\O(D)[T^{-1}]_q)^r$ and $(K((T))_q)^r$. Recall that 
$\nabla$ is a connection with respect to $Td/dT$, 
therefore its generalized index coincides with its 
absolute index 
(cf. Remark \ref{Remark : chiabs=chigen}).

\comm{J'ai ajouté le commentaire sur le choix de 
$d/dT$.}
\begin{lemma}
\label{Lemma : computation chigen O(D)[T^-1]_n}
The differential equation $\Fs_{|C_{b_D}}^{\mathrm{Robba}}$ is Fredholm at $b_D$ if, and only if, the operator $T^N\nabla$ has finite index on $\O(D)[T^{-1}]_q^r$.  
In this case, we have
\begin{equation}
\label{eq : equality chi=chigen grqpleb}
\chi(\O(D)[T^{-1}]_q^r,T^N\nabla)
\;=\; N\cdot r+\chiabs_{b_{D}}(\Fs_{|C_{b_D}}^{\mathrm{Robba}}) + \Irr_{b_D}(\Fs)\;.
\end{equation}
\end{lemma}
\begin{proof}
By Proposition~\ref{Prop : chirel=Irr} and Remark~\ref{Remark : chiabs=chigen}, $\Fs_{|C_{b_D}}^{\mathrm{Robba}}$ is Fredholm at $b_D$ if, and only if, $\nabla$ is and, in this case, we have 
\begin{equation}
\chi_{b_D}^{\mathrm{gen}}(\O(C)^r,\nabla)
\;=\;\chiabs_{b_D}(\Fs)
\;=\; \chiabs_{b_{D}}(\Fs_{|C_{b_D}}^{\mathrm{Robba}})  + \Irr_{b_D}(\Fs)\;.
\end{equation}
By Lemma~\ref{Lemma : additivity gen ind}, Proposition~\ref{Proposition : compass} and Lemma~\ref{Lem : gen ind of a constant}, $\nabla$ is Fredholm at~$b_{D}$ if, and only if, $T^N \nabla$ is and that, in this case, we have  
\begin{equation}
\chi_{b_D}^{\mathrm{gen}}(\O(C)^r,T^N\nabla)\;=\; 
N\cdot r + 
\chiabs_{b_{D}}(\Fs_{|C_{b_D}}^{\mathrm{Robba}})  
+\Irr_{b_D}(\Fs) \;.
\end{equation}

Let us now compute this generalized index in another way. To do so, we consider the operator $T^N \nabla$ acting on~$\Os(C)^r$ and use the decomposition~\eqref{eq: O(C) deco O(D_0)+O(D_1)} with $m=-q-1$ and $n=-q$ (see Lemma~\ref{lem:b0b1}). From Definition \ref{Def : } and the stability of $(\O(D)[T^{-1}]_q)^r$ under $T^N\nabla$, we deduce that $T^N\nabla$ has finite generalized index at~$b_{D}$ if, and only if, $T^N\nabla$ has finite index on $\O(D)[T^{-1}]_q^r$ and that 
\begin{equation}\label{eq : chi_n = chigen}
\chi_{b_D}^{\mathrm{gen}}(\O(C)^r,T^N\nabla)\;=\;\chi((\O(D)[T^{-1}]_q)^r,T^N\nabla)\;.
\end{equation}
Indeed the truncation \eqref{eq : truncated op u_k} 
applied to $T^N\nabla$ equals $T^N\nabla$ itself.

\comm{J'ai ajouté ce court commentaire sur la troncation}

If~$K$ is not trivially valued, then, by Remark~\ref{rem:Fredholmnontriv}, $T^N\nabla$ has finite generalized index at~$b_{D}$ if, and only if, it is Fredholm at~$b_{D}$, and the result follows. 

If~$K$ is trivially valued, then, by item i)-(a) of Section~\ref{section : some situations FIn}, $\Fs_{|C_{b_D}}^{\mathrm{Robba}}$ is Fredholm at~$b_D$. The previous arguments show that this implies that $T^N\nabla$ has finite index on $\O(D)[T^{-1}]_q^r$ and that~\eqref{eq : equality chi=chigen grqpleb} holds. The result follows.
\end{proof}

%

\begin{lemma}
\label{Lemma : computation chigen K((T))_n}
The operator $T^N\nabla$ has finite index on $K((T))_q^r$ and we have
\begin{equation}
\label{eq : equality chi=chigen grqpleb-formal}
\chi(K((T))_q^r,T^N\nabla)
\;=\;
N\cdot r-\Irr_{0}(\Fs)\;.
\end{equation}
\end{lemma}
\begin{proof}
We now consider the operator $\nabla$ as a connection of $\M$, \textit{i.e.} as a differential equation on the punctured open unit disk~$C=\{0<|T|<1\}$ over~$K$ endowed with the trivial valuation. We denote by~$b_{0}$ (resp.~$b_{1}$) the germ of segment at the open boundary out of~0 (resp. out of the Gauss point). By Section~\ref{Remark : index annulus trivial valuation}, we know that the radii of~$\M$ are log-affine along the skeleton of~$C$, that $\M^\textrm{Robba}$ is Fredholm at~$b_{0}$ and~$b_{1}$ and that, for $i=0,1$, $\chiabs_{b_{i}}(\M^\textrm{Robba}) = 0$.

We may now apply Lemma~\ref{Lemma : computation chigen O(D)[T^-1]_n} to~$\M$. Since $K$ is trivially valued, $\O(D)[T^{-1}]_q=K((T))_q$. It follows that $T^N\nabla$ has finite index on $K((T))_q^r$ and that

\comm{J'ai ajouté que $\O(D)[T^{-1}]_q=K((T))_q$}
\begin{equation}
\label{eq : equality chi=chigen grqpleb-formal}
\chi(K((T))_q^r,T^N\nabla)
\;=\;
N\cdot r + \Irr_{b_{1}}^F(\M) \;=\; N\cdot r - \Irr_{b_{0}}^F(\M)\;.
\end{equation}
Where $\Irr_{b_{0}}^F(\M)$ is as usual the irregularity 
on $\M$, and the superscript $^F$ merely indicates that 
we are considering the trivial valuation on $K$. 
The result now follows from Proposition~\ref{Prop : Irr-form=Irr-x-1}.

\comm{J'ai ajouté l'indice $^F$ et quelque mot pour 
l'expliquer, de sorte à être cohérent avec les notations de la Proposition~\ref{Prop : Irr-form=Irr-x-1}.}
%
%
\end{proof}

We now continue the proof of Theorem 
\ref{Thm : Index meromorphic disk}. 
By~\eqref{eq:K((T))TNnabla} and Section~\ref{Remark : index annulus trivial valuation}, 
we know that $T^N\nabla$ has finite index on $K((T))^r$ and that 
\begin{equation}
\chi(K((T))^r,T^N\nabla)\;=\;0\;.
\end{equation}
Therefore, we know that $T^N\nabla$ has finite index on 
two terms of the exact sequence at the bottom of the 
diagram \eqref{eq : diagram key point D(*0)} and the value of the corresponding indexes. 
It follows from Lemma \ref{Lemma : additivity of Index} that $T^N\nabla$ has finite index on 
$K((T))^r/K((T))_q^r$ and that 
\begin{equation}\label{eq : chi-microsol}
\chi\Bigl(\frac{K((T))^r}{K((T))_q^r},T^N\nabla\Bigr)\;=\;
-N\cdot r+\Irr_{0}(\Fs)\;.
\end{equation}

Now, we consider the  exact sequence at the top of the 
diagram \eqref{eq : diagram key point D(*0)}. 
We have the equality $\frac{K((T))^r}{K((T))_q^r}=
\frac{\O(D)[T^{-1}]^r}{\O(D)[T^{-1}]_q^r}$. 
Therefore \eqref{eq : chi-microsol} computes also the 
index of $T^N\nabla$ on the right hand term of the 
sequence at the top of the diagram.
\comm{J'ai ajouté quelques mots ci plus haut ...}

We deduce, by 
Lemma \ref{Lemma : additivity of Index} again,
that $T^N\nabla$ has finite index 
on $\O(D)[T^{-1}]^r$ if, and only if, it has finite index 
on $\O(D)[T^{-1}]^r_q$. By Lemma~\ref{Lemma : computation chigen O(D)[T^-1]_n}, this happens exactly when $\Fs_{|C_{b_D}}^{\mathrm{Robba}}$ is Fredholm at~$b_{D}$. It follows that (a) and (b) are equivalent.

%


\medskip

Let us now assume that (a) and (b) hold. By~\eqref{eq:ODTNnabla}, Lemmas 
\ref{Lemma : additivity of Index}, 
\ref{Lemma : computation chigen O(D)[T^-1]_n} and~\eqref{eq : chi-microsol},
we have
\begin{eqnarray}
\chidr(D(*0),\Fc)&\;=\;&\chi(\O(D)[T^{-1}]^r,T^N\nabla)\\
&\;=\;&
\chi(\O(D)[T^{-1}]_q^r,T^N\nabla) 
+\chi\Bigl(\frac{\O(D)[T^{-1}]^r}{(\O(D)
[t^{-1}]_q)^r},T^N\nabla\Bigr)
\label{eq : chi-split 0 infty mero}\\
&=&\chiabs_{b_{D}}(\Fs_{|C_{b_D}}^{\mathrm{Robba}}) +\Irr_{b_D}(\Fs) + \Irr_0(\Fs)\;.
\end{eqnarray}
The equivalence between (c) and (d) follows.
\end{proof}


}\fi

In this section, we temporarily substitute the notation of 
Section~\ref{section : meromorphic Settings punctured 
disk} and we use the following setting instead. 


\begin{setting}
\label{Setting : section 4 pseudo disk L/K finite}
Let~$D$ be an open pseudo-disk 
(cf. \cite[Definition 1.1.8]{NP-IV})
and assume that there exists a finite 
extension~$L$ of~$K$ such that 
$D\otimes_{K} L$ is a (finite) disjoint union of open disks and affine lines over~$L$.\footnote{Recall that an open pseudo-disk is a connected $K$-analytic curve that is not 
compact and has empty analytic skeleton. In particular $D$ has genus~0 and $\chi_c(D)=1$. 
By~\cite[Lemma 1.1.26]{NP-III}, for every non-trivially 
valued algebraically closed maximally complete extension~$M$ of~$K$, $D \otimes_{K} M$ is a disjoint union of open disks and affine lines but it could happen that this decomposition does not descend to any finite extension of~$K$.} 
Denote by~$b_D$ the germ of segment at the open boundary of~$D$. Let~$Z$ be a finite set of rigid points 
of~$D$ and let~$\Fc$ be a differential equation on~$D(*Z)$. 

We set $Y:= D - Z$, $\Fs := \Fc_{|Y}$ and
\begin{equation}\label{eq : global irr}
\Irr_{Y}(\Fs) \;=\;
-\sum_{z\in Z}\Irr_z(\Fs)-\Irr_{b_D}(\Fs)\;.
\end{equation}
\end{setting}

%
%
%
%


\begin{theorem}
\label{Thm : Index meromorphic disk}
We set notations as in Setting 
\ref{Setting : section 4 pseudo disk L/K finite} and we 
assume moreover that 
\begin{enumerate}
\item $\Fc$ is a free $\O(D)[\ast Z]$-module;
\item the radii of $\Fs$ are all $\log$-affine 
along the germ of segment $b_D$ at the open 
boundary of $D$.
\end{enumerate}
Let $C_{b_D}$ be an open pseudo-annulus in~$Y$ 
containing $b_D$ such that $\Fs$ has log-affine radii 
along $\Gamma_{C_{b_D}}$.

Then, the following conditions are equivalent:
\begin{enumerate}
\item[(a)] $\Fc$ has finite-dimensional meromorphic 
cohomology groups $\Hdr^i(D(*Z),\Fc)$;
\item[(b)] $\Fs_{|C_{b_D}}^{\mathrm{Robba}}$ is Fredholm at $b_D$.
\end{enumerate}
Moreover, in this case we have
\begin{equation}
\label{eq : meromorphic index-lot of poles}
\chidr(D(*Z),\Fc)\;=\; 
\chi_c(Y)\cdot\mathrm{rank}(\Fc)-
\Irr_{Y}(\Fs)+
\chiabs_{b_{D}}(\Fs_{|C_{b_D}}^{\mathrm{Robba}})\;.
\end{equation}
In particular, the following assertions are equivalent:
\begin{enumerate}
\item[(c)] one has 
\begin{equation}
\label{eq : assumption iii) Liouville punctured disk-bis}
\chiabs_{b_D}(
\Fs_{|C_{b_D}}^{\mathrm{Robba}})\;=\;0\;;
\end{equation}
\item[(d)] the index formula holds:
\begin{equation}
\label{eq : meromorphic index formula punctured disk-bis}
\chidr(D(*Z),\Fc)\;=\; 
\chi_c(Y)\cdot\mathrm{rank}(\Fc)-
\Irr_{Y}(\Fs)\;.
\end{equation}
\end{enumerate}
\end{theorem}
\begin{proof}
The open pseudo-disk $D$ is cohomologically Stein, which means that 
coherent sheaves have no higher 
cohomology on it. This holds in particular for~$\Omega^1_{D}$. Since~$\Fc$ may be written as a direct limit of coherent 
sheaves, the same result holds for it. A spectral 
sequence argument now shows that 
$\Hdr^0(D(*Z),\Fc)$ and $\Hdr^1(D(*Z),\Fc)$ coincide 
respectively with the kernel and the cokernel of the 
morphism
\begin{equation}\label{eq:nablaan-bis}
\nabla : \Fc(D) \;\to\; \Fc(D) \otimes_{\Os(D)} 
\Omega^1_{D}(D)\;.
\end{equation} 

Let~$K'$ be a finite extension of~$K$. Set $D' := D \otimes_{K} K'$ and denote by~$\Fc'$ the pull-back of~$\Fc$ on~$D'$ and by~$Z'$ the preimage of~$Z$ in~$D'$. The same argument as above shows that $\Hdr^0(D'(*Z'),\Fc')$ and $\Hdr^1(D'(*Z'),\Fc')$ coincide 
respectively with the kernel and the cokernel of the 
morphism
\begin{equation}\label{eq:nabla'an-bis}
\nabla' : \Fc'(D') \;\to\; \Fc'(D') \otimes_{\Os(D')} 
\Omega^1_{D'}(D')\;,
\end{equation} 
which is nothing but the morphism~\eqref{eq:nablaan-bis} tensored by~$K'$ over~$K$. By exactness of the tensor product, we deduce that, for $i=0,1$, we have a canonical isomorphism
\begin{equation}
\Hdr^i(D(\ast Z),\Fc)\otimes_{K}K' \simto \Hdr^i(D'(\ast Z'),\Fc')\;.
\end{equation}
Since the behavior of the other invariants of the statement with respect to scalar extension is known, it is easy to check that it is enough to prove the statement with~$D$ replaced by~$D'$, or even a connected component of~$D'$. By choosing~$K'$ appropriately, we may ensure that all those connected components are standard open pseudo-disks and that all the points of~$Z'$ are $K'$-rational. Consequently, from now on, we assume that $D := \{|T| < r_{D}\}$ with $0<r_{D}\le+\infty$ and that all the point of~$Z$ are $K$-rational. 

Since $\Omega^1_{D}$ and $\mathcal{F}$ are both 
free, we can identify the connection~\eqref{eq:nablaan-bis} 
with an endomorphism
\begin{equation}\label{eq:nablaglobal-bis}
\nabla:(\O_{D}[\ast Z])(D)^r\;\to\;
(\O_{D}[\ast Z])(D)^r
\end{equation}
of the form $d/dT-G(T)$, with 
$G(T)\in M_r((\O_{D}[\ast Z])(D))$. We will now study 
this endomorphism.

Denote by $z_{1},\dotsc,z_{n}$ the points of~$Z$. We 
then have
\begin{equation}
(\O_{D}[\ast Z])(D) = \Os(D)[(T-z_{1})^{-1},\dotsc,(T-z_{n})^{-1}]
\end{equation} 
and a direct sum decomposition
\begin{equation}
(\O_{D}[\ast Z])(D) = \Os(D) \oplus \bigoplus_{i=1}^n (T-z_{i})^{-1} K[(T-z_{i})^{-1}]\;.
\end{equation}

We will now compute the generalized indexes 
of~$\nabla$ with respect to this decomposition (see 
Definition~\ref{Def: genindgeneralclassical}). 
\begin{remark} 
We observe that one cannot deduce immediately 
that the index over $\O(D)[T^{-1}]$ is the sum of the 
generalized indexes because \emph{the spaces are 
not Fr\'echet} and we cannot apply Proposition 
\ref{Prop : chi=sumchigen}.\footnote{More precisely, the 
problem is that, in order to apply 
Propositions \ref{Prop : chi=sumchigen} and 
\ref{Prop : Fred --> compact perturbation}, we 
would need to define a Fr\'echet topology on 
$\O(D)[\{(T-z_i)^{-1}\}_i]$ for which it is the 
topological sum of 
$\O(D)$ and $\oplus_i (T-z_i)^{-1}K[(T-z_i)^{-1}]$, 
where the latter is considered as a Fr\'echet space with respect 
to the trivial valuation on $K$. 
We were unable to find a topology with these 
properties.} 
Therefore, we will provide along this proof 
an ad hoc argument of 
\emph{comparison between the formal and 
convergent situations} to prove that the global index over $\O(D)$ is 
the sum of the local indexes.
\end{remark}

We will call generalized 
index of~$\Fc$ at~$b_{D}$ the generalized index 
corresponding to the factor~$\O(D)$ and denote it 
(with an abuse) by 
$\chi^\mathrm{gen}_{b_{D}}(\Fc)$. 
Similarly, for every 
$i\in\{1,\dotsc,n\}$, we denote by~$b_{i}$ the germ of 
segment out of the point~$z_{i}$, call generalized index 
of~$\Fc$ at~$b_{i}$ the generalized index 
corresponding to the factor 
$(T-z_{i})^{-1}K[(T-z_{i})^{-1}]$ and denote it by 
$\chi^\mathrm{gen}_{b_{i}}(\Fc)$. We explain below 
that we have actually already studied those generalized 
indexes.


First, remember that we have investigated the generalized indexes of~$\nabla$ acting on~$\Fs(C_{b_{D}})$ with respect to a decomposition of the form~\eqref{eq: O(C) deco O(D_0)+O(D_1)}. Let us consider this decomposition with~$m=-1$ and~$n=0$. We have $D_{0} = D$ and the generalized indexes of~$\Fc$ and~$\Fs$ at~$b_{D}$ coincide, in the sense that one exists if, and only if, the other does and that, in this case, they have the same value. Indeed, the two endomorphisms of~$\Os(D)$ induced by~$\nabla$ by formula~\eqref{eq : truncated op u_k} already coincide. In particular, by Lemma 
\ref{Lemma: Def of chiabs chigen-chigen d} and Proposition~\ref{Prop : chirel=Irr}, if~$K$ is not trivially valued, then~$\Fc$ has finite generalized index at~$b_{D}$ if, and only if, $\Fs_{|C_{b_D}}^{\mathrm{Robba}}$ is Fredholm at $b_D$ and, in this case, 
we have
\begin{equation}\label{eq:chigenbDIrr-bis}
\chi^\mathrm{gen}_{b_{D}}(\Fc) \;=\; 
r+\chiabs_{b_D}(\Fs)
\;=\; r+\chiabs_{b_{D}}(\Fs_{|C_{b_D}}^{\mathrm{Robba}}) + \Irr_{b_{D}}(\Fs),
\end{equation}
where $r$ is the rank of $\Fc$.
 
If~$K$ is trivially valued, then item i)-(a) of Section~\ref{section : some situations FIn} ensures that $\Fs_{|C_{b_D}}^{\mathrm{Robba}}$ is Fredholm at $b_D$ (and that its generalized index is~0), hence that~$\Fc$ has finite generalized index at~$b_{D}$. By Proposition~\ref{Prop : chirel=Irr}, \eqref{eq:chigenbDIrr-bis} still holds.


Second, let $i\in\{1,\dotsc,n\}$ and denote by~$\M_{i}$ the differential equation over~$K((T-z_{i}))$ induced by~$\Fc$  by tensoring~\eqref{eq:nablaglobal-bis} with~$K((T-z_{i}))$. We see it as a differential equation on the punctured open unit disk over~$K$ endowed with the trivial valuation. The direct sum~\eqref{eq: O(C) deco O(D_0)+O(D_1)} then reads
\begin{equation}
K((T-z_{i})) = (T-z_{i})^{-1} K[(T-z_{i})^{-1}] \oplus K[[T-z_{i}]]\;.
\end{equation} 
We call $b_i^F$ the germ of segment out of $z_i$ in the 
trivially valued case.
By arguments similar to those above, we show that the 
generalized index of~$\Fc$ at $b_i$ coincides with that of~$\M_{i}$ at~$b_{i}^F$. In particular,  Section~\ref{Rk : formal gen indexes} ensures that this generalized index always exist and is equal to $\Irr_{z_{i}}^F(\M_i)-r$ (see~\eqref{eq : Irr_b_0^F=chigen} and Lemma~\ref{Lemma: Def of chiabs chigen-chigen d}). By Proposition~\ref{Prop : Irr-form=Irr-x-1}, we have
\begin{equation}\label{eq:chigen0DIrr-bis}
\chi^\mathrm{gen}_{b_{i}}(\Fc) \;=\; 
\chiabs_{b_i^F}(\M_i)-r\;=\;
\Irr_{z_i}^F(\M_i)-r\;=\;
\Irr_{z_{i}}(\Fs)-r\;.
\end{equation}

%
%
%

The map $\nabla=d/dT-G(T)$ acts naturally on 
$(\O_D[\ast Z])(D)^r$ and on $K((T-z_{i}))^r$ for 
every $i\in\{1,\dotsc,n\}$. Let $N\ge 0$ such that the 
map $\prod_{i=1}^n (T-z_{i})^N\cdot\nabla$ stabilizes 
$\O(D)^r$ and $K[[T-z_{i}]]^r$ for every 
$i\in\{1,\dotsc,n\}$. 
Set $P(T) := \prod_{i=1}^n (T-z_{i})^N$.

Since the multiplication by~$P(T)$ is invertible on both 
$(\O_D[\ast Z])(D)^r$ and $K((T-z_{i}))^r $ for every 
$i\in\{1,\dotsc,n\}$, $\nabla$ has finite index exactly 
when $P(T) \nabla$ has and, in this case, we 
have
\begin{equation}\label{eq:ODZTNnabla-bis}
\chi((\O_D[\ast Z])(D)^r,P(T)\nabla)\;=\;
\chi((\O_D[\ast Z])(D)^r,\nabla)
\end{equation}
and, for each $i\in\{1,\dotsc,n\}$,
\begin{equation} \label{eq:K((T-zi))TNnabla-bis}
\chi(K((T-z_{i}))^r,P(T)\nabla)\;=\;
\chi(K((T-z_{i}))^r,\nabla)\;.
\end{equation}


We now consider the exact sequences 
\begin{equation}\label{eq:exseqOD-bis}
0 \to \O(D)^r \to (\O_D[\ast Z])(D)^r \to \frac{(\O_D[\ast Z])(D)^r}{\O(D)^r} \to 0
\end{equation}
and, for each $z\in Z$, 
\begin{equation}\label{eq:exseqKTz-bis}
0 \to K[[T-z]]^r \to K((T-z))^r \to \frac{K((T-z))^r}{K[[T-z]]^r} \to 0\;.
\end{equation}
Note that the operator $P(T)\nabla$ acts on all the terms of those exact sequences and that we have an isomorphism 
\begin{equation}\label{eq:isoquotients-bis}
\frac{(\O_D[\ast Z])(D)^r}{\O(D)^r} \simeq  \bigoplus_{z\in Z} \frac{K((T-z))^r}{K[[T-z]]^r}
\end{equation}
that is compatible with this action.

\begin{lemma}
\label{Lemma : computation chigen O(D)-bis}
The differential equation 
$\Fs_{|C_{b_D}}^{\mathrm{Robba}}$ is Fredholm at 
$b_D$ if, and only if, the operator $P(T)\nabla$ has finite 
index on $\O(D)^r$.  
In this case, we have
\begin{align}
\label{eq : equality chi=chigen grqpleb-bis}
\chi(\O(D)^r,P(T)\nabla) 
& 
\;=\; (1- \partial_{b_{D}}(P(T))) \cdot r+\chiabs_{b_{D}}(\Fs_{|C_{b_D}}^{\mathrm{Robba}}) + \Irr_{b_D}(\Fs)\\
& \;=\; (1+ Nn) r+\chiabs_{b_{D}}(\Fs_{|C_{b_D}}^{\mathrm{Robba}}) + \Irr_{b_D}(\Fs)\;.
\end{align}
\end{lemma}
\begin{proof}
By Proposition~\ref{Prop : chirel=Irr} and 
Lemma~\ref{Lemma: Def of chiabs chigen-chigen d}, 
$\Fs_{|C_{b_D}}^{\mathrm{Robba}}$ is Fredholm at 
$b_D$ if, and only if, $\nabla$ is and, in this case, we 
have 
\begin{equation}
\chi_{b_D}^{\mathrm{gen}}(\O(C_{b_{D}})^r,\nabla)
\;=\;r+\chiabs_{b_D}(\Fs)
\;=\; r+\chiabs_{b_{D}}(\Fs_{|C_{b_D}}^{\mathrm{Robba}})  + \Irr_{b_D}(\Fs)\;.
\end{equation}
By Lemma~\ref{Lem : gen ind of a constant}, $P(T)$ is Fredholm at~$b_{D}$ and we have
\begin{equation}
\chi_{b_D}^{\mathrm{gen}}(\O(C_{b_{D}})^r,P(T))\;=\;  - \partial_{b_{D}}(P(T)) \cdot r\;.
\end{equation}
By assumption, $C_{b_{D}}$ is an open pseudo-annulus in $Y = D-Z$ containing~$b_{D}$. We deduce that, for each $z\in Z$, $|T-z|$ is log-affine of slope~1 along its skeleton when computed in the direction where~$|T|$ increases, that is to say opposite to~$b_{D}$. It follows that $\partial_{b_{D}}(P(T)) = -Nn$ and 
$\chi_{b_D}^{\mathrm{gen}}(\O(C_{b_{D}})^r,P(T))=Nnr$.

By Lemma~\ref{Lemma : additivity gen ind} and Proposition~\ref{Proposition : compass}, $\nabla$ is Fredholm at~$b_{D}$ if, and only if, $P(T) \nabla$ is and, in this case, we have  
\begin{equation}
\chi_{b_D}^{\mathrm{gen}}(\O(C_{b_{D}})^r,P(T)\nabla)\;=\; (1+Nn)r + \chiabs_{b_{D}}(\Fs_{|C_{b_D}}^{\mathrm{Robba}})  +\Irr_{b_D}(\Fs) \;.
\end{equation}

Let us now compute this generalized index in another way. To do so, we consider the operator $P(T) \nabla$ acting on~$\Os(C_{b_{D}})^r$ and use the decomposition~\eqref{eq: O(C) deco O(D_0)+O(D_1)} with $m=-1$ and $n=0$. 
Notice that $\O(D)^r$ is stable under $P(T)\nabla$, 
therefore the truncation \eqref{eq : truncated op u_k} 
applied to $P(T)\nabla:
\O(C_{b_D})^r\to\O(C_{b_D})^r$ is equal 
to $P(T)\nabla$ itself 
acting on $\O(D)^r$. It follows, by Definition \ref{Def : }, that $P(T)\nabla$ has finite generalized index at~$b_{D}$ if, and only if, $P(T)\nabla$ has finite index on $\O(D)^r$ and that 
\begin{equation}\label{eq : chi_n = chigen-bis}
\chi_{b_D}^{\mathrm{gen}}(\O(C_{b_{D}})^r,P(T)\nabla)\;=\;\chi(\O(D)^r,P(T)\nabla)\;.
\end{equation}

If~$K$ is not trivially valued, then, by 
Lemma~\ref{Lemma: Def of chiabs chigen-chigen d}, 
$P(T)\nabla$ has finite generalized index at~$b_{D}$ if, 
and only if, it is Fredholm at~$b_{D}$, and the result 
follows. 

If~$K$ is trivially valued, then, by Corollary~\ref{Cor. H^i(C)=H^i(C') restr}, $\Fs_{|C_{b_D}}^{\mathrm{Robba}}$ is Fredholm at~$b_D$. The previous arguments show that this implies that $P(T)\nabla$ has finite index on $\O(D)^r$ and that~\eqref{eq : equality chi=chigen grqpleb-bis} holds. The result follows.
\if{By assumption, $C_{b_{D}}$ is an open pseudo-annulus in $Y = D-Z$ containing~$b_{D}$. In particular, for each $i\in\{1,\dotsc,n\}$, $|T-z_{i}|$ is log-affine of slope~1 along its skeleton. 
}\fi
\end{proof}

\begin{lemma}
\label{Lemma : computation chigen K[[T]]}
Let $z\in Z$. The operator $P(T)\nabla$ has finite index on $K[[T-z]]^r$ and we have
\begin{equation}
\label{eq : equality chi=chigen grqpleb-formal-bis}
\chi(K[[T-z]]^r,P(T)\nabla)
\;=\;
(1+N)r-\Irr_{z}(\Fs)\;.
\end{equation}
\end{lemma}
\begin{proof}
Let $i\in \{1,\dotsc,n\}$ such that $z=z_{i}$. We consider the operator $\nabla$ as a connection on $\M_{i}$, \textit{i.e.} as a differential equation on the punctured open unit disk~$C_i=\{0<|T-z_i|<1\}$ 
over~$K$ endowed with the trivial valuation. We denote by~$b_{0}^F$ (resp.~$b_{1}^F$) the germ of segment at the open boundary out of~0 (resp. out of the Gauss point). By Section~\ref{Remark : index annulus trivial valuation}, we know that the radii of~$\M_{i}$ are log-affine along the skeleton of~$C_i$, that $\M_i^\textrm{Robba}$ is Fredholm at~$b_{0}^F$ and~$b_{1}^F$ and that, for $j=0,1$, $\chiabs_{b_{j}^F}(\M_i^\textrm{Robba}) = 0$.

Set $D_i=\{|T-z_i|<1\}$ over~$K$ endowed with the trivial valuation. 
We have $\O(C_i)=\O(D_i)[(T-z_i)^{-1}]=K((T-z_i))$ 
and $\O(D_i)=K[[T-z_i]]$. Hence we may now apply 
Lemma~\ref{Lemma : computation chigen O(D)-bis} 
to~$\M_{i}$, seen as a differential module over 
$D_i(*\{z_i\})$. It follows that $P(T)\nabla$ has finite 
index on $K[[T-z_i]]^r$ and that
\begin{equation}
\label{eq : equality chi=chigen grqpleb-formal}
\chi(K[[T-z_i]]^r,P(T)\nabla)
\;=\;
(1- \partial_{b_{1}^F}(P(T)) )\cdot r + \Irr_{b_{1}^F}(\M_{i}) \;=\; (1- \partial_{b_{1}^F}(P(T))) \cdot r - \Irr_{b_{0}^F}(\M_{i})\;.
\end{equation}
Note that, on the skeleton of~$C$, $|T-z_{i}|$ is log-affine of slope~1 when computed in the direction where $|T-z_{i}|$ increases and $|T-z_{j}|$ is constant, for each $j\ne i$ (i.e. $P$ has no zeros in $C$). We deduce that $\partial_{b_{1}^F}(P(T)) = - N$. Finally, by Proposition~\ref{Prop : Irr-form=Irr-x-1}, we have $\Irr_{b_{0}^F}(\M_{i}) = \Irr_{z_i}(\Fs)$ and the 
result follows.
%
%
%
%
\end{proof}

We now continue the proof of Theorem 
\ref{Thm : Index meromorphic disk}. 
By~\eqref{eq:K((T-zi))TNnabla-bis} and Section~\ref{Remark : index annulus trivial valuation}, for each $z\in Z$, 
we know that $P(T)\nabla$ has finite index on $K((T-z))^r$ and that 
\begin{equation}
\chi(K((T-z))^r,P(T)\nabla)\;=\;0\;.
\end{equation}
Therefore, we know that $P(T)\nabla$ has finite index on 
two terms of the exact sequence~\eqref{eq:exseqKTz-bis} and the value of the corresponding indexes. 
It follows from Lemma \ref{Lemma : additivity of Index} that $P(T)\nabla$ has finite index on 
$K((T-z))^r/K[[T-z]]^r$ and that 
\begin{equation}\label{eq:chiquotientTz}
\chi\Bigl(\frac{K((T-z))^r}{K[[T-z]]^r},P(T)\nabla\Bigr)\;=\;
\Irr_{z}(\Fs)-(1+N)r\;.
\end{equation}

Now, looking at the exact sequence \eqref{eq:exseqOD-bis} and using the isomorphism~\eqref{eq:isoquotients-bis}, 
we deduce, by Lemma~\ref{Lemma : additivity of Index} again,
that $P(T)\nabla$ has finite index 
on $(\O_D[\ast Z])(D)^r$ if, and only if, it has finite index 
on $\O(D)^r$. By Lemma~\ref{Lemma : computation chigen O(D)-bis}, this happens exactly when $\Fs_{|C_{b_D}}^{\mathrm{Robba}}$ is Fredholm at~$b_{D}$. It follows that (a) and (b) are equivalent.

\medskip

Let us now assume that (a) and (b) hold. By~\eqref{eq:ODZTNnabla-bis}, Lemmas 
\ref{Lemma : additivity of Index}, 
\ref{Lemma : computation chigen O(D)-bis} 
and~\eqref{eq:chiquotientTz}, and by the fact that 
$\chi_c(D-Z)=1-n$, we have
\begin{eqnarray}
\chidr(D(*Z),\Fc)&\;=\;&\chi((\O_D[\ast Z])(D)^r,P(T)\nabla)\\
&\;=\;&
\chi(\O(D)^r,P(T)\nabla) 
+\chi\Bigl(\frac{(\O_D[\ast Z])(D)^r}{\O(D)^r},P(T)\nabla\Bigr)
\label{eq : chi-split 0 infty mero}\\
&=&\Bigl[(1+Nn)r+\chiabs_{b_{D}}(\Fs_{|C_{b_D}}^{\mathrm{Robba}}) +\Irr_{b_D}(\Fs) \Bigr]+ \Bigl[\Bigl(\sum_{z\in Z} \Irr_z(\Fs)\Bigr)-(1+N)nr\Bigr]
\qquad\qquad
\\
&=& (1-n)r+
\chiabs_{b_{D}}(\Fs_{|C_{b_D}}^{\mathrm{Robba}})-\Irr_{D-Z}(\Fs)\;.\end{eqnarray}
This proves \eqref{eq : meromorphic index-lot of poles}. The equivalence between (c) and (d) follows.
\end{proof}


\if{\begin{corollary}
\label{Cor : Index meromorphic disk}
Denote by~$b_{D}$ the germ of segment at the open boundary of~$D$. Assume that there exists an open pseudo-annulus~$C_{b_{D}}$ of $C := D-\{0\}$ 
containing $b_D$ such that 
\begin{enumerate}
\item $\Fs$ has log-affine radii along $\Gamma_{C_{b_D}}$;
\item for every germ of segment~$b$ in~$\Gamma_{C_{b_D}}$, $\Fs_{|C_{b_D}}^{\mathrm{Robba}}$ is Fredholm at $b$.
\end{enumerate}

\comment{Ajouter une remarque faisant le lien entre cette hypoth\`ese et la condition Liouville~?}
\comm{Je l'ai ajouté après le corollaire}
Then, $\Fc$ has finite-dimensional meromorphic cohomology groups $\Hdr^i(D(*0),\Fc)$.

Moreover, if, for every germ of segment~$b$ in~$\Gamma_{C_{b_D}}$, we have 
\begin{equation}
\chiabs_{b}(\Fs) = \Irr_{b}(\Fs)\;,
\end{equation} 
then the index formula holds:
\begin{equation}
\label{eq : meromorphic index formula punctured disk}
\chidr(D(*0),\Fc)\;=\;\Irr_0(\Fs)+\Irr_{b_D}(\Fs)\;=\;-\mathrm{Irr}_C(\Fs)\;.
\end{equation}
\end{corollary}
\begin{proof}
Let~$D'$ be a relatively compact sub-disk of~$D$ containing~0 whose germ of segment at the open boundary~$b_{D'}$ is contained in~$\Gamma_{C_{b_D}}$. Since~$D'$ is contained in a closed disk and since the ring of functions on such a disk is a PID, $\Fc_{|D'}$ is a free $\Os(D')[T^{-1}]$-module. It then follows from Theorem~\ref{Thm : Index meromorphic disk} that $\Fc'$ has finite-dimensional meromorphic cohomology groups $\Hdr^i(D'(*0),\Fc_{|D'})$. Moreover, under the additional assumption of the statement, we have 
\begin{equation}
\chidr(D'(*0),\Fc)\;=\;\Irr_0(\Fs)+\Irr_{b_{D'}}(\Fs) \;=\; \Irr_0(\Fs)+\Irr_{b_{D}}(\Fs)
\end{equation}
since the radii of~$\Fs$ are log-affine on $\Gamma_{C_{b_D}}$.

Set $C' := D' \cap C_{b_{D}}$. Denote by~$c'$ the germ of segment at the open boundary of~$C'$ that is different from~$b_{D'}$. By Lemma~\ref{lem:b0b1chidr}, $\Fs_{|C'}$ has finite-dimensional de Rham cohomology. Moreover, under the additional assumption of the statement, we have 
\begin{equation}
\chidr(C',\Fc)\;=\;\Irr_{c'}(\Fs)+\Irr_{b_{D'}}(\Fs)\;=\;0
\end{equation}
since the radii of~$\Fs$ are log-affine on $\Gamma_{C_{b_D}}$.

The result now follows from the Mayer-Vietoris exact 
sequence.
\end{proof}
}\fi
%
%
%
%
%

The presence of the freeness assumption on~$\Fc$ in the statement of Theorem~\ref{Thm : Index meromorphic disk} is due to the fact that 
we cannot extend the scalars to a large field where $\Fc$ becomes free, as we did in most of the proofs until now, due to the lack of descent results in the context of differential equation with meromorphic singularities. To 
remove this assumption, we will need to impose 
conditions at more than one germ. 




\begin{corollary}
\label{Cor : Index meromorphic disk}

We maintain the notation of Setting 
\ref{Setting : section 4 pseudo disk L/K finite}.
Let~$C_{0}$ be a relatively compact open 
pseudo-annulus in~$D$ such that $Z$ is contained in the 
connected component of $D-C_0$ that does not contain 
$b_D$. In particular, $C_0$ does not meet~$Z$. 
Let~$\{x,y\}$ be the relative boundary of $C_0$ in $D$ 
and assume that~$y$ lies between~$x$ and the open 
boundary of~$D$. Denote by~$\{b_{x},b_y\}$ the 
germs of segment at the open boundary of $C_0$ 
(directed as usual towards the interior of $C_0$).

%

Assume that, for each $b \in \{b_{D},b_{x},b_{y}\}$, 
there exists an open sub-pseudo-annulus~$C_{b}$ in 
$D-Z$ representing $b$ such that
\begin{enumerate}
\item $\Fs$ has log-affine radii along $\Gamma_{C_{b}}$;\footnote{This is automatically true if $b\in\{b_x,b_y\}$ by relative compactness.}
\item $\Fs_{|C_{b}}^{\mathrm{Robba}}$ is Fredholm at~$b$.
\end{enumerate}
Then, $\Fc$ has finite-dimensional meromorphic 
cohomology groups $\Hdr^i(D(*Z),\Fc)$. 
In this case 
\if{
if we set
\begin{equation}
I\;:=\;\chiabs_{b_0}(\Fs_{|C_{b}}^{\mathrm{Robba}})
+
\chiabs_{b_1}(\Fs_{|C_{b_{D'}}}^{\mathrm{Robba}})
+
\chiabs_{b_D}(\Fs_{|C_{b_D}}^{\mathrm{Robba}})\;,
\end{equation}
then}\fi 
we have
\begin{equation}
\chidr(D(*Z),\Fc)\;=\;
\chi_c(Y)\cdot\mathrm{rank}(\Fc)-\Irr_{Y}(\Fs)
+\chiabs_{b_D}(\Fs_{|C_{b_D}}^{\mathrm{Robba}})\;.
\end{equation}
In particular, if $\chiabs_{b_D}(\Fs_{|C_{b_D}}^{\mathrm{Robba}})=0$, the index formula holds:
\begin{equation}
\label{eq : meromorphic index formula punctured disk}
\chidr(D(*Z),\Fc)\;=\;
\chi_c(Y)\cdot\mathrm{rank}(\Fc)-\Irr_{Y}(\Fs)\;.
\end{equation}
\end{corollary}
\begin{proof}
By the same argument as in the first part of the proof of Theorem~\ref{Thm : Index meromorphic disk}, we may reduce to the case where~$D$ is an open disk or the affine line. 

Denote by~$C$ the connected component of~$D-\{x\}$ 
containing~$b_{x}$. It is an open pseudo-annulus with 
open boundary~$\{b_{x},b_{D}\}$. Note that it is 
disjoint from~$Z$.

Denote by~$D'$ the connected component of~$D-\{y\}$ containing~$b_{y}$. It is an open disk containing~$Z$, with open boundary $b_{D'}=b_y$. Moreover, since~$D'$ is contained in a closed disk and since the ring of functions on such a disk is a PID, $\Fc_{|D'}$ is a 
free $\Os(D')[\ast Z]$-module. 
Therefore, we can apply 
Theorem~\ref{Thm : Index meromorphic disk}: the 
cohomology of $\Fc_{|D'}$ is finite-dimensional and 
one has (cf. \eqref{eq : meromorphic index-lot of poles})
\begin{equation}
\chidr(D'(*Z),\Fc_{|D'})\;=\;
\chi_c(D'-Z)\cdot r -
\Irr_{D'-Z}(\Fs)+
\chiabs_{b_{D'}}(\Fc_{|C_{b_{D'}}}^{\mathrm{Robba}})\;.
\end{equation}
Let us now consider the open pseudo-annulus $C'=C\cap D'$. 
We have $\Fc_{|C}=\Fs_{|C}$ and the assumptions of 
Lemma \ref{lem:b0b1chidr} are fulfilled 
over $C$ and over $C'$. 
Hence the cohomologies of $\Fc_{|C'}$ and of 
$\Fc_{|C}$ are finite-dimensional and one has
\begin{eqnarray}
\chidr(C(*Z),\Fc_{|C})&\;=\;&
\chidr(C,\Fs_{|C})\;=\;
\chiabs_{b_x}(\Fs)+\chiabs_{b_D}(\Fs)\;,\\
\chidr(C'(*Z),\Fc_{|C'})&\;=\;&
\chidr(C',\Fs_{|C'})\;=\;
\chiabs_{b_x}(\Fs)+\chiabs_{b_{D'}}(\Fs)\;.
\end{eqnarray}
By Mayer-Vietoris, we find
\begin{eqnarray}
\chi(D(*Z),\Fc)&\;=\;&
\chi(D'(*Z),\Fc_{|D'})+\chi(C(*Z),\Fc_{|C})-
\chi(C'(*Z),\Fc_{|C'})\\
&=&\chi_c(D'-Z)\cdot r -
\Irr_{D'-Z}(\Fs)+
\chiabs_{b_{D'}}(\Fc_{|C_{b_{D'}}}^{\mathrm{Robba}})+
\chiabs_{b_D}(\Fs)-\chiabs_{b_{D'}}(\Fs)
\end{eqnarray}
Now,  since $D'$ contains $Z$, 
one has $\chi_c(D'-Z)=\chi_c(Y)$.
Moreover, by Proposition \ref{Prop : chirel=Irr}, one has 
$\chiabs_{b_D}(\Fs)=
\chiabs_{b_D}(\Fs_{|C_{b_D}}^{\mathrm{Robba}}) 
+\Irr_{b_D}(\Fs)$ and analogous relations hold for 
$\chiabs_{b_x}(\Fs)$ and $\chiabs_{b_{D'}}(\Fs)$. 
Therefore 
\begin{eqnarray}
\Irr_{D'-Z}(\Fs)&\;=\;&-\sum_{z\in Z}\Irr_z(\Fs)-\Irr_{b_{D'}}(\Fs)\\
&=&\Irr_{Y}(\Fs)+\Irr_{b_D}(\Fs)-\Irr_{b_{D'}}(\Fs)\\
&=&\Irr_{Y}(\Fs)+\chiabs_{b_D}(\Fs)-\chiabs_{b_D}(\Fs_{|C_{b_D}}^{\mathrm{Robba}})-\chiabs_{b_{D'}}(\Fs)+\chiabs_{b_{D'}}(\Fs_{|C_{b_{D'}}}^{\mathrm{Robba}})\;.
\end{eqnarray}
The  claim follows.
\end{proof}

\begin{corollary}
\label{Cor : index-mero disk-Z-with Liouville}
We maintain the setting \ref{Setting : section 4 pseudo disk L/K finite}. 
Assume that
\begin{enumerate}
\item $\Fs$ has log-affine radii along the germ of 
segment $b_D$ at the open boundary of $D$;
\item $\Fs$ is 
free of Liouville numbers along $b_D$ 
(cf. Definition~\ref{def:NLgerm}).
\end{enumerate}
Then, $\Fc$ has finite-dimensional meromorphic 
cohomology groups $\Hdr^i(D(*Z),\Fc)$ and we have 
the index formula
\begin{equation}
\chidr(D(*Z),\Fc)\;=\;
\chi_c(Y)\cdot\mathrm{rank}(\Fc)-\Irr_{Y}(\Fs)\;.
\end{equation}
\end{corollary}
\begin{proof}
By Corollary \ref{Cor. H^i(C)=H^i(C') restr},
the assumptions of Corollary 
\ref{Cor : Index meromorphic disk} are fulfilled.
%
\end{proof}

In the following example we analyze what happens for a 
rank one differential equation with an individual 
singularity on $D$.
\begin{example}
Assume that $D=\{|T|<r_D\}$ as in 
\eqref{section : meromorphic Settings punctured disk}. 
Let $g(T)=\sum_{n\geq n_0}a_nT^n\in \O(D)[T^{-1}]$, 
and consider the differential equation 
$\Fc\;:\;T\frac{d}{dT}(y)=g(T)y$ on $D(*0)$ with 
$\log$-affine radii at $b_D$.

If $\Fc$ is not of type Robba at the open boundary of 
$D$, then $\Fs_{|C_{b_D}}^{\mathrm{Robba}}=0$ and 
the conditions of Corollary 
\ref{Cor : index-mero disk-Z-with Liouville} are 
automatically satisfied. In this case, the cohomology of 
$\Fc$ over $D(*0)$ is finite dimensional.

Assume now that $\Fc$ is of type Robba at $b_D$, in 
particular $\Irr_{b_D}(\Fc)=0$.
Set $g^+(T):=\sum_{n\geq 1}a_nT^n$, 
$g^-(T):=\sum_{n= n_0}^{-1}a_nT^n$.
Then $\Fc$ decomposes, over $\O(D)[T^{-1}]$, as 
\begin{equation}
\Fc\;=\;\Fc^+\otimes\Ns(a_0)\otimes\Fc^-\;,
\end{equation}
where $\Fc^\pm$ and $\Ns(a_0)$ are associated with 
the equations $T\frac{d}{dT}(y)=g^\pm(T)y$ and 
$T\frac{d}{dT}(y)=a_0\cdot y$ respectively.
The concavity of the radius of convergence 
function on the skeleton of $D-\{0\}$ implies
that $\Fc^+$, $\Ns(a_0)$, $\Fc^-$ are all 
of type Robba (see 
\cite[Proposition 1.1, p.501]{Rk1} 
for more details). In this case $\Fc^+$ is moreover 
trivial on $D(*0)$, and there exists a pseudo-annulus 
$C_{b_D}$ in $D-\{0\}$ containing $b_D$ such that 
$\Fc^-_{|C_{b_D}}$ is trivial. Therefore, 
$\chiabs_{b_D}(\Fc) =\chiabs_{b_D}(\Ns(a_0))$ 
and the latter equals $0$ 
if, and only if, the residue $a_{0}$ of $g(T)$ satisfies 
$\mathrm{type}_{b_D}(a_{0})=1$ 
(cf. Definition \ref{Def : type_b}). 

In other words if $\Fc$ is of Robba type at $b_D$, then 
its cohomology on $D(*0)$ is finite dimensional if and 
only if $\mathrm{type}_{b_D}(a_0)=1$ 
(this is true in particular if $a_0$ is non-Liouville).
\end{example}

\section{Comparison results}
\label{App : Some comparison results}
In this section we list some consequence of the index 
Theorem \ref{Thm : Index meromorphic disk} about the 
comparison between formal, 
meromorphic and analytic de Rham cohomologies of a 
differential equation over an open disk~$D$ with a 
meromorphic singularity at $0$. 
Some statements generalize certain foundational 
results of Clark \cite{Clark} and Baldassarri 
\cite{Balda-Turritin}. More precisely, 
the results of \cite{Clark} and Baldassarri 
\cite{Balda-Turritin} hold 
under a ``\emph{boundary condition}'' of non-Liouvilleness of the \emph{formal} exponents at $0$. 
The novelty of our approach is twofold : on the one hand 
our boundary condition consists in a 
systematic use of the absolute index which
does not involves any Liouvilleness 
of the exponents, on the other hand our boundary 
condition arises \emph{at the open boundary of $D$, not 
at $0$}. This allows to obtain finer comparison results. 

Unless specific mention, we 
maintain the settings of Section 
\ref{section : meromorphic Settings punctured disk}. We recall that $D = \{|T|<r_D\}$, $b_D$ and~$b_{0}$ are the germs of segment at the open boundary of~$D$ and out of~0 respectively, $\Fc$ is a differential equation over~$D$ possibly with meromorphic singularities at~0 
and we have $\Fs=\Fc_{|D-\{0\}}$, 
$\Fc_0^\dag=\Fc\otimes K(\{T\})$ and 
$\M=\Fc\otimes K((T))$. 

\subsection{Meromorphic vs. formal theories}
\label{Meromorphic vs. formal theories}
In this section, we compare the meromorphic and the 
formal theories of differential equations.
We maintain the notations of Section 
\ref{section : meromorphic Settings punctured disk}.

\subsubsection{Meromorphic vs. formal cohomologies.}
In this section we compare the meromorphic and formal 
\emph{cohomologies}. The results of this sub-section are 
generalizations of a result of D.~Clark 
\cite{Clark, CLark-2}.
 


We begin by an easy claim that does not require freeness 
on $\Fc$.
\begin{lemma}
\label{Lemma : injectiona and surjection res mero form}
\label{Lemma : mero vs formal coh}
The natural map 
\begin{equation}\label{eq : iso mero formal coh-sec5}
\Hdr^i(D(*0),\Fc)\;\xrightarrow{\;\quad\;}\;
\Hdr^i(K((T)),\M)
\end{equation}
is injective for 
$i=0$ and surjective for $i=1$.

If~$\Fc$ has finite-dimensional de Rham cohomology, then we have
\begin{equation}\label{S4eq : chidr leq chifrom}
\chidr(D(*0),\Fc)\;\leq\; 0\;=\;\chidr(K((T)),\M)\;
\end{equation}
and the following are equivalent:
\begin{enumerate}
\item $\chidr(D(*0),\Fc)=0$;
\item for $i=0,1$, the restriction maps 
\eqref{eq : iso mero formal coh-sec5} are isomorphisms.
\end{enumerate}

\end{lemma}
\begin{proof}
It is clear that the map $\Hdr^0(D(*0),\Fc) \to \Hdr^0(K((T)),\M)$ between the 
vector spaces of formal and convergent solutions is injective.

On the other hand, the restriction morphism
\begin{equation} 
(\Os_{D}[*0])(D) = \O(D)[T^{-1}] \to K((T))
\end{equation}
has dense image. We may see~$\M$ as a differential equation on an open annulus over $K$ endowed with the trivial valuation. By Remark~\ref{rem:logafftrivval}, its radii are log-affine on the skeleton. By Corollary~\ref{Cor. H^i(C)=H^i(C') restr}, $\Hdr^1(K((T)),\M)$ is finite-dimensional and this remains true after base-change to any complete extension of~$K$. Let us consider the commutative diagram
\begin{equation}\label{S1eq : diagram dense}
\xymatrix{
\Fc(D(*0)) \ar[r]^-{\nabla}\ar[d] & (\Fc\otimes_{\Os_{D}}\Omega^1_{D})(D(*0))\ar[d]\ar[r]&
\Hdr^1(D(*0),\Fc)\ar[d]\ar[r] &0\\ 
M\ar[r]^-{\nabla}&M\otimes_{K((T))} K((T))\, dT\ar[r]&
\Hdr^1(K((T)),M)\ar[r]&0
}
\end{equation}
\if{\begin{equation}\label{S1eq : diagram dense}
\begin{tikzcd}
\Fc(D(*0)) \arrow{r}{\nabla}\arrow{d} & (\Fc\otimes_{\Os_{D}}\Omega^1_{D})(D(*0))\arrow{d}\arrow{r}&
\Hdr^1(D(*0),\Fc)\arrow{d}\arrow{r} &0\\ 
M\arrow{r}{\nabla}&M\otimes_{K((T))} K((T))\, dT\arrow{r}&
\Hdr^1(K((T)),M)\arrow{r}&0
\end{tikzcd}
\end{equation}
}\fi
It is possible to put suitable structures, namely Banachoid structures over~$K$ (see~\cite{Banachoid}), on the spaces of the bottom line such that the middle vertical map has dense image, hence the right hand vertical map too. Since $\Hdr^1(K((T)),\M)$ is finite-dimensional and separated, this map is indeed surjective. We refer to the proof of~\cite[Lemma~4.17]{Banachoid} for details.



Assume now that~$\Fc$ has finite-dimensional de Rham 
cohomology. We have just proved that 
$\mathrm{dim}\;\Hdr^1(D(*0),\Fc)
\geq\mathrm{dim}\;\Hdr^1(K((T)),\M)$ and 
$\mathrm{dim}\;\Hdr^0(D(*0),\Fc)
\leq\mathrm{dim}\;\Hdr^0(K((T)),\M)$. The result now follows from Corollary~\ref{Cor. H^i(C)=H^i(C') restr} and a computation of dimensions.
%
\end{proof}
\if{
\begin{lemma}
\label{Lemma : mero vs formal coh}
Assume that $\Fc$ is free as $\O_D[*0]$-module 
and that it has finite dimensional cohomology 
groups $\Hdr^i(D(*0),\Fc)$. 
Then the following are equivalent
\begin{enumerate}
\item $\chidr(D(*0),\Fc)=0$;
\item For $i=0,1$, the restriction maps
\begin{equation}\label{eq : iso mero formal coh-sec5}
\Hdr^i(D(*0),\Fc)\;\xrightarrow{\;\sim\;}\;
\Hdr^i(K((T)),\M)
\end{equation}
are isomorphisms.
\end{enumerate}
\end{lemma}
\begin{proof}
Lemma 
\ref{Lemma : injectiona and surjection res mero form} 
implies that if i) holds, then we must have ii) by a matter 
of dimensions. On the other hand, if ii) holds, then 
$\chidr(D(*0),\Fc)=\chidr(K((T)),\M)$. 
By Theorem \ref{Thm : index of finite opens} 
we have $\chidr(K((T)),\M)=0$ (cf. Section 
\ref{Remark : index annulus trivial valuation}). The 
claim follows.
\end{proof}
}\fi

\begin{corollary}
\label{Corollary : hyp=> resiso mero-form}
Assume that 
\begin{enumerate}
\item $\mathcal{F}$ is free as  
$\O_D[*0]$-module;
\item the radii of $\Fs$ are all $\log$-affine 
along the germ of segment $b_D$ at the open 
boundary of $D$;
\item $\Fs$ is Fredholm at $b_D$ and 
\begin{equation}
\chiabs_{b_D}(\Fs)\;=\;\Irr_{b_D}(\Fs)\;;
\end{equation}
\item $\Irr_{b_{D}}(\Fs) + \Irr_{b_{0}}(\Fs) = 0$.
\end{enumerate}
Then $\Fc$ has finite-dimensional de Rham cohomology, we have $\chidr(D(*0),\Fc)=0$ 
and the restriction maps 
\eqref{eq : iso mero formal coh-sec5} are isomorphisms.
\end{corollary}
\begin{proof}
It is a direct consequence of Theorem 
\ref{Thm : Index meromorphic disk} and Lemma 
\ref{Lemma : mero vs formal coh}.
\end{proof}

\begin{remark}
\label{Rk : Conditions for finiteness D(*0)}
The conclusions of Corollary 
\ref{Corollary : hyp=> resiso mero-form} remain 
the same if we replace the conditions i), ii), iii) by the 
conditions i),ii),iii),iv) of Corollary 
\ref{Cor : Index meromorphic disk}, or alternatively
by the conditions i),ii) of Corollary 
\ref{Cor : index-mero disk-Z-with Liouville}, where of 
course $Z=\{0\}$. 
\end{remark}

\begin{remark}
\label{rk : trivializati Y'=GY then rational}
A direct interesting consequence of the above claim 
asserts, in down to earth terms, that
if $Y'=G(T)Y$, with $G\in M_r(\O(D)[T^{-1}])$, 
is a differential equation satisfying the 
assumptions of Corollary
\ref{Corollary : hyp=> resiso mero-form}, or more 
simply condition i) of Lemma 
\ref{Lemma : mero vs formal coh}, 
then any formal solution 
$Y\in K((T))^r$ is actually convergent outside $0$ 
and lies in $\O(D)[T^{-1}]^r\subseteq K((T))^r$. 
\end{remark}

\begin{remark}
The assumptions of Corollary  
\ref{Corollary : hyp=> resiso mero-form} 
do not imply that the radii of $\Fs$ are all 
log-affine along the whole $\Gamma_C$. 
In particular, the largest 
disk on which the maps 
\eqref{eq : iso mero formal coh-sec5} are isomorphisms
may be bigger than the largest disk $D$ such 
that $\Fs$ has log-affine radii on $\Gamma_{D-\{0\}}$.
For instance, if 
$\Fc=\Fc_1\oplus\Fc_2$, where $\Fc_1$ and $\Fc_2$ 
are rank one meromorphic differential
equations on $D(*0)$ having log-affine radii
along $\Gamma_C$ and if the 
radii of $\Fc_1$ and $\Fc_2$ 
cross at one point $x\in\Gamma_C$, 
then $\Irr_C(\Fc)=0$, but the 
radii of $\Fc$ are not log-affine.
\end{remark}

\begin{remark}\label{Rk :  subdisk D' of D mero-formal}
Let $D'$ be a sub-disk of $D$. Assume that the restrictions maps~\eqref{eq : iso mero formal coh-sec5} are isomorphisms for~$D$ and~$D'$ and $i=0,1$. 
Then, for $i=0,1$, 
the restriction maps
\begin{equation}\label{Rk-eq: res coh mero D(*0)}
\Hdr^i(D(*0),\Fc)\;\xrightarrow{\;\sim\;}\;
\Hdr^i(D'(*0),\Fc_{|D'})\;\xrightarrow{\;\sim\;}\;
\Hdr^i(K((T)),\M)
\end{equation}
are isomorphisms too.

Notice that if the radius of $D'$ is small enough, then 
items i), ii) and iv) of Corollary 
\ref{Corollary : hyp=> resiso mero-form} 
are automatically satisfied by Lemma 
\ref{lem:merosinglinear-1}.
\end{remark}

\begin{remark}
Assume that $K$ is trivially valued. 
If the radius $r_D$ of $D$ satisfies $r_D\leq 1$, then 
$\O(D)[T^{-1}]=K((T))$, therefore the above lemmas 
are trivially true. 
However, if $r_D> 1$, then one has 
$\Irr_C(\Fs)=0$ if and only if $\Fs$ 
has log-affine radii along $\Gamma_C$ and 
$\Fs=\Fs^{\mathrm{Robba}}$. 
Indeed, if $x_{0,1}$ denotes the 
point at the boundary of the annulus 
$\{|T|=1\}$, then all the radii of $\Fs$ 
are solvable at $x_{0,1}$ (cf. Section 
\ref{Remark : index annulus trivial valuation}).
Therefore, if $r_D>1$, the above lemmas only involve 
Robba modules.
\end{remark}

\subsubsection{$K(\{T\})$ vs. $K((T))$.}
In this section, we state some consequences of the 
above results about differential modules over $K(\{T\})$.

The same proof as that of Lemma \ref{Lemma : 
injectiona and surjection res mero form} gives 
the following result.

\begin{lemma}
\label{Lemma : res=iso merodag form}
The natural map 
\begin{equation}
\Hdr^i(K(\{T\}),\Fc_0^\dag)\to\Hdr^i(K((T)),\M)
\end{equation}
is injective for $i=0$ and surjective for $i=1$. 
In particular, if the cohomology groups 
$\Hdr^i(K(\{T\}),\Fc_0^\dag)$ are finite-dimensional, then
the following conditions are equivalent:
\begin{enumerate}
\item $\chidr(K(\{T\}),\Fc_0^\dag)=0$;
\item for $i=0,1$, the restriction maps
\begin{equation}\label{eq : H^i mero = H^i K(T)}
\Hdr^i(K(\{T\}),\Fc_0^\dag)\;\xrightarrow{\;\sim\;}\;
\Hdr^i(K((T)),\M)
\end{equation}
are isomorphisms.
\hfill$\Box$
\end{enumerate}
\end{lemma}

\begin{lemma}\label{Lemma : H^i K(T)=lim-1}
For $i=0,1$, we have a natural isomorphism of $K$-vector spaces
\begin{equation}
\Hdr^i(K(\{T\}),\Fc_0^\dag)
\;\simto\;\varinjlim_{D'}\Hdr^i(D'(\ast0),\Fc_{|D'})\;,
\end{equation}
where $D'$ runs through the set of subdisks of $D$ containing $0$.

\end{lemma}
\begin{proof}
The proof is formally the same as that of Lemma 
\ref{Lemma : H^i Robba=lim}.
\end{proof}

\begin{lemma}\label{Lemma : H^i K(T)=lim-2}
Let $D_1\supset D_2\supset \cdots$ be 
a decreasing sequence of disks such that 
$\bigcap_nD_n=\{0\}$. 
If, for each $n\geq 1$, the cohomology group 
$\Hdr^i(D_n(*0),\Fc_{|D_n})$ is finite-dimensional 
and $\chidr(D_n(*0),\Fc_{|D_n})=0$, 
then, for each $n\geq 1$ and $i=0,1$, 
we have isomorphisms
\begin{equation}
\Hdr^{i}(D_n(*0),\Fc_{|D_n})\;\xrightarrow{\;\sim\;}\;
\Hdr^{i}(K(\{T\}),\Fc_0^\dag)\;\xrightarrow{\;\sim\;}\;
\Hdr^{i}(K((T)),\M)\;.
\end{equation}
\end{lemma}
\begin{proof}
The claim follows from Lemma~\ref{Lemma : mero vs formal coh}, 
Remark \ref{Rk :  subdisk D' of D mero-formal} 
and Lemma \ref{Lemma : H^i K(T)=lim-1}.
\end{proof}

\begin{corollary}
\label{Cor : comparison mero-formal K(t) and K((t))}
Let $D_1\supset D_2\supset \cdots$ be a decreasing 
sequence of disks with $\bigcap_nD_n=\{0\}$. 
Let $b_n$ be the germ of segment at the open 
boundary of $D_n$. 
Assume that, for~$n$ large enough, $\Fs$ is Fredholm at $b_{n}$ and $\chiabs_{b_{n}}(\Fs)=\Irr_{b_{n}}(\Fs)$. Then, for $i=0,1$ and~$n$ large enough, the space 
$\Hdr^i(D_n(*0),\Fc_{|D_n})$ is finite-dimensional 
and we have natural isomorphisms
\begin{equation}
\Hdr^i(D_n(*0),\Fc_{|D_n})
\;\xrightarrow{\;\sim\;}\;
\Hdr^i(K(\{T\}),\Fc_0^\dag)\;
\xrightarrow{\;\sim\;}\;
\Hdr^i(K((T)),\M)\;\;\;.
\end{equation}
In particular, the de Rham cohomology groups 
$\Hdr^i(K(\{T\}),\Fc_0^\dag)$ are finite-dimensional and we have
\begin{equation}
\chidr(K(\{T\}),\Fc_0^\dag)
\;=\;0\;.
\end{equation}
\end{corollary}
\begin{proof}
For~$n$ large enough, $\Fc_{|D_{n}}$ is a free $\O(D_{n})[\ast 0]$-module and, by Lemma \ref{lem:merosinglinear-1}, $\Fs_{|D_{n}}$ has log-affine radii along $\Gamma_{D_n-\{0\}}$. By Theorem \ref{Thm : Index meromorphic disk} and Proposition \ref{Prop : chirel=Irr}, the 
cohomology groups $\Hdr^i(D_n(*0),\Fc_{|D_n})$ are 
finite-dimensional and we have
\begin{equation}
\chidr(D_n(*0),\Fc_{|D_n}) = \Irr_{D_n-\{0\}}(\Fs_{|D_n-\{0\}}) = 0\;.
\end{equation}
The claim then follows from Lemma 
\ref{Lemma : H^i K(T)=lim-2}.
\end{proof}

\subsubsection{Descent of morphisms.}
\label{Descent of morphisms}
In this section, we give conditions to ensure that a morphism between differential modules over 
$K((T))$ comes from a morphism over $\O(D)[T^{-1}]$. This generalizes a result of F.~Baldassarri \cite[Theorem 2]{Balda-Turritin} 
(cf. Section
\ref{Remark : classical exponents and Baldassarri result}).


\begin{lemma}
\label{Lemma: descent of morphisms-1}
Let $\Fc$ and $\Gc$ be differential equations on $D(*0)$ 
(resp. over $K(\{T\})$) and let 
\begin{equation}
\beta\;:\;\Fc\otimes K((T))\;\xrightarrow{\;\quad\;}\;
\Gc\otimes K((T))\;.
\end{equation}
be a morphism of differential equations over $K((T))$. 
Assume that $\Hom(\Fc,\Gc)$ 
has finite-dimensional de Rham cohomology and satisfies
\begin{equation}
\label{S4eq: chidrHom=0 descent meroform}
\chidr(D(*0),\Hom(\Fc,\Gc))\;=\;0\qquad\textrm{(resp. 
$\chidr(K(\{T\}),\Hom(\Fc,\Gc))\;=\;0$)}\;.
\end{equation}
\if{
\begin{enumerate}
\item $\Hom(\Fc,\Gc)=\Fc^*\otimes\Gc$ is 
free as $\O_D[*0]$-module;

\item $\Hom(\Fc,\Gc)$ 
has finite-dimensional de Rham cohomology and satisfies
\begin{equation}
\label{S4eq: chidrHom=0 descent meroform}
\chidr(D(*0),\Hom(\Fc,\Gc))\;=\;0\qquad\textrm{(resp. 
$\chidr(K(\{T\}),\Hom(\Fc,\Gc))\;=\;0$)}\;.
\end{equation}
\end{enumerate} 
}\fi
Then, there exists a unique morphism
\begin{equation}
\alpha\;:\;\Fc\;\xrightarrow{\;\quad\;}\;\Gc
\end{equation}
such that $\beta=\alpha\otimes 1$. Moreover, $\alpha$ is a monomorphism (resp.
epimorphism, isomorphism) if, and only if, so is 
$\beta$.
\end{lemma}
%
%
\begin{proof}
By Lemma \ref{Lemma : mero vs formal coh}, 
one has 
\begin{eqnarray}
\Hom^\nabla(\Fc,\Gc)&\;=\;&
\Hdr^0(D(*0),\Hom(\Fc,\Gc))\\
&\;\cong\;&
\Hdr^0(K((T)),\Hom(\Fc,\Gc)\otimes K((T)))\\
&\;\cong\;&
\Hdr^0(K((T)),\Hom(\Fc\otimes K((T)),\Gc\otimes K((T))))\;.
\end{eqnarray}
Therefore, there exists a unique $\alpha:\Fc\to\Gc$ such that 
$\alpha\otimes 1=\beta$. 

Denote by $\mathcal{K}$ and $\mathcal{C}$ the kernel 
and cokernel of $\alpha$ respectively. 
If $E\subseteq D$ is a closed 
disk containing~$0$, the terms of the sequence
$0\to \mathcal{K}_{|E}\to\Fc_{|E}\to
\Gc_{|E}\to\mathcal{C}_{|E}\to 0$ 
are free as $\O_{E}[*0]$-modules. 
Therefore, $\beta$ is epi or mono if, and 
only if, so is $\Fc_{|E}\to\Gc_{|E}$.

Now, $\mathcal{K}$ and $\mathcal{C}$ are locally free 
over $D(*0)$ by Proposition~\ref{prop:eqcatStein}. 
Therefore they are zero over $E(*0)$ if and only if 
they are zero on the whole $D(*0)$.
\if{If $\alpha$ is an epimorphism or a monomorphism, then 
so is $\alpha\otimes 1=\beta$. Indeed 
$\beta=\alpha_{|E(*0)}\otimes 1$, and all the terms of 
the sequence are free over $E(*0)$.

Assume that $\beta$ is surjective. Then 
$\mathcal{C}\otimes K((T))=0$. This implies that for all 
closed sub-disk of $D$ one has 
$\mathcal{C}_{|E(*0)}=0$, because it is a free 
$\O_E[*0]$-module. 
Therefore, $\mathcal{C}$ is itself zero, and $\alpha$ is 
an epimorphism.

Assume that $\beta$ is a monomorphism. 
Then $\mathcal{K}_{|E(*0)}=0$ because all the 
modules over $E(*0)$ are free. Therefore 
$\mathcal{K}=0$.
}\fi

The case where $\Fc,\Gc$ are differential modules over 
$K(\{T\})$ follows analogously from 
Lemma \ref{Lemma : res=iso merodag form} and the 
fact that the tensor product 
$-\otimes_{K(\{T\})}K((T))$ is fully faithful.
\end{proof}

\if{
\begin{remark}
Notice that $\Fc$ and $\Gc$ are not assumed to be free 
$\O(D)[T^{-1}]$-modules in Lemma 
\ref{Lemma: descent of morphisms-1}. 
\end{remark}
}\fi
It follows from 
Corollary \ref{Corollary : hyp=> resiso mero-form} 
and Remark 
\ref{Rk : Conditions for finiteness D(*0)} that we have the following criterion.

\begin{corollary}
\label{Cor: descent of morphisms-2}
Let $\Fc$ and $\Gc$ be differential equations over 
$D(*0)$ and let 
$\beta:\Gc\otimes K((T))\to\Fc\otimes K((T))$ 
be a morphism. Assume that the differential equation 
$\Hom(\Fc,\Gc)$ satisfies the assumptions of 
Corollary \ref{Corollary : hyp=> resiso mero-form} (resp. 
the assumptions of Remark 
\ref{Rk : Conditions for finiteness D(*0)}).
\if{
\begin{enumerate}
\item $\Hom(\Fc,\Gc)=\Fc^*\otimes\Gc$ is  free as 
$\O_D[*0]$-module;
\item the radii of $\Hom(\Fc,\Gc)$ 
are all $\log$-affine along the germ of segment 
$b_D$ at the open boundary of $D$ (cf. Example 
\ref{Example: liouville of End(F) counterexample}).
\item $\Hom(\Fc,\Gc)$ is Fredholm at $b_D$ and 
(cf. Proposition \ref{Prop : chirel=Irr})
\begin{equation}
\chiabs_{b_D}(\Hom(\Fc,\Gc))\;=\;
\Irr_{b_D}(\Hom(\Fc,\Gc))\;.
\end{equation}
\item $\Irr_C(\Hom(\Fc,\Gc))=0$.
\end{enumerate}
}\fi
Then, there exists a morphism of differential 
equations
$\alpha:\Fc\simto\Gc$ such that 
$\beta=\alpha\otimes 1$. 
Moreover, $\alpha$ is a monomorphism (resp.
epimorphism, isomorphism) if, and only if, so is 
$\beta$.\hfill$\Box$
\end{corollary}
 
\subsubsection{Descent of 
Turrittin-Hukuhara-Levelt 
decomposition.}
\label{Section : Descent Turritin-Balda}
We maintain the notation of Section 
\ref{section : meromorphic Settings punctured disk} 
where $\M=\Fc\otimes K((T))$ and $r$ is its rank.
A well-known result of Turrittin-Hukuhara-Levelt 
\cite{Turritin} (cf. \cite{VS}) proves that
there exist a natural number $n$, a finite extension 
$K'/K$ and rank-one modules 
$\N_1,\ldots,\N_r$ over $K'((T^{\frac{1}{n}}))$ 
such that 
$\M\otimes_{K((T))}K'((T^{\frac{1}{n}}))$ is 
successive extension of $\N_1,\ldots,\N_r$, 
\textit{i.e.} $\N_1,\ldots,\N_r$ 
is a Jordan-Hölder sequence for the differential module 
$\M\otimes_{K((T))}K'((T^{\frac{1}{n}}))$. 

To simplify the exposition, we assume $K'=K$ and 
$n=1$; the general case will be analyzed in 
Remark \ref{Rk : descente on Turrittin}.
\if{
if 
$\mathbb{G}_{m,K'}^{\mathrm{an}}\to
\mathbb{G}_{m,K}^{\mathrm{an}}$ is the 
composition of the projection relative to 
$K'/K$ and the map sending $x$ into $x^n$, then
the inverse image $D'(*0)$ of $D(*0)$ 
is again a disk centered at $0$ 
with radius $r_{D'}=r_D^{1/n}$. 
For our purposes, we can directly work over $D'(*0)$.
}\fi
Assume that $\M$ admits a Jordan-Hölder 
sequence $\N_1,\ldots,\N_r$ 
formed by rank one modules. 
Classical computations show 
that $\M$ is then isomorphic to 
a direct sum of modules of type
\begin{equation}
\N_i\otimes\mathrm{U}_m\;,
\end{equation}
where, for $m\geq 1$, we denote by 
$\mathrm{U}_m$ the standard 
$m$-dimensional unipotent object.
This is the free differential module of rank $m$ 
over $\O(D)[T^{-1}]$ with connection 
$\nabla=\nabla(T\cdot \frac{d}{dT}):
\mathrm{U}_m\to \mathrm{U}_m$ 
given in the basis 
$e_1,\ldots,e_m$ of  $\mathrm{U}_m$ by 
$\nabla(e_i)=e_{i+1}$ for all 
$i=1,\ldots,m$ and 
$\nabla(e_m)=0$. 

It is easily seen that each $\N_i$ may 
be represented, in a basis $n_i\in\N_i$, 
by an equation of the form 
$T\frac{d}{dT}(y)=g_i(T)y$, with 
$g_i(T)\in K[T^{-1}]$. This equation defines a 
differential module $\mathcal{N}_i$ over $K[T,T^{-1}]$
such that $\mathcal{N}_i\otimes K((T))\cong\N_i$.
Therefore, in the basis $n_i\otimes e_j$, the matrix of 
$\nabla(T\frac{d}{dT}):
\N_i\otimes\mathrm{U}_m\to\N_i\otimes\mathrm{U}_m$
is in the Jordan form 
\begin{equation}\label{eq : Jordan Form g_i}
\left(\begin{smallmatrix}
g_i&1&0&\cdots&0\\
0&g_i&1&\cdots&0\\
\cdots&&\cdots&&\cdots\\
0&\cdots&&g_i&1\\
0&\cdots&&0&g_i
\end{smallmatrix}\right)\;.
\end{equation}
The fact that $\M$ is direct sum of the 
$\N_i\otimes \mathrm{U}_m$ implies that whole 
matrix of the connection of $\M$ has such blocks 
along the diagonal and zeroes elsewhere. 
In particular, the matrix lies in $M_r(K[T,T^{-1}])$ and 
it defines a differential module $\mathcal{M}$
over $K[T,T^{-1}]$ such that 
$\mathcal{M}\otimes K((T))\cong\M$. 
The family $\mathcal{N}_1,\ldots,\mathcal{N}_r$ is a 
Jordan-Hölder sequence of $\mathcal{M}$ over 
$K[T,T^{-1}]$. It is not hard to show that, up to 
isomorphisms, $\mathcal{M}$ is the unique differential 
module over $K[T,T^{-1}]$ such that 
\begin{enumerate}
\item $\mathcal{M}\otimes K((T))\cong\M$;
\item $\mathcal{M}$ is extension of (i.e. it has a 
Jordan-Hölder sequence formed by) rank-one modules 
that are regular singular at $\infty$.\footnote{Recall that 
the isomorphism class of a rank one 
differential module over $\mathbb{G}_{m}$ defined by 
an equation $T\frac{d}{dT}(y)=g(T)y$, with 
$g(T)=\sum_ia_iT^i\in K[T,T^{-1}]$, is completely 
determined by the tuple $(a_i)_{i\neq 0}$ and by the 
class of $a_0\in K/\mathbb{Z}$. Such a module is 
regular at $\infty$ (resp. at $0$) if and only if 
$g(T)\in K[T^{-1}]$ (resp $g(T)\in K[T]$), or 
equivalently if its irregularity at infinity 
is equal to $0$ (cf. Section 
\ref{Section : ANVSFORMAL irr}).}
\end{enumerate}

From Lemma \ref{Lemma: descent of morphisms-1} 
(resp. Corollary \ref{Cor: descent of morphisms-2}) 
we obtain the following corollary. It is called 
\emph{Katz's 
canonical extension of $\M$} and denoted also by
\begin{equation}
\mathcal{M}\;=\;\mathrm{Can}(\M)\;.
\end{equation}

\begin{corollary}
\label{Corollary : descent Turrittin}
We maintain notation of section 
\ref{section : meromorphic Settings punctured disk}. Assume as above that $K=K'$ and $n=1$.
If the assumptions of 
Lemma \ref{Lemma: descent of morphisms-1} (resp.
Corollary \ref{Cor: descent of morphisms-2}) are fulfilled 
by the differential module $\mathrm{Hom}(\Fc,\mathcal{M}_{|D})$ or by 
$\mathrm{Hom}(\mathcal{M}_{|D},\Fc)$, then
the formal isomorphism 
$\Fc\otimes K((T))\simto\mathcal{M}\otimes K((T))$ 
descends to an isomorphism 
\begin{equation}\label{eq: sdftgvplq}
\Fc\;\xrightarrow{\;\sim\;}\;\mathcal{M}_{|D}
\end{equation}
over $\O(D)[T^{-1}]$.\hfill$\Box$
\end{corollary}
\begin{remark}
The finite dimensionality of the cohomology groups 
$\Hdr^i(D(*0),-)$ and the condition 
$\chidr(D(*0),-)=0$ are 
fulfilled by $\mathrm{Hom}(\Fc,\mathcal{M}_{|D})$ 
(resp. $\mathrm{Hom}(\mathcal{M}_{|D},\Fc)$) 
if for all $i=1,\ldots,r$ they hold for 
$\Fc^*\otimes(\mathcal{N}_i)_{|D}$ (resp. 
$\Fc\otimes(\mathcal{N}_i)_{|D}^*$). 
Indeed, $\mathcal{M}$ is direct sum of modules 
of type $(\mathcal{N}_i)_{|D}\otimes\mathcal{U}_m$, 
where 
$\mathcal{U}_m$ is the standard $m$-dimensional 
unipotent object over $D(*0)$.
Therefore $\mathrm{Hom}(\Fc,\mathcal{M}_{|D})=
\Fc^*\otimes\mathcal{M}_{|D}$ 
is direct sum of modules of type 
$\Fc^*\otimes (\mathcal{N}_i)_{|D}\otimes\mathcal{U}_m$. 
These last fit into exact sequences 
\begin{equation}
0\to\Fc^*\otimes (\mathcal{N}_i)_{|D} \to
\Fc^*\otimes 
(\mathcal{N}_i)_{|D}\otimes\mathcal{U}_m 
\to
\Fc^*\otimes (\mathcal{N}_i)_{|D}\otimes\mathcal{U}_{m-1} 
\to0
\end{equation}
and the claim follows by induction. 
\end{remark}

\begin{corollary}
\label{Cor : descent turrittin K(T)}
Assume that the conditions of Corollary \ref{Cor : 
comparison mero-formal K(t) and K((t))} are 
fulfilled by 
$\mathrm{Hom}(\Fc\otimes K(\{T\}),
\mathcal{M}\otimes K(\{T\}))$ or 
$\mathrm{Hom}(\mathcal{M}\otimes K(\{T\}),
\Fc\otimes K(\{T\}))$. Then the formal isomorphism 
$\Fc\otimes K((T))\simto\mathcal{M}\otimes K((T))$ 
descends to an isomorphism 
$\Fc\otimes K(\{T\})\simto
\mathcal{M}\otimes K(\{T\})$.\hfill$\Box$
\end{corollary}

\begin{remark}
\label{Rk : descente on Turrittin}
At the beginning of this section we have supposed 
$K'=K$ and $n=1$. 
We now drop this assumption 
and place ourself in the general situation where
$\mathcal{M}$ is defined after pull-back, that is 
over the ring $K'[T^{\frac{1}{n}},T^{-\frac{1}{n}}]$. 
The aim of this remark is to recall that 
$\mathcal{M}$ descends to a differential module over 
$K[T,T^{-1}]$ 
and to prove that the isomorphism obtained in Corollary 
\ref{Corollary : descent Turrittin} descends to an 
isomorphism over $D$.

Without loss of generality we can assume that $K'/K$ is 
Galois. It is a classical result of N.M.~Katz \cite{Katz-Can} 
that $\mathcal{M}$ descends into a differential module 
over $K[T,T^{-1}]$, denoted by $\mathrm{Can}(\M)$, which 
is defined as the fixed points of $\mathcal{M}$ 
by the action of the Galois group 
$\mathrm{Gal}(K'((T^{1/n}))/K((T)) )$ (this last group 
is implicitly identified with the Galois group of the Galois 
covering $[n]:
\mathbb{G}_{m,K'}\to\mathbb{G}_{m,K}$ 
given by the multiplication by $n$).
By Katz's construction, the isomorphism 
\begin{equation}
\beta'\;:\;\M\otimes_{K((T))} K'((T^{1/n}))
\;\xrightarrow{\;\sim\;}\;
\mathcal{M}\otimes_{K'[T^{\frac{1}{n}},
T^{-\frac{1}{n}}]} 
K'((T^{1/n}))
\end{equation} 
descends to 
an isomorphism 
\begin{equation}
\beta\;:\;\M\;\xrightarrow{\;\sim\;}\;
\mathrm{Can}(\M)\otimes_{K[T,T^{-1}]}K((T))\;.
\end{equation}
More generally, $\mathrm{Can}$ is a fully faithful functor 
associating to a differential module 
over $K((T))$ a differential module
over $K[T,T^{-1}]$. This functor is called Katz's 
\emph{canonical extension}.

Now, denote by $D'$ the inverse image of $D$ by 
$[n]^{\mathrm{an}}$. It is an open pseudo-disk centered at 
$0$ too. We claim that the isomorphism 
$\alpha':
\Fc\otimes_{\O(D)[T^{-1}]} 
\O(D')[T^{-1/n}]\xrightarrow{\;\sim\;}
\mathcal{M}_{|D'}$ obtained in Corollary 
\ref{Corollary : descent Turrittin} descends to 
an isomorphism 
\begin{equation}\label{eq : alpha F- Can}
\alpha\;:\;
\Fc\;\xrightarrow{\;\sim\;}\;\mathrm{Can}(\M)_{|D}\;.
\end{equation}
Indeed, it is enough to prove that $\alpha'$ 
commutes with the action of the Galois group, but this 
follows from the fact that its restriction 
$\alpha\otimes 1:
\Fc\otimes K'((T^{1/n}))\simto
\mathcal{M}\otimes K'((T^{1/n}))$ coincides with 
$\beta'$ (by the identification 
$\Fc\otimes K((T))=\M$). This shows that $\alpha'$ 
descends.

Finally we notice that, by Corollary 
\ref{Cor : push-f of inde gen-1}, the assumptions 
of Lemma \ref{Lemma: descent of morphisms-1} or 
Corollary \ref{Cor: descent of morphisms-2} are 
invariants by pull-back. Therefore those assumptions 
can be done directly on 
$\Hom(\Fc,\mathrm{Can}(\M)_{|D})$ or 
$\Hom(\mathrm{Can}(\M)_{|D},\Fc)$. 
\end{remark}

We summarize the above Remark in the following 
\begin{corollary}
\label{Corollary : descent Turrittin - CAN}
We maintain notation of section 
\ref{section : meromorphic Settings punctured disk}.
If the assumptions of 
Lemma \ref{Lemma: descent of morphisms-1} (resp.
Corollary \ref{Cor: descent of morphisms-2}) are fulfilled 
by $\mathrm{Hom}(\Fc,\mathrm{Can}(\M)_{|D})$ or by 
$\mathrm{Hom}(\mathrm{Can}(\M)_{|D},\Fc)$, then
the formal isomorphism 
$\Fc\otimes K((T))\simto\mathrm{Can}(\M)\otimes K((T))$ 
descends to an isomorphism 
\begin{equation}\label{eq: sdftgvplq}
\Fc\;\xrightarrow{\;\sim\;}\;\mathrm{Can}(\M)_{|D}
\end{equation}
over $\O(D)[T^{-1}]$.\hfill$\Box$
\end{corollary}

\subsubsection{Relations with Baldassarri's theorem.}
\label{Remark : classical exponents and Baldassarri result}
We maintain the notations of Sections \ref{Descent of morphisms} and \ref{Section : Descent Turritin-Balda}.
In this section we discuss the relations of Corollary 
\ref{Corollary : descent Turrittin} 
with Baldassarri's theorem 
\cite[Theorem 2]{Balda-Turritin}. We begin by 
recalling the classical notion of formal exponents, and 
state Baldassarri's theorem.

For $f=\sum_{i}a_iT^i$, 
we denote by $\mathrm{res}_0(f):=a_{0}$ the
residue of the differential form $f\cdot \frac{dT}{T}$ and 
by $v_T(f)=\min\{i,\;a_i\neq 0\}$ its $T$-adic valuation.

The classical definition of the exponent of a formal 
differential module $\M$ over 
$K((T))$ involves the \emph{indicial polynomial}. 
However, if the rank one modules $\N_1,\ldots,\N_r$ 
over $K((T))$ are given and if 
$e_i$ is the class modulo $\mathbb{Z}$ of 
$\mathrm{res}_0(g_i)$ 
(cf. \eqref{eq : Jordan Form g_i}), 
then the \emph{formal exponent}  of $\M$ is the 
multi-set $\{e_1,\ldots,e_r\}$ of elements of 
$K/\mathbb{Z}$.\footnote{We 
notice that if we consider the trivial valuation on $K$ 
and if $\M$ is seen as a differential module over the 
pseudo-annulus $\{ 0<|T|<1\}$ over $K$, then the 
formal exponent of the Robba part of $\M$ coincides with 
the exponent of Definition 
\ref{Def : exponent pseudo-an + germ of seg}.}
Moreover, an easy computation shows that the formal exponent 
of $\mathrm{End}(\M)$ is the multi-set 
\begin{equation}
\label{eq : formal exponent of End}
\{e_i-e_j\}_{i,j=1,\ldots,r}\;.
\end{equation}
Theorem \cite[Theorem 
2]{Balda-Turritin} asserts that if the residual field of $K$ 
has positive characteristic\footnote{The 
main ingredient of the proof of Baldassarri is 
the result of Clark \cite{Clark} that gives, under the 
same assumptions as those of Baldassarri's Theorem, 
the equality 
$\Hdr^0(K(\{T\}),
\mathrm{End}(\Fc))\cong\Hdr^0(K((T)),
\mathrm{End}(\M))$. 
Up to this ingredient, the 
proof of Baldassarri works over general base fields. 
Most probably the proof of Clark may be extended to a 
general field too, but we did not investigated in this direction 
(recall that 
if the residual field has characteristic $0$, there are no 
Liouville numbers, therefore no assumptions at all).} 
and if the formal exponent of $\mathrm{End}(\M)$ does 
not 
contain any Liouville number, then the isomorphism 
$\Fc\otimes K((T))\simto\mathcal{M}\otimes K((T))$ 
descends to an isomorphism 
$\Fc\otimes K(\{T\})\simto
\mathcal{M}\otimes K(\{T\})$ over 
$K(\{T\})$. The idea of the proof is that these 
isomorphisms are \emph{solutions} of the differential 
modules associated with the internal $\Hom$'s, 
therefore claim will follows from an 
isomorphisms between the formal and meromorphic 
de Rham cohomologies, which is a result of Clark (cf. \cite{Clark}).

We now discuss the relation with 
Corollary \ref{Corollary : descent Turrittin}. 
The conclusion of this last is stronger than 
the result of Baldassarri since 
it provides more precision about the disk of existence 
$D$ of the isomorphism $\Fc\simto\mathcal{M}_{|D}$. 
Moreover, in Lemma 
\ref{Lemma : formal Liouv imply ch-me Liouv for 
Hom(F,M)} below, 
we show that if the characteristic of the residual field of 
$K$ is positive, the assumptions of Baldassarri's theorem 
imply those of Corollary \ref{Cor : descent turrittin K(T)}. Note however that the proof of Lemma~\ref{Lemma : formal Liouv imply ch-me Liouv for Hom(F,M)} itself uses Baldassarri's theorem.

\begin{remark}
\label{Remark : Assumption about exponents Hom End}
The assumptions of Corollary 
\ref{Corollary : descent Turrittin}
involve $\Hom(\Fc,\mathcal{M}_{|D})$ 
while those of \cite[Theorem 2]{Balda-Turritin} 
involve $\mathrm{End}(\M)$. 
We now discuss this (apparent) discrepancy of 
assumptions.

If $\Fc$ is a differential equation on $D(*0)$ Baldassarri 
defines its formal exponent as the formal exponent of 
$\Fc\otimes K((T))$ (this last has been just defined 
cf. 
\eqref{eq : formal exponent of End}). 
We now notice that the scalar extension 
$-\otimes K((T))$ 
commutes with the internal $\mathrm{Hom}$ and we have
\begin{equation}
\mathrm{End}(\Fc)\otimes K((T))\;\cong\; 
 \End(\M)\;\cong\;\Hom(\Fc,\mathcal{M}_{|D})\otimes K((T))\;.
\end{equation}
Therefore, the formal exponent of 
$\Hom(\Fc,\mathcal{M}_{|D})$ coincides by 
definition with that of 
$\mathrm{End}(\Fc)$ and we may interpret the 
assumption of \cite{Balda-Turritin} as an 
assumption on the formal exponent of 
$\Hom(\Fc,\mathcal{M}_{|D})$.

However, in the context of Corollary 
\ref{Corollary : descent Turrittin} it seems clear 
that no assumptions on $\mathrm{End}(\Fc)$ 
can imply an isomorphism 
$\Fc\simto\mathcal{M}_{|D}$. 
For instance, if $\Fc$ has rank one, then 
$\mathrm{End}(\Fc)=\Fc^*\otimes\Fc$ is 
automatically trivial  as a differential equation over 
$D(*0)$ (hence it satisfies any condition of Lemma 
\ref{Lemma: descent of morphisms-1} and
Corollary \ref{Cor: descent of morphisms-2}, 
or about the exponents), 
while it is easy to provide examples of rank one 
modules $\Fc$ and $\mathcal{M}$ that are non 
isomorphic  as differential equations over 
$D(*0)$, but whose formal 
restrictions to $K((T))$ are isomorphic. 
For instance, an easy example is given by 
the rank equations 
$\mathcal{M}: y'=0$ (trivial) and $\Fc: y'=y$ 
(exponential). 
If $\omega$ is the radius of convergence of $\exp(T)$, 
then $\Fc$ is trivial in the disk $\{|T|<\omega\}$ and 
non-trivial (i.e. non isomorphic to $\mathcal{M}$) over 
any larger disk. 
\end{remark}

\begin{remark}\label{Rk : Liouville assumptions}
Assume that the characteristic $p$
of the residual field  $\widetilde{K}$ is positive.
We notice that 
the definition of formal exponent provided above 
(cf. \eqref{eq : formal exponent of End})
does not agree with that of Christol-Mebkhout 
\cite{Ch-Me-II} at the germ of segment $b_0$ out of $0$
(cf. Definition 
\ref{Def : exponent pseudo-an + germ of seg}). 
The reason is that this last only involves the 
\emph{Robba part} of a module, whose dimension can 
be strictly smaller than that of $\Fc$: the two 
multisets do not have the same number of elements.
More precisely, let 
$C_\varepsilon=\{0<|T|<\varepsilon\}$ be an
annulus with unspecified outer radius $\varepsilon$ 
and let $\Fc$ and $\mathcal{M}$ be as in Section 
\ref{Section : Descent Turritin-Balda}. 
If an \emph{analytic} isomorphism 
$\Fc_{|C_\varepsilon}\simto
\mathcal{M}_{|C_\varepsilon}$ 
is provided over $C_\varepsilon$, then
the definition of \cite{Ch-Me-II} 
of the exponent of 
$\Fc^{\mathrm{Robba}}_{|C_\varepsilon}$ 
consists in the multi-set 
\begin{equation}
\{e_i\;\textrm{ such that }\;
\mathcal{N}_i=\mathcal{N}_i^{\mathrm{Robba}}\;,
\;i=1,\ldots,r\}\;.
\end{equation}
While that of $\mathrm{End}(\Fc)^{\mathrm{Robba}}_{|C_\varepsilon}$ 
is the multi-set
\begin{equation}\label{eq : p-adic exponent end}
\{e_i-e_j\;
\textrm{ such that }\;
\mathcal{N}_i\otimes\mathcal{N}_j^*=
(\mathcal{N}_i\otimes\mathcal{N}_j^*)^{\mathrm{Robba}}
\;,\;i,j=1,\ldots,r\}\;.
\end{equation}
\if{of the residues 
$\mathrm{res}_0(g_i)$ modulo $\mathbb{Z}$
of those $g_i$ verifying 
$\mathcal{N}_i=\mathcal{N}_i^{\mathrm{Robba}}$ 
along $b_0$.}\fi 
One proves that 
$\mathcal{N}_i=\mathcal{N}_i^{\mathrm{Robba}}$ 
if, and only if, $v_T(g_i)\geq 0$ and 
$e_i\in\mathbb{Z}_p$ 
(resp. $\mathcal{N}_i\otimes\mathcal{N}_j^*=
(\mathcal{N}_i\otimes
\mathcal{N}_j^*)^{\mathrm{Robba}}$ if, and only if, $v_T(g_i-g_j)\geq 0$ and 
$e_i-e_j\in\mathbb{Z}_p$). 

In this case, the formal exponent of $\Fc$ (resp. 
$\End(\Fc)$) contains the
Christol-Mebkhout exponent of $\Fc$ at $b_0$ and the 
inclusion may be strict. 
In particular, if the formal exponent does not contain any 
Liouville number, then so does the Christol-Mebkhout 
exponent (indeed Liouville numbers lie in 
$\mathbb{Z}_p$). 
The converse is not true in general.

However, in the general case, if an isomorphism 
$\Fc_{|C_\varepsilon}\simto
\mathcal{M}_{|C_\varepsilon}$ is not given (for 
instance because we do not assume any condition on the 
exponents nor about the indexes), 
the definition of the Christol-Mebkhout exponent is more 
involved, and the relation with the classical 
definition \eqref{eq : formal exponent of End} 
is not clear to us.\footnote{It is known that if $\Fs$ is of type Robba and if the 
Christol-Mebkhout exponent of  
$\mathrm{End}(\Fs)$ along $b_0$
does not contain any Liouville number, then 
$\Fs_{|C_\varepsilon}$ 
decomposes into rank one 
sub-quotients (cf. Theorem 
\ref{Thm : deco in rk 1 Ch-Me Robba}). 
However, it is not clear that this implies that 
$\M$ decomposes over $K((T))$ into rank one 
sub-quotients too. Moreover, in the case where $\M$ 
decomposes over $K((T))$, 
it is not clear that the decomposition 
descends to $K(\{T\})$, and if it does it is not clear 
that it coincides with that of 
$\Fs_{|C_{\varepsilon}}$.}
\end{remark}

We now 
prove that the assumptions of 
\cite[Theorem 2]{Balda-Turritin} imply those 
of Corollary \ref{Cor : descent turrittin K(T)} 
for the differential equation 
$\Hom(\Fc,\mathcal{M}_{|D})\otimes K(\{T\})$.
\begin{lemma}
\label{Lemma : formal Liouv imply ch-me Liouv for Hom(F,M)}
We maintain the assumptions 
established all along Section 
\ref{Section : Descent Turritin-Balda}. Assume that the 
residual field of $K$ has positive characteristic and 
consider the following properties
\begin{enumerate}
\item the formal exponent of 
$\Hom(\Fc,\mathcal{M}_{|D})$ 
(i.e. the formal exponent of $\mathrm{End}(\M)$, 
cf. Remark 
\ref{Remark : Assumption about exponents Hom End}) 
has no Liouville numbers;
\item the Christol-Mebkhout exponent of 
$\Hom(\Fc,\mathcal{M}_{|D})$ at the germ of segment 
$b_0$ out of $0$ (cf. Definition 
\ref{Def : exponent pseudo-an + germ of seg}) 
has no Liouville numbers.
\end{enumerate}
Then, the first property implies the second one. 

In this case 
$\Hom(\Fc,\mathcal{M}_{|D})$ satisfies 
the condition $\Fin_{b_0}$ 
and $\Fc\otimes K(\{T\})\simto
\mathcal{M}\otimes K(\{T\})$. In 
particular the assumptions of 
Corollary \ref{Cor : descent turrittin K(T)} and 
\ref{Cor : comparison mero-formal K(t) and K((t))} are 
fulfilled (cf. Section \ref{section : some situations FIn}).
\end{lemma}
\begin{proof}
\if{there exists sequence of disk 
$D_n$ in $D$, with open boundary $b_n$, 
satisfying $\cap_nD_n=\{0\}$, 
such that for all $n$ the equation 
$H_n:=\Hom(\Fc_{|D_n},
\mathcal{M}_{|D_n})$ verifies 
$\chiabs_{b_n}(H_n)=\Irr_{b_n}(H_n)$. 
}\fi
If i) holds, the assumptions of Baldassarri's result are 
fulfilled, therefore we have an isomorphism 
$\Fc\otimes K(\{T\})\simto
\mathcal{M}\otimes K(\{T\})$. 
In particular, there exists an 
unspecified pseudo-annulus 
$C_\varepsilon=\{0<|T|<\varepsilon\}$ together with an 
isomorphism $\Fc_{|C_\varepsilon}\simto
\mathcal{M}_{|C_\varepsilon}$. 
Now, by Remark \ref{Rk : Liouville assumptions}, 
the Christol-Mebkhout exponent of 
$\Hom(\Fc,\mathcal{M}_{|D})$ at $b_0$ is 
included into the formal exponent (as multi-sets), and 
hence it has no Liouville numbers too. Therefore ii) holds.

\if{Assume now that ii) holds. 
Let $b_\varepsilon$ be the germ of segment 
at the open boundary of the disk 
$D_\varepsilon=\{|T|<\varepsilon\}$. 
By Section \ref{section : some situations FIn} 
there exists a $\varepsilon'$ such that for all 
$\varepsilon\leq \varepsilon'$ the equation 
$H_\varepsilon:=\Hom(\Fc_{|D_\varepsilon},
\mathcal{M}_{|D_\varepsilon})$ verifies 
$\chiabs_{b_\varepsilon}(H_\varepsilon)
=\Irr_{b_\varepsilon}(H_\varepsilon)$. Therefore, 
Corollary \ref{Corollary : descent Turrittin} applies 
and we have an isomorphism 
$\Fc_{|D_\varepsilon}\simto
\mathcal{M}_{|D_\varepsilon}$ over $\O(D_\varepsilon)[T^{-1}]$. 
By Remark \ref{Rk : Liouville assumptions} it follows that 
the formal exponent of $\Hom(\Fc,\mathcal{M}_{|D})$ 
has no Liouville numbers (indeed the 
Christol-Mebkhout exponent of 
$\Hom(\Fc,\mathcal{M}_{|D})$ at $b_0$ is 
included into the formal exponent, and 
any component of the formal exponent that lies in 
$\mathbb{Z}_p$ belongs actually to Christol-Mebkhout 
one).}\fi
The last assertion follows from 
Section \ref{section : some situations FIn}.
\end{proof}
\if{\comm{Pour Jérôme,

c'est terrible ! Si on suppose que l'exposant de Ch-Me 
de $\mathrm{End}(\Fc)$ lelong de $b_0$ n'a pas de 
nombres de Liouville, alors on sait par le Théorème 
\ref{Thm : deco in rk 1 Ch-Me Robba} que $\Fc$ a une 
suite de Jordan-Hölder formée par des modules de rang 
un. Le module $\mathcal{M}$ aussi a par construction la 
même propriété. Mais le fait d'avoir un isomorphisme 
$\Fc\otimes K((T))\simto\mathcal{M}\otimes K((T))$ ne 
semble entrainer un isomorphisme sur 
$\O(C_\varepsilon)$. \\

Le problème est que les modules de rang un fournis par 
le théorème \ref{Thm : deco in rk 1 Ch-Me Robba} ne 
sont pas forcement méromorphes en $0$.\\

Mais je ne suis pas capable de touver un exemple.\\

Toutefois, en rang un c'est vrai: si deux modules de rang 
un sont isomorphes sur $K((T))$ ils le sont aussi sur 
$K(\{T\})$. En effet il est facile de voir, sans aucune 
hypothèse (Liouville ou autre), que les classes 
d'isomorphisme de modules de rang un
sur $K((T))$ sont les mêmes que sur $K(\{T\})$. C'est à 
dire que tout module de rang un s'étend 
automatiquement à $K(\{T\})$. Du coup on ne pourra 
pas trouver un contre-exemple en rang un.\\

Est-ce que par hasard tu vois comment démontrer que 
l'hypothèse Liouville de Ch-Me sur $\mathrm{End}(\Fc)$ 
le long de $b_0$ entraine un isomorphisme 
$\Fc\otimes \O(C_\varepsilon)\simto
\mathcal{M}\otimes \O(C_\varepsilon)$?\\

C'est lié au fait de trouver un module sur $D(*0)$ qui 
soit triviale sur $K((T))$, mais non triviale sur 
$K(\{T\})$.}
}\fi
\subsubsection{The restriction functor.}
\label{Section : The restriction functor}
The above comparison results 
permit to deduce an equivalence of 
categories between certain sub-categories of differential 
equations over $D(*0)$ and $K((T))$.

\begin{definition}
Consider a derivation $d:R\to R$ on a ring $R$. 
We denote by 
\begin{equation}
d-\mathrm{Mod}(R)
\end{equation}
the category whose objects are  
$R$-modules $\N$ together with a connection 
$\nabla:\N\to \N$ satisfying the Leibniz rule with 
respect to $d$ and whose morphisms are $R$-linear 
maps commuting with the connections. 
The group of 
morphisms between two objects $\M$ and $\N$ will be 
denoted by $\Hom^{\nabla}(\M,\N)$.

\if{If $R=K(\{T\})$, we denote by 
\begin{equation}
d-\mathrm{Mod}(K(\{T\}),\Fin)
\end{equation}
the full subcategory of $d-\mathrm{Mod}(K(\{T\}))$ 
formed by differential modules satisfying $\Fin_C$ for 
some unspecified pseudo-annulus 
$C=\{0<|T|<\varepsilon\}$.}\fi
\end{definition}

\begin{definition}
For $\M,\N\in d-\mathrm{Mod}(R)$ we denote 
by $\mathrm{Ext}_{d-\mathrm{Mod}(R)}(\M,\N)$ the 
Yoneda extension 
group (whose elements are equivalence classes 
of exact sequences 
$0\to \N\to \mathrm{E}\to \M\to0$ in 
$d-\mathrm{Mod}(R)$). If the category 
$d-\mathrm{Mod}(R)$ is clear, we simply write 
$\mathrm{Ext}(\M,\N)$.

Let $\mathcal{C}$ be a sub-category of 
$d-\mathrm{Mod}(R)$. We say that $\mathcal{C}$ is 
stable by extensions if
for each exact sequence 
$0\to\M\to\mathrm{E}\to\N\to 0$ 
in $d-\mathrm{Mod}(R)$, where 
$\M$ and $\N$ belong to $\mathcal{C}$, the middle 
term $\mathrm{E}$ belongs to $\mathcal{C}$ too.
\end{definition}

The boundary conditions (as well as the Liouville 
conditions on the exponents) 
are not stable by tensor product nor by internal 
$\Hom$. Hence, we consider a 
full sub-category $\mathcal{C}$ of 
$d-\mathrm{Mod}(\O(D)[T^{-1}])$ such that for all 
$\Fc,\Gc\in\mathcal{C}$ the  
differential equation $\Hom(\Fc,\Gc)$ satisfies the following properties:
\begin{enumerate}
\item $\Hom(\Fc,\Gc)$ is free; 
\item $\Hom(\Fc,\Gc)$ has finite-dimensional de Rham cohomology;
\item $\chidr(D(*0),\Hom(\Fc,\Gc))=0$.
\end{enumerate}

\begin{corollary}\label{Cor : equivalente D(*0)-form}
The restriction functor  
\begin{equation}
Res^{D(*0)}_{K((T))}\;:\;
\mathcal{C}
\;\xrightarrow{\quad}\;
d-\mathrm{Mod}(K((T)))\;.
\end{equation}
is fully faithful. 
Moreover, for $\Fc,\Gc\in\mathcal{C}$, if we denote by $F=\Fc\otimes K((T))$ and 
$G=\Gc\otimes K((T))$ their images in 
$d-\mathrm{Mod}(K((T)))$, we have an isomorphism 
of Yoneda extension groups
\begin{equation}\label{eq : stability by ext mero-form-111}
\mathrm{Ext}_{d-\mathrm{Mod}(\O(D)[T^{-1}])}(\Fc,
\Gc)
\;\xrightarrow{\;\sim\;}\;
\mathrm{Ext}_{d-\mathrm{Mod}(K((T)))}(F,G)\;.
\end{equation}
In particular, 
if $\mathcal{C}$ 
is stable by extensions in 
$d-\mathrm{Mod}(\O(D)[T^{-1}])$, 
then so is its essential image in 
$d-\mathrm{Mod}(K((T)))$. 
\end{corollary}
\begin{proof}
The full-faithfullness follows from Lemma~\ref{Lemma : mero vs formal coh}
applied to the differential equation $\Hom(\Fc,\Gc)$ since
\begin{equation}
\Hom^{\nabla}(\Fc,\Gc)\;=\;
\Hdr^0(\Hom(\Fc,\Gc))\;=\;
\Hdr^0(\Hom(F,G))\;=\;
\Hom^{\nabla}(F,G)\;.
\end{equation}
Let $A$ be one of the rings $\O(D)[T^{-1}]$ 
and $K((T))$ and let $\M,\N\in d-\mathrm{Mod}(A)$.
It is a general fact that, when~$M$ and~$N$ are projective, we have isomorphisms
$\mathrm{Ext}(\M,\N)\cong
\Hdr^1(A,\M^*\otimes \N)\cong
\Hdr^1(A,\Hom(\M, \N))$ (cf. 
\cite[Lemma~5.3.3 and Remark~5.3.4]{Kedlaya-book}). 
The isomorphism \eqref{eq : stability by ext mero-form-111} then follows from Proposition~\ref{prop:eqcatStein} and Lemma 
\ref{Lemma : mero vs formal coh}. 
%
%
%
%
\end{proof}

We now state a variant of the above corollary 
for the field $K(\{T\})$. 
It is clear that the existence of a sequence of disks as in 
Corollary 
\ref{Cor : comparison mero-formal K(t) and K((t))} 
only depends on $\Fc_0^\dag$, and not 
on the particular equation~$\Fc$ defined 
over some disk~$D$ centered at~$0$ satisfying 
$\Fc_0^\dag=\Fc\otimes K(\{T\})$. 
Let $\mathcal{B}$ be a full-subcategory of 
$d-\mathrm{Mod}(K(\{T\}))$ such that for all 
$\Fc_0^\dag,\Gc_0^\dag\in\mathcal{B}$, the equation 
$\Hom(\Fc_0^\dag,\Gc_0^\dag)$ has finite-dimensional 
cohomology groups 
$\Hdr^i(K(\{T\}),\Hom(\Fc_0^\dag,\Gc_0^\dag))$
and satisfies
\begin{equation}
\chidr(K(\{T\}),\Hom(\Fc_0^\dag,\Gc_0^\dag))\;=\;0\;.
\end{equation}

The proof of the following Corollary is similar to that of 
Corollary \ref{Cor : equivalente D(*0)-form} using Lemma~\ref{Lemma : res=iso merodag form} instead of Lemma~\ref{Lemma : mero vs formal coh}.

\begin{corollary}\label{Cor : equivalente merodag-form}
The scalar-extension functor  
\begin{equation}
\mathrm{Res}^{K(\{T\})}_{K((T))}\;:\;
\mathcal{B}
\;\xrightarrow{\quad}\;
d-\mathrm{Mod}(K((T)))\;.
\end{equation}
is fully faithful. Moreover, 
for $\Fc_0^\dag,\Gc_0^\dag\in\mathcal{B}$, if we denote by $F$ and $G$ their images in 
$d-\mathrm{Mod}(K((T)))$, we 
have we have an isomorphis of Yoneda extension groups
\begin{equation}\label{eq : stability by ext mero-form}
\mathrm{Ext}_{d-\mathrm{Mod}(K(\{T\}))}(\Fc_0^\dag,\Gc_0^\dag)
\;\xrightarrow{\;\sim\;}\;
\mathrm{Ext}_{d-\mathrm{Mod}(K((T)))}(F,G)\;.
\end{equation}

In particular, if $\mathcal{B}$ is stable by extensions in 
$d-\mathrm{Mod}(K(\{T\}))$, then so is 
its essential image in $d-\mathrm{Mod}(K((T)))$. 
\hfill$\Box$
\end{corollary}

\subsection{Meromorphic vs. analytic theories}\label{Meromorphic vs. analytic theories}
In this section we compare the meromorphic and the 
analytic theories of differential equations.

We maintain the notations of Section \ref{section : 
meromorphic Settings punctured disk}. In particular, we have $C=D-\{0\}$ and $\Fs=\Fc_{|C}$.
We further consider an open sub-pseudo-annulus 
\begin{equation}\label{eq : C' mero anal section45}
C'\;\subseteq\; C=D-\{0\}
\end{equation}
such that $\Gamma_{C'}\subseteq\Gamma_C$.
We denote by $b_{0,C'},b_{1,C'}$ the germs of segment 
at the boundary of $C'$.

\subsubsection{Meromorphic vs. analytic cohomologies.}

\label{Section :  mero VS anal}

\begin{lemma}
\label{Lemma : injectiona and surjection res mero anal}
The following hold: 
\begin{enumerate}
\item the natural map 
$\Hdr^0(D(*0),\Fc)\to\Hdr^0(C',\Fs_{|C'})$ 
is injective;

\item if $\Hdr^1(C',\Fs_{|C'})$ is finite-dimensional, the 
map $\Hdr^1(D(*0),\Fc)\to\Hdr^1(C',\Fs_{|C'})$ is surjective;

\item if $\Hdr^1(D(*0),\Fc)$ and 
$\Hdr^1(C',\Fs_{|C'})$ are finite-dimensional, then 
\begin{equation}\label{S4eq : chidr leq chi anal}
\chidr(D(*0),\Fc)\;\leq\;\chidr(C',\Fs_{|C'})\;.
\end{equation}
\end{enumerate}
Moreover, in the setting of iii), the following conditions 
are equivalent:
\begin{enumerate}
\item[(a)] $\chidr(D(*0),\Fc)=\chidr(C',\Fs)$;
\item[(b)] for $i=0,1$, the restriction maps
\begin{equation}\label{eq : iso mero anal coh-sec5}
\Hdr^i(D(*0),\Fc)\;\xrightarrow{\;\sim\;}\;
\Hdr^i(C',\Fs_{|C'})
\end{equation}
are isomorphisms.
\end{enumerate}
\end{lemma}
\begin{proof}
The only difficulties lie in point~ii). It follows from \cite[Lemma~4.17]{Banachoid} (using also Corollary~\ref{cor:cohomologypseudoannulusLiouville} and Remark~\ref{rem:logafftrivval} when~$K$ is trivially valued).

The rest of the proof is analogous to that of Lemma~\ref{Lemma : mero vs formal coh}.
%
\end{proof}

\begin{theorem}
\label{Theorem : restriction meromorphic}
Assume that 
\begin{enumerate}
\item $\mathcal{F}$ is a free 
$\O_D[*0]$-module;
\item the radii of $\Fc$ are all log-affine 
at the open boundary of $D$;
\item $\Fs$ is Fredholm at 
$b_{0,C'}$, $b_{1,C'}$ and \eqref{eq:chiabs01} holds;
\item $\Fs$ is Fredholm at $b_D$ and 
satisfies \eqref{eq : chiabs=Irr};


\item $\Irr_{C'}(\Fs)=\Irr_{C}(\Fs)$.
\end{enumerate}
Then, $\Fc$ and~$\Fs$ have finite-dimensional de Rham cohomology on~$D(*0)$ 
and~$C'$ respectively and, for every $i\geq 0$, the natural 
morphism
\begin{equation}
\label{eq : iso D(*0) C mero to anal disk}
\Hdr^i(D(*0),\Fc)\;\xrightarrow{\;\;\sim\;\;}\;
\Hdr^i(C',\Fs_{|C'})\;
\end{equation}
is an isomorphism.
\end{theorem}
\begin{proof}
The claim follows from Theorems~\ref{Thm : Index meromorphic disk} 
and~\ref{Thm : index of finite opens} and Lemma 
\ref{Lemma : injectiona and surjection res mero anal}.
\end{proof}

In the following statement we 
temporarily suspend the notations of Section 
\ref{section : meromorphic Settings punctured disk}

\begin{corollary}
\label{Coro : Mero = analif Liouville}
Let $Y$ be a quasi-smooth $K$-analytic curve and $Z$ 
be a locally finite subset of $K$-rational points of~$Y$. 
Let $\Fc$ be a differential equation on $Y(*Z)$. 
Set $X:=Y-Z$ and $\Fs:=\Fc_{|X}$. 

For each $z\in Z$, let $D_z\subseteq Y$ be 
an open disk centered at $z$. 
We assume that for $z\neq z'$ one has 
$D_z\cap D_{z'}=\emptyset$. 
Denote by $b_{D_z}$ the open boundary of 
$D_z$ and by $b_z$ the 
germ of segment out of 
$z$.

Then the following statement are equivalent
\begin{enumerate}
\item for each $i \ge 0$, we have a natural isomorphism
\begin{equation}
\Hdr^i(Y(*Z),\Fc)\; \xrightarrow[]{\sim}\;
\Hdr^i(X,\Fs)\;.
\end{equation}
\item for each $i \ge 0$, and each $z\in Z$ 
we have a natural isomorphism
\begin{equation}
\Hdr^i(D_z(*\{z\}),\Fc)\; \xrightarrow[]{\sim}\;
\Hdr^i(D_z-\{z\},\Fs)\;.
\end{equation}
\end{enumerate}
In particular, these conditions hold 
if, for all $z\in Z$, one has
\begin{enumerate}
\item[ii-a)] $\Fc$ is free as $\O_{D_z}[*z]$-module;
\item[ii-b)] if $D_z$ is a connected component of $Y$, then $\Fs$ has log-affine radii at $b_{D_z}$;
\item[ii-c)] $\Fs$ is Fredholm at $b_{D_z}$ and $b_z$ and 
satisfies \eqref{eq:chiabs01}. 
\if{(cf. condition ii) of Proposition \ref{Prop : chirel=Irr}):
\begin{equation}
\chiabs_{b_{D_z}}(\Fs)=\Irr_{b_D}(\Fs)\;,\qquad
\chiabs_{b_{z}}(\Fs)=\Irr_{z}(\Fs)\;.
\end{equation}
}\fi
\end{enumerate}
\end{corollary}
\begin{proof}
\if{
By Theorem~\ref{thm:descent}, we can assume 
that~$K$ is algebraically closed and non-trivially valued, 
and that $z$ is a $K$-rational point of $Y$. 

\comm{Manqhe la déscente pour les $\Hdr^i(Y(*Z),\Fc)$ 
...}
}\fi
%
%
%
Clearly,

Set $D:=\bigcup_{z\in Z}D_z$. We have $X\cup D=Y$ 
and $X\cap D=\bigcup_{z\in Z}C_{z}$, where 
$C_z=D_z-\{z\}$. 
The Mayer-Vietoris exact sequences for the meromorphic and analytic de Rham cohomologies fit into a two-line diagram:
\begin{equation}
\xymatrix{
\ar[r]&\ar[d]\Hdr^i(Y(*Z),\Fc)\ar[r]&\ar[d]\Hdr^i(X(*Z),\Fc)\oplus\Hdr^i(D(*Z),\Fc)\ar[r]&\ar[d]\Hdr^i((X\cap D)(*Z),\Fc)\ar[r]&\\
\ar[r]&\Hdr^i(X,\Fs)\ar[r]&\Hdr^i(X,\Fs)\oplus\Hdr^i(D-Z,\Fs)\ar[r]&
\Hdr^i(X\cap D,\Fs)\ar[r]&\\
}
\end{equation}

By Remark~\ref{rk : mero=anal over U}, we have 
$\Hdr^i(X(*Z),\Fc)=\Hdr^i(X,\Fs)$ and
$\Hdr^i((X\cap D)(*Z),\Fc)=\Hdr^i((X\cap D),\Fs)$. 

Assume that ii) holds. First of all, we have 
$\Hdr^i(D(*Z),\Fc)=\oplus_{z\in Z}\Hdr^i(D_z(*z),\Fc)$
and
$\Hdr^i(D-Z,\Fc)=\oplus_{z\in Z}\Hdr^i(C_z,\Fc)$. 
This means that in the above 
diagram, all vertical arrows are isomorphisms up to the 
$\Hdr^i(Y(*Z),\Fc)\to\Hdr^i(X,\Fs)$. By the five lemma, 
it has to be an isomorphism too.
A symmetric argument shows that i) implies ii).

Assume now that ii-a), ii-b), ii-c) hold. We can apply 
Theorem~\ref{Theorem : restriction meromorphic} 
to obtain an isomorphism 
$\Hdr^i(D_z(*z),\Fc)\simto\Hdr^i(C_z,\Fs)$, for all 
$z\in Z$. Item ii) is then satisfied.
\end{proof}

\begin{remark}
\label{rk : trivialization Y'=GY then rational}
An interesting consequence of Corollary   
\ref{Coro : Mero = analif Liouville} is the 
fact that any global \emph{analytic} 
solution of $\nabla$ on $X$ is actually 
\emph{meromorphic} on $Y$ with poles in $Z$.
\end{remark}

Corollary \ref{Coro : Mero = analif Liouville} together 
with item vi) of Section 
\ref{section : some situations FIn} imply 
the following meromorphic analogue of 
Proposition~\ref{Prop : H^i=0 if not solvablegfz}.
\begin{proposition}
\label{Prop : H^i=0 if not solvablegfz mero}
We maintain the notations of Corollary 
\ref{Coro : Mero = analif Liouville}.
Assume that the differential equation 
$\Fs:=\Fc_{|X}$ over $X$ 
satisfies  the conditions of situation~$1$ of 
Proposition~\ref{Prop : H^i=0 if not solvablegfz}. 
Then, for all $i$, we have $\Hdr^i(Y(*Z),\Fc)=0$.
\hfill$\Box$
\end{proposition}

\begin{remark}\label{rem:restrictionLiouvilleD*0}
Let~$D_1$ be a subdisk of~$D$ centered at~0 and set 
$C_1 := D_1-\{0\}$. Assume that the assumptions of 
Theorem \ref{Theorem : restriction meromorphic}  
hold on $D$ and on $D_1$ with respect to the same 
$C'$. 
In this situation, it follows from 
\eqref{eq : iso D(*0) C mero to anal disk}
that, for all $i\geq 0$, the natural restrictions
\begin{equation}
\label{eq : iso restr mero to mero disk anal hyp}
\Hdr^i(D(*0),\Fc)\;\xrightarrow{\;\;\sim\;\;}\;
\Hdr^i(D_1(*0),\Fc_{|D_1})\;
\end{equation}
are isomorphisms. 

\if{

\comm{
Pour Jérôme, comme pour la Remarque 
\ref{Rk :  subdisk D' of D mero-formal}, 
j'ai le problème de savoir si 
on a l'isomorphisme 
\eqref{eq : iso restr mero to mero disk anal hyp} avec 
des hypothèses plus faibles. 
Plus précisément, peut on démontrer ça ?\medskip

Let~$D_1$ be a subdisk of~$D$ centered at~0 and set 
$C_1 := D_1-\{0\}$. Denote by $b_D$ and $b_{D_1}$ the 
respective open boundaries. 
Assume that $\mathcal{F}$ is a free 
$\O(D)[T^{-1}]$-module whose the radii are all 
$\log$-affine along the open boundary of $D$, having 
finite index at $b_D$ and $b_{D_1}$ and satisfying
\eqref{eq : assumption iii) Liouville punctured disk} 
 on $D$ and on $D_1$. Then we have 
\eqref{eq : iso restr mero to mero disk anal hyp}.}

}\fi
\end{remark}

\subsubsection{$D(*0)$ vs. $\mathfrak{R}_b$.}
We maintain the notations of Section 
\ref{section : meromorphic Settings punctured disk}.
In this section, we consider a germ of segment 
$b\subset\Gamma_C$ and the Robba ring 
$\mathfrak{R}_{b}$ at $b$.

\begin{corollary}\label{Comparison Rigid formal-1}
Assume that 
\begin{enumerate}
\item $\Fc$ is free as $\O(D)[T^{-1}]$-module;
\item $\Fc$ has finite-dimensional de Rham cohomology 
on $D(*0)$;
\item there exists a sequence  $C_1\supseteq C_2\supseteq\cdots$ of pseudo-annuli 
having $b$ at their boundary such that 
$\bigcap_nC_n=\emptyset$ and a complete valued extension~$L$ of~$K$ with non-trivial valuation such that, for each $n$, the equation
$(\Fc_{L})_{|(C_n)_{L}}$ has finite-dimensional de Rham cohomology;
\item for each $n$, one has $\chidr(C_n,\Fc_{|C_n})=\chidr(D(*0),\Fc)$.
\end{enumerate}
Then, for each $n$, and $i=0,1$ we have canonical isomorphisms
\begin{equation}\label{eq : identif}
\Hdr^i(D(*0),\Fc)\;\xrightarrow{\;\sim\;}\;
\Hdr^i(C_n,\Fs_{|C_n})\;\xrightarrow{\;\sim\;}\;
\Hdr^i(\mathfrak{R}_{b},\Fs_{|
\mathfrak{R}_{b}})\;.
\end{equation}
If moreover 
$\chidr(D(*0),\Fc)=0$, then 
one also has isomorphisms 
$\Hdr^i(D(*0),\Fc)\simto\Hdr^i(K((T)),\M)$.
\end{corollary}
\begin{proof}
The claim follows from Lemma 
\ref{Lemma : injectiona and surjection res mero anal},
Corollary \ref{Cor : zero index over the robba ring}.
\end{proof}
\begin{remark}
\label{Rk : If the radii are affine L=K}
If the radii of $\Fc$ are log-affine at $b$, 
item iii) of Corollary \ref{Comparison Rigid formal-1} 
can be relaxed by assuming $L=K$ (cf. item ii) of 
Corollary \ref{Cor : zero index over the robba ring}).
\end{remark}

\if{Using Theorems \ref{Thm : Index meromorphic disk}, 
\ref{Theorem : restriction meromorphic}, 
Lemmas \ref{Lemma : mero vs formal coh}, 
\ref{Lemma : H^i K(T)=lim-2} and 
Corollary \ref{Cor : zero index over the robba ring} we 
now obtain a set of boundary conditions that imply the 
conditions of Corollary \ref{Comparison Rigid formal-1}. 

\comment{Ce que tu \'ecris ici n'est pas clair. Pour montrer que les conditions du corollaire qui suit impliquent celle du corollaire \ref{Comparison Rigid formal-1}, il me semble que le th\'eor\`eme \ref{Theorem : restriction meromorphic} suffit.

Je sugg\`ere d'enlever la phrase introductive et d'ajouter une preuve structur\'ee comme suit.

By Theorem~\ref{Theorem : restriction meromorphic}, the conditions i) to v) imply the conditions of Corollary~\ref{Comparison Rigid formal-1}. Point (a) follows from ... Point (b) follows from ...

Je n'\'ecris pas les d\'etails parce que je ne comprends pas vraiment comment d\'eduire les \'enonc\'es des r\'ef\'erences que tu donnes. En particulier, j'ai l'impression qu'il faudrait demander \eqref{eq : chiabs=Irr} en $b_{n}$.
}
\comm{J'applique vraiment le Théorème 
\ref{Theorem : restriction meromorphic} 
(qui est une conséquence du Theorème 
\ref{Thm : Index meromorphic disk} quand même, mais 
puisqu'on s'est fatigué à le demontrer il vaut 
mieux l'utiliser...). 

Je fais une hypothèse de finitude en $b_n$, qui n'est pas 
\eqref{eq : chiabs=Irr}, mais qui est quand même 
forte et entraine le théorème d'indice.

J'utilise aussi le 
Corollary \ref{Cor : zero index over the robba ring} pour 
avoir l'égalité entre $C_n$ et $\mathfrak{R}_b$.

Les Lemmes \ref{Lemma : mero vs formal coh}, 
\ref{Lemma : H^i K(T)=lim-2} interviennent uniquement 
pour demontrer (a) et (b).

Je crois que j'utilise tout ce que j'ai indiqué 
finalement ... et tu as raison qu'il faut une preuve ...
}
\comment{Je ne comprends pas ta r\'eponse. Je ne comprends toujours pas comment on peut d\'emontrer le r\'esultat avec les hypoth\`eses de l'\'enonc\'e.}
\comm{Voila une preuve :

Les points i) à v) sont exactement les hypothèses du théorème 4.6.2 pour avoir pour tout n l'isomorphisme

$H^i(D(*0)) --> H^i(C_n).$

Cela entraine que pour tout n la restriction $H^i(C_n)-->H^i(C_n+1)$ est un isomorphisme.

Par le Corollaire 3.7.3 cela entraine qu'ils sont tous isomorphes à $H^i(R_b)$ (Robba) c'est à dire qu'on a les isomorphismes (4.111).

Maintenant, si $b=b_0$ on a droit de se demander si cela est aussi égal à $H^i( K({T}) )$. C'est la question (a) :

Comme on a supposé $Irr_{C_n}=Irr_C$, et comme les rayons sont affines à l'origine, on a $Irr_{C_n}=0$ pour tout n. 

Donc, par le théorème 4.4.1, on a $chi(D(*0))=0$, et donc les conditions du Lemme 4.5.1 sont satisfaites, et le point (a) est donc vrai.

Si b n'est pas forcement égale à $b_0$, on peut quand même utiliser le 4.5.1 pour avoir 

$H^i( D(*0) )-->H^i( K((T)) )$

mais ce n'est pas automatique comme avant, il faut supposer $Irr_C=0$ pour appliquer le théorème 4.4.1 et déduire que $chi(D(*0))=0$, pour remplir les conditions du 4.5.1.}

}\fi
\begin{corollary}\label{Comparison Rigid formal-2}
Let $C_1\supseteq C_2\supseteq\cdots$ be a sequence 
of pseudo-annuli having $b$ at their boundary such that 
$\bigcap_nC_n=\emptyset$. For each~$n\ge1$, denote by~$b_{n}$ the germ of segment at the open boundary of~$C_n$ that is not~$b$. 

Assume that 
\begin{enumerate}
\item $\Fc$ is a free $\O(D)[T^{-1}]$-module;
\item the radii of $\Fc$ are log-affine at the open 
boundary of $D$;
\item 
$\Fc$ is Fredholm and 
satisfies \eqref{eq : chiabs=Irr} 
at $b_D$;

\item for each $n\geq 1$, $\Fc$ is Fredholm at $b$ and $b_n$ 
and satisfies \eqref{eq:chiabs01}  over $C_n$;

\item for each $n\ge1$, one has 
$\Irr_{C_n}(\Fs)=\Irr_C(\Fs)$.
\end{enumerate}
Then, we have the isomorphisms \eqref{eq : identif}. 

In this situation we have moreover the following facts.
\begin{enumerate}
\item[(a)] If $\Irr_C(\Fs)=0$, then the spaces in 
\eqref{eq : identif} are also isomorphic to 
$\Hdr^i(K((T)),\M)$. 
\item[(b)] If $b=b_0$ is the germ of segment out of 
$0$, the spaces in \eqref{eq : identif} are also 
isomorphic to $\Hdr^i(K(\{T\}),\Fc_{0}^\dag)$.
\end{enumerate}
%
\end{corollary}
\begin{proof}
By Theorem~\ref{Theorem : restriction meromorphic}, 
the conditions i) to v) imply the conditions of 
Corollary~\ref{Comparison Rigid formal-1}. 
The isomorphisms~\eqref{eq : identif} then follow from it and from Corollary \ref{Cor : zero index over the robba ring}.

Let us now prove (a). Assumptions $i)$, $ii)$, and $iii)$ 
ensure that Theorem \ref{Thm : 
Index meromorphic disk} holds and we have the index 
formula \eqref{eq : meromorphic index formula punctured disk-bis}. If 
$\Irr_{C}(\Fs)=0$, then $\chidr(D(*0),\Fc)=0$ and we 
can apply Lemma~\ref{Lemma : mero vs formal coh}. 
Item (a) follows.

Let us now prove point~(b). 
Assume that~$b=b_{0}$.   
By Lemma~\ref{lem:merosinglinear-1}, 
the radii of~$\Fs$ are log-affine on~$\Gamma_{C_{n}}$ 
for~$n$ big enough. It follows by $v)$ that 
$\Irr_{C}(\Fs)=\Irr_{C_n}(\Fs)=0$ for all $n\geq 1$. 
If $D_n:=C_n\cup\{0\}$, one sees that the assumptions 
$i)$ to $v)$ hold as well over $D_n$. Lemma 
\ref{Lemma : H^i K(T)=lim-2} then implies (b).
\end{proof}

\begin{remark}
Assume that $b=b_D$ and that $\Fc$ 
has the property that all its radii approach $1$ as 
$x$ approaches the open boundary $b_D$\footnote{This property is called 
\emph{solvability} in \cite{Ch-Me-III,Ch-Me-IV}.}. 
Then it is not hard to prove (using the concavity along 
$\Gamma_C$ of the partial heights of the convergence 
Newton polygon) that 
the condition $\Irr_C(\Fs)=0$ implies that the radii are log-affine along $\Gamma_C$. 

In Section \ref{Thm : Index meromorphic disk}, 
starting from a differential module 
over $\mathfrak{R}_{b_D}$, we will find a lattice $\Fc$ 
over $\O(D)[T^{-1}]$. It remains an open question to 
know whether one can find such a lattice satisfying 
$\Irr_{C}(\Fc)=0$.
\end{remark}

The following Corollary takes into account differential 
equations over $K(\{T\})$.

\begin{corollary}
\label{Cor : res=iso merodag anal-tre}
Let $\Fc_0^\dag$ be a differential equation over 
$K(\{T\})$ and let $\Fs_0^\dag:=
\Fc_0^\dag\otimes_{K(\{T\})}\mathfrak{R}_0$.
Let $D_1\supset D_2\supset\cdots$ be a sequence of 
disks with $\bigcap_nD_n=\{0\}$ and let 
$C_n:=D_n-\{0\}$. Let $b_{n}$ be the germ of 
segment at the open boundary of $D_n$ and $b_0$ be the 
germ of segment out of $0$.
Assume that $\Fc_0^\dag$ comes from an equation 
$\Fc$ over $D_1(*0)$.

Consider the following conditions.
\begin{enumerate}
\item There exists $n_0\geq 1$ 
such that, 
for each $n\geq n_0$, the cohomology groups 
$\Hdr^i(D_{n}(*0),\Fc_{|D_n})$ and 
$\Hdr^i(C_{n},\Fc_{|C_n})$ 
are finite-dimensional and one has $\chidr(D_n(*0),\Fc_{|D_n})=
\chidr(C_n,\Fs_{|C_n})=0$.
\item There exists $n_0'\geq 1$ such that, 
for each $n\geq n'_0$, $\Fc$ is Fredholm at both 
$b_{n}$ and~$b_0$ and it satisfies 
\eqref{eq : chiabs=Irr} on both germs of segments.
\end{enumerate}
Then, ii) implies i). Moreover, if i) holds, then the cohomology groups 
$\Hdr^i(K(\{T\}),\Fc_0^\dag)$ and 
$\Hdr^i(\mathfrak{R}_0,\Fs_0^\dag)$
are finite-dimensional and, for each $n\geq n_0$ and $i=0,1$, we have isomorphisms
\begin{equation}
\Hdr^i(K((T)),\M)\xleftarrow{\sim}
\Hdr^i(K(\{T\}),\Fc_0^\dag)\xleftarrow{\sim}
\Hdr^i(D_n(*0),\Fc_{|D_n})\xrightarrow{\sim}
\Hdr^i(C_n,\Fs_{|C_n})\xrightarrow{\sim}
\Hdr^i(\mathfrak{R}_0,\Fs_0^\dag).
\end{equation}
\end{corollary}
\begin{proof}
Assume that ii) holds. Then i) follows from Theorems~\ref{Thm : Index meromorphic disk} and~\ref{Thm : index of finite opens} together with Lemma~\ref{lem:merosinglinear-1} to ensure that the radii are log-affine on the skeleta for~$n$ big enough.

Assume that i) holds. The result now follows from Lemmas~\ref{Lemma : H^i K(T)=lim-2}, \ref{Lemma : injectiona and surjection res mero anal} and Corollary~\ref{Cor : zero index over the robba ring} (together with Lemma~\ref{lem:merosinglinear-1} again to ensure log-affinity.)
\end{proof}

\subsubsection{Descent of morphisms.}
We maintain the notations of Section 
\ref{section : meromorphic Settings punctured disk}.
Let $C'$ be as in  
\eqref{eq : C' mero anal section45}.
\begin{corollary}
Let $\Fc,\Gc$ be differential equations over $D(*0)$. Let 
\begin{equation}
\beta\;:\;\Fc_{|C'}\to\Gc_{|C'}
\end{equation}
be a morphism. Assume that 
\begin{enumerate}
\item $\Hom(\Fc,\Gc)$ is free as $\O_D[*0]$-module;
\item $\Hdr^1(D(*0),\Hom(\Fc,\Gc))$ and 
$\Hdr^1(C',\Hom(\Fc_{|C'},\Gc_{|C'}))$ are finite-dimensional;
\item $\chidr(D(*0),\Hom(\Fc,\Gc))=
\chidr(C',\Hom(\Fc_{|C'},\Gc_{|C'}))$.
\end{enumerate}
Then there exists a unique morphism $\alpha:\Fc\to\Gc$ 
over $D(*0)$ such that $\beta=\alpha\otimes 1$. 
Moreover, $\alpha$ is a monomorphism (resp.
epimorphism, isomorphism) if, and only if, so is 
$\beta$.
\end{corollary}
\begin{proof}
The proof is similar to that of Lemma \ref{Lemma: 
descent of morphisms-1} and by using Lemma 
\ref{Lemma : injectiona and surjection res mero anal}.
\if{
one has 
\begin{equation}
\Hom^\nabla(\Fc,\Gc)\;=\;
\Hdr^0(D(*0),\Hom(\Fc,\Gc))\;=\;
\Hdr^0(C',\Hom(\Fc_{|C'},\Gc_{|C'}))\;=\;
\Hom^\nabla(\Fc_{|C'},\Gc_{|C'})\;.
\end{equation}
Therefore, there exists $\alpha:\Fc\to\Gc$ such that 
$\alpha\otimes 1=\beta$. 

Denote by $\mathcal{K}$ and $\mathcal{C}$ the kernel 
and cokernel of $\alpha$ respectively. Let $C=D-\{0\}$. 
The terms of the exact sequence
$0\to \mathcal{K}_{|C}\to\Fc_{|C}\to
\Gc_{|C}\to\mathcal{C}_{|C}\to 0$ 
are free as $\O(C)$-modules. Therefore, 
$\alpha_{|C}$ is epi or mono if and only if so is 
$\beta=(\alpha_{|C})_{|C'}$.
Now, $\mathcal{K}$ and $\mathcal{C}$ are locally free 
$\O_D[*0]$-modules, 
therefore they are zero over $C$ if and only if 
they are zero on the whole $D(*0)$.
}\fi
\end{proof}

\if{
\begin{remark}
Notice that $\Fc$ and $\Gc$ are not assumed to be free 
$\O(D)[T^{-1}]$-modules in Lemma 
\ref{Lemma: descent of morphisms-1}. 
\end{remark}
}\fi
It follows from 
Theorem \ref{Theorem : restriction meromorphic}
that we have the following criterion.

\begin{corollary}
\label{Cor: descent of morphisms-3}
Let $\Fc$ and $\Gc$ be differential equations over 
$D(*0)$ and let 
$\beta:\Fc_{|C'}\to\Gc_{|C'}$ 
be a morphism. Assume that 
\begin{enumerate}
\item $\Hom(\Fc,\Gc)$ is a free 
$\O_D[*0]$-module;
\item the radii of $\Hom(\Fc,\Gc)$ are all 
$\log$-affine at the open boundary $b_D$ of $D$;
\item  $\Hom(\Fc,\Gc)$ is Fredholm and satisfies 
\eqref{eq : chiabs=Irr} at $b_D$;
\item  $\Hom(\Fc,\Gc)$ is Fredholm at $b_{0,C'}$ 
and at $b_{1,C'}$ and 
\eqref{eq:chiabs01} holds over $C'$;
\item $\Irr_C(\Hom(\Fc,\Gc))=
\Irr_{C'}(\Hom(\Fc_{|C'},\Gc_{|C'}))$.
\end{enumerate}
Then, there exists a unique morphism of differential 
equations
$\alpha:\Fc\to\Gc$ such that 
$\beta=\alpha\otimes 1$. 
Moreover, $\alpha$ is a monomorphism (resp.
epimorphism, isomorphism) if, and only if, so is 
$\beta$.\hfill$\Box$
\end{corollary}

\subsubsection{The restriction functors.}
In this section we compare the categories of differential 
equations over $D(*0)$, $K(\{T\})$, $K((T))$ and 
$\mathfrak{R}_0$ under some conditions on the indexes.

Let $D$ be an open disk centered at $0$, $C:=D-\{0\}$ 
and let $b$ be a germ of segment in $\Gamma_C$.

Let $\mathcal{A}$ be a full sub-category of the category 
of differential equations over $D(*0)$ with the property 
that for all $\Fc,\Gc\in\mathcal{A}$ the differential 
equation $\mathcal{H}:=\mathrm{Hom}(\Fc,\Gc)$ 
satisfies items i), ii), iii) and iv) of Corollary 
\ref{Comparison Rigid formal-1} 
(with respect to an 
unspecified sequence of open pseudo-annuli $\{C_n\}_n$). 

Denote by $\mathcal{A}'$ a full sub-category of 
$\mathcal{A}$ with the property that for all 
$\Fc,\Gc\in\mathcal{A}'$ the equation $\mathcal{H}$ 
satisfies moreover $\chidr(D(*0),\mathcal{H})=0$.

\begin{corollary}\label{Cor : equivalente D(*0)-Robba}
The restriction functor  
\begin{equation}
Res^{D(*0)}_{\mathfrak{R}_b}\;:\;
\mathcal{A}
\;\xrightarrow{\quad}\;
d-\mathrm{Mod}(\mathfrak{R}_b)\;.
\end{equation}
is fully faithful. 
Moreover, for $\Fc,\Gc\in\mathcal{A}$, 
if we denote by $\Fc_b=\Fc\otimes \mathfrak{R}_b$ and 
$\Gc_b=\Gc\otimes \mathfrak{R}_b$ 
their images in 
$d-\mathrm{Mod}(\mathfrak{R}_b)$, we 
have an isomorphism 
of Yoneda extension groups
\begin{equation}\label{eq : stability by ext mero-form-1}
\mathrm{Ext}_{d-\mathrm{Mod}(\O(D)[T^{-1}])}(\Fc,
\Gc)
\;\xrightarrow{\;\sim\;}\;
\mathrm{Ext}_{d-\mathrm{Mod}(\mathfrak{R}_b)}(\Fc_b,\Gc_b)\;.
\end{equation}

In particular, 
if $\mathcal{A}$ 
is stable by extensions in 
$d-\mathrm{Mod}(\O(D)[T^{-1}])$, 
then so is its essential image in 
$d-\mathrm{Mod}(\mathfrak{R}_b)$. 

Moreover the restriction functors
\begin{equation}
d-\mathrm{Mod}(K((T)))\;\xleftarrow{\;\;\;\;}\;
\mathcal{A}'\;\xrightarrow{\;\;\;\;}\;
d-\mathrm{Mod}(\mathfrak{R}_b)
\end{equation}
are both fully faithful.
\end{corollary}
\begin{proof}
The proof is similar to that of Corollary 
\ref{Cor : equivalente D(*0)-form} by using 
Corollary \ref{Comparison Rigid formal-1} instead of Lemma~\ref{Lemma : injectiona and surjection res mero form}.
\end{proof}

Let $\mathcal{D}$ be a full sub-category of 
$d-\mathrm{Mod}(K(\{T\}))$ satisfying the property 
that if $\Fc_0^\dag,\Gc_0^\dag\in\mathcal{D}$, 
then there exists an 
unspecified sequence of open disks 
$D_1\supset D_2\supset\cdots$, with
$\bigcap_nD_n=\{0\}$, such that if $C_n:=D_n-\{0\}$ 
and if the differential equation 
$\mathrm{Hom}(\Fc_0^\dag,\Gc_0^\dag)$ comes from 
a differential equation~$\mathcal{H}$ over $D_1(*0)$, 
then for all $n$ the cohomology groups 
$\Hdr^i(D_{n}(*0),\mathcal{H}_{|D_n})$ and 
$\Hdr^i(C_{n},\mathcal{H}_{|C_n})$ 
are finite-dimensional and  one has 
\begin{equation}
\chidr(D_n(*0),\mathcal{H}_{|D_n})\;=\;
\chidr(C_n,\mathcal{H}_{|C_n})\;=\;0\;.
\end{equation}
For a differential equation $\Fc_0^\dag$ over $\O(D)[T^{-1}]$ we set as usual 
$F:=\Fc_0^\dag\otimes K((T))$, 
$\Fs_0^\dag=\Fc_0^\dag\otimes\mathfrak{R}_0$. 

\begin{corollary}\label{Cor : equivalente mero-AN}
The scalar-extension functors  
\begin{equation}
d-\mathrm{Mod}(K((T)))\;\xleftarrow{\quad}\;
\mathcal{D}
\;\xrightarrow{\quad}\;
d-\mathrm{Mod}(\mathfrak{R}_0)\;.
\end{equation}
are fully faithful.
Moreover, for $\Fc,\Gc\in\mathcal{D}$, we have an 
identity of Yoneda extension groups
\begin{equation}\label{eq : stability by ext mero-form}
\mathrm{Ext}_{d-\mathrm{Mod}(K((T)))}
(F,G)
\;\xleftarrow{\;\sim\;}\;
\mathrm{Ext}_{d-\mathrm{Mod}(K(\{T\}))}
(\Fc_0^\dag,
\Gc_0^\dag)
\;\xrightarrow{\;\sim\;}\;
\mathrm{Ext}_{d-\mathrm{Mod}(\mathfrak{R}_0)}
(\Fs_0^\dag,\Fs_0^\dag)\;.
\end{equation}
\end{corollary}
\begin{proof}
The result follows from 
Corollary~\ref{Cor : res=iso merodag anal-tre}.
\end{proof}

\subsection{Formal vs. convergent decompositions.}
\label{section : Formal VS Convergent decompositions.}
To any differential module over~$K(\{T\})$, we can 
associate two Newton polygons: the formal one, attached 
to its image in $d-\mathrm{Mod}(K((T)))$, and the 
convergence one, attached to its image in 
$d-\mathrm{Mod}(\mathfrak{R}_0)$. These polygons 
are related by the fact that the formal one is 
(up to a transformation) the 
derivative of the convergence one (cf. 
Proposition~\ref{Prop : Irr-form=Irr-x-1}). 
This does not depend on any  
assumption at the boundary. 

In this section we show the relations between these two 
decompositions under some assumptions.
Namely, we maintain the assumptions of Section 
\ref{section : meromorphic Settings punctured disk}
and assume moreover that there exists a disk $D'$ containing $0$ 
and an isomorphism of $\O(D')[T^{-1}]$-differential 
modules 
\begin{equation}\label{eq : descend Can hyp-section 4}
\Fc_{|D'}\;\xrightarrow{\;\sim\;}\;
\mathrm{Can}(\M)_{|D'}\;,
\end{equation}
where $\mathrm{Can}(\M)$ 
denotes Katz's \emph{canonical extension} of $\M$.
For instance, by Lemma \ref{Lemma: descent of morphisms-1}, we may assume that 
$\mathrm{Hom}(\Fc^\dag_0,
\mathrm{Can}(\M)^\dag_0)$ 
has finite-dimensional de Rham 
cohomology over $K(\{T\})$ and that
\begin{equation}
\chidr(K(\{T\}),\mathrm{Hom}(\Fc^\dag_0,
\mathrm{Can}(\M)^\dag_0))\;=\;0\;.
\end{equation} 


When~$K$ is trivially valued, we have 
$K((T))=K(\{T\})=\mathfrak{R}_0$ and 
$\M=\Fc_0^\dag=\Fs_0^\dag$. 
In this situation, by \cite[Section 5.7]{NP-III}, the 
decomposition of~$\M$ by the slopes of the formal 
polygon coincides with its decomposition by the radii 
(\textit{i.e.} by the slopes of its convergence polygon). 


Since $\mathrm{Can}$ is a functor, the decomposition of 
$\M$ by the slopes of the \emph{formal} Newton 
polygon gives a decomposition of $\Fc_0^\dag$, that 
we will call \emph{the formal decomposition of 
$\Fc_0^\dag$}.
Corollary \ref{Cor : Formal VS conv deco} below 
compares the formal decomposition of $\Fc_0^\dag$ 
over $K(\{T\})$ with the decomposition by the radii of 
$\Fs^\dag_0$ over $\mathfrak{R}_0$ (with no 
assumptions on the valuation of~$K$).

\begin{definition}
We endow  
$\mathbb{Q}\times\mathbb{R}_{>0}$ with the 
lexicographic order: 
$(s_1,\alpha_1)\leq (s_2,\alpha_2)$ if  
$s_1<s_2$, or 
$s_1=s_2 \textrm{ and }\alpha_1\leq \alpha_2$.

\end{definition}

For all $\rho$ let 
$x_{0,\rho}$ be the Berkovich point at the boundary of 
the disk $\{|T|\leq \rho\}$. 
The differential module~$\Fc_{0}^\dag$ comes from a 
differential module~$\Fc$ over $\Os(D)[T^{-1}]$ for 
some open disk~$D$ centered at~0. By 
Lemma~\ref{lem:merosinglinear-1} and 
\cite{NP-III}, 
there exists $\eps \in \ERRE_{>0}$ and, for every $i=1,
\ldots,r=\mathrm{rank}(\Fc_0^\dag)$, $(s_{i},\alpha_i) 
\in \QQ \times \ERRE_{>0}$ such that 
the $i$-th radius of $\Fs_0^\dag$ 
along $]0,x_{0,\varepsilon}[$ is
\begin{equation}
\R_{i}(\Fs_0^\dag,x_{0,\rho})\;=\; 
\alpha_i\cdot\rho^{s_i}\;,\qquad \forall\;\rho\in
]0,\varepsilon[\;.
\end{equation}
Notice that, since by definition one has
$\R_{1}(\Fs_0^\dag,x_{0,\rho})\leq
\R_{2}(\Fs_0^\dag,x_{0,\rho})\leq
\cdots\leq\R_{r}(\Fs_0^\dag,x_{0,\rho})$, 
where $r$ is the rank of $\Fc_0^\dag$, then
we must have 
\begin{equation}
s_1\;\leq\; s_2\;\leq\;\cdots\;\leq\; s_r\;.
\end{equation}
The decomposition of $\Fs_0^\dag$ by the radii
over $\mathfrak{R}_0$ can be written as follows: 
\begin{equation}\label{eq : deco conv around 0}
\Fs_0^\dag\;=\;
\bigoplus_{(s,\alpha)\in
\mathbb{Q}\times \mathbb{R}_{>0}}
\Fs_0^\dag(s,\alpha)\;,
\end{equation}
where all the radii of $\Fs(s,\alpha)$ are all equal to 
$\alpha\cdot \rho^s$ along $]0,x_{0,\varepsilon}[$.

\begin{definition}
Set
\begin{equation}
\Fs_0^\dag(s)\;:=\;\bigoplus_{s'=s}\Fs_0^\dag(s',\alpha')\;,
\end{equation}
where the sum runs on all the factors of 
\eqref{eq : deco conv around 0} such that $s'=s$.
The decomposition
\begin{equation}
\label{eq : deco by the derived convergence polygon}
\Fs_0^\dag\;=\;
\bigoplus_{s\in\mathbb{Q}}\Fs_0^\dag(s)\;,
\end{equation}
is called decomposition of $\Fs_0^\dag$ 
\emph{by the derivative of it convergence Newton 
polygon}.
\end{definition}
\begin{remark}\label{Remark : K trval deco M}
Assume that $K$ is trivially valued. Then 
$K((T))=\mathfrak{R}_0$ and we have the above 
decomposition also for $\M$. In this case, it follows 
from an easy computation \cite[Section 5.7]{NP-III} 
that $\alpha_i=1$ for all $i=1,\ldots,r$. 
Hence, there is no distinction between the 
decomposition \eqref{eq : deco conv around 0} 
of $\M$ by the radii 
(i.e. by the slopes of 
its convergence polygon) and its decomposition 
\eqref{eq : deco by the derived convergence polygon} 
by the derivative of its convergence Newton polygon. 
Both coincide with its formal decomposition, \textit{i.e.} 
the decomposition by the slopes of its formal Newton 
polygon (which is indeed, up to a 
transformation, the derivative of the 
convergence Newton polygon by 
Proposition~\ref{Prop : Irr-form=Irr-x-1}). 
\end{remark}

\begin{corollary}\label{Cor : Formal VS conv deco}

The decomposition 
\eqref{eq : deco by the derived convergence polygon} 
of $\Fs_0^\dag$ by the derived convergence Newton 
polygon coincides with the decomposition of 
$\Fs_0^\dag$ induced by that of 
$\Fc_0^\dag$ by the formal Newton polygon.
\if{The following properties hold:
\begin{enumerate}
\item The decomposition of $\M$ by the slopes of its 
formal Newton polygon descends into a decomposition 
of $\Fc_0^\dag$ 
(cf. \eqref{eq : descend Can hyp-section 4}).
\item The decomposition \eqref{eq : deco by the derived convergence polygon} of $\Fs_0^\dag$ by the 
derived convergence Newton polygon descends 
into a decomposition of $\Fc_0^\dag$.
\item These two decompositions of $\Fc_{0}^\dag$ 
coincide.
\end{enumerate}
}\fi
\end{corollary}
\begin{proof}
The claim follows immediately from the fact that 
the derivative of the convergence Newton polygon of 
$\Fs_0^\dag$ coincides with the formal Newton polygon 
of $\M$ (up to a transformation). See Proposition 
Proposition~\ref{Prop : Irr-form=Irr-x-1}.
\end{proof}

\appendix

\section{Local Liouville conditions.}  
\label{Liouville condition}  

%

In the literature, exponents are defined only over standard 
annulus and differential modules that are free over $\O$. 
In this section, we extend the definition to open 
pseudo-annuli and germ of segments, and we allow non-freeness. Moreover, we allow differential equations with 
non-affine radii.
We also recall and adapt to our setting some classical 
definitions and results that are mainly due to 
Christol-Mebkhout \cite{Ch-Me-II}, 
Dwork \cite{Dwork-Exponents}, \cite{DGS} 
and Kedlaya \cite{Kedlaya-draft}.

\subsection{Exponents}
\subsubsection{Oriented pseudo-annuli}

\begin{definition}
Let~$C$ be an open pseudo-annulus. An 
\emph{orientation} of~$C$ is the datum of 
a germ of segment~$b$ in the open 
boundary of~$C$. Once an orientation~$b$ is chosen, 
we sometimes call \emph{upper germ} (resp. 
\emph{lower germ}) of~$C$ the germ~$b$ (resp. the 
germ in the open boundary of~$C$ that is not~$b$).

The pair $(C,b)$ is called an \emph{oriented open 
pseudo-annulus}.
\end{definition}

It is convenient to fix a choice of orientation when we have an annulus embedded in the affine line.

\begin{definition}\label{Def : standard orientation}
Fix a coordinate~$T$ on $\mathbb{A}^{1,\an}_{K}$. Let $C = \{r_{1} < |T| < r_{2}\}$, with $0\le r_{1} < r_{2} \le +\infty$, be a standard open pseudo-annulus. The \emph{standard orientation} of~$C$ is the germ $b \in \partial^o C$ pointing away from~$\infty$, \emph{i.e.} the germ represented by $]x_{r_{2}},x_{r_{2}-\eps}[$ for $\eps$ small enough.

\end{definition}

Note that an isomorphism of open pseudo-annuli induces a bijection between their open boundaries.

\begin{definition}
\label{Definition oriented pseudo-annuli preserved}
Let~$(C,b)$ and $(C',b')$ be oriented open pseudo-annuli and let $f : C \to C'$ be an isomorphism. We say that~$f$ \emph{preserves the orientation} if it sends~$b$ to~$b'$ and that~$f$ \emph{reverses the orientation} otherwise.
\end{definition}

Note that, if~$f$ is an automorphism of an open pseudo-annulus~$C$, then it preserves the orientation for some choice of orientation if it does for the other. Therefore, we can speak about orientation preserving automorphisms without actually choosing an orientation.

\begin{remark}
Let~$(C,b)$ be an oriented open pseudo-annulus. 

Let~$C'$ be an open sub-pseudo-annulus of~$C$ such that $\Gamma_{C'} \subseteq \Gamma_{C}$. There exists a unique germ~$b'$ in the open boundary of~$C'$ pointing in the same direction as the germ~$b$. We call it the induced orientation on~$C'$.

Let~$L$ be a complete valued extension of~$K$. Let~$C''$ be a connected component of~$\pi^{-1}_{L/K}(C)$. It is an open pseudo-annulus whose open boundary contains precisely one germ~$b''$ above~$b$. We call it the induced orientation on~$C''$.
\end{remark}

The following statement ensures that the 
isomorphism class of a Robba module is stable by 
infinitesimal isomorphisms. 
This is an example of infinitesimal deformation.

\begin{proposition}[\protect{\cite[Theorem 4.3.1]{Inf-Def}}]\label{prop:deformation}
Let $C$ be an open pseudo-annulus. 
Let $\Fs$ be a finite differential equation over $C$ of 
Robba type (i.e. $\Fs=\Fs^{\mathrm{Robba}}$, 
cf. Definition \ref{Def.: Rrobba part}). 
Let $\sigma:C\simto C$ be a 
$K$-automorphism of $C$. Assume that, for each complete valued extension~$L$ of~$K$, each connected component of $C_{L} - \Gamma_{C_{L}}$ is globally fixed by~$\sigma_{L}$.
Then, we have an isomorphism of differential equations 
$\sigma^*\Fs\cong\Fs$.
\end{proposition}
\begin{proof}
By definition, such a $\sigma$ is 
an \emph{infinitesimal automorphism} with respect to 
the empty pseudo-triangulation of~$C$ in the sense of 
\cite[Definition 3.0.2]{Inf-Def}. 
Moreover, since $\Fs=\Fs^{\mathrm{Robba}}$, it 
trivially satisfies the $\sigma$-compatibility condition 
\cite[Section 4.2]{Inf-Def}. 
Hence, it follows from \cite[Theorem 4.3.1]{Inf-Def} that 
we have an isomorphism $\sigma^*(\Fs)\cong\Fs$. 
\end{proof}

\begin{remark}\label{rem:deformation}
Assume that~$C$ is a standard open pseudo-annulus: $C = \{r< |T|<s\}$. Then, the condition of the proposition is satisfied if, and only if, $\sigma$ is of the form $T \mapsto T+h(T)$ with $|h(x_{\rho})| < \rho$ for each $\rho \in (r,s)$.

In general, let~$\Omega$ be a spherically complete valued extension of~$K$ that is algebraically closed and with value group~$\ERRE_{+}$. Let $C_{1},\dotsc,C_{n}$ be the connected components of~$C_{\Omega}$. They are standard open annuli and the condition of the proposition is satisfied if, and only if, for each $i\in\{1,\dotsc,n\}$, $(\sigma_{\Omega})_{|C_{i}}$ is of the form above.
\end{remark}

\subsubsection{Type of a number,
Liouville numbers.}


We denote by $|.|$ the absolute value of $K$, and by 
$x_\rho$ the point at the boundary of the disk 
$\{|T|\leq \rho\}$. 
We consider the open pseudo-annulus 
\begin{equation}\label{eq : C notation liouville}
C\;:=\;\{r_1<|T|<r_2\}\;,\qquad r_{1},r_{2} \in [0,+\infty]\;,\qquad r_1<r_2\;.
\end{equation}
Denote by $b_1$ and $b_2$ the germs of segments at 
the open boundary of $C$ represented, 
for $\eps$ small enough, by 
$]x_{r_1},x_{r_1+\eps}[$ and 
$]x_{r_2-\eps},x_{r_2}[$ respectively.

\begin{definition}\label{Def : type_b}
Let $e\in K$. We set
\begin{eqnarray}
\mathrm{type}_{b_2}(e,|.|)&\;:=\;&
\liminf_{n\to+\infty}|e+n|^{1/n}\;,\;\\
\mathrm{type}_{b_1}(e,|.|)&\;:=\;&
\liminf_{n\to+\infty}|e-n|^{1/n}\;,\;\\
\mathrm{type}(e,|.|)&\;:=\;&\min\Bigl(
\;\mathrm{type}_{b_1}(e,|.|)\;,
\;\mathrm{type}_{b_2}(e,|.|)\;\Bigr)\;.\label{annexe-liouvile (C.4)}
\end{eqnarray}
For all germ of segment $b$ in $\Gamma_C$ that is 
oriented as $b_i$ we set
\begin{equation}
\mathrm{type}_b(e,|.|)\;:=\;
\mathrm{type}_{b_i}(e,|.|)\;.
\end{equation}
We often write $\mathrm{type}_{b}(e)$ or 
$\mathrm{type}(e)$ if no confusion is 
possible.
\end{definition}

\begin{lemma}
\label{Lemma : Liouville are Transcendental}
Let $b$ be a germ of segment in $\Gamma_C$. 
The following properties hold.
\begin{enumerate}
\item For all $e\in K$, one has 
$\mathrm{type}_{b_1}(e,|.|)=
\mathrm{type}_{b_2}(-e,|.|)$, and 
$\mathrm{type}(e,|.|)=
\mathrm{type}(-e,|.|)$. 
\item For all $e\in K$ and $n\in\mathbb{Z}$, one has 
$\mathrm{type}_b(e+n,|.|)=\mathrm{type}_b(e,|.|)$.
\item If $\psi:C\simto C$ is an isomorphism then
\begin{equation}
\mathrm{type}_{\psi(b)}(e)\;=\;\left\{
\begin{array}{ll}
\mathrm{type}_{b}(e)&\textrm{ if $\psi$ preserves 
the orientation of $C$,}\smallskip\\
\mathrm{type}_{b}(-e)&\textrm{ if $\psi$ reverses 
the orientation of $C$.}
\end{array}
\right.
\end{equation}
\item For all $e\in K$, one has $\mathrm{type}(e,|.|)
\leq 1$. 
\item If the restriction of $|.|$ to $\mathbb{Z}$ is the 
trivial absolute value, then $\mathrm{type}(e,|.|)=1$ 
for all $e\in K$.
\item If the restriction of $|.|$ to $\mathbb{Z}$ is a 
$p$-adic absolute value and if we have either 
$e\notin\mathbb{Z}_p$ or $e\in 
\mathbb{Z}_p\cap
\mathbb{Q}^{\mathrm{alg}}$, then 
$\mathrm{type}(e,|.|)=1$.
\end{enumerate}
\end{lemma}
\begin{proof}
The evoked equalities in i), ii) and iii) 
follow easily from the definition.

iv) Let $e\in K$. For all $n\in\Z$, we have $|e\pm n|^{1/n}\leq \max(1,|e|)^{1/n}$. We deduce that 
$\mathrm{type}(e)\leq 1$. 

v) Assume that $|.|$ is trivial on $\mathbb{Z}$. In particular, $\mathbb{Z}$ is complete, hence closed in~$K$. If $e\in \mathbb{Z}$, then the claim is clear. If $e\notin \mathbb{Z}$, then its distance $d:=\inf_{n\in\mathbb{Z}} (|e-n|)$ to~$\mathbb{Z}$ is non-zero. It follows that $\mathrm{type}(e,|.|)\ge 1$, hence $\mathrm{type}(e,|.|)=1$, by iv).

vi) Assume that $|.|$ is $p$-adic on $\mathbb{Z}$. For all $e\notin\mathbb{Z}_p$, we have $d>0$, hence $\mathrm{type}(e)=1$, as before. For $e\in\mathbb{Z}_p\cap\mathbb{Q}^{\mathrm{alg}}$, a proof is given in
\cite[Prop. 11.3.4]{Ch-Ro} or 
\cite[Ch.VI, Prop. 1.1]{DGS}. 
%
%
\end{proof}

\begin{definition}[Liouville numbers]
Let $e\in K$. We say that $e$ is Liouville with respect to $|.|$ 
if $\mathrm{type}(e)<1$.
If $\mathrm{type}(e)=1$, 
we say that $e$ is non-Liouville with respect to $|.|$.
\end{definition}


Lemma \ref{Lemma : Liouville are Transcendental} 
shows that Liouville numbers only arise in a $p$-adic 
context and that they are transcendental  
numbers in $\mathbb{Z}_p$. 

\if{It also  
allows to extend the definition of 
$\mathrm{type}_b$ to any oriented good germ of 
segments in $X$.
\begin{definition}
Let $e\in K$. Let $b$ be a good germ of segment in $X$ 
represented by a pseudo-annulus $C\subseteq X$. 
Let~$\Omega$ be a spherically complete and 
algebraically closed field extension of~$K$ such that 
$|\Omega|=\mathbb{R}_{\geq 0}$. 
By~\cite[Proposition~3.2]{Liu}), $C_\Omega$ may 
be identified with an analytic domain of 
$\mathbb{P}^{1,\mathrm{an}}_\Omega$. 
It is hence a finite disjoint union 
$C_\Omega=C_1\sqcup\ldots\sqcup C_n$ of standard 
open pseudo-annuli over~$\Omega$ (see Definition 
\ref{Def : standard pseudo-annulus}). 
Let $b_1,\ldots,b_n$ be the inverse images of $b$ in 
$C_\Omega$. We chose a coordinate function 
of each $C_i$ that identifies $C_i$ with a standard  $b_i$ 
with a germ of segment in the affine line that is oriented 
as $b_2$.
\end{definition}
}\fi
\subsubsection{Modules of type $\Ns(e)$.}
\label{Section : N(e)}

\if{\begin{proof}
\comment{Je change les $K$ en $L$ pour la coh\'erence. L'ancienne version est en commentaire en dessous.}

If $L$ is trivially valued, then $\mathrm{type}(e)=1$ and~$\O(C)$ is one of the rings $L((T))$, 
$L[T,T^{-1}]$ or $L((T^{-1}))$. The claim then 
follows from a direct computation (similar to 
\cite[Théorème 11.3.2]{Ch-Ro}).

Assume that $L$ is not trivially valued. 
By Theorem~\ref{thm:descent}, 
we can assume that $L$ is algebraically closed and spherically 
complete and that $|L|=\mathbb{R}$. 
Set $L_e:=T\frac{d}{dT}-e= 
T\circ(\frac{d}{dT}-\frac{e}{T})$. 
Since the multiplication by $T$ is invertible in 
$\O(C)$, 
$\frac{d}{dT}-\frac{e}{T}:\O(C)\to\O(C)$ 
has finite index if, and only if, 
so has $L_e:\O(C)\to\O(C)$. 
Now, by Proposition \ref{Prop : truncation compact}, 
any differential operator has the compactness property 
of Definition \ref{Def. Compactness property}, so, by 
Proposition \ref{Prop : chi=sumchigen}, $L_e$ 
has finite index if, and only if, it has finite generalized 
indexes on $\O(D_0)$ and $T^{-1}\O(D_1)$ 
(cf. notations as in \eqref{eq : deco annulus ML}). In 
other words, $L_e$ has finite index if and only if 
the indexes of the truncated operators 
$p_k\circ L_e\circ i_k$ (cf. \eqref{eq : u_k (def)}) are finite. 
Since $L_e:\O(C)\to\O(C)$ stabilizes 
$\O(D_0)$ and $T^{-1}\O(D_1)$, the truncated 
operators $p_k\circ L_e\circ i_k$ coincide 
with the restrictions $(L_e)_{|\O(D_0)}$ and 
$(L_e)_{|T^{-1}\O(D_1)}$. 
Now, a classical direct computation 
(see \cite[Théorème 11.3.2]{Ch-Ro} or 
\cite[Section 4.19]{Ro-I}) 
shows that the indexes of $(L_e)_{|\O(D_0)}$ and 
$(L_e)_{|T^{-1}\O(D_1)}$ are finite if, and only if, 
$\mathrm{type}(e)=1$.

%
\end{proof}
}\fi
We maintain the notation 
\eqref{eq : C notation liouville}. 
For $e\in K$, we denote by 
\begin{equation}\label{eq : N(e)}
\Ns(e)
\end{equation}
the rank one differential module associated with the 
differential equation 
$\frac{d}{dT}(y)=\frac{e}{T}\cdot y$ over~$C$. 
Clearly, we have $\Ns(e)\otimes\Ns(e')=\Ns(e+e')$,
$\Ns(e)^*=\Ns(-e)$ and $\Ns(e+n) \simeq \Ns(e)$ for 
all $n\in\mathbb{Z}$. 

\begin{notation}
We set 
\begin{equation}
E(K) = \begin{cases}
K^\circ \textrm{ if } \mathrm{char}(\tilde{K}) = 0;\\
\Z_{p} \textrm{ if } \mathrm{char}(\tilde{K}) = p>0.
\end{cases}
\end{equation}
\end{notation}

\begin{lemma}\label{lem:NeRobba}
Let $e\in K$. Then, the radius of convergence function 
$x\mapsto\R_{1}(x,\Ns(e))$
of~$\Ns(e)$ is constant on $C$. 

It is identically equal to~1 on~$\Gamma_{C}$ (i.e. 
$\Ns(e) = \Ns(e)^\mathrm{Robba}$) if, and only if 
$e\in E(K)$.

Moreover, in this case, if~$f$ is a $K$-linear automorphism of~$C$, we have $f^*(\Ns(e)) \simeq \Ns(e)$ if~$f$ preserves the orientation and $f^*(\Ns(e)) \simeq \Ns(-e)$ otherwise.
\end{lemma}
\begin{proof}
The Taylor solution of the equation 
$\frac{d}{dT}(Y)=\frac{e}{T}\cdot Y$ at the generic 
point $t_x$ is $Y(T,t_x)=
\sum_{s\geq 0}\tbinom{e}{s}t_x^{-s}
(T-t_x)^s$, 
where 
$\tbinom{e}{s}=\frac{e(e-1)(e-2)\cdots(e-s+1)}{s!}$. 
It follows that 
$\R_{1}(x,\Ns(e))=
\liminf_s\tbinom{e}{s}^{-1/s}$, which is a constant 
function on $C$.

Now, the exact computation of the radius is estimated in 
\cite[Chapter IV, Proposition 7.3]{DGS} in the case of 
positive residual characteristic (and it was known since 
Robba \cite{RoIV}).

If the residual characteristic is $0$ we may argue as 
follows. The absolute value on induced on $\mathbb{Z}$ 
by that of $K$ is trivial. Therefore, if $|e|>1$, then 
$|\tbinom{e}{s}|=|e|^s$ and the radius equals 
$|e|^{-1}<1$. On the other hand, if $|e|\leq 1$, 
then $e$ belongs at most to an individual 
unit open disk centered at an integer. 
Therefore, $|\tbinom{e}{s}|$ is eventually constant and 
the radius is $1$.

Let us prove the last part of the statement. Let~$f$ be a 
$K$-linear automorphism of~$C$. 

Assume that $f$ preserves the orientation.
We can write $f(T)=q(T+h(T))$ with $q\in K$, $|q|=1$, 
$h\in\O(C)$ and $|h(x_\rho)|<\rho$ for all 
$\rho\in]r_1,r_2[$. Therefore $f$ is the composition of 
$f_1(T)=qT$ and $f_2:=f\circ f_1^{-1}$. It is easy to show directly that we have
$f_{1}^*\Ns(e)\cong\Ns(e)$. By Remark~\ref{rem:deformation}, the automorphism~$f_{2}$ satisfies the assumptions of 
Proposition~\ref{prop:deformation} 
(i.e. it is an infinitesimal automorphism 
\cite[Definition 3.0.2]{Inf-Def}), hence we have $f_{2}^*\Ns(e)\cong\Ns(e)$. The result follows.


Now assume that $f$ reverses the orientation. Then there exists a $K$-linear automorphism~$g$ of~$C$ of the form $g : T \mapsto a T^{-1}$ for some $a\in K$ such that $g \circ f$ preserves the orientation. For the previous case, we have $(g\circ f)^*(\Ns(e)) \simeq \Ns(e)$. A direct computation shows that we have $(g^{-1})^*(\Ns(e)) \simeq \Ns(-e)$ and the result follows.
\end{proof}

%
%
%
%
%
%
%
%

\begin{lemma}\label{Lemma : Liouville iff index O}
\label{Lemma : H^i N(e)}
Let $e\in K$ and $b$ be a germ of segment in 
$\Gamma_C$. Then
\begin{enumerate}
\item $\Ns(e)$ is Fredholm at $b$ (cf. Definition 
\ref{def:Fredholmnabla})
if and only if 
$\mathrm{type}_b(e)=1$. 
In this case, we have $\chiabs_b(\Ns(e))=0$;
\item $\Ns(e)$ has 
finite dimensional de Rham cohomology over $C$ 
if, and only if, 
$\mathrm{type}(e)=1$. 
In this case, we have $\chidr(C,\Ns(e))=0$.
\item One has
\begin{equation}
\mathrm{dim}\;
\Hdr^0(C,\Ns(e))
\;=\;\left\{
\begin{array}{ll}
1&\textrm{if $e\in\mathbb{Z}$}
\;;\\
0&\textrm{if $e\notin\mathbb{Z}$}\;.
\end{array}
\right.
\end{equation} 
and
\begin{equation}
\qquad\qquad\qquad\qquad\mathrm{dim}\;
\Hdr^1(C,\Ns(e))\;=\;\left\{
\begin{array}{ll}
1&\textrm{if \;$e\in\mathbb{Z}$}\;;\\
0&\textrm{if \;$e\notin\mathbb{Z}$ and if $e$ is non-Liouville}\;;\\
+\infty&\textrm{if $e$ is Liouville}\;.
\end{array}
\right.
\end{equation}
\end{enumerate}
\end{lemma}
\begin{proof}
Items i) and ii) result from a direct computation. 
See \cite[Théorème 11.3.2]{Ch-Ro} or 
\cite[Section 4.19]{Ro-I}. Item iii) follows from ii) and 
from straightforward computations.
\end{proof}

\if{\begin{lemma}
Let $e\in K$. Then, we have
\begin{equation}
\mathrm{dim}\;
\Hdr^0(C,\Ns(e))
\;=\;\left\{
\begin{array}{ll}
1&\textrm{if $e\in\mathbb{Z}$}
\;;\\
0&\textrm{if $e\notin\mathbb{Z}$}
\end{array}
\right.
\end{equation}
and 
\begin{equation}
\qquad\qquad\qquad\qquad\mathrm{dim}\;
\Hdr^1(C,\Ns(e))\;=\;\left\{
\begin{array}{ll}
1&\textrm{if \;$e\in\mathbb{Z}$}\;;\\
0&\textrm{if \;$e\notin\mathbb{Z}$ and if $e$ is non-Liouville}\;;\\
+\infty&\textrm{if $e$ is Liouville}\;.
\end{array}
\right.\qquad\qquad\Box
\end{equation}
\end{lemma}
\begin{proof}
It is not difficult to check that the differential equation~$\Ns(e)$ admits a non-zero global solution if, and only if, $e\in\mathbb{Z}$. As a consequence, if $e\in\mathbb{Z}$, then $\Ns(e)$ is trivial and the claim holds. 

If $e\notin\mathbb{Z}$, then we have $\Hdr^0(C,\Ns(e))=0$. The rest of the result follows from a standard computation. It could also be deduced from Lemma~\ref{Lemma : Liouville iff index O}.

\comment{J'ai modifi\'e la r\'edaction.}


\end{proof}
}\fi

\begin{lemma}\label{Lemma : Hom and Ext of N(e)}
Let $e,e'\in K$. 
Denote by 
$\mathrm{Hom}^{\nabla}(\Ns(e),\Ns(e'))$ the 
space of $\O(C)$-linear homomorphisms commuting 
with the connections and by $\mathrm{Ext}^{1}(\Ns(e),\Ns(e'))$ 
Yoneda's group of extensions 
(whose elements are equivalence 
classes of exact sequences 
$0\to \Ns(e')\to E\to \Ns(e)\to0$). Then, we have
\begin{equation}
\mathrm{dim}\;
\mathrm{Hom}^{\nabla}(\Ns(e),\Ns(e'))
\;=\;\left\{
\begin{array}{ll}
1&\textrm{if $e'-e\in\mathbb{Z}$}
\;;\\
0&\textrm{if $e'-e\notin\mathbb{Z}$}
\end{array}
\right.
\end{equation}
and 
\begin{equation}
\mathrm{dim}\;
\mathrm{Ext}^{1}(\Ns(e),\Ns(e'))
\;=\;
\left\{
\begin{array}{ll}
1&\textrm{if \;$e'-e\in\mathbb{Z}$}\;;\\
0&\textrm{if \;$e'-e\notin\mathbb{Z}$, and if $e'-e$ is non-Liouville}\;;\\
+\infty&\textrm{if \;$e'-e$ is Liouville}\;.
\end{array}
\right.
\end{equation}
\end{lemma}
\begin{proof}
Since $\Ns(e-e')=\Ns(s)\otimes\Ns(e')^\vee$, we have 
classical isomorphisms 
\begin{eqnarray}
\Hdr^0(C,\Ns(e'-e))&\;=\;&
\mathrm{Hom}^\nabla(\Ns(e),\Ns(e'))\;;\\
\Hdr^1(C,\Ns(e'-e))&=&
\mathrm{Ext}^1(\Ns(e),\Ns(e'))\;,
\end{eqnarray} 
(cf. \cite[Lemma~5.3.3 and 
Remark~5.3.4]{Kedlaya-book}). 
The claim then follows 
from Lemma \ref{Lemma : H^i N(e)}.
\end{proof}

\begin{lemma}\label{lem:invariantextension}
Let~$\Fs$ be a differential equation on~$C$ with log-affine radii along~$\Gamma_{C}$. Set 
$r':=\mathrm{rank}(\Fs^{\mathrm{Robba}})$. Assume that $\Fs^{\mathrm{Robba}}$ may be written as extension of $\Ns(e_{1}),\dotsc,\Ns(e_{r'})$ for $e_{1},\dotsc,e_{r'} \in K$. Then, we have $e_{1},\dotsc,e_{r'} \in R(K)$. 

Moreover, if $\Fs^{\mathrm{Robba}}$ may be written as extension of $\Ns(e'_{1}),\dotsc,\Ns(e'_{r'})$ for $e'_{1},\dotsc,e'_{r'} \in K$, then, there exists a permutation $\sigma\in\mathfrak{S}_{r'}$ such that, for each $i\in\{1,\dotsc,r'\}$, we have $e'_{\sigma(i)} - e_{i} \in \Z$.
\end{lemma}
\begin{proof}
The first part of the lemma follows from Lemma~\ref{lem:NeRobba} and the second from Lemma~\ref{Lemma : Hom and Ext of N(e)}.
\end{proof}

\subsubsection{The group of exponents.}

Assume that the residue field of $K$ has positive characteristic~$p$. Let~$m$ be a positive integer. 
In this setting, Christol and Mebkhout 
introduce in \cite[Sections 4 and 5]{Ch-Me-II} 
a certain equivalence relation 
$\stackrel{\mathfrak{E}}{\sim}$ on 
$(\mathbb{Z}_p/\mathbb{Z})^m$ and define the 
$p$-adic group of exponents of rank $m$ as
\begin{equation}
\mathfrak{E}_m
\;:=\;(\mathbb{Z}_p/\mathbb{Z})^m/
\stackrel{\mathfrak{E}}{\sim}\;.
\end{equation}
We omit the definition of 
$\stackrel{\mathfrak{E}}{\sim}$ which is technical 
and not essential for our purposes.

\begin{definition}
Let  $\mathfrak{e}\in\mathfrak{E}_m$
and let $e=(e_1,\ldots,e_m)\in\mathbb{Z}_p^m$ be a 
lift of $\mathfrak{e}$. 
\begin{itemize}
\item If, for all $i=1,\ldots,m$, the number $e_i$ is non-Liouville, 
we say that $\mathfrak{e}$ is non-Liouville.
\item If, for all $i,j=1,\ldots,m$, the difference 
$e_i-e_j$ is non-Liouville, we say that 
$\mathfrak{e}$ has non-Liouville differences.
\end{itemize}
By \cite[Def.4.3-1, Prop. 4.3-4, Thm. 4.4-7]{Ch-Me-II} 
(see also \cite[Definition 3.4.18]{Kedlaya-draft})
these properties only depend of 
$\mathfrak{e}$ and not of the particular choice of the 
lift~$e$.
\end{definition}
\begin{remark}
Unfortunately, Liouville numbers do not behave well with 
respect to the usual operations and do not form a group 
in general. In particular, there are elements of  
$\mathfrak{E}_m$ 
satisfying one point of the above definition but not 
both.
\end{remark}

\begin{lemma}[\protect{%
\cite[Prop. 4.4-10]{Ch-Me-II}}]
\label{Lemma : DNL implies multiset}
Assume that $\mathfrak{e}\in\mathfrak{E}_m$ has 
non-Liouville differences. Then there exists $(\bar{e}_1,\ldots,
\bar{e}_m)\in(\mathbb{Z}_p/\mathbb{Z})^m$ such that the class of $\mathfrak{e}$ for $\stackrel{\mathfrak{E}}{\sim}$ in 
$(\mathbb{Z}_p/\mathbb{Z})^m$ is given by all
the vectors of the form 
$(\bar{e}_{\sigma(1)},\ldots,\bar{e}_{\sigma(m)})$ 
where $\sigma$ runs through the group of permutations of $\{1,\dotsc,m\}$.

\if{That is the orbit of 
$(\bar{e}_{1},\ldots,\bar{e}_{m})$ 
by the action of the group of all permutations
$\mathfrak{S}_m$ acting on 
$(\mathbb{Z}_p/\mathbb{Z})^m$ by permutation of 
the coordinates.}\fi

In other words, we can identify $\mathfrak{e}$ with a 
multiset of $m$ elements of 
$\mathbb{Z}_p/\mathbb{Z}$.\footnote{i.e. 
a set of elements of $\mathbb{Z}_p/\mathbb{Z}$ 
counted with multiplicities such that the sum of the 
multiplicities equals $m$.}
\hfill$\Box$
\end{lemma}

\begin{remark}
In the situation of Lemma \ref{Lemma : DNL implies 
multiset}, we will sometimes talk about the exponent, 
\emph{i.e.} the multiset, and sometimes about the 
exponents (plural), \emph{i.e.} the elements of this 
multiset.
\end{remark}

%
%
%
%
%

\subsubsection{The exponent of a differential module of 
type Robba in positive residue characteristic.}
\label{section : Liouville conditions}


%
%
%


In this section, we assume that the residue field of~$K$ has positive characteristic~$p$. 

Let us consider the affine $K$-analytic line $\mathbb{A}^{1,\an}_{K}$ and fix a coordinate~$T$ on it. The definition that follows depends on it.

Let $C := \{r_{1} < |T| < r_{2}\}$, with $0< r_{1} < r_{2} < +\infty$, be an open annulus.
Let~$\Fs$ be a free differential equation over $C$ such that the radii of $\Fs$ are $\log$-affine along~$\Gamma_C$. Set $r' := \mathrm{rank}(\Fs^{\mathrm{Robba}})$. In this setting, Christol and Mebkhout define the exponent 
$\mathfrak{e}(\Fs^{\mathrm{Robba}})\in\mathfrak{E}_{r'}$ 
of $\Fs^{\mathrm{Robba}}$ (see \cite[Definition 5.3-6]{Ch-Me-II} or \cite[Definition 3.4.11]{Kedlaya-draft}).

\begin{remark}\label{rem:indepexp}
Let $j : K\to L$ be an isometric extension of complete 
valued fields. It induces a morphism 
$\pi_{L/K} : \E{1}{L} \to \E{1}{K}$ and we get a 
differential equation~$\pi_{L/K}^*(\Fs)$ over the open 
annulus~$\pi_{L/K}^{-1}(C)$. It follows from the 
definition that  we have 
\begin{equation}\label{eq : inv expo scalar ext}
\mathfrak{e}(\pi_{L/K}^*(\Fs)^{\mathrm{Robba}}) = \mathfrak{e}(\Fs^{\mathrm{Robba}})\;.
\end{equation}
More precisely, $j$ may be extended to a morphism $j : \wKa \to \widehat{L^\alg}$. Using Kedlaya's definition, the invariance boils down to the fact that, if~$\zeta$ is a root of unity in~$\wKa$ and $A \in \Z_{p}$, then $j(\zeta)$ is a root of unity and $j(\zeta^A) = j(\zeta)^A$.


\end{remark}

\begin{proposition}
\label{prop:orientationannulus}
Let $C$ be as above and let~$f$ be a $K$-linear automorphism of~$C$. Let $\Fs$ be a differential equation over $C$ with 
log-affine radii along $\Gamma_C$. 
Then $f^*(\Fs)^{\mathrm{Robba}}=
f^*(\Fs^{\mathrm{Robba}})$.

If $f$ preserves (resp. reverses) the orientation, then
the exponents of $(f^*\Fs)^{\mathrm{Robba}}$ 
coincide with (resp. the opposite of) those of 
$\Fs^{\mathrm{Robba}}$. 
\end{proposition}
\begin{proof}
The radii are invariant by isomorphisms, so 
$f^*(\Fs)^{\mathrm{Robba}}=
f^*(\Fs^{\mathrm{Robba}})$. 

By \eqref{eq : inv expo scalar ext} we can 
assume that $K$ is algebraically closed, spherically 
complete and that $|K|=\mathbb{R}_{\geq 0}$. In this 
case $\Fs$ is free. By 
Remark~\ref{Lemma: resriction of NL}, 
we can assume that $C$ is an annulus. 

In this case, if $f$ preserves the orientation, 
the equality of the exponents follows from 
\cite[Proposition 5.5-4]{Ch-Me-II}. 

Assume that $f$ reverses the orientation. 
Then, we can write 
$f=f_1\circ f_2$, where $f_2$ preserves the orientation 
and $f_1$ is of the form $x\mapsto ax^{-1}$ with $a\in K^*$. 
Since the claim holds when $f$ preserves the orientation, 
we can assume that $f=f_1$. In this case, $f$ is an 
isometric automorphism of $\O(C)$, therefore 
the properties of existence and of convergence required 
in \cite[Definition 13.5.2]{Kedlaya-book} are 
transported by pull-back. 
The only property that changes 
concerns the action of the $p^n$-th root of unity $\xi$ 
appearing in the definition. This action is given by 
$x\mapsto \xi\cdot x$ on $C$ and 
it produces an isomorphism 
$\xi^*:\Fs\simto\xi^*\Fs$ by 
deformation of the connection as in \cite{Inf-Def}. 
Explicitly, one has  $\xi^*(m)=
\sum_{i\geq 0}\frac{(\xi m-m)^i}{i!}\nabla^i(m)$, for 
all $m\in\Fs$. Moreover, we have 
$(\xi^{-1})^*=(\xi^*)^{-1}$. 
Now, the multiplication by $\xi$ satisfies 
$\xi\circ f_1=f_1\circ\xi^{-1}$ as an endomorphism of 
$\O(C)$.  Therefore, the definition of the exponent of 
$f_1^*(\Fs)$ involves the pull-back by $(\xi^{-1})^*$ 
instead that of $\xi^*$ and this produces the effect that 
the exponent changes sign. The claim follows.
\end{proof}

We now extend the definition of exponent to oriented virtual open annuli.

\begin{definition}[Exponent of $\Fs$]
\label{def:exponentcharp}
Let $C$ be an oriented virtual open annulus over $K$. Let~$\Fs$ be a 
differential equation over $C$ with $\log$-affine radii
along~$\Gamma_C$. Set 
$r':=\mathrm{rank}(\Fs^{\mathrm{Robba}})$.

Let $\Omega$ be a complete valued field extension of~$K$ that is algebraically 
closed, spherically complete and such that $|\Omega|
=\mathbb{R}_{\geq 0}$. Denote by $\pi_{\Omega/K}:C_\Omega\to C$ the projection morphism.
Let~$C'$ be a connected component of~$C_{\Omega}$ and endow it with the orientation induced by that of~$C$. It is an oriented open annulus on which~$\pi_{\Omega/K}^*(\Fs)$ is free. Choose an embedding~$\varphi$ of~$C$ into the affine line such that the image of the upper germ of~$C$ points away from~$\infty$ and identify~$C$ with its image. We define the exponent of~$\Fs^{\mathrm{Robba}}$ as
\begin{equation}\label{eq : gtrd}
\mathfrak{e}(\Fs^{\mathrm{Robba}}) \;:=\; \mathfrak{e}((\pi_{\Omega/K}^{*}(\Fs)_{|C'})^{\mathrm{Robba}}) \in\mathfrak{E}_{r'}\;.
\end{equation}
\end{definition}

\begin{remark}
Assume that $C = \{r_{1} < |T| < r_{2}\}$ as in the beginning of the section. If we endow~$C$ with the standard orientation (see 
Definition~\ref{Def : standard orientation}), then, in the previous definition, one may choose~$\varphi$ to be the inclusion into~$\mathbb{A}^{1,\an}_{K}$, hence the definition coincides with that of Christol and Mebkhout.
\end{remark}

\begin{lemma}
\label{Lemma : exp lies in Z_p}
Definition \ref{def:exponentcharp} does not depend on the choices of~$\Omega$, $C'$ and~$\varphi$. 
\end{lemma}
\begin{proof}
The independence of~$\varphi$ follows from Proposition~\ref{prop:orientationannulus}, so we need only consider what happens when~$\Omega$ and~$C$ change.

Let $(\Omega_{1},C'')$ be another choice. We can find a 
larger field $\Omega_{2}$ with the same properties 
containing both $\Omega_1$ and $\Omega$. By Remark~\ref{rem:indepexp}, we have $\mathfrak{e}((\pi_{\Omega_{2}/K}^{*}(\Fs)_{|C'_{\Omega_{2}}})^{\mathrm{Robba}}) = \mathfrak{e}((\pi_{\Omega/K}^{*}(\Fs)_{|C'})^{\mathrm{Robba}})$ and $\mathfrak{e}((\pi_{\Omega_{2}/K}^{*}(\Fs)_{|C''_{\Omega_{2}}})^{\mathrm{Robba}}) = \mathfrak{e}((\pi_{\Omega_{1}/K}^{*}(\Fs)_{|C''})^{\mathrm{Robba}})$. It follows that we may assume that~$\Omega_{1}=\Omega$.

Let us embed~$\wKa$ into~$\Omega'$. The projections~$C'_{0}$ and~$C''_{0}$ of~$C'$ and~$C''$ are two connected components of~$\pi^{-1}_{\wKa/K}(C)$. It follows that there exists a $K$-linear isometric automorphism of~$\wKa$ sending~$C'_{0}$ to~$C''_{0}$. By~\cite[Corollary~2.18]{NP-II}, it extends to a $K$-linear isometric automorphism of~$\Omega$ that sends~$C'$ to~$C''$ and the result follows from Remark~\ref{rem:indepexp} again, using the fact that $\sigma^*(\Fs) = \Fs$.
\end{proof}

\begin{remark}\label{rem:exponentscalarextensionannulus}
It follows readily from the definition that the exponent is invariant under extension of scalars. More precisely, in the setting of Definition~\ref{def:exponentcharp}, for each complete valued extension~$L$ of~$K$ and each connected component~$C'$ of $\pi_{L/K}^{-1}(C)$ with the induced orientation, we have
\begin{equation}
\mathfrak{e}(\pi^*_{L/K}(\Fs)_{|C'}^{\mathrm{Robba}})\;=\; \mathfrak{e}(\Fs^{\mathrm{Robba}})\;.
\end{equation} 
Il also follows from Proposition~\ref{prop:orientationannulus} that, if we reverse the orientation of~$C$, then the exponent is changed into its opposite.
\end{remark}

\begin{lemma}\label{Lemma: resriction of NL}
Let $C$ be an oriented virtual open annulus over $K$. 
Let~$\Fs$ be a differential equation over $C$ with 
$\log$-affine radii along~$\Gamma_C$. 
Let $C'\subseteq C$ be a virtual open annulus such that 
$\Gamma_{C'} \subseteq \Gamma_C$ and endow it with 
the induced orientation. Then, the 
exponent of $\Fs$ coincides with the exponent of its 
restriction $\Fs_{|C'}$:
\begin{equation}
\mathfrak{e}(\Fs^{\mathrm{Robba}})\;=\;
\mathfrak{e}(\Fs_{|C'}^{\mathrm{Robba}})\;.
\end{equation}

\end{lemma}
\begin{proof}
See \cite[after D\'efinition~5.3-6]{Ch-Me-II} or 
\cite[Theorem~3.4.16]{Kedlaya-draft}.
\end{proof}

\begin{definition}[Exponent, case of pseudo-annuli and 
good germs of segments]
\label{Def : exponent pseudo-an + germ of seg}
Let $C$ be an oriented pseudo-annulus and let $\Fs$ be a 
differential equation with log-affine radii along 
$\Gamma_C$. Let~$C'$ be a relatively compact open sub-pseudo-annulus of~$C$ such that $\Gamma_{C'}\subset \Gamma_C$ and endow it with the induced orientation. By \cite[Lemma 1.1.28]{NP-IV},
$C'$ is a virtual open annulus and we define the exponent of 
$\Fs^{\mathrm{Robba}}$ as 
\begin{equation}
\mathfrak{e}(\Fs^{\mathrm{Robba}})\;:=\;
\mathfrak{e}(\Fs_{|C'}^{\mathrm{Robba}})\;.
\end{equation}
By Lemma \ref{Lemma: resriction of NL}, it does not depend on the choice of $C'$. 


Let $b$ be a good germ of segment in a curve $X$ and let 
$\Fs$ be a differential equation on $X$ with log-affine 
radii along $b$. Let~$C$ be an open pseudo-annulus in~$X$ whose skeleton represents~$b$ and such that $\Fs_{|C}$ has 
log-affine radii along $\Gamma_C$. Endow~$C$ with the orientation~$b$. Then, we define the exponent of 
$\Fs^{\mathrm{Robba}}$ at~$b$ as $\mathfrak{e}(\Fs_{|C}^{\mathrm{Robba}})$.
By Lemma \ref{Lemma: resriction of NL}, it does not depend on the choice of~$C$.

\end{definition}

\begin{remark}\label{rem:exponentscalarextensiongerm}
The exponent over a pseudo-annulus or a good germ of segment is invariant under extension of scalars, in the sense of Remark~\ref{rem:exponentscalarextensionannulus}. The exponent over a pseudo-annulus is changed into its opposite when the orientation of the pseudo-annulus is reversed.
\end{remark}

\subsubsection{The exponent of a differential module of 
type Robba in residue characteristic~0.} 
\label{section : Liouville conditions char 0}
In this section, we assume that the residue field of~$K$ 
has characteristic~$0$ (this includes in particular the 
case where $K$ is trivially valued).

Let us consider the affine $K$-analytic line $\mathbb{A}^{1,\an}_{K}$ and fix a coordinate~$T$ on it. The definition that follows depends on it. The theory of exponents will be based on the following result.

\begin{theorem}[\protect{\cite[Theorem 
3.3.6]{Kedlaya-draft}}]
\label{Thm : Kedlaya exponent char 0}
Let $C:=\{r_1<|T|<r_2\}$, with 
$0\le r_1<r_2\le \infty$, be a standard open pseudo-annulus.

Let~$\Fs$ be a free differential equation over~$C$ such that the radii of $\Fs$ are $\log$-affine along $\Gamma_C$. Set 
$r':=\mathrm{rank}(\Fs^{\mathrm{Robba}})$.
Then, there exists a basis of 
$\Fs^{\mathrm{Robba}}$ in which 
the associated 
differential equation has the form 
\begin{equation}
\frac{d}{dT}-\frac{1}{T}\cdot B\;,\qquad 
\textrm{with}\quad B\in M_{r'\times r'}(K^{\circ})\;.
\end{equation}
\hfill$\Box$
\end{theorem}

\begin{remark}\label{rem:indepcoordinatechar0}
Let $e_{1},\dotsc,e_{r'} \in K^\circ$ be the eigenvalues of the matrix~$B$ of the theorem. Let~$K'$ be a complete valued extension of~$K$ in which the matrix~$B$ can be put in Jordan form. Then, for each connected component~$C'$ of~$\pi^{-1}_{K'/K}(C)$, the differential equation $\pi_{K'/K}^*(\Fs)^{\mathrm{Robba}}_{|C'}$ may be written as an extension of the modules $\Ns(e_{1}),\dotsc,\Ns(e_{r'})$. 

By Lemma~\ref{lem:invariantextension}, the multiset $\{e_{1},\dotsc,e_{r'}\}$ of elements $K^\circ/\Z$ is independent of the choices.
\end{remark}

\begin{definition}[Exponent of $\Fs$]
\label{def:exponentchar0}
Let $C$ be an oriented standard open pseudo-annulus over $K$ 
and let~$\Fs$ be a free
differential equation over $C$ such 
that the radii of $\Fs$ are all $\log$-affine 
along~$\Gamma_C$. Set 
$r'=\mathrm{rank}(\Fs^{\mathrm{Robba}})$.

Choose an embedding~$\varphi$ of~$C$ into the affine 
line such that the image of the upper germ of~$C$ 
points away from~$\infty$ and identify~$C$ with its 
image. We define the exponent 
$\mathfrak{e}(\Fs^{\mathrm{Robba}})$ 
of $\Fs^{\mathrm{Robba}}$ as 
the finite multisubset of 
$
\bigl(K^\circ/\mathbb{Z}\bigr)^{r'}/\mathfrak{S}_{r'}$
formed by the eigenvalues of a matrix 
$B\in M_{r'\times r'}(K^\circ)$ associated with
$\Fs^{\mathrm{Robba}}$ by Theorem 
\ref{Thm : Kedlaya exponent char 0}.
\end{definition}

By Remark~\ref{rem:indepcoordinatechar0}, the exponent is independent of the choice of~$\varphi$.

\begin{remark}\label{lem:exponentinvariantchar0}
It follows from the definition that the exponent is 
invariant under restriction and extension of scalars. More 
precisely, in the setting of 
Definition~\ref{def:exponentchar0}, for each open 
sub-pseudo-annulus~$C'$ of~$C$ such that 
$\Gamma_{C'} \subseteq \Gamma_C$ endowed with the 
induced orientation, we have
\begin{equation}
\mathfrak{e}(\Fs_{|C'}^{\mathrm{Robba}})\;=\;
 \mathfrak{e}(\Fs^{\mathrm{Robba}})
\end{equation}
and, for each complete valued extension~$L$ of~$K$ and each connected component~$C''$ of $\pi_{L/K}^{-1}(C)$ endowed with the induced orientation, we have
\begin{equation}\label{eq : strets}
\mathfrak{e}(\pi^*_{L/K}(\Fs)_{|C''}^{\mathrm{Robba}})\;=\; \mathfrak{e}(\Fs^{\mathrm{Robba}})\;.
\end{equation} 
\end{remark}

%

Since, in residue characteristic~0, exponents do not seem to be invariant by the Galois action (compare with Remark~\ref{rem:indepexp}), we will not try to define them in a more general setting.

\subsection{Liouville conditions.}
We now explain what is an equation free of Liouville 
numbers. We firstly define this notion over  open 
pseudo-annuli (cf. Definition 
\ref{Def : Free of LN over R}), then over good germs of 
segments in $X$ on which the radii 
of~$\Fs$ are log-affine (cf. Definition \ref{def:NLgerm}).

We maintain the notations of the above sections.

\begin{definition}[Liouville conditions, case of pseudo-annuli]
\label{Def : Free of LN over R}
Let~$C$ be an oriented open pseudo-annulus. Let $\Fs$ be a differential equation on~$C$ whose radii are $\log$-affine along~$\Gamma_C$. 

We say that~$\Fs$ is \emph{free of Liouville numbers} 
along $\Gamma_C$ if either~$\widetilde{K}$ has characteristic~0 or~$\widetilde K$ has characteristic $p>0$ and the following conditions hold:
\begin{enumerate}
\item the exponent $\mathfrak{e}(\Fs^{\mathrm{Robba}})$ is 
non-Liouville;
\item the exponent $\mathfrak{e}(\Fs^{\mathrm{Robba}})$ has 
non-Liouville differences.
\end{enumerate}

\if{
We say that $\Fs$ is \emph{strongly free of Liouville 
numbers} (along $\Gamma_C$) if 
either~$\widetilde{K}$ has characteristic~0 
or~$\widetilde K$ has characteristic $p>0$, $\Fs$ is free 
of Liouville numbers and, in addition, there exists an 
open sub-pseudo-annulus~$C'$ of~$C$ with 
$\Gamma_{C'}\subseteq\Gamma_C$ such that
\begin{enumerate}
\item[iii)] the radii of $\End(\Fs)$ are all $\log$-affine 
along $\Gamma_{C'}$;
\item[iv)] endowing~$C'$ with the induced orientation, the restriction 
$\End(\Fs)_{|C'}=\End(\Fs_{|C'})$ is 
free of Liouville numbers along $\Gamma_{C'}$.
\end{enumerate}
}\fi

We say that \emph{$\Fs$ is strongly free of Liouville numbers 
along $\Gamma_C$} if both~$\Fs$ and $\End(\Fs)$ have log-affine radii along 
$\Gamma_C$ and are free of Liouville 
numbers along it.

%
\end{definition}

\begin{remark}
Assume that~$K$ has residual characteristic~0 and 
that~$C$ is a standard open pseudo-annulus. 
In this case, the above definition states that any 
differential equation is free of Liouville numbers 
as soon as it has log-affine radii along $\Gamma_C$. 
This is consistent with 
the fact that, in this setting, 
there are no Liouville numbers by item v) of 
Lemma~\ref{Lemma : Liouville are Transcendental}.
\end{remark}

The following results follows readily from Lemma~\ref{Lemma: resriction of NL} and Remark~\ref{rem:exponentscalarextensiongerm}.

\begin{lemma}\label{lem:indepLpseudo}
Let~$C$ be an oriented open pseudo-annulus. Let $\Fs$ be a differential equation on~$C$ whose radii are $\log$-affine along~$\Gamma_C$. 

Let~$C' \subseteq C$ be an open pseudo-annulus such that $\Gamma_{C'} \subseteq \Gamma_{C}$ and endow it with the induced orientation. Then, $\Fs$ is 
free of Liouville numbers along 
$\Gamma_C$ if, and only if, $\Fs_{|C'}$  is 
free of Liouville numbers along $\Gamma_{C'}$.

Let~$L$ be a complete valued extension of~$K$ and let~$C''$ be a connected component of $\pi_{L/K}^{-1}(C)$. Endow it with the induced orientation. Then, $\Fs$ is  free 
of Liouville numbers along 
$\Gamma_C$ if, and only if, $\pi^*_{L/K}(\Fs)_{|C''}$  is  free 
of Liouville numbers along $\Gamma_{C''}$.
\hfill$\Box$
\end{lemma}

We now state the principal notion introduced 
in this appendix: being free of Liouville numbers 
at a germ of segment $b$. Notice that we do not require 
log-affineness of the radii along $b$.

\begin{definition}[Liouville condition, 
case of good germs]
\label{def:NLgerm}
Let~$b$ be a good germ of segment in~$X$. 
We say that~$\Fs$ is free 
of Liouville numbers along~$b$ if, for each open 
pseudo-annulus~$C$ whose skeleton represents~$b$, 
there exists an 
open sub-pseudo-annulus $C'\subseteq C$ with 
$\Gamma_{C'}\subseteq\Gamma_C$ (possibly 
not representing $b$) such that~$\Fs$ has log-affine radii along~$\Gamma_{C'}$ and is free of Liouville numbers along it.

We say that \emph{$\Fs$ is strongly free of Liouville numbers along~$b$} if both $\Fs$ and $\End(\Fs)$ are free of Liouville numbers along it.\footnote{We notice that 
the annulus $C'$ on which $\Fs$ is free of Liouville 
numbers is possibly not the same as the one on which $\End(\Fs)$ 
is free of Liouville numbers. They may be disjoint, and 
we do not exclude the case where there are no loci on 
which both $\Fs$ and $\End(\Fs)$ are free of Liouville 
numbers.}

%
%
\if{Let~$b$ be a good germ of segment in~$X$. 
Assume that all the radii of~$\Fs$ are log-affine 
on~$b$. 
We say that~$\Fs$ is free (resp. strongly free) 
of Liouville numbers along~$b$ if there exists an open 
pseudo-annulus~$C$ whose skeleton represents~$b$ 
such that
\begin{enumerate}
\item all the radii of $\Fs$ are $\log$-affine 
on~$\Gamma_{C}$;
\item endowing~$C$ with the orientation~$b$, 
$\Fs_{|C}$ is free (resp. strongly free) of Liouville 
numbers along~$\Gamma_C$.
\end{enumerate}
This implies the same property on every open 
sub-pseudo-annulus $C'$ of $C$, with 
$\Gamma_{C'}\subseteq\Gamma_C$.

Now, $\Fs$ is strongly free of Liouville numbers at $b$ if 
moreover \emph{for all} open pseudo-annulus~$C$ as above we 
can find a open sub-annulus $C''\subseteq C$, possibly 
not representing $b$, with 
$\Gamma_{C''}\subseteq \Gamma_C$, such that 
$\End(\Fs)$ is free of Liouville numbers along 
$\Gamma_{C''}$.
}\fi
\end{definition}



In particular, it makes sense to say that~$\Fs$ 
is free of Liouville numbers on~$b$, for any germ of 
segment~$b$ out of a point $x\in X$. 

\begin{remark}\label{Rk : uncontrolled end}
In the \cite[Section 3.4.2]{NP-V} we will 
use in an essential 
way the freeness of Liouville numbers of the differential 
equation $\End(\Fs)$. We notice that, even if $\Fs$ has 
log-affine radii along 
$\Gamma_C$, the same property need not hold for the 
differential equation $\End(\Fs)$. The 
relation between the radii of $\Fs$ and those of 
$\End(\Fs)=\Fs\otimes\Fs^*$ is unfortunately unclear.

Notice moreover that, $\End(\Fs)$ always has a 
non-trivial submodule generated by the identity 
endomorphism, so if the radii of $\End(\Fs)$ are 
log-affine along $\Gamma_C$ we always have 
$\End(\Fs)^{\mathrm{Robba}}\neq 0$ (even in the 
case where $\Fs^{\mathrm{Robba}}=0$). More 
precisely, if the radii of $\End(\Fs)$ are log-affine along 
$\Gamma_C$, we have
\begin{equation}
\End(\Fs^{\mathrm{Robba}})
\;\subseteq\;
\End(\Fs)^{\mathrm{Robba}}
\end{equation}
and the inclusion is strict whenever 
$\Fs\neq\Fs^{\mathrm{Robba}}$. Moreover applying  
 the decomposition from Definition~\ref{Def.: Rrobba part}
to $\Fs$ and
$\End(\Fs)=\Fs^*\otimes\Fs$, one sees that
$\End(\Fs^{\mathrm{Robba}})$ is a non-trivial 
direct summand of $\End(\Fs)$, and hence also of
$\End(\Fs)^{\mathrm{Robba}}$.
\end{remark}

\begin{remark}\label{rem:automaticLiouville}
\begin{enumerate}
\item It follows from Lemma~\ref{lem:indepLpseudo} 
that the property of being free or strongly free of 
Liouville numbers along a good germ of segment is 
invariant under extension of scalars.
\item Definition \ref{def:NLgerm} generalizes  
Definition \ref{Def : Free of LN over R} in the following 
sense.
Let $b$ be a good germ of segment represented by 
an open pseudo-annulus $C$ and assume that $\Fs$ 
is free of Liouville numbers along $b$. 
If the radii of $\Fs$ are all log-affine along $\Gamma_C$, 
then $\Fs$ is free of Liouville numbers all along 
the whole $\Gamma_C$ by Lemma 
\ref{Lemma: resriction of NL}.

\item We recall that $\Fs$ is 
automatically free of Liouville numbers at $b$ 
in the following cases:
\begin{enumerate}
\item $b$ is a germ of segment out of a point $x$ on 
which $\Fs$ has spectral non solvable radii. Indeed, in 
that case, by continuity, the radii of $\Fs$ are spectral 
log-affine and spectral non solvable along $b$, hence 
$\Fs^{\mathrm{Robba}}=0$.

In this case, as observed in Remark 
\ref{Rk : uncontrolled end}, 
the radii of $\End(\Fs)$ are never all spectral 
non-solvable and we cannot say anything about its 
freeness of Liouville numbers.

\item $b$ is the germ of segment at the open boundary of a virtual open disk~$D$ (with empty pseudo-triangulation) and $\Fs$ is a differential equation 
over $D$ whose radii are constant along $b$. In this case, if $C$ is any virtual open 
annulus representing $b$, the Robba part 
$(\Fs_{|C})^{\mathrm{Robba}}$ of $\Fs_{|C}$ 
is a trivial differential equation 
(see Lemma \ref{Lemma : Free of Liouville over a disk 
with constant radii}) and the result follows. 


In this situation, we cannot say anything about the 
affineness of the radii of $\End(\Fs)$ nor about its freeness 
of Liouville numbers.
\end{enumerate}
\end{enumerate}
\end{remark}

\if{\begin{remark}
\label{rk : non st by restr}
In the situation of  Lemma \ref{Lemma: resriction of NL}, if~$\Fs_{|C}$ is free of Liouville numbers along~$\Gamma_{C}$, then, by Remark~\ref{rem:indepLpseudo}, it is also free of Liouville numbers for every open pseudo-annulus $C' \subseteq C$ whose skeleton represents~$b$. Beware that this property does not hold anymore for the strong Liouville condition, since the definition involves another pseudo-annulus whose skeleton may not represents~$b$. In particular, although it is a property of the germ, it may not be stable by localization to a neighborhood of this germ.
\end{remark}
}\fi

\if{
\comm{J'ai ENCORE modifiéé la remarque conclusive 
suivante, ça me genait d'affirmer un truc que je n'avais 
pas démontré jusqu'au bout ... j'ai alors dit que je sais le 
démontrer seulement dans le cadre free of Liouville... et 
que le cas général devrait marcher pareil ...}
\comment{Je ne comprends pas. Le fait que l'orientation 
change le signe est d\'emontr\'e, non~? On le fait en 
citant CM.}
\comm{Non, on ne l'a démontré que pour le modules du 
type $\Ns(e)$ ... 
Ch-Me ne font que le cas qui preserve l'orientation 
malheureusement... je penses raiment que ça ne doit pas 
être tellement difficile, mais je n'ai pas envie de le 
faire...}
\begin{remark}
Although it is not relevant for the sequel, we point out 
here that we guess that it might be proven that Definition 
\ref{Def : Free of LN over R} 
is invariant by change of orientation of $C$. 
The strategy of proof we have in mind consists in proving 
that the pull-back by an automorphism reversing the 
orientation changes the exponent into its opposite and 
being non-Liouville is stable by change of sign by item i) 
of Lemma \ref{Lemma : Liouville are Transcendental}.

This proof works for modules of rank one and also 
for modules that becomes successive extension of rank 
one modules after an extension of the scalars. 
In particular, it works for modules that are free or 
strongly free of Liouville numbers with log-affine radii 
along the skeleton of a pseudo-annulus (cf. Theorem 
\ref{Thm : deco in rk 1 Ch-Me Robba}) which is all we 
need in this paper.
\end{remark}
}\fi


We finally explain how to take into account the case where the 
equation has some meromorphic singularities. Let $X$ be a quasi-smooth $K$-analytic curve, $Z$ a 
locally finite set of rigid points and $\Fc$ a differential 
equation on $X$ with meromorphic singularities at $Z$.
Set $Y:=X-Z$ and $\Fs:=\Fc_{|Y}$. 

A good germ of segment $b$ of $X$ can be represented 
by an open pseudo-annulus $C_b$ not intersecting~$Z$ if 
and only if $b$ is a good germ of segment of $Y$ (in 
other words $Z$ does not accumulate at $b$). In this case, the restriction 
$\Fc_{|C_b}=\Fs_{|C_b}$ has no 
meromorphic singularities and the notions 
developed in the previous section apply.

As a general principle, when speaking about the 
exponents and Liouville conditions, we say that~$\Fc$ has 
a given property if~$\Fs$ has that property. In particular, 
in the paper we will speak freely about the exponents of 
$\Fc$ (that are by definition those of $\Fs$) and 
the fact that $\Fc$ is free or strongly free of Liouville 
numbers (which means that $\Fs$ is).


\subsection{Decomposition theorem of type Fuchs.} 

In this section we provide some statements about 
Liouville conditions that will be systematically used in 
the sequel.
%
%
%
%
%
%
%

The following theorem extends the major result of 
\cite{Ch-Me-II} to the case of pseudo-annuli.

\begin{theorem}
\label{Thm : deco in rk 1 Ch-Me Robba}
Assume that $K$ is spherically complete, 
algebraically closed and that $|K|=\mathbb{R}_{\ge0}$.

Let $0\leq r_1<r_2\leq +\infty$ and let $\Fs$ be a differential equation over the open 
pseudo-annulus $C:=\{r_1<|T|<r_2\}$, 
with log-affine radii along $\Gamma_C$. Assume that the 
exponent of $\Fs^{\mathrm{Robba}}$ has non-Liouville differences. Then $\Fs^{\mathrm{Robba}}$ 
is a successive extension of rank one differential modules.

More precisely, choose a multiset $\mathfrak{e}'$ 
of~$(K^\circ/\Z)^r$ lifting 
$\mathfrak{e}(\Fs^{\mathrm{Robba}})$. Choose a 
subset~$\mathfrak{e}''$ of~$K^\circ$ that contains 
exactly one representative of each element of 
$\mathfrak{e}'$. Then, we have
\begin{equation}
\Fs^{\mathrm{Robba}}\;=\;
\bigoplus_{e\in\mathfrak{e}''}
E(e)\;,
\end{equation}
where $E(e)$ is a successive extension of the rank one 
module $\Ns(e)$ (cf. \eqref{eq : N(e)}) with dimension equal to the multiplicity of $e \mod \Z$ in $\mathfrak{e}'$.
\end{theorem}
\begin{proof}
If the residue characteristic of $K$ is $0$, then the claim 
follows from Remark~\ref{rem:indepcoordinatechar0}.



If the residue characteristic of $K$ is positive, then the 
claim is proved in \cite{Ch-Me-II} when 
$0<r_1<r_2<+\infty$. The general case follows easily.
%
%
%
\end{proof}

%
\begin{corollary}
\label{Cor. H^i(C)=H^i(C') restr}
Let $C$ be an open pseudo-annulus and let $\Fs$ be 
a differential equation with log-affine radii along 
$\Gamma_C$ that is free of Liouville numbers. 
Then $C$ has finite-dimensional de Rham cohomology 
and we have 
\begin{equation}\label{eq : chi=0 pseudo-annulus}
\chidr(C,\Fs)\;=\;0\;.
\end{equation}
For each germ of segment $b$ in the skeleton 
$\Gamma_C$, $\Fs$ and $\Fs^\mathrm{Robba}$ are Fredholm at~$b$ and one has
\begin{equation}\label{eq : chiabs=0 pseudo-annulus, app}
\chiabs_b(\Fs)\;=\;\Irr_b(\Fs)\;\textrm{ and }\;\;\chiabs_b(\Fs^{\mathrm{Robba}})\;=\;0\;.
\end{equation}
In particular, $\Fs$ satisfies~$\Fin_{b}$.

Moreover, if $C'\subseteq C$ is an inclusion of open 
pseudo-annuli such that 
$\Gamma_{C'}\subseteq\Gamma_C$, then,
for $i=0,1$, the natural restriction
\begin{equation}\label{eq : res C to C' log-affine}
\Hdr^i(C,\Fs)\;\xrightarrow{\;\sim\;}\;
\Hdr^i(C',\Fs_{|C'})
\end{equation}
is an isomorphism. 
\end{corollary}
\begin{proof}
We can assume that $K$ is algebraically closed, spherically complete and that $|K|=\mathbb{R}_{\geq 0}$. This follows from Theorem~\ref{thm:descent} for the part about cohomology and from the definitions for the part about the indexes. We may also assume that $\Fs=\Fs^{\mathrm{Robba}}$ by Propositions~\ref{Cor : Coh M = Coh MRobbaanna} and~\ref{Prop : chirel=Irr}.

By Theorem 
\ref{Thm : deco in rk 1 Ch-Me Robba}, 
$\Fs$ is then extension of rank one 
modules of type $\Ns(e)$. We may assume that~$\Fs$ has rank one, 
by the five lemma for the part about cohomology and by Lemma~\ref{Lemma : additivity on exact sequence} for the part about the indexes. In this case, the results follow from Lemma \ref{Lemma : H^i N(e)}. 
%
%


Let us now prove the final statement. The map \eqref{eq : res C to C' log-affine} is 
clearly injective for $i=0$ and surjective for $i=1$ by Lemma~\ref{Lemma : H^1 surjectif}. Since the indexes are 
zero on $C$ and $C'$, the map 
\eqref{eq : res C to C' log-affine} is necessarily an 
isomorphism.
%
\end{proof}

\if{
\comment{Je ne comprends pas ce que la d\'efinition fait l\`a.}

\begin{corollary}
\label{Cor : Ext injective}
Let $C'\subseteq C$ be an inclusion of open 
pseudo-annuli such that 
$\Gamma_{C'}\subseteq\Gamma_C$. 
Let $\Fs,\Gs$ be two differential equations on $C$ 
such that 
\begin{enumerate}
\item $\Fs^*\otimes\Gs$ has 
log-affine radii along $\Gamma_C$;
\item $\Fs^*\otimes\Gs$ is free of Liouville numbers.
\end{enumerate}
Then, the restriction map between Yoneda's groups 
of extensions
\begin{equation}\label{eq : res C to C' log-affine EXT}
\mathrm{Ext}^1(\Fs,\Gs)\;\xrightarrow{\;\sim\;}\;
\mathrm{Ext}^1(\Fs_{|C'},\Gs_{|C'})
\end{equation}
is bijective.
\end{corollary}
\begin{proof}
The claim follows from Corollary 
\ref{Cor. H^i(C)=H^i(C') restr} since 
$\mathrm{Ext}^1(\Fs,\Gs)=
\Hdr^1(C,\Fs^*\otimes\Gs)$ and similarly for~$C'$
(cf. \cite[Lemma 5.3.3]{Kedlaya-book}).

\comment{C'est assez abusif~: Kedlaya n'est pas dans ce cadre-l\`a. Pour bien faire les choses, je pense qu'il faudrait faire un changement de base, puis reprendre les arguments de la preuve de la proposition~\ref{Prop : H^i=0 if not solvablegfz}.

Ceci dit, j'ai l'impression que tu ne vas utiliser cela que pour des couronnes relativement compactes (auquel cas on a bien un faisceau libre, etc.).}
\end{proof}
\begin{remark}\label{Remark : exp Fotimes G}
Notice that the exponents of $\Fs^*\otimes\Gs$ are not 
related to those of $\Fs$ and $\Gs$. Namely, it is easy 
to find an example of rank one  
modules $\Fs$ and $\Gs$ that are \emph{not of Robba 
type} (so they do not have exponents) 
such that $\Fs^*\otimes\Gs$ is of \emph{Robba type} 
with arbitrary exponents. 

Nevertheless, 
if $\Fs$ and $\Gs$ are both of \emph{Robba type}, 
then the 
exponents of $\Fs^*\otimes\Gs$ are given by all the 
differences $e'-e$ where  
$e\in\mathfrak{e}(\Fs^{\mathrm{Robba}})$, 
$e'\in\mathfrak{e}(\Gs^{\mathrm{Robba}})$. Indeed, by Remark~\ref{rem:indepLpseudo}, 
we can reduce~$C$ and work over an open annulus.

\comment{J'ai mis une r\'ef\'erence \`a la remarque plut\^ot qu'au lemme  \ref{Lemma: resriction of NL}.}

 In this case, Theorem \ref{Thm : deco in rk 1 Ch-Me Robba} 
 shows that $\Fs$ and $\Gs$ are 
successive extensions of rank one differential modules of 
the form $\Ns(e)$ and it is then easy to compute the 
exponents of $\Fs^*\otimes\Gs$.
\end{remark}

\emph{Continuation of proof of 
Theorem \ref{Thm : deco in rk 1 Ch-Me Robba}.} We now prove Theorem~\ref{Thm : deco in rk 1 Ch-Me Robba} for an arbitrary open pseudo-annulus~$C$. By Remark~\ref{Rk : pseudo-annuli}, we may assume that it is of the form $C = \{r_{1}< |T|<r_{2}\}$, with $0\leq r_1<r_2\leq+\infty$. 

\comment{Je ne comprends pas. C'est une ancienne version, non~? Tu as suppos\'e que~$C$ est de cette forme dans l'\'enonc\'e du th\'eor\`eme...}

Consider an increasing sequence of open annuli 
$\{C_n\}_n$ such that $C=\cup_n C_n$. 
For all $n$, we have an isomorphism 
\begin{equation}
\psi_n\;:\;
\Fs^{\mathrm{Robba}}_{|C_n}\;
\xrightarrow{\;\sim\;}\;
\bigoplus_{e\in\mathfrak{e}(\Fs^{\mathrm{Robba}})}E_n(e)\;.
\end{equation}
It is enough to prove that we can find such a family of $\psi_n$'s that is compatible with the 
restrictions $C_{n+1}\to C_n$.

By Lemma \ref{Lemma: resriction of NL}, the exponents 
are independent of~$n$ and, by Lemma \ref{Lemma : H^i N(e)}, if $e\neq e' \pmod{\mathbb{Z}}$, then, 
for all $n,k\geq 0$, we have 
$\Hom^\nabla(E_{n+k}(e)_{|C_n},E_{n}(e')) = 0$,

\comment{J'ai enlev\'e la version avec les Ext qui me semble inutile ici. \c{C}a demande peut-\^etre quand m\^eme des pr\'ecisions, non~? C'est $E(e)$, pas $N(e)$.}

\comm{Par le théorème de Jordan-Holder l'image de $E_n(e)$ dans $E_m(e')$ a un unique constituant simple (avec multiplicité) qui est N(e) ou 0. Or si N(e) n'est pas isomorphe à N(e') ca doit être 0.

Le problème que je vois ici, c'est plutôt celui de montrer que une extension non triviale sur C reste non triviale après restriction à une sous couronne C'. C'est pour cela que j'ai introduit le lemme 2.1.26.}

\comment{Oui, d'accord, mais alors il faut citer ce lemme ici.}

hence $\psi_n$ has to send $E_{n+k}(e)_{|C_n}$ into $E_n(e)$, 
for all $e$. We may now assume that 
$\mathfrak{e}(\Fs^{\mathrm{Robba}})$ 
is constituted by a single element $e$. 

Now, by Remark \ref{Remark : exp Fotimes G} and 
\eqref{eq : res C to C' log-affine EXT} it follows that 
$E_{n+k}(e)_{|C_n}$ has to be is isomorphic to 
$E_n(e)$. 

\comment{Je ne comprends pas l'argument. Mais est-ce que ce n'est pas d\'ej\`a clair de toute fa\c{c}on~?}

The claim follows.

\end{proof}
}\fi





The following result is now a direct consequence of Corollary~\ref{cor:cohomologypseudoannulus}.

\begin{corollary}\label{cor:cohomologypseudoannulusLiouville}
Let $C$ be an open pseudo-annulus and let $b_0$ and 
$b_1$ be the germs of segment at its open boundary. 
Let $\Fs$ be a differential equation over $C$. Assume that $\Fs$ is free of 
Liouville numbers at $b_0$ and $b_1$.

Then, the following assertions are equivalent:
\begin{enumerate}[a)]
\item for each $i\ge 0$, $\Hdr^i(C,\Fs)$ is finite dimensional;
\item the total height of~$\Fs$ is log-affine along~$b_{0}$ and~$b_{1}$.
\end{enumerate}

Moreover, when these properties hold, we have 
\begin{equation}\label{eq:chidrpseudo-annulus}
\chidr(C,\Fs)
\;=\;\Irr_{b_0}(\Fs)+
\Irr_{b_1}(\Fs)\;.
\end{equation}
\hfill$\Box$
\end{corollary}

\subsection{Some useful result.}
In this section we state some useful results that are 
frequently used in the paper.

\begin{lemma}
\label{Lemma : Free of Liouville over a disk with constant radii}
Let~$D$ be an open pseudo-disk. Denote 
by~$b$ its germ of segment at infinity. Let~$\Fs$ be a 
differential equation over~$D$ of rank~$r$. 
Assume that, for every $i\in\{1,\dotsc,r\}$, the $i$-th 
radius is log-affine on $b$ and we have 
$\partial_{b} \R_{i}(-,\Fs) = 0$. 
Then, $\Fs$ is free of Liouville numbers along~$b$.
\end{lemma}
\begin{proof}
Let $C$ be an open pseudo-annulus whose 
skeleton represents $b$ and on which all the radii of~$\Fs$ are constant. 
In particular, the radii $\R_{i}(-,\Fs)$ 
(before localization to $C$) are either 
spectral non-solvable on $b$ or over-solvable on $b$.
We deduce that the Robba part 
$(\Fs_{|C})^{\mathrm{Robba}}$ is the restriction 
to~$C$ of the maximal submodule of~$\Fs$ that is 
trivial over $D$. Hence 
$(\Fs_{|C})^{\mathrm{Robba}}$ is the trivial 
equation and $\Fs$ is free of Liouville numbers.
\end{proof}

%

We have the following interesting consequence.

\begin{corollary}
Let $\Fs$ be a differential equation on a quasi-smooth 
$K$-analytic curve $X$. Let $\Gamma$ be a subgraph of~$X$ such that the radii of~$\Fs$ are locally constant on $X\setminus \Gamma$. Then, the following assertions are 
equivalent:
\begin{enumerate}
\item $\Fs$ is free of Liouville numbers along 
all good germs of segments in $X$;
\item $\Fs$ is free of Liouville numbers along all good germs 
of segments inside $\Gamma$.
\hfill$\Box$
\end{enumerate}
\end{corollary}
%

\begin{lemma}\label{Sub-quotients-Liouville}
Let $C$ be an open pseudo-annulus and let $\Fs$ 
 be a differential equation on $C$. Assume that~$\Fs$ (resp. $\End(\Fs)$) is free 
of Liouville numbers along 
$\Gamma_C$. Let $\Fs'$ be a sub-quotient of $\Fs$. 
Then, $\Fs'$ (resp. $\End(\Fs')$) 
is free of Liouville numbers along $\Gamma_C$.

\if{If, moreover, $\End(\Fs)$ is also 
free of Liouville numbers along $\Gamma_C$, 
then so is $\End(\Fs')$. 

More precisely, if $C'\subseteq C$ is a virtual open sub-
annulus with $\Gamma_{C'} \subseteq \Gamma_{C}$ 
and where $\End(\Fs)$ satisfies properties iii) and iv) of 
Definition \ref{Def : Free of LN over R}, then 
$\End(\Fs')$ satisfies the same properties over~$C'$.
}\fi
\end{lemma}
\begin{proof}
We may assume that $K$ has positive residue characteristic.
Moreover, it is clear that if $\Fs'$ is a sub-quotient of $\Fs$ then 
$\End(\Fs')=(\Fs')^*\otimes\Fs'$ is a 
sub-quotient of $\End(\Fs)=\Fs^*\otimes\Fs$, so it is enough to prove the result for~$\Fs$.

Along~$\Gamma_{C}$, all the radii are spectral, hence the family of radii of any sub-quotient belongs to the family of radii of the original module (see \cite[Theorem 10.6.2]{Kedlaya-book}). This shows that the radii of $\Fs'$ 
are all $\log$-affine along $\Gamma_C$. 
The claim now follows from Lemma
\ref{Lemma : ex-seq-exp} below.
\end{proof}
\if{
\comm{Là, j'ai reperé un erreur : la conclusion du lemme 
suivant n'est pas un if and only if, mais seulement un 
"alors" le problème concerne les différences, si les 
différences de deux exposants de $\Fs_1$ n'est pas 
Liouville, et si la même chose vaut pour $\Fs_2$, rien ne 
dit que la différence d'un exposant de $\Fs_1$ et d'un 
autre de $\Fs_2$ ne soit pas Liouville ...}

\comment{OK. J'ai reformul\'e la derni\`ere partie. J'ai 
aussi ajout\'e l'hypoth\`ese de caract\'eristique $>0$.}
\comm{J'ai déplacé plus bas l'hypothèse de 
caract\'eristique $>0$}
\comment{\c Ca ne me semble pas une bonne id\'ee. Les 
exposants ne sont pas d\'efinis dans cette g\'en\'eralit\'e 
en caract\'eristique r\'esiduelle nulle.}
\comm{C'est parce que la définition qu'on a donné n'est 
pas complète ? Par le futur on aura forcement envie de 
parler d'exposants sur une pseudo-couronne au bord ... 
je penses toujours que c'est un erreur de ne pas avoir 
donné la définition, même si elle n'est pas très 
intrinsèque ... pour pas se fatiguer maintennat on va se 
fatiguer toute la vie à distinguer des  cas ...}
\comment{Ce n'est pas une question de se fatiguer ou 
pas, c'est plus profond que \c ca. Si je comprends bien ce 
que tu veux faire, la d\'efinition n'est pas invariante par 
extension des scalaires, par exemple.}
\comm{Pourquoi tu dis ça ? En caract nulle ce sont les 
valeurs propres de la matrice $B$, ils sont donc les zéros 
du polynôme caractéristique, et c'est un ensemble stable 
par Galois. Je 
crois que c'est ça que Kedlaya veut dire par 
``\emph{galois invariant exponent}'' dans le Théorème 
3.3.6 de son papier de 2015. C'est clair que ce n'est pas 
un multi-set de $K$, mais je penses qu'il est bien défini 
est stable par extension de $K$. Fais moi comprendre 
mieux ce qui te trouble stp...}
\comment{Si tu parles d'une couronne, alors c'est d\'efini 
dans ce papier. Si tu parles d'une couronne virtuelle, 
alors il n'y a pas de coordonn\'ee et je ne sais pas de 
quelle matrice tu parles.}
\comm{Je ne comprends pas pourquoi tu ne peux pas 
définir l'exposant sur une couronne virtuelle 
comme dans la définition 
\ref{def:exponentcharp} et ensuite reproduite 
pas par pas la preuve du Lemme 
\ref{Lemma : exp lies in Z_p}. Tout ce qu'il te faut serait 
uniquement un analogue de 
\eqref{eq : inv expo scalar ext} en caractéristique 
résiduelle nulle et cela me semble 
faisable comme je te l'indiquait ci plus haut en regardant 
les valeurs propres de la matrice $B$. Il me semble qu'on trouve un truc du type 
\begin{equation}\label{hmlokhtgtyght}
\mathfrak{e}(\pi_{L/K}^*(\Fs)^{\mathrm{Robba}}) = j(\mathfrak{e}(\Fs^{\mathrm{Robba}}))
\end{equation}
où $j:\wKa\to\widehat{L^{\mathrm{alg}}}$ 
serait un prolongement de $j:K\to L$.
Maintenant l'exposant est un multi-set dans $\wKa$ 
stable globalement par Galois 
(car c'est les zéros du polynôme caractéristique de la 
matrice $B$), donc 
\eqref{hmlokhtgtyght} montre qu'il est 
invariant par changement de base (à moins de 
permutations, mais pour les 
multisets les permutations ne font rien)... 
Pourquoi ça ne marcherait pas ? }
\comment{Ce qui ne marche pas, c'est le passage d'une composante connexe $C'$ à une composante connexe $C''$. Ces deux composantes sont définies sur une extension $K'$ de $K$. Pour chacune on a une matrice que l'on peut choisir à coefficients dans $K'$ d'après Kedlaya. Mais l'action de Galois qui envoie $C'$ sur $C''$ va changer les valeurs propres.

Et si tu dis que tu choisis une composante connexe et que ça ne te dérange pas, tu as quand même un problème parce que quand tu vas changer de base, tes deux composantes vont réapparaître.}

\comm{Oui, je vois à quoi tu penses, mais ce n'est pas 
génant pour moi. Ce que je te propose est de définir 
l'exposant de $\Fs$ sur une pseudo-couronne $C$ 
comme celui de $(\Fs_{\wKa})_{|C'}$ où $C'$ est une 
composante connexe de $C_{\wKa}$ \emph{MODULO 
ACTION DE GALOIS}.\\

En effet, le pull-back par un élément $\sigma$ de Galois 
permute les composantes connexes, mais 
$\sigma^*(\Fs)\cong\Fs$, du coup si 
$\sigma:C''\simto C'$ alors la matrice $B$ sur $C'$ se 
tranforme par pull-back dans $\sigma^{-1}(B)$ sur $C''$ 
et le multi-setde $E$ de ses valeurs propres sur $C'$ 
devient $\sigma^{-1}(E)$ sur $C''$.\protect{\footnote{Il est important de 
remarquer que un exposant est un élément de l'orbite par 
Galois de l'ensemble $E$ et que cela est différent de 
considerer un ensemble dont chaque élément represente 
une orbite d'un élément de $E$. Autrement dit, il faut 
que l'action de Galois agisse sur tous les éléments de 
l'ensemble $E$ simultanement.}}\\

Cette définition n'est pas genant pour moi. \\

Après, pour exprimer des propriétés comme celles du 
Lemme \ref{Lemma : ex-seq-exp} ci plus bas on va dire 
que une fois fixé la composante connexe $C'$ le lemme a 
un sens, et l'action de Galois transporte les unclusions 
des multi-sets. Et à partir de la on peut supposer 
$K=\wKa$...\\

Qu'en penses tu ?}
\comment{Je pense que \c ca ne marche toujours pas. Si tu pars d'une couronne virtuelle sur $K$, tu vas d\'efinir un exposant modulo $Gal(\wKa/K)$. Mais que devient cet exposant si on change de base \`a un corps $L$ o\`u la couronne virtuelle devient un nombre fini de couronnes~? Tu ne retrouves pas exactement les exposants (qui sont bien d\'efinis \`a ce niveau) mais un truc bizarre. En plus, \c ca a l'air de d\'ependre du $K$ choisi alors que \c ca ne devrait pas.

En gros, je suis d'accord qu'il est possible de donner la d\'efinition que tu veux, mais je n'en vois pas l'int\'er\^et et je pense qu'elle a de mauvaises propri\'et\'es.}

}\fi
\begin{lemma}[\protect{\cite[Proposition 5.4-3, and 
Théorèmes 5.4-5, 5.4-6]{Ch-Me-II}}]
\label{Lemma : ex-seq-exp}
Let $0\to\Fs_1\to\Fs_2\to\Fs_3\to 0$ be an exact 
sequence of differential equation over an open pseudo-annulus~$C$. If $\Fs_2$ has $\log$-affine radii along 
$\Gamma_C$, then so have $\Fs_1$ and $\Fs_3$. In 
this case, we have an exact sequence
\begin{equation}
0\to\Fs_1^{\mathrm{Robba}}\to
\Fs_2^{\mathrm{Robba}}\to
\Fs_3^{\mathrm{Robba}}\to 0
\end{equation}
and the exponent of $\Fs_2$ is the union of 
those of $\Fs_1$ and 
$\Fs_3$. 

In particular, if $K$ has positive residue characteristic, then: 
\begin{enumerate}
\item The exponent of $\Fs_2^{\mathrm{Robba}}$ is 
non-Liouville if, and only if, so are the exponents of 
$\Fs_1^{\mathrm{Robba}}$ and 
$\Fs_3^{\mathrm{Robba}}$.
\item If the exponent of $\Fs_2^{\mathrm{Robba}}$ 
has non-Liouville differences, then so have 
the exponents of 
$\Fs_1^{\mathrm{Robba}}$ and 
$\Fs_3^{\mathrm{Robba}}$.
\item If each element of the exponent of $\Fs_1^{\mathrm{Robba}}$ belongs to the exponent of $\Fs_3^{\mathrm{Robba}}$ or the other way round and if the exponents of 
$\Fs_1^{\mathrm{Robba}}$ and 
$\Fs_3^{\mathrm{Robba}}$ both have non-Liouville 
differences, then so has $\Fs_2^{\mathrm{Robba}}$. 
\hfill$\Box$
\end{enumerate}
\end{lemma}

\subsection{A characterization of the exponents.}
\label{A cohomological characterization of the 
exponents}


Here, we recall a well-known cohomological characterization of the exponents that is often 
useful. 

Let $C := \{r_{1}< |T|<r_{2}\}$, with $0\leq r_1<r_2\leq+\infty$, 
and let 
$\Ns(e)$ be as in \eqref{eq : N(e)}.
For $k\geq 0$, we consider the 
differential module
\begin{equation}
\Ls_k\;:=\;
\bigoplus_{i=0}^k\O(C)\cdot\log(T)^i\;,
\end{equation}
where $\log(T)$ is a formal symbol and the connection is given by the action of $d/dT$ on $\log(T)^i$ as one expects. The union 
\begin{equation}
\Ls\;:=\;\bigcup_k\Ls_k
\end{equation}
is a differential ring (algebraically 
isomorphic to the ring of polynomials in the 
indeterminate $\log(T)$ with coefficients in $\O(C)$).

The following characterization follows from Theorem~\ref{Thm : deco in rk 1 Ch-Me Robba} and Lemma~\ref{Lemma : H^i N(e)}.

\begin{lemma}
\label{Lemma : cohomological characterization}
Let $C:=\{r_1<|T|<r_2\}$, with 
$0\leq r_1<r_2\leq +\infty$. 
Let $\Fs$ be a differential equation with $\log$-affine 
radii along $\Gamma_C$ such that 
the exponent $\mathfrak{e}(\Fs^{\mathrm{Robba}})$ 
has non-Liouville differences. Then, an element
$e\in K$ belongs to the multiset 
$\mathfrak{e}(\Fs^{\mathrm{Robba}})$
if, and only if, one has 
\begin{equation}\label{eq : efpq}
\Hdr^0(C,\Fs^{\mathrm{Robba}}
\otimes\Ns(-e))\;\neq\; 0\;.
\end{equation}
In this case the multiplicity of~$e$ in the multiset 
$\mathfrak{e}(\Fs^{\mathrm{Robba}})$ equals the 
dimension 
\begin{equation}\label{eq : indfil}
\mathrm{dim}\;\Hdr^0(C,\Fs^{\mathrm{Robba}}\otimes
\Ns(-e)\otimes\Ls)\;=\;
\lim_{k\to\infty}\mathrm{dim}\;
\Hdr^0(C,\Fs^{\mathrm{Robba}}\otimes
\Ns(-e)\otimes\Ls_k)\;.
\end{equation}
The space $\Hdr^0(C,\Fs^{\mathrm{Robba}}\otimes \Ns(-e)\otimes\Ls)$ is also the space of solutions of 
$\Fs^{\mathrm{Robba}}\otimes\Ns(-e)$ with values in 
the differential ring $\Ls$.
\end{lemma}
\begin{proof}
By Theorem~\ref{thm:descent},
we can assume that $K$ is algebraically closed, 
spherically complete, and $|K|=\mathbb{R}_{\geq 0}$.
The first part of the claim, concerning \eqref{eq : efpq}, 
then follows from Theorem 
\ref{Thm : deco in rk 1 Ch-Me Robba}.

We now prove the second part of the claim. 
Replacing $\Fs^{\mathrm{Robba}}$ with 
$\Fs^{\mathrm{Robba}}\otimes\Ns(-e)$ we can assume 
that $e=0$. In this case the block $E(0)$ appearing in 
Theorem \ref{Thm : deco in rk 1 Ch-Me Robba} is 
represented in a basis by the equation $Y'=GY$, where 
$G$ is a standard nilpotent matrix in the Jordan form.
The solutions of this equations are linear combinations of 
$\{1,\log(T),\log(T)^2,\ldots\}$. So 
$E(0)$ is trivialized by the differential 
ring $\Ls$. In other words $E(0)\otimes\Ls$ is a trivial 
differential module over $\Ls$.
The limit expression \eqref{eq : indfil} follows from the 
fact that $\Hdr^0$ commutes with inductive limits.
\end{proof}
\begin{remark}
\label{Remark : exponents of Fotimes L_k}
Sometimes the dimension of 
$\Hdr^0(C,\Fs^{\mathrm{Robba}}\otimes
\Ns(-e)\otimes\Ls_k)$ can be computed with 
an index formula, and it happens that 
it is necessary to check the Liouville 
condition on $\Fs^{\mathrm{Robba}}\otimes
\Ns(-e)\otimes\Ls_k$. For this, we observe that 
$\Ls_k$ is a successive extension of the trivial 
rank one equation $y'=0$, hence 
$\Fs^{\mathrm{Robba}}\otimes
\Ns(-e)\otimes\Ls_k$ is a successive extension of 
$\Fs^{\mathrm{Robba}}\otimes
\Ns(-e)$. 
For this reason, 
$\Fs^{\mathrm{Robba}}\otimes
\Ns(-e)\otimes\Ls_k$ is free of 
Liouville numbers along $\Gamma_C$ 
if, and only if, so is $\Fs^{\mathrm{Robba}}\otimes
\Ns(-e)$ (cf. item (iii) of Lemma \ref{Lemma : ex-seq-exp}).
Notice moreover that, by Lemma 
\ref{Lemma : cohomological characterization},
the exponents of $\Fs^{\mathrm{Robba}}\otimes
\Ns(-e)$ are those of $\Fs^{\mathrm{Robba}}$ 
to which we subtract $e$ component by component. 
\end{remark}

\if{The arguments of the proof of 
Lemma~\ref{lem:basicnbgh} also work in the presence 
of meromorphic singularities (cf. Definition 
\ref{Def : mero-overc-Hdr^i}).


\begin{lemma}
\label{Lemma : Independence on U mero over}
Let~$P$ be a quasi-smooth $K$-analytic curve and 
let~$Z$ be locally finite set of rigid points of~$P$. 
Let $\Fc$ be an overconvergent differential equation 
on~$P$ with poles on~$Z$. Assume that $\Fc$ is free of 
overconvergent Liouville numbers at every point of 
$\partial P$. Then, for every elementary 
neighborhood~$U$ of~$P$ in~$P'$ that is adapted to 
$\Fc'$ and for every $i\ge 0$, we have a canonical 
isomorphism
\begin{equation}
\qquad\qquad\qquad\quad
\Hdr^i(U(*Z),\Fc'_{|U})\;\xrightarrow[]{\sim}\; \Hdr^i(P^\dag(*Z),\Fc)  \;.
\qquad\qquad\qquad\Box
\end{equation}
\end{lemma}
}\fi

\bibliographystyle{amsalpha}
\bibliography{NP}

\def\cprime{$'$}
\providecommand{\bysame}{\leavevmode\hbox to3em{\hrulefill}\thinspace}
\providecommand{\MR}{\relax\ifhmode\unskip\space\fi MR }
\providecommand{\MRhref}[2]{%
  \href{http://www.ams.org/mathscinet-getitem?mr=#1}{#2}
}
\providecommand{\href}[2]{#2}
\begin{thebibliography}{{Dwo}97}

\bibitem[AB01]{Andre-Balda-Book}
Yves Andr{\'e} and Francesco Baldassarri, \emph{De {R}ham cohomology of
  differential modules on algebraic varieties}, Progress in Mathematics, vol.
  189, Birkh\"auser Verlag, Basel, 2001. \MR{1807281 (2002h:14031)}

\bibitem[AC18]{Caro-Abe}
Tomoyuki Abe and Daniel Caro, \emph{Theory of weights in {{\(p\)}}-adic
  cohomology}, Am. J. Math. \textbf{140} (2018), no.~4, 879--975 (English).

\bibitem[Ado76]{Adolphson}
A.~Adolphson, \emph{An index theorem for {$p$}-adic differential operators},
  Trans. Amer. Math. Soc. \textbf{216} (1976), 279--293. \MR{0387284 (52
  \#8127)}

\bibitem[And04]{Andre-Comparison}
Yves Andr{\'e}, \emph{Comparison theorems between algebraic and analytic de
  {Rham} cohomology (with emphasis on the {{\(p\)}}-adic case)}, J. Th{\'e}or.
  Nombres Bordx. \textbf{16} (2004), no.~2, 335--355 (English).

\bibitem[AS63]{Atiyah-Singer-announcement}
Michael~F. Atiyah and I.~M. Singer, \emph{The index of elliptic operators on
  compact manifolds}, Bull. Am. Math. Soc. \textbf{69} (1963), 422--433
  (English).

\bibitem[AS68]{Atiyah-Singer}
Michael~F. Atiyah and Isadore~Manuel Singer, \emph{The index of elliptic
  operators. {I}}, Ann. Math. (2) \textbf{87} (1968), 484--530 (English;
  Russian).

\bibitem[Bal82]{Balda-Turritin}
Francesco Baldassarri, \emph{Differential modules and singular points of
  $p$-adic differential equations}, Adv. Math. \textbf{44} (1982), 155--179.

\bibitem[Bal87]{Balda-Comparaison}
F.~Baldassarri, \emph{Comparaison entre la cohomologie alg\'ebrique et la
  cohomologie {$p$}-adique rigide \`a coefficients dans un module
  diff\'erentiel. {I}. {C}as des courbes}, Invent. Math. \textbf{87} (1987),
  no.~1, 83--99. \MR{862713 (88d:14014)}

\bibitem[Bal88]{Balda-Comparison-II}
\bysame, \emph{Comparaison entre la cohomologie alg{\'e}brique et la
  cohomologie p-adique rigide {\`a} coefficients dans un module
  diff{\'e}rentiel. {II}: {Cas} des singularit{\'e}s r{\'e}guli{\`e}res {\`a}
  plusieures variables. ({Comparison} between the algebraic cohomology and the
  rigid p-adic cohomology with coefficients in a differential module. {II}:
  {Case} of regular singularities with several variables)}, Math. Ann.
  \textbf{280} (1988), no.~3, 417--439 (French).

\bibitem[Bal10]{Balda-Inventiones}
Francesco Baldassarri, \emph{Continuity of the radius of convergence of
  differential equations on {$p$}-adic analytic curves}, Invent. Math.
  \textbf{182} (2010), no.~3, 513--584. \MR{2737705 (2011m:12015)}

\bibitem[Ber86]{Berthelot-rigide}
Pierre Berthelot, \emph{G{\'e}om{\'e}trie rigide et cohomologie des
  vari{\'e}t{\'e}s alg{\'e}briques de {{\(caract\acute eristique\quad p\)}}.
  (rigid geometry and cohomology of algebraic varieties of
  {{\(characteristic\quad p\)}})}, M{\'e}m. Soc. Math. Fr., Nouv. S{\'e}r.
  \textbf{23} (1986), 7--32 (French).

\bibitem[Ber90]{Ber}
Vladimir~G. Berkovich, \emph{Spectral theory and analytic geometry over
  non-{A}rchimedean fields}, Mathematical Surveys and Monographs, vol.~33,
  American Mathematical Society, Providence, RI, 1990.

\bibitem[BV07]{DV-Balda}
F.~Baldassarri and L.~Di Vizio, \emph{Continuity of the radius of convergence
  of $p$-adic differential equations on berkovich spaces}, arXiv, 2007,
  \url{http://arxiv.org/abs/0709.2008}, pp.~1--22.

\bibitem[Car15]{Caro}
Daniel Caro, \emph{On the preservation of the overconvergence for the direct
  image of a proper and flat morphism}, Ann. Sci. {\'E}c. Norm. Sup{\'e}r. (4)
  \textbf{48} (2015), no.~1, 131--169 (French).

\bibitem[CD94]{Ch-Dw}
G.~Christol and B.~Dwork, \emph{Modules diff\'erentiels sur des couronnes},
  Ann. Inst. Fourier (Grenoble) \textbf{44} (1994), no.~3, 663--701.
  \MR{MR1303881 (96f:12008)}

\bibitem[Chi88a]{Chiarellotto-Comparison-2}
Bruno Chiarellotto, \emph{A comparison theorem in p-adic cohomology}, Ann. Mat.
  Pura Appl. (4) \textbf{153} (1988), 115--131 (English).

\bibitem[Chi88b]{Chiarellotto-Comparison}
\bysame, \emph{Sur le th{\'e}or{\`e}me de comparaison entre cohomologie de {De}
  {Rham} alg{\'e}brique et p-adique rigide. ({On} the comparison between the
  algebraic {De} {Rham} and the rigid p-adic cohomology)}, Ann. Inst. Fourier
  \textbf{38} (1988), no.~4, 1--15 (French).

\bibitem[Chr83]{Ch}
Gilles Christol, \emph{Modules diff\'erentiels et \'equations diff\'erentielles
  {$p$}-adiques}, Queen's Papers in Pure and Applied Mathematics, vol.~66,
  Queen's University, Kingston, ON, 1983.

\bibitem[Chr10]{Ch-Irrobba}
Gilles Christol, \emph{Exposants $p$-adiques et solutions dans les couronnes},
  2010, \url{https://webusers.imj-prg.fr/~gilles.christol/exposants.pdf},
  pp.~1--14.

\bibitem[{Cla}66]{Clark}
D.N. {Clark}, \emph{{A note on the p-adic convergence of solutions of linear
  differential equations.}}, {Proc. Am. Math. Soc.} \textbf{17} (1966),
  262--269 (English).

\bibitem[CM93]{Ch-Me-I}
G.~Christol and Z.~Mebkhout, \emph{Sur le th\'eor\`eme de l'indice des
  \'equations diff\'erentielles {$p$}-adiques. {I}}, Ann. Inst. Fourier
  (Grenoble) \textbf{43} (1993), no.~5, 1545--1574. \MR{1275209 (95j:12009)}

\bibitem[CM97]{Ch-Me-II}
\bysame, \emph{Sur le th\'eor\`eme de l'indice des \'equations
  diff\'erentielles {$p$}-adiques. {II}}, Ann. of Math. (2) \textbf{146}
  (1997), no.~2, 345--410. \MR{1477761 (99a:12009)}

\bibitem[CM00]{Ch-Me-III}
\bysame, \emph{Sur le th\'eor\`eme de l'indice des \'equations
  diff\'erentielles {$p$}-adiques. {III}}, Ann. of Math. (2) \textbf{151}
  (2000), no.~2, 385--457. \MR{1765703 (2001k:12014)}

\bibitem[CM01]{Ch-Me-IV}
\bysame, \emph{Sur le th\'eor\`eme de l'indice des \'equations
  diff\'erentielles {$p$}-adiques. {IV}}, Invent. Math. \textbf{143} (2001),
  no.~3, 629--672. \MR{1817646 (2002d:12005)}

\bibitem[CM02]{Astx}
Gilles Christol and Zoghman Mebkhout, \emph{\'{E}quations diff\'erentielles
  {$p$}-adiques et coefficients {$p$}-adiques sur les courbes}, Ast\'erisque
  (2002), no.~279, 125--183, Cohomologies $p$-adiques et applications
  arithm{\'e}tiques, II. \MR{1922830 (2003i:12014)}

\bibitem[CR94]{Ch-Ro}
G.~Christol and P.~Robba, \emph{\'{E}quations diff\'erentielles {$p$}-adiques},
  Actualit\'es Math\'ematiques., Hermann, Paris, 1994, Applications aux sommes
  exponentielles. [Applications to exponential sums].

\bibitem[Cre98]{Crew-Finiteness}
Richard Crew, \emph{Finiteness theorems for the cohomology of an overconvergent
  isocrystal on a curve}, Ann. Sci. \'Ecole Norm. Sup. (4) \textbf{31} (1998),
  no.~6, 717--763. \MR{1664230 (2000a:14023)}

\bibitem[CT12]{Caro-Tsuzuki}
Daniel Caro and Nobuo Tsuzuki, \emph{Overholonomicity of overconvergent
  {{\(F\)}}-isocrystals over smooth varieties}, Ann. Math. (2) \textbf{176}
  (2012), no.~2, 747--813 (English).

\bibitem[Del70]{Deligne-Reg-Sing}
Pierre Deligne, \emph{\'{E}quations diff\'erentielles \`a points singuliers
  r\'eguliers}, Lecture Notes in Mathematics, Vol. 163, Springer-Verlag,
  Berlin, 1970. \MR{MR0417174 (54 \#5232)}

\bibitem[DGS94]{DGS}
B.~Dwork, G.~Gerotto, and F.~J. Sullivan, \emph{An introduction to
  {$G$}-functions}, Annals of Mathematics Studies, vol. 133, Princeton
  University Press, Princeton, NJ, 1994.

\bibitem[DMR07]{Correspondance-Malgrange-Ramis}
Pierre Deligne, Bernard Malgrange, and Jean-Pierre Ramis, \emph{Singularit\'es
  irr\'eguli\`eres}, Documents Math\'ematiques (Paris) [Mathematical Documents
  (Paris)], 5, Soci\'et\'e Math\'ematique de France, Paris, 2007,
  Correspondance et documents. [Correspondence and documents]. \MR{2387754
  (2009d:32033)}

\bibitem[Duc]{Duc}
Antoine Ducros, \emph{La structure des courbes analytiques},
  \url{http://www.math.jussieu.fr/~ducros/livre.html}.

\bibitem[{Dwo}97]{Dwork-Exponents}
Bernard~M. {Dwork}, \emph{{On exponents of $p$-adic differential modules.}},
  {J. Reine Angew. Math.} \textbf{484} (1997), 85--126 (English).

\bibitem[Gro66]{Grothendieck-DR-coh}
A.~Grothendieck, \emph{On the {De} {Rham} cohomology of algebraic varieties},
  Publ. Math., Inst. Hautes {\'E}tud. Sci. \textbf{29} (1966), 95--103
  (English).

\bibitem[Har75]{Hartshorne-de-Rham}
Robin Hartshorne, \emph{On the {De} {Rham} cohomology of algebraic varieties},
  Publ. Math., Inst. Hautes {\'E}tud. Sci. \textbf{45} (1975), 5--99 (English).

\bibitem[HTT08]{HTT-Dmodules}
Ryoshi Hotta, Kiyoshi Takeuchi, and Toshiyuki Tanisaki, \emph{{$D$}-modules,
  perverse sheaves, and representation theory}, Progress in Mathematics, vol.
  236, Birkh\"auser Boston, Inc., Boston, MA, 2008, Translated from the 1995
  Japanese edition by Takeuchi. \MR{2357361 (2008k:32022)}

\bibitem[Kat87a]{Katz-Can}
Nicholas~M. Katz, \emph{On the calculation of some differential {G}alois
  groups}, Invent. Math. \textbf{87} (1987), no.~1, 13--61. \MR{MR862711
  (88c:12010)}

\bibitem[Kat87b]{Katz-cyclic-vect}
\bysame, \emph{A simple algorithm for cyclic vectors}, Amer. J. Math.
  \textbf{109} (1987), no.~1, 65--70. \MR{878198 (88b:13001)}

\bibitem[Ked06a]{Kedlaya-finiteness}
Kiran~S. Kedlaya, \emph{Finiteness of rigid cohomology with coefficients}, Duke
  Math. J. \textbf{134} (2006), no.~1, 15--97. \MR{2239343 (2007m:14021)}

\bibitem[Ked06b]{Kedlaya-Weil-II}
\bysame, \emph{Fourier transforms and {{\(p\)}}-adic `{Weil} {II}'}, Compos.
  Math. \textbf{142} (2006), no.~6, 1426--1450 (English).

\bibitem[Ked08]{Kedlaya-practice}
\bysame, \emph{{{\(p\)}}-adic cohomology: from theory to practice}, \(p\)-adic
  geometry. Lectures from the 2007 10th Arizona winter school, Tucson, AZ, USA,
  March 10--14, 2007, Providence, RI: American Mathematical Society (AMS),
  2008, pp.~175--203 (English).

\bibitem[Ked09]{Ked-p-adic-cohomology}
\bysame, \emph{{$p$}-adic cohomology}, Algebraic geometry---{S}eattle 2005.
  {P}art 2, Proc. Sympos. Pure Math., vol.~80, Amer. Math. Soc., Providence,
  RI, 2009, pp.~667--684. \MR{2483951 (2010e:14014)}

\bibitem[Ked10]{Kedlaya-book}
\bysame, \emph{p-adic differential equations}, Cambridge Studies in Advanced
  Mathematics, vol. 125, Cambridge Univ. Press, 2010.

\bibitem[Ked15]{Kedlaya-draft}
\bysame, \emph{Local and global structure of connections on nonarchimedean
  curves}, Compos. Math. \textbf{151} (2015), no.~6, 1096--1156. \MR{3357180}

\bibitem[Ked22]{Kedlaya-book-2}
Kiran~Sridhara Kedlaya, \emph{{{\(p\)}}-adic differential equations (to
  appear)}, 2nd edition ed., Camb. Stud. Adv. Math., vol. 199, Cambridge:
  Cambridge University Press, 2022 (English).

\bibitem[KS17]{Kedlaya-Shiho}
Kiran~S. Kedlaya and Atsushi Shiho, \emph{Corrigendum to: ``{Local} and global
  structure of connections on nonarchimedean curves''}, Compos. Math.
  \textbf{153} (2017), no.~12, 2658--2665 (English).

\bibitem[Laz62]{Lazard}
Michel Lazard, \emph{Les z\'eros des fonctions analytiques d'une variable sur
  un corps valu\'e complet}, Inst. Hautes \'Etudes Sci. Publ. Math. (1962),
  no.~14, 47--75. \MR{0152519 (27 \#2497)}

\bibitem[Liu87]{Liu}
Qing Liu, \emph{Ouverts analytiques d'une courbe alg\'ebrique en g\'eom\'etrie
  rigide}, Ann. Inst. Fourier (Grenoble) \textbf{37} (1987), no.~3, 39--64.
  \MR{916273 (89c:14032)}

\bibitem[LS07]{Le-Stum-Book}
Bernard Le~Stum, \emph{Rigid cohomology}, Camb. Tracts Math., vol. 172,
  Cambridge: Cambridge University Press, 2007 (English).

\bibitem[Mal74]{Malgrange-Irreg}
Bernard Malgrange, \emph{Sur les points singuliers des \'equations
  diff\'erentielles}, Enseignement Math. (2) \textbf{20} (1974), 147--176.
  \MR{0368074 (51 \#4316)}

\bibitem[Man65]{Man}
Y.~I. Manin, \emph{Moduli fuchsiani}, Ann. Scuola Norm. Sup. Pisa (3)
  \textbf{19} (1965), 113--126.

\bibitem[Meb89]{Mebkhout-book}
Z.~Mebkhout, \emph{Syst{\`e}mes diff{\'e}rentiels. {Le} formalisme des six
  op{\'e}rations de {Grothendieck} pour les {{\(D_X\)}}-modules
  coh{\'e}rents.}, Trav. Cours, vol.~35, Paris: Hermann, 1989 (French).

\bibitem[Mon72]{Monsky-finiteness-alg}
Paul Monsky, \emph{Finiteness of {De} {Rham} cohomology}, Am. J. Math.
  \textbf{94} (1972), 237--245 (English).

\bibitem[Pon00]{Pons}
Emilie Pons, \emph{Non-solvable differential modules. {Radii} of convergence
  and indexes}, Rend. Semin. Mat. Univ. Padova \textbf{103} (2000), 21--45
  (French).

\bibitem[PP13a]{NP-III}
Andrea Pulita and Jérôme Poineau, \emph{{The convergence Newton polygon of a
  $p$-adic differential equation III : local and global decomposition
  theorems}}, arxiv, 2013, \url{http://arxiv.org/abs/1308.0859}, pp.~1--81.

\bibitem[PP13b]{NP-IV}
\bysame, \emph{{The convergence Newton polygon of a $p$-adic differential
  equation IV : controlling graphs}}, arxiv, 2013,
  \url{http://arxiv.org/abs/1308.0859?????}, pp.~1--81.

\bibitem[PP15]{NP-II}
J\'er\^ome {Poineau} and Andrea {Pulita}, \emph{{The convergence Newton polygon
  of a $p$-adic differential equation. II: Continuity and finiteness on
  Berkovich curves.}}, {Acta Math.} \textbf{214} (2015), no.~2, 357--393
  (English).

\bibitem[PP22a]{Banachoid}
J\'er\^ome Poineau and Andrea Pulita, \emph{Banachoid spaces},
  \url{https://www-fourier.ujf-grenoble.fr/~pulitaa/Publications/Banachoid.pdf},
  2022.

\bibitem[PP22b]{NP-V}
Andrea Pulita and Jérôme Poineau, \emph{The convergence newton polygon of a
  $p$-adic differential equation v : Local index theorems}, arxiv, 2022,
  \url{https://arxiv.org/abs/????}, pp.~1--95.

\bibitem[Pul07]{Rk1}
Andrea Pulita, \emph{Rank one solvable {$p$}-adic differential equations and
  finite abelian characters via {L}ubin-{T}ate groups}, Math. Ann. \textbf{337}
  (2007), no.~3, 489--555. \MR{MR2274542}

\bibitem[Pul14]{HDR}
\bysame, \emph{Equations diffrentielles p-adiques}, Mémoire d'habilitation à
  diriger des recherches, Université Montpellier II,
  \url{https://hal.archives-ouvertes.fr/tel-03667006}, 2014.

\bibitem[{Pul}15]{NP-I}
Andrea {Pulita}, \emph{{The convergence Newton polygon of a $p$-adic
  differential equation. I: Affinoid domains of the Berkovich affine line.}},
  {Acta Math.} \textbf{214} (2015), no.~2, 307--355 (English).

\bibitem[Pul16]{Inf-Def}
Andrea Pulita, \emph{Infinitesimal deformation of $p$-adic differential
  equations on berkovich curves.}, To appear in Math. Annalen (2016), 39 pages.

\bibitem[Ram78]{Ramis-Devissage-Gevrey}
J.-P. Ramis, \emph{D\'evissage {G}evrey}, Journ\'ees {S}inguli\`eres de {D}ijon
  ({U}niv. {D}ijon, {D}ijon, 1978), Ast\'erisque, vol.~59, Soc. Math. France,
  Paris, 1978, pp.~4, 173--204. \MR{542737 (81g:34010)}

\bibitem[Rob75]{Ro-I}
P.~Robba, \emph{On the index of {$p$}-adic differential operators. {I}}, Ann.
  of Math. (2) \textbf{101} (1975), 280--316. \MR{0364243 (51 \#498)}

\bibitem[Rob76]{Ro-II}
\bysame, \emph{On the index of {$p$}-adic differential operators. {II}}, Duke
  Math. J. \textbf{43} (1976), no.~1, 19--31. \MR{0389910 (52 \#10739)}

\bibitem[Rob84]{RoIII}
P.~Robba, \emph{Indice d'un op\'erateur diff\'erentiel {$p$}-adique. {III}.
  application to twisted exponential sums}, Asterisque \textbf{119-120} (1984),
  191--266.

\bibitem[Rob85]{RoIV}
\bysame, \emph{Indice d'un op\'erateur diff\'erentiel {$p$}-adique. {IV}. {C}as
  des syst\`emes. {M}esure de l'irr\'egularit\'e dans un disque}, Ann. Inst.
  Fourier (Grenoble) \textbf{35} (1985), no.~2, 13--55.

\bibitem[SB22]{Scholze-Bhatt}
Peter Scholze and Bhargav Bhatt, \emph{Prisms and prismatic cohomology}, 2022,
  \url{https://arxiv.org/abs/1905.08229}, pp.~1--125.

\bibitem[Sch02]{Schneider}
Peter Schneider, \emph{Nonarchimedean functional analysis}, Springer Monographs
  in Mathematics, Springer-Verlag, Berlin, 2002. \MR{1869547 (2003a:46106)}

\bibitem[Ser56]{GAGA}
Jean-Pierre Serre, \emph{Algebraic geometry and analytic geometry}, Ann. Inst.
  Fourier \textbf{6} (1956), 1--42 (French).

\bibitem[{Set}97]{CLark-2}
Minoru {Setoyanagi}, \emph{{Note on Clark's theorem for $p $-adic
  convergence.}}, {Proc. Am. Math. Soc.} \textbf{125} (1997), no.~3, 717--721
  (English).

\bibitem[Tsu98]{Tsu-swan}
Nobuo Tsuzuki, \emph{The local index and the {S}wan conductor}, Compositio
  Math. \textbf{111} (1998), no.~3, 245--288. \MR{MR1617130 (99g:14021)}

\bibitem[Tsu09]{Tsuzuki-rigide}
\bysame, \emph{Rigid cohomology}, Sugaku (Mathematics) \textbf{61} (2009),
  64--82.

\bibitem[Tur55]{Turritin}
H.~Turritin, \emph{Convergent solutions of ordinary homogeneous differential
  equations in the neighborhood of an irregular singular point}, Acta Math
  \textbf{93} (1955), 27--66.

\bibitem[vdPS03]{VS}
M.~van~der Put and M.~F. Singer, \emph{Galois theory of linear differential
  equations}, vol. 328, Springer-Verlag, Berlin, 2003.

\bibitem[You92]{Young}
Paul~Thomas Young, \emph{Radii of convergence and index for {$p$}-adic
  differential operators}, Trans. Amer. Math. Soc. \textbf{333} (1992), no.~2,
  769--785. \MR{1066451 (92m:12015)}

\end{thebibliography}

\end{document}